\documentclass{amsart}
\usepackage{amsmath,amsfonts,amsthm,amscd,amssymb}
\usepackage{enumerate,graphicx,oldgerm}
\usepackage[style=list]{glossary}
\usepackage[all]{xy}

\renewcommand{\glossaryname}{Glossary of notation}

\theoremstyle{plain}

\newtheorem{thm}{Theorem}[section]
\newtheorem{lem}[thm]{Lemma}
\newtheorem{prop}[thm]{Proposition}
\newtheorem{cor}[thm]{Corollary}
\newtheorem*{thm*}{Theorem}
\newtheorem*{rem*}{A posteriori remark}

\theoremstyle{definition}

\newtheorem{defn}[thm]{Definition}
\newtheorem*{NB}{Nota Bene}
\newtheorem{rem}[thm]{Remark}

\renewcommand{\t}{\mathfrak{t}}
\newcommand{\p}{\mathfrak{p}}
\newcommand{\s}{\mathfrak{s}}
\renewcommand{\ss}{\mathbf{s}}
\newcommand{\nn}{\mathfrak{n}}
\newcommand{\rr}{\mathfrak{r}}
\newcommand{\cA}{\mathcal{A}}
\newcommand{\cL}{\mathcal{L}}
\newcommand{\bS}{\mathbf{S}}

\newcommand{\cc}{\mathfrak{c}}
\DeclareMathOperator{\nf}{nf}
\DeclareMathOperator{\e}{exp}

\DeclareMathOperator{\dec}{dec}
\DeclareMathOperator{\nfd}{CD}
\DeclareMathOperator{\ext}{ext}
\DeclareMathOperator{\sol}{sol}
\DeclareMathOperator{\Aut}{Aut}
\DeclareMathOperator{\lef}{left}
\DeclareMathOperator{\rig}{right}
\DeclareMathOperator{\sign}{sign}
\DeclareMathOperator{\Hom}{Hom}
\DeclareMathOperator{\id}{id}
\DeclareMathOperator{\Ker}{Ker}
\DeclareMathOperator{\ncl}{ncl}
\DeclareMathOperator{\rel}{rel}
\DeclareMathOperator{\tp}{tp}
\DeclareMathOperator{\ET}{ET}
\DeclareMathOperator{\D}{D}
\DeclareMathOperator{\Sh}{Sh}
\DeclareMathOperator{\const}{const}
\DeclareMathOperator{\az}{alph}
\DeclareMathOperator{\edge}{\mathbf{e}}
\DeclareMathOperator{\comp}{comp}

\newcommand{\VV}{\mathfrak{V}}
\renewcommand{\AA}{\mathfrak{A}}
\newcommand{\BB}{\mathfrak{C}}
\newcommand{\PT}{\mathcal{PT}}
\newcommand{\TB}{\mathcal{T}}
\renewcommand{\P}{\mathcal{P}}

\newcommand{\kk}{\mathfrak{k}}
\newcommand{\m}{\mathfrak{m}}
\newcommand{\BC}{\mathcal{BC}}
\newcommand{\BS}{\mathcal{BS}}
\newcommand{\BD}{\mathcal{BD}}
\newcommand{\GE}{\mathcal{GE}}
\newcommand{\M}{\mathcal{M}}
\newcommand{\B}{\mathcal{B}}

\newcommand{\N}{\mathbb{N}}
\newcommand{\BR}{\mathbb{R}}
\newcommand{\Z}{\mathbb{Z}}
\newcommand{\factor}[2]{{\raise0.7ex\hbox{$#1$} \!\mathord{\left/ {\vphantom {#1 {#2}}}\right.\kern-\nulldelimiterspace}\!\lower0.7ex\hbox{${#2}$}}}
\newcommand{\GG}{\ensuremath{\mathbb{G}}}
\newcommand{\HH}{\ensuremath{\mathbb{H}}}
\newcommand{\FF}{\ensuremath{\mathbb{F}}}
\newcommand{\Tr}{\ensuremath{\mathbb{T}}}
\newcommand{\MM}{\ensuremath{\mathbb{M}}}
\newcommand{\ov}{\overline}
\newcommand{\BA}{\ensuremath{\mathbb{A}}}
\newcommand{\gpo}{\langle\Upsilon,\Re_\Upsilon\rangle}
\newcommand{\gpof}[1]{\langle\Upsilon_{#1},\Re_{\Upsilon_{#1}}\rangle}
\newcommand{\lra}{\leftrightarrows}

\title[Equations over Partially Commutative Groups]{On Systems of Equations over Free Partially Commutative Groups}

\author[M. Casals-Ruiz]{Montserrat Casals-Ruiz}\thanks{The first author is supported by Programa de Formaci\'{o}n
de Investigadores del Departamento de Educaci\'{o}n, Universidades e Investigaci\'{o}n del Gobierno Vasco}
\address{Montserrat Casals-Ruiz,
Department of Mathematics and Statistics, McGill University,
805 Sherbrooke St. West, Montreal,
Quebec H3A 2K6, Canada}
\email{casalsruiz@math.mcgill.ca}
\author[I. Kazachkov]{Ilya V. Kazachkov}\thanks{The second author is supported by Richard.~H.~Tomlinson Fellowship}
\address{Ilya V. Kazachkov,
Department of Mathematics and Statistics, McGill University,
805 Sherbrooke St. West, Montreal,
Quebec H3A 2K6, Canada}
\email{kazachkov@math.mcgill.ca}

\makeindex
\makeglossary
\subjclass{Primary 20F10;\\  Secondary 20F36}
\keywords{Equations in groups, partially commutative group, right-angled Artin group, Makanin-Razborov diagrams}
\begin{document}
\begin{abstract}
Using an analogue of Makanin-Razborov diagrams, we give an effective description of the solution set of systems of equations over a partially commutative group (right-angled Artin group) $\GG$. Equivalently, we give a parametrisation of $\Hom(G, \GG)$, where $G$ is a finitely generated group.
\end{abstract}
\maketitle
\begin{flushright}
\parbox{1.5in}{\textit{``Divide et impera''}\\ Principle of government of Roman Senate}
\end{flushright}
\setcounter{tocdepth}{2}
\tableofcontents

\section{Introduction}

Equations are one of the most natural algebraic objects; they have been studied at all times and in all branches of mathematics and lie at its core.  Though the first equations to be considered were equations over integers, now equations are considered over a multitude of other algebraic structures: rational numbers, complex numbers, fields, rings, etc. Since equations have always been an extremely important tool for a majority of branches of mathematics, their study developed into a separate, now classical, domain - algebraic geometry.

The algebraic structures we work with in this paper are groups. We would like to mention some important results in the area that established the necessary basics and techniques and that motivate the questions addressed in this paper. Though this historical introduction is several pages long, it is by no means complete and does not give a detailed account of the subject. We refer the reader to \cite{LS} and to \cite{LynEqns} and references there for an extensive survey of the early history of equations in groups.

\bigskip

It is hard to date the beginning of the theory of equations over groups, since even the classical word and conjugacy problems formulated by Max Dehn in 1911 can be interpreted as the compatibility problem for very simple equations.
Although the study of equations over groups now goes back almost a century, the foundations of a general theory of algebraic geometry over groups, analogous to the classical theory over commutative rings, were only set down relatively recently by G.~Baumslag, A.~Miasnikov and V.~Remeslennikov, \cite{BMR1}.

Given a system of equations, the two main questions one can ask are whether the system is compatible and whether one can describe its set of solutions. It is these two questions that we now address.

\subsubsection*{\textbf{Nilpotent and solvable groups}}
The solution of these problems is well-known for finitely generated abelian groups. V.~Roman'kov showed that the compatibility problem is undecidable for nilpotent groups, see \cite{romanN}. Furthermore, in \cite{repin1} N.~Repin refined the result of Roman'kov and showed that there exists a finitely presented nilpotent group for which the problem of recognising whether equations in one variable have solutions is undecidable. Later, in \cite{repin2}, the author showed that for every non-abelian free nilpotent group of sufficiently large nilpotency class the compatibility problem for one-variable equations in undecidable. The compatibility problem is also undecidable for free metabelian groups, see \cite{roman}.

Equations from the viewpoint of first-order logic are simply atomic formulas. Therefore, recognising if a system of equations and inequations over a group has a solution is equivalent to the decidability of the universal theory (or, which is equivalent, the existential theory) of this group. This is one of the reasons why often these two problems are intimately related. In general, both the compatibility problem and the problem of decidability of the universal theory are very deep. For instance, in \cite{romanNilp} Roman'kov showed that the decidability of the universal theory of a free nilpotent group of class 2 is equivalent to the decidability of the universal theory of the field of rational numbers - a long-standing problem which, in turn, is equivalent to the Diophantine problem over rational numbers. For solvable groups, O.~Chapuis in \cite{Chapuis} shows that if the universal theory of a free solvable group of class greater than or equal to 3 is decidable then so is the Diophantine problem over rational numbers.

\subsubsection*{\textbf{Free groups and generalisations}}
In the case of free groups, both the compatibility problem and the problem of describing the solution set were long-standing problems. In this direction, works of R.~Lyndon and A.~Malcev, which are precursors to the solution of these problems, are of special relevance.

\subsubsection*{One-variable equations}
One of the first types of equations to be considered was one-variable equations. In \cite{Lyndon1} R.~Lyndon solved the compatibility problem and gave a description of the set of all solutions of a single equation in one variable (over a free group). Lyndon proved that the set of solutions of a single equation can be defined by a finite system of ``parametric words''. These parametric words were complicated and the number of parameters on which they depended was restricted only by the type of each equation considered. Further results were obtained by K.~Appel \cite{Appel68} and A.Lorents \cite{Lorents63}, \cite{Lorents68}, who gave the exact form of the parametric words, and Lorents extended the results to systems of equations with one variable. Unfortunately Appel's published proof has a gap, see p.87 of \cite{BaumReviews}, and Lorents announced his results without proof. In the year 2000, I.~Chiswell and V.~Remeslennikov gave a complete argument, see \cite{ChRem}. Instead of giving a description in terms of parametric words, they described the structure of coordinate groups of irreducible algebraic sets (in terms of their JSJ-decompositions) and, thereby, using the basic theory of algebraic geometry over free groups, they obtained a description of the solution set (viewed as the set of homomorphisms from the coordinate group to a free group). Recently, D.~Bormotov, R.~Gilman and A.~Miasnikov in their paper \cite{BGM}, gave a polynomial time algorithm that produces a description of the solution set of one-variable equations.

The parametric description of solutions of one-variable equations gave rise to a conjecture that the solution set of any system of equations could be described by a finite system of parametric words. In 1968, \cite{Appel68}, Appel showed that there are equations in three and more variables that can not be defined by a finite set of parametric words. Therefore, the method of describing the solution set in terms of parametric words was shown to be limited.

\subsubsection*{Two-variable equations}
In an attempt to generalise the results obtained for one-variable equations, a more general approach, involving parametric functions, was suggested by Yu.~Hmelevski\u{i}. In his papers \cite{Hm1}, \cite{Hm2}, he gave a description of the solution set and proved the decidability of the compatibility problem of some systems of equations in two variables over a free group. This approach was later developed by Yu.~Ozhigov in \cite{Ozhig}, who gave a description of the solution set and proved the decidability of the compatibility problem for arbitrary equations in two variables.

However, it turned out that this method was not general either. In \cite{Razborov84}, A.~Razborov showed that there are equations whose set of solutions cannot be represented by a superposition of a finite number of parametric functions.

Recently, in \cite{Tou}, N.~Touikan, using the approach developed by Chiswell and Remeslennikov for one-variable equations, gave a description (in terms of the JSJ-decomposition) of coordinate groups of irreducible algebraic sets defined by a system of equations in two variables.

\subsubsection*{Quadratic equations}
Because of their connections to surfaces and automorphisms thereof, quadratic equations have been studied since the beginning of the theory of equations over groups.

The first quadratic equation to be studied was the commutator equation $[x,y]=[a,b]$, see \cite{Niel}. A description of the solution set of this equation was given in  \cite{Mal62} by A.~Malcev in terms of parametric words in automorphisms and minimal solutions (with respect to these automorphisms).

Malcev's powerful idea was later developed by L.~Commerford  and C.~Edmunds, see \cite{ComEd}, and by R.~Grigorchuk and P.~Kurchanov, see \cite{GrKur89}, into a general method of describing the set of solutions of all quadratic equations over a free group. A more geometric approach to this problem was given by M.~Culler, see \cite{Cul81} and A.~Olshanskii \cite{OlshSMJ}.

Quadratic equations and Malcev's idea to use automorphisms and minimal solutions play a key role in the modern approach to describing the solution set of arbitrary systems of equations over free groups.

Because of their importance, quadratic equations were considered over other groups. The decidability of compatibility problem for quadratic equations over one relator free product of groups was proved by A.~Duncan in \cite{Dunc} and over certain small cancellation groups by Commerford in \cite{Comerf}. In \cite{Lys}, I.~Lys\"{e}nok reduces the description of solutions of quadratic equations over certain small cancellation groups to the description of solutions of quadratic equations in free groups. Later, Grigorchuk and Lys\"{e}nok gave a description of the solution set of quadratic equations over hyperbolic groups, see \cite{GrLys}.

\subsubsection*{Arbitrary systems of equations}
A major breakthrough in the area was made by G.~Makanin in his papers \cite{Makanin} and \cite{Mak82}. In his work, Makanin devised a process for deciding whether or not an arbitrary system of equations $S$ over a free monoid (over a free group) is compatible. Later, using similar techniques, Makanin proved that the universal theory of a free group is decidable, see \cite{Mak84}. Makanin's result on decidability of the universal theory of a free group together with an important result of Yu.~Merzlyakov \cite{Merz} on quantifier elimination for positive formulae over free groups proves that the positive theory of a free group is decidable.

Makanin's ideas were developed in many directions. Remarkable progress was made by A.~Razborov. In his work \cite{Razborov1}, \cite{Razborov3}, Razborov refined Makanin's process and used automorphisms and minimal solutions to give a complete description of the set of solutions of an arbitrary system of equations over a free group in terms of, what is now called, Makanin-Razborov diagrams (or $\Hom$-diagrams). In their work \cite{KMIrc}, O.~Kharlampovich and A.~Miasnikov gave an important insight into Razborov's process and provided algebraic semantics for it. Using the process and having understanding of radicals of quadratic equations, the authors showed that the solution set of a system of equations can be canonically represented by the union of solution sets of a finite family of NTQ systems and gave an effective embedding of finitely generated fully residually free groups into coordinate groups of NTQ systems ($\omega$-residually free towers), thereby giving a characterisation of such groups. Then, using Bass-Serre theory, they proved that finitely generated fully residually free groups are finitely presented and that one can effectively find a cyclic splitting of such groups. Analogous results were proved by Z.~Sela using geometric techniques, see \cite{Sela1}. Later, Kharlampovich and Myasnikov in \cite{IFT} and Sela in \cite{Sela1}, developed Makanin-Razborov diagrams for systems of equations with parameters over a free group. These Makanin-Razborov diagrams encode the Makanin-Razborov diagrams of the systems of equations associated with each specialisation of the parameters. This construction plays a key role in a generalisation of Merzlyakov's theorem, in other words, in the proof of existence of Skolem functions for certain types of formulae (for NTQ systems of equations), see \cite{IFT}, \cite{Sela2}.

These results are an important piece of the solution of two well-known conjectures formulated by A.~Tarski around 1945: the first of them states that the elementary theories of non-abelian free groups of different ranks coincide; and the second one states that the elementary theory of a free group is decidable. These problems were recently solved by O.~Kharlampovich and A.~Miasnikov in \cite{KhMTar} and the first one was independently solved by Z.~Sela in \cite{SelaTar}.

In another direction, Makanin's result (on decidability of equations over a free monoid) was developed by K.Schulz, see \cite{Schulz}, who proved that the compatibility problem of equations with regular constraints over a free monoid is decidable. Later V.~Diekert, C.~Guti\'{e}rrez and C.~Hagenah, see \cite{DGH}, reduced the compatibility problem of systems of equations over a free group with rational constraints to compatibility problem of equations with regular constraints over a free monoid.

Since then, one of the most successful methods of proving the decidability of the compatibility problem for groups (monoids) has been to reduce it to the compatibility problem over a free group (monoid) with constraints. The reduction of compatibility problem for torsion-free hyperbolic groups to free groups was made by E.~Rips and Z.~Sela in \cite{RipsSela}; for relatively hyperbolic groups with virtually abelian parabolic subgroups by F.~Dahmani in \cite{Dahmani}; for partially commutative monoids by Yu.~Matiasevich in \cite{Mat} (see also \cite{DMM}); for partially commutative groups by V.~Diekert and A.~Muscholl in \cite{DM}; for graph products of groups by V.~Diekert and M.~Lohrey in \cite{DL}; and for HNN-extensions with finite associated subgroups and for amalgamated products with finite amalgamated subgroups by M.~Lohrey and G.~S\'{e}nizergues in \cite{LohrSeni}.

The complexity of Makanin's algorithm has received a lot of attention. The best result about arbitrary systems of equations over monoids is due to W.~Plandowski. In a series of two papers \cite{Pl1, Pl2} he gave a new approach to the compatibility problem of systems of equations over a free monoid and showed that this problem is in PSPACE. This approach was further extended by Diekert, Guti\'{e}rrez and Hagenah, see \cite{DGH} to systems of equations over free groups. Recently, O.~Kharlampovich, I.~Lys\"{e}nok, A.~Myasnikov and N.~Touikan have shown that solving quadratic equations over free groups is NP-complete, \cite{KhLMT}.

Another important development of the ideas of Makanin is due to E.~Rips and is now known as the Rips' machine. In his work, Rips interprets Makanin's algorithm in terms of partial isometries of real intervals, which leads him to a classification theorem of finitely generated groups that act freely on $\BR$-trees. A complete proof of Rips' theorem was given by D.~Gaboriau, G.~Levitt, and F.~Paulin, see \cite{GLP}, and, independently, by M.~Bestvina and M.~Feighn, see \cite{BF}, who also generalised Rips' result to give a classification theorem of groups that have a stable action on $\BR$-trees. Recently, F.~Dahmani and V.~Guirardel announced the decidability of the compatibility problem for virtually free groups generalising Rips' machine.

The existence of analogues of Makanin-Razborov diagrams has been proven for different groups. In \cite{SelaHyp}, for torsion-free hyperbolic groups by  Z.~Sela; in \cite{groves}, for torsion-free relatively hyperbolic groups relative to free abelian subgroups by D.~Groves; for fully residually free (or limit) groups in \cite{EJSJ}, by O.~Kharlampovich and A.~Miasnikov and in \cite{Alibeg}, by E.~Alibegovi\'{c}.

\subsubsection*{\textbf{Other results}}
Other well-known problems and results were studied in relation to equations in groups. In analogy to Galois theory, the problem of adjoining roots of equations was considered in the first half of the twentieth century leading to the classical construction of an HNN-extension of a group $G$  introduced by B.~Higman, B.~Neumann and H.~Neumann in \cite{HNN49} as a construction of an extension of a group $G$ in which the ``conjugacy'' equation $x^{-1}gx=g'$ is compatible. Another example is the characterisation of the structure of elements $g$ for which an equation of the form $w=g$ is compatible, also known as the endomorphism problem. A well known result in this direction is the solution of the endomorphism problem for the commutator equation $[x,y]=g$ in a free group obtained in \cite{Wicks} by M.~Wicks in terms of, what now are called, Wicks forms. Wicks forms were generalised to free products of groups and to higher genus by A.~Vdovina in \cite{Vdov}, to hyperbolic groups by S.~Fulthorp in \cite{Fulth} and to partially commutative groups by S.~Shestakov, see \cite{Shest}. The endomorphism problem for many other types of equations over free groups was studied by P.~Schupp, see \cite{Schupp}, C.~Edmunds, see \cite{Edm1}, \cite{Edm2} and was generalised to free products of groups by L.~Comerford and C.~Edmunds, see \cite{ComEdm}.

\bigskip

We now focus on equations over partially commutative groups. Recall that a partially commutative group $\GG$ is a group given by a presentation of the form $\langle a_1, \dots, a_\rr \mid  R \rangle$, where $R$ is a subset of the set $\{ [a_i, a_j] \mid i,j=1, \dots, \rr, i \neq j\}$. This means that the only relations imposed on the generators are commutation of some of the generators. In particular, both free abelian groups and free groups are partially commutative groups.

As we have mentioned above, given a system of equations, the two main questions one can ask are whether the system is compatible and whether one can describe its set of solutions. It is known that the compatibility problem for systems of equations over partially commutative groups is decidable, see \cite{DM}. Moreover, the universal  (existential) and positive theories of partially commutative groups are also decidable, see \cite{DL} and \cite{CK1}. But, on the other hand, until now there were not even any partial results known about the description of the solution sets of systems of equations over partially commutative groups.

Nevertheless, in the case of  partially commutative groups, other problems, involving particular equations, have been studied. We would like to mention two papers by S.~Shestakov, see \cite{Shest} and \cite{Shest2}, where the author solves the endomorphism problem for the equations $[x,y]=g$ and $x^2y^2=g$ correspondingly, and a result of J.~Crisp and B.~Wiest from \cite{crw} stating that partially commutative groups satisfy  Vaught's conjecture, i.e. that if a tuple $(X,Y,Z)$ of elements from $\GG$ is a solution of the equation $x^2 y^2 z^2=1$, then $X$, $Y$ and $Z$ pairwise commute.

In this paper, we effectively construct an analogue of Makanin-Razborov diagrams and use it to give a description of the solution set of systems of equations over partially commutative groups. It seems to the authors that the importance of the work presented in this paper lies not only in
the construction itself but in the fact that it enables consideration of analogues of the (numerous) consequences of the classical Makanin-Razborov process in more general circumstances.

The classes of groups for which Makanin-Razborov diagrams have been constructed so far are generalisations of free groups with a common feature: all of them are CSA-groups (see Lemma 2.5 in \cite{groves1}). Recall that a group is called CSA if every maximal abelian subgroup is malnormal, see \cite{MyasExpo2}, or, equivalently, a group is CSA if every non-trivial abelian subgroup is contained in a unique maximal abelian subgroup. In Proposition 9 in \cite{MyasExpo2}, it is proved that if a group is CSA, then it is commutative transitive (the commutativity relation is transitive) and thus the centralisers of its (non-trivial) elements are abelian, it is directly indecomposable, has no non-abelian normal subgroups, has trivial centre, has no non-abelian solvable subgroups and has no non-abelian Baumslag-Solitar subgroups. This shows that the CSA property imposes strong restrictions on the structure of the group and, especially, on the structure of the centralisers of its elements. Therefore, numerous classes of groups, even geometric, are not CSA.

The CSA property is important in the constructions of analogues of Makanin-Razborov diagrams constructed before. The fact that partially commutative groups are \emph{not} CSA, conceptually, shows that this property is not essential for constructing Makanin-Razborov diagrams and that the approach developed in this paper opens the possibility for other groups to be taken in consideration: graph products of groups, particular types of HNN-extensions (amalgamated products), partially commutative-by-finite, fully residually partially commutative groups,  particular types of small cancellation groups, torsion-free relatively hyperbolic groups, some torsion-free CAT(0) groups and more.

On the other hand, as mentioned above, Schulz proved that Makanin's process to decide the compatibility problem of equations carries over to systems of equations over a free monoid with regular constraints. In our construction, we show that Razborov's results can be generalised to systems of equations (over a free monoid) with constraints characterised by certain axioms. Therefore, two natural problems arise: to understand for which classes of groups the description of solutions of systems of equations reduces to this setting, and to understand which other constraints can be considered in the construction of Makanin-Razborov diagrams.

In another direction, the Makanin-Razborov process is one of the main ingredients in the effective construction of the JSJ-decomposition for fully residually free groups and in the characterisation of finitely generated fully residually free groups as given by Kharlampovich and Myasnikov in \cite{EJSJ} and \cite{KMIrc}, correspondingly. Therefore, the process we construct in this paper may be one of the techniques one can use to understand the structure of finitely generated fully residually partially commutative groups, or, which in this case is equivalent, of finitely generated residually partially commutative groups, and thus, in particular, the structure of all finitely generated residually free groups.

In the case of free groups, Rips' machine leads to a classification theorem of finitely generated groups that have a stable action on $\BR$-trees. Therefore, our process may be useful for understanding finitely generated groups that act in a certain way (that can be axiomatised) on an $\BR$-tree. A more distant hope is that the process may lead to a classification of finitely generated groups that act freely on a (nice) CAT(0) space.

The structure of subgroups of partially commutative groups is very complex, see \cite{Wise1}, and some of the subgroups exhibit odd finiteness properties, see \cite{BB}. Nevertheless, recent results on partially commutative groups suggest that the attempts to generalise some of the well-known results for free groups have been, at least to some extent, successful. We would like to mention here the recent progress of R.~Charney and K.~Vogtmann on  the outer space of partially commutative groups, of M.~Day on the generalisation of peak-reduction algorithm (Whitehead's theorem) for partially commutative groups, and of a number of authors on the structure of automorphism groups of partially commutative groups, see \cite{Day1}, \cite{Day2}, \cite{CCV}, \cite{GPR}, \cite{DKR}. This and the current paper makes us optimistic about the future of this emerging area.

\medskip

To get a global vision of the long process we describe in this paper, we would like to begin by a brief comparison of the method of solving equations over a free monoid (as the reader will see the setting we reduce the study of systems of equations over partially commutative groups is rather similar to it) to Gaussian elimination in linear algebra. Though, technically very different, we want to point out that the general guidelines of both methods have a number of common features. Given a system of $\m$ linear equations in $k$ variables over a field $K$, the algorithm in linear algebra firstly encodes this system as an $\m \times (k+1)$ matrix with entries in $K$. Then it uses elementary Gauss transformations of matrices (row permutation, multiplication of a row by a scalar, etc) in a particular order to take the matrix to a row-echelon form, and then it produces a description of the solution set of the system of linear equations with the associated matrix in row-echelon form. Furthermore, the algorithm has the following property: every solution of the (system of equations corresponding to the) matrix in row-echelon form gives rise to a solution of the original system and, conversely, every solution of the original system is induced by a solution of the (system of equations corresponding to the) matrix in row-echelon form. Let us compare this algorithm with its group-theoretic counterpart, a variation of which we present in this paper.

Given a system of equations over a free monoid, we introduce a combinatorial object - a generalised equation, and establish a correspondence between systems of equations over a free monoid and generalised equations (for the purposes of this introduction, the reader may think of generalised equations as just systems of equations over a monoid). Graphic representations of generalised equations will play the role of matrices in linear algebra.

We then define certain transformations of generalised equations, see Sections \ref{se:5.1} and \ref{se:5.2half}. One of the differences is that in the case of systems of linear equations, applying an elementary Gauss transformation to a matrix one obtains a \emph{single} matrix. In our case, applying some of the (elementary and derived) transformations to a generalised equation one gets a finite family of generalised equations. Therefore, the method we describe here, results in an oriented rooted tree of generalised equations instead of a sequence of matrices.

A case-based algorithm, described in Section \ref{se:5.2}, applies transformations in a specific order to generalised equations. The branch of the algorithm terminates, when the generalised equation obtained is ``simple'' (there is a formal definition of being ``simple''). The algorithm then produces a description of the set of solutions of ``simple'' generalised equations.

One of the main differences is that a branch of the procedure described may be \emph{infinite}. Using particular automorphisms of coordinate groups of generalised equations as parameters, one can prove that a finite tree is sufficient to give a description of the solution set. Thus, the parametrisation of the solutions set will be given using a finite tree, recursive groups of automorphisms of finitely many coordinate groups of generalised equations and solutions of ``simple'' generalised equations.

We now briefly describe the content of each of the sections of the paper.

In Section \ref{sec:prelim} we establish the basics we will need throughout the paper.

The goal of Section \ref{sec:red} is to present the set of solutions of a system of equations over a partially commutative group as a union of sets of solutions of finitely many constrained generalised equations over a free monoid. The family of solutions of the collection constructed of  generalised equations describes all solutions of the initial system over a partially commutative group. The term ``constrained generalised equation over a free monoid (partially commutative monoid)'' can be misleading, since their solutions are not tuples of elements from a free monoid (partially commutative monoid), but tuples of reduced elements of the partially commutative group, that satisfy the equalities in a free monoid (partially commutative monoid).

This reduction is performed in two steps. Firstly, we use an analogue of the notion of a cancellation tree for free groups  to reduce the study of systems of equations over a partially commutative group to the study of constrained generalised equations over a partially commutative monoid. We show that van Kampen diagrams over partially commutative groups can be taken to a ``standard form'' and therefore the set of solutions of a given system of equations over a partially commutative group defines only finitely many types of van Kampen diagram in standard form, i.e. finitely many different cancellation schemes. For each of these cancellation schemes, we then construct a constrained generalised equation over a partially commutative monoid.

We further show that to a given generalised equation over a partially commutative monoid one can associate a finite collection of (constrained) generalised equations over a free monoid. The family of solutions of the generalised equations from this collection describes all solutions of the initial generalised equation over a partially commutative monoid. This reduction relies on the ideas of Yu.~Matiyasevich, see \cite{Mat}, V.~Diekert and A.~Muscholl, see \cite{DM}, that state that there are only finitely many ways to take a product of words in the trace monoid (written in normal form) into normal form. We apply these results to reduce the study of solutions of generalised equations over a partially commutative monoid to the study of solutions of constrained generalised equations over a free monoid.

In Section \ref{se:5} in order to describe the solution set of a constrained generalised equation over a free monoid we describe a branching rewriting process for constrained generalised equations. For a given generalised equation, this branching process results in a locally finite, possibly infinite, oriented rooted tree. The vertices of this tree are labelled by (constrained) generalised equations over a free monoid. Its edges are labelled by epimorphisms of the corresponding coordinate groups. Moreover, for every solution $H$ of the original generalised equation, there exists a path in the tree from the root vertex to another vertex $v$ and a solution $H^{(v)}$ of the generalised equation corresponding to $v$ such that the solution $H$ is a composition of the epimorphisms corresponding to the edges in the tree and the solution $H^{(v)}$.  Conversely, to any path from the root to a vertex $v$ in the tree, and any solution $H^{(v)}$ of the generalised equation labelling $v$, there corresponds a solution of the original generalised equation.

The tree described is, in general, infinite. In Lemma \ref{3.2} we give a characterisation of the three types of infinite branches that it may contain: namely linear, quadratic and general. The aim of the remainder of the paper is, basically, to define the automorphism groups of coordinate groups that are used in the parametrisation of the solution sets and to prove that, using these automorphisms all solutions can be described by a finite tree.

Sections \ref{sec:minsol} and \ref{sec:periodstr} contain technical results used in Section \ref{se:5.3}.

In Section \ref{sec:minsol} we use automorphisms to introduce a reflexive, transitive relation on the set of solutions of a generalised equation. We use this relation to introduce the notion of a minimal solution and describe the behaviour of minimal solutions with respect to transformations of generalised equations.

In Section \ref{sec:periodstr} we introduce the notion of periodic structure. Informally, a periodic structure is an object one associates to a generalised equation that has a solution that contains subwords of the form $P^k$ for arbitrary large $k$. The aim of this section is two-fold. On one hand, to understand the structure of the coordinate group of such a generalised equation and to define a certain finitely generated subgroup of its automorphism group. On the other hand, to prove that, using automorphisms from the automorphism group described, either one can bound the exponent of periodicity $k$ or one can obtain a solution of a proper generalised equation.

Section \ref{se:5.3} contains the core of the proof of the main results. In this section we deal with the three types of infinite branches and construct a finite oriented rooted subtree $T_0$ of the infinite tree described above. This tree has the property that for every leaf either one can trivially describe the solution set of the generalised equation assigned to it, or from every solution of the generalised equation associated to it, one gets a solution of a proper generalised equation using automorphisms defined in Section \ref{sec:periodstr}.

The strategy for the linear and quadratic branches is common: firstly, using a combinatorial argument, we prove that in such infinite branches a generalised equation repeats thereby giving rise to an automorphism of the coordinate group. Then, we show that such automorphisms are contained in a recursive group of automorphisms. Finally, we prove that minimal solutions with respect to this recursive group of automorphisms factor through sub-branches of bounded height.

The treatment for the general branch is more complex. We show that using automorphisms defined for quadratic branches and automorphisms defined in Section \ref{sec:periodstr}, one can take any solution to a solution of a proper generalised equation or to a solution of bounded length.

In Section \ref{5.5.5} we prove that the number of proper generalised equations through which the solutions of the leaves of the finite tree $T_0$ factor is finite and this allows us to extend $T_0$ and obtain a tree $T_{\dec}$ with the property that for every leaf either one can trivially describe the solution set of the generalised equation assigned to it, or the edge with an end in the leaf is labelled by a proper epimorphism. Since partially commutative groups are equationally Noetherian and thus any sequence of proper epimorphisms of coordinate groups is finite, an inductive argument, given in Section \ref{se:5.5'}, shows that we can construct a tree $T_{\ext}$ with the property that for every leaf one can trivially describe the solution set of the generalised equation assigned to it.

In the last section we construct a tree $T_{\sol}$ as an extension of the tree $T_{\ext}$ with the property that for every leaf the coordinate group of the generalised equation associated to it is a fully residually $\GG$ partially commutative group and one can trivially describe its solution set.

The following theorem summarises one of the main results of the paper.

\begin{thm*}
Let $\GG$ be the free partially commutative group with the underlying commutation graph $\mathcal{G}$ and let ${G}$ be a finitely generated {\rm(}$\GG$-{\rm)}group. Then the set of all {\rm(}$\GG$-{\rm)}homomorphisms  $\Hom({G},\GG)$ {\rm(}$\Hom_\GG({G},\GG)$, correspondingly{\rm)} from ${G}$ to $\GG$ can be effectively described by a finite rooted tree. This tree is oriented from the root, all its vertices except for the root and the leaves are labelled by coordinate groups of generalised equations. The leaves of the tree are labelled by $\GG$-fully residually $\GG$ partially commutative groups $\GG_{w_i}$ {\rm(}described in Section {\rm\ref{se:5.5}}{\rm)}.

Edges from the root vertex correspond to a finite number of {\rm(}$\GG$-{\rm)}homomorphisms from ${G}$ into coordinate groups of generalised equations.  To each vertex group we assign a group of automorphisms. Each edge {\rm(}except for the edges from the root and the edges to the final leaves{\rm)} in this tree is labelled by an epimorphism, and all the epimorphisms are proper. Every {\rm(}$\GG$-{\rm)}homomorphism from ${G}$ to $\GG$ can be written as a composition of the {\rm(}$\GG$-{\rm)}homomorphisms corresponding to the edges, automorphisms of the groups assigned to the vertices, and a {\rm(}$\GG$-{\rm)}homomorphism $\psi=(\psi_j)_{j\in J}$, $|J|\le 2^\rr$  into $\GG$, where $\psi_j:\HH_j[Y]\to \HH_j$ and $\HH_j$ is the free partially commutative subgroup of $\GG$ defined by some full subgraph of $\mathcal{G}$.
\end{thm*}

\begin{rem*}
In his work on systems of equations over a free group \cite{Razborov1}, \cite{Razborov3}, A.~Razborov uses a result of J.~McCool on automorphism group of a free group {\rm(}that the stabiliser of a finite set of words is finitely presented{\rm)}, see {\rm \cite{McCool}}, to prove that the automorphism groups of the coordinate groups associated to the vertices of the tree $T_{\sol}$ are finitely generated. When this paper was already written M.~Day published two preprints, {\rm\cite{Day1}} and {\rm\cite{Day2}} on automorphism groups of partially commutative groups. The authors believe that using the results of this paper and techniques developed by M.~Day, one can prove that the automorphism groups used in this paper are also finitely generated.
\end{rem*}

\section{Preliminaries}\label{sec:prelim}

\subsection{Graphs and relations}

In this section we introduce notation for graphs that we use in this paper. Let $\Gamma=(V(\Gamma),E(\Gamma))$ be an oriented graph, where $V(\Gamma)$ is the set of vertices of $\Gamma$ and $E(\Gamma)$ is the set of edges of $\Gamma$. If an edge $e:v\to v'$ has \emph{origin} $v$ and \emph{terminus} $v'$, we sometimes write $e=v\to v'$. We always denote the paths in a graph by letters $\p$ and $\s$, and cycles by the letter $\cc$. To indicate that a path $\p$ begins at a vertex $v$ and ends at a vertex $v'$ we write $\p(v,v')$. If not stated otherwise, we assume that all paths we consider are simple. For a path $\p(v,v')=e_1\dots e_l$ by ${(\p(v,v'))}^{-1}$ we denote the reverse (if it exists) of the path $\p(v,v')$, that is a path $\p'$ from $v'$ to $v$, $\p'= e_l^{-1}\dots e_1^{-1}$, where $e_i^{-1}$ is the inverse of the edge $e_i$, $i=1,\dots l$.

Usually, the edges of the graph are labelled by certain words or letters. The \emph{label} of a path $\p=e_1\dots e_l$ is the concatenation of labels of the edges $e_1\dots e_l$.

Let $\Gamma$ be an oriented rooted tree, with the root at a vertex $v_0$ and such that for any vertex $v\in V(\Gamma)$ there exists a unique path $\p(v_0,v)$ from $v_0$ to $v$. The length of this path from $v_0$ to $v$ is called the \index{height!of a vertex}\emph{height of the vertex $v$}. The number $\max\limits_{v\in V(\Gamma)} \{\hbox{height of } v\}$ is called the \index{height!of a tree}\emph{height of the tree $\Gamma$}. We say that a vertex $v$ of height $l$ is \emph{above} a vertex $v'$ of height $l'$ if and only if $l>l'$ and there exists a path of length $l-l'$ from $v'$ to $v$.

\bigskip

Let $S$ be an arbitrary finite set. Let $\Re$ be a symmetric binary relation on $S$, i.e. $\Re\subseteq S\times S$ and if $(s_1,s_2)\in \Re$ then $(s_2,s_1)\in \Re$. Let $s\in S$, then by $\Re(s)$ we denote the following set: \glossary{name={$\Re(x)$},description={set of elements related with $x$},sort=R}
$$
\Re(s)=\{s_1\in S\mid \Re(s,s_1)\}.
$$

\subsection{Elements of algebraic geometry over groups}
\label{se:2-4}

In \cite{BMR1} G.~Baumslag, A.~Miasnikov and V.~Remeslennikov lay down the foundations of algebraic geometry over groups and introduce group-theoretic counterparts of basic notions from algebraic geometry over fields. We now recall some of the basics of algebraic geometry over groups. We refer to \cite{BMR1} for details.

Algebraic geometry over groups centers around the notion of a \index{group@$G$-group}\emph{$G$-group}, where $G$ is a fixed group generated by a set $A$. These $G$-groups can be likened to algebras over a unitary commutative ring, more specially a field, with $G$ playing the role of the coefficient ring. We therefore, shall consider the category of $G$-groups, i.e. groups which contain a designated subgroup isomorphic to the group $G$. If $H$ and $K$ are $G$-groups then a homomorphism $\varphi: H \rightarrow K$ is a \index{homomorphism@$G$-homomorphism}\emph{$G$-homomorphism} if $\varphi(g)= g$ for every $g \in G$. In the category of $G$-groups morphisms are $G$-homomorphisms; subgroups are \index{subgroup@$G$-subgroup}$G$-subgroups, etc. By \glossary{name={$\Hom_G(H,K)$}, description={set of $G$-homomorphisms from $H$ to $K$}, sort=H} $\Hom_G(H,K)$ we denote the set of all $G$-homomorphisms from $H$ into $K$. A $G$-group $H$ is termed \index{group@$G$-group!finitely generated}{\em finitely generated $G$-group} if there exists a finite subset $C \subset H$ such that the set $G \cup C$ generates $H$. It is not hard to see that the free product $G \ast F(X)$ is a free object in the category of $G$-groups, where $F(X)$ is the free group with basis $X = \{x_1, x_2, \ldots,  x_n\}$. This group is called the free $G$-group with basis $X$, and we denote it  by $G[X]$.

For any element $s\in G[X]$ the formal equality $s=1$ can be treated, in an obvious way, as an \index{equation over a group}\emph{equation} over $G$. In general, for a subset  $S \subset G[X]$ the formal equality $S=1$ can be treated as \index{system of equations over a group}{\em a system of equations} over $G$ with coefficients in $A$. In other words, every equation is a word in $(X \cup A)^{\pm 1}$. Elements from $X$ are called {\em variables}, and elements from $A^{\pm 1}$ are called {\em coefficients} or {\em constants}. To emphasize this we sometimes write $S(X,A) = 1$.

A \index{solution!of a system of equations} {\em solution} $U$ of the system $S(X) = 1$ over a group $G$ is a tuple of elements $g_1, \ldots, g_n \in G$ such that every equation from $S$ vanishes at $(g_1, \ldots, g_n)$, i.e. $S_i(g_1, \ldots, g_n)=1$ in $G$, for all $S_i\in S$. Equivalently, a solution $U$ of the system $S = 1$ over $G$ is a $G$-homomorphism \glossary{name={$\pi_U$}, description={homomorphism defined by the solution $U$}, sort=P}$\pi_U: G[X] \to G$ induced by the map $\pi_U: x_i\mapsto g_i$ such that $S\subseteq \ker(\pi_U)$. When no confusion arises, we abuse the notation and write $U(w)$, where $w\in G[X]$, instead of $\pi_U(w)$.

Denote by \glossary{name=$\ncl\langle S\rangle$, description={normal closure of $S$}, sort=N}$\ncl\langle S\rangle$ the normal closure of $S$ in $G[X]$. Then every solution of $S(X) = 1$ in $G$ gives rise to a $G$-homomorphism $\factor{G[X]}{\ncl\langle S\rangle )} \to G$, and vice-versa. The set of all solutions over $G$ of the system $S=1$ is denoted by \glossary{name={$V_G(S)$},description={algebraic set defined by $S$ over $G$}, sort=V}$V_G(S)$ and is called the \index{algebraic set}{\em algebraic set defined by} $S$.

For every system of equations $S$ we set the \index{radical!of a system}{\em radical of the system $S$}  to be the following subgroup of  $G[X]$:\glossary{name=$R(S)$, description={radical of the system $S$}, sort=R}
$$
R(S) = \left\{ T(X) \in G[X] \ \mid \ \forall g_1,\dots,\forall g_n  \left( S(g_1,\dots,g_n) = 1 \rightarrow T(g_1,\dots, g_n) = 1\right) \right\}.
$$
It is easy to see that  $R(S)$ is a normal subgroup of $G[X]$ that contains $S$. There is a one-to-one correspondence between algebraic sets $V_G(S)$ and radical subgroups $R(S)$ of $G[X]$. This correspondence is described in Lemma \ref{lem:rad} below. Notice that if $V_G(S) = \emptyset$, then $R(S) = G[X]$.

It follows from the definition that
$$
R(S)=\bigcap\limits_{U\in V_G(S)}\ker(\pi_U).
$$
In the lemma below we summarise the relations between radicals, systems of equations and algebraic sets. Note the similarity of these relations with the ones in algebraic geometry over fields, see \cite{Eisenbud}.

For a subset $Y\subseteq G^n$ define the \index{radical!of a set}\emph{radical of $Y$} to be \glossary{name={$R(Y)$}, description={radical of a set $Y$}, sort=R}
$$
R(Y) = \left\{ T(X) \in G[X] \ \mid \ T(g_1,\dots,g_n)=1 \hbox{ for all } (g_1,\dots,g_n)\in Y \right\}.
$$

\begin{lem} \label{lem:rad}\
\begin{enumerate}
    \item The radical of a system $S \subseteq G[X]$ contains the normal closure $\ncl \langle S \rangle$ of $S$.
    \item Let\ $Y_1$ and $Y_2$  be  subsets of $G^n$ and $S_1$, $S_2$   subsets of $G[X]$.
If $Y_1  \subseteq Y_2$  then $R(Y_1) \supseteq R(Y_2)$. If $S_1  \subseteq S_2$ then $R(S_1) \subseteq R(S_2)$.
    \item For any family of sets $\left\{Y_i \mid i \in I \right\}$, $Y_i \subseteq G^n $, we have
$$
R\left(\bigcup\limits_{i \in I} {Y_i }\right) = \bigcap\limits_{i \in I} R(Y_i ).
$$
    \item A normal subgroup $H$ of the group $G[X]$ is the radical of an algebraic set over $G$ if and only if $ R(V_G(H)) = H$.
    \item A set $Y \subseteq G^n$ is algebraic over $G$ if and only if $ V_G (R(Y)) = Y$.
    \item Let $Y_1 ,Y_2  \subseteq G^n $ be two algebraic sets, then
$$
Y_1  = Y_2 \hbox{ \textit{if and only if} }  R(Y_1 ) = R(Y_2 ).
$$
Therefore the radical of an algebraic set describes it uniquely.
\end{enumerate}
\end{lem}

The quotient group \glossary{name=$G_{R(S)}$, description={coordinate group of the system $S$ over $G$}, sort=G}
$$
G_{R(S)}=\factor{G[X]}{R(S)}
$$
is called the \index{coordinate group!of an algebraic set}\index{coordinate group!of a system of equations}{\em coordinate group} of the algebraic set  $V_G(S)$ (or of the system $S$). There exists a one-to-one correspondence between the algebraic sets and  coordinate groups $G_{R(S)}$. More formally, the categories of algebraic sets and coordinate groups are equivalent, see Theorem 4, \cite{BMR1}.

A $G$-group $H$ is called \index{equationally Noetherian@$G$-equationally Noetherian group}{\em $G$-equationally Noetherian} if every system $S(X) = 1$ with coefficients from $G$ is equivalent over $G$ to a finite subsystem $S_0 = 1$, where $S_0 \subset S$, i.e. the system $S$ and its subsystem $S_0$ define the same algebraic set. If $G$ is $G$-equationally Noetherian, then we say that $G$ is equationally Noetherian. If $G$ is equationally Noetherian then the Zariski topology over $G^n$ is {\em Noetherian} for every $n$, i.e., every proper descending chain of closed sets in $G^n$ is finite. This implies that every algebraic set $V$ in $G^n$ is a finite union of irreducible subsets, called \index{irreducible!component}{\em irreducible components} of $V$, and such a decomposition of $V$ is unique. Recall that a closed subset $V$ is \index{irreducible!algebraic set}{\em irreducible} if it is not a union of two proper closed (in the induced topology) subsets.

\medskip

If $V_G(S) \subseteq G^n$ and $V_G(S') \subseteq G^m$ are algebraic sets, then a map $\phi: V_G(S) \to V_G(S')$ is a
\index{morphism of algebraic sets}\emph{morphism} of algebraic sets if there exist $f_1,\dots,f_m \in G[x_1,\ldots ,x_n]$ such that, for any $(g_1,\ldots,g_n) \in V_G(S)$,
$$
\phi(g_1,\dots,g_n)=(f_1(g_1,\dots,g_n),\ldots,f_m(g_1,\dots,g_n)) \in V_G(S').
$$
Occasionally we refer to morphisms of algebraic sets as \emph{word mappings}.

Algebraic sets $V_G(S)$ and $V_G(S')$ are called \index{isomorphism of algebraic sets}\emph{isomorphic} if there exist morphisms $\psi: V_G(S) \rightarrow V_G(S')$  and $\phi: V_G(S') \rightarrow V_G(S)$ such that $\phi \psi  = \id_{V_G(S)}$ and $\psi \phi = \id_{V_G(S')}$.

Two systems of equations $S$ and $S'$ over $G$ are called \index{equivalence!of systems of equations}\emph{equivalent} if the algebraic sets $V_G(S)$ and  $V_G(S')$ are isomorphic.

\begin{prop}
Let $G$ be a group and let $V_G(S)$ and $V_G(S')$ be two algebraic sets over $G$. Then the algebraic sets $V_G(S)$ and $V_G(S')$ are isomorphic if and only if the coordinate groups $G_R(S)$ and $G_{R(S')}$ are $G$-isomorphic.
\end{prop}

The notions of an equation, system of equation, solution of an equation, algebraic set defined by a system of equations, morphism between algebraic sets and equivalent systems of equations are categorical in nature. These notion carry over, in an obvious way, from the case of groups to the case of monoids.

\subsection{Formulas in the languages $\cL_{A}$ and $\cL_{G}$}
\label{se:2-3}

In this section we recall some basic notions of first-order logic and model theory. We refer the reader to \cite{ChKe} for details.

Let $G$ be a group generated by the set $A$. The standard first-order language of group theory, which we denote by \glossary{name=$\cL$, description={first-order language of groups}, sort=L}$\cL$, consists of a symbol for multiplication $\cdot$, a symbol for inversion $^{-1}$, and a symbol for the identity $1$.  To deal with $G$-groups, we have to enrich the language $\cL$ by all the elements from $G$ as constants. In fact, as $G$ is generated by $A$, it suffices to enrich the language $\cL$ by the constants that correspond to elements of $A$, i.e. for every element of $a\in A$ we introduce a new constant $c_a$. We denote language $\cL$ enriched by constants from $A$ by \glossary{name={$\cL_A$}, description={first-order language of groups enriched by constants from $A$}, sort=L}$\cL_{A}$, and by constants from $G$ by \glossary{name={$\cL_G$}, description={first-order language of groups enriched by constants from $G$}, sort=L}$\cL_G$. In this section we further consider only the language $\cL_A$, but everything stated below carries over to the case of the language $\cL_G$.

A group word in variables $X$ and constants $A$ is a word $W(X,A)$ in the alphabet $(X\cup A)^{\pm 1}$. One  may consider the word $W(X,A)$ as a term in the language $\cL_A$. Observe that every term in the language $\cL_A$ is  equivalent modulo the axioms of group theory to a group word in variables $X$ and constants $A$. An \index{formula!atomic}{\em atomic formula}  in the language $\cL_A$ is a formula of the type $W(X,A) = 1$, where $W(X,A)$ is a group word in $X$ and $A$.  We interpret atomic formulas in $\cL_A$ as equations over $G$, and vice versa.  A {\em Boolean combination} of atomic formulas in the language $\cL_A$ is a disjunction of conjunctions of atomic formulas and their negations. Thus every Boolean combination $\Phi$  of atomic  formulas in $\cL_A$ can be written in the form $\Phi =  \bigvee\limits_{i=1}^m\Psi_i$, where each $\Psi_i$ has one of following form:
$$
\bigwedge\limits_{j = 1}^n(S_j(X,A) = 1),  \hbox{ or } \bigwedge\limits_{j =1}^n(T_j(X,A) \neq 1),  \hbox{ or }
\bigwedge\limits_{j = 1}^n (S_j(X,A) = 1) \ \wedge \  \bigwedge_{k = 1}^m (T_k (X,A) \neq 1).
$$

It follows from general results on disjunctive normal forms in propositional logic that every quantifier-free formula in the language $\cL_\cA$ is logically equivalent (modulo the axioms of group theory) to a Boolean combination of  atomic ones. Moreover, every formula $\Phi$ in $\cL_A$ with \emph{free variables} $Z=\{z_1,\ldots ,z_k\}$  is logically equivalent to a formula of the type
$$
Q_1x_1 Q_2 x_2 \ldots Q_l x_l \Psi(X,Z,A),
$$
where  $Q_i \in \{\forall, \exists \}$, and  $\Psi(X,Z,A)$ is a Boolean combination of atomic formulas in variables from $X \cup Z$.
Introducing fictitious quantifiers, if necessary, one can always rewrite the formula $\Phi$ in the form
$$
\Phi(Z)  = \forall x_1 \exists y_1 \ldots \forall x_k \exists y_k \Psi(x_1, y_1, \ldots, x_k, y_k, Z).
$$

A first-order formula $\Phi$ is called a \index{sentence}\emph{sentence}, if $\Phi$ does not contain free variables.

A sentence $\Phi$ is called \index{formula!universal}\emph{universal} if and only if $\Phi$ is equivalent to a formula of the type:
$$
\forall x_1 \forall x_2 \ldots \forall x_l \Psi(X,A),
$$
where $\Psi(X,A)$ is a Boolean combination of atomic formulas in variables from $X$. We sometimes refer to universal sentences as to universal formulas.

A \index{quasi identity}\emph{quasi identity} in the language $\cL_A$  is a universal formula of the form
\[
\forall x_1 \cdots\forall x_l \left(\mathop  \bigwedge \limits_{i = 1}^m \;r_i (X,A) = 1 \rightarrow s(X, A) = 1\right),
\]
where $r_i (X,A)$ and $S(X,A)$ are terms.

\subsection{First order logic and algebraic geometry}

The connection between algebraic geometry over groups and logic has been shown to be very deep and fruitful and, in particular, led to a solution of the Tarski's problems on the elementary theory of free group, see \cite{KhMTar}, \cite{SelaTar}.

In \cite{Rem} and \cite{AG2} A.~Myasnikov and V.~Remeslennikov established relations between universal classes of groups, algebraic geometry and residual properties of groups, see Theorems \ref{thm:coordgr} and \ref{thm:ircoordgr} below. We refer the reader to \cite{AG2} and \cite{K} for proofs.

In order to state these theorems, we shall make us of the following notions. Let $H$ and $K$ be $G$-groups.  We say that a family of $G$-homomorphisms ${\mathcal F} \subset \Hom_G(H,K)$ {\em $G$-separates} ({\em $G$-discriminates}) $H$ into $K$ if for every non-trivial element $h \in H$ (every finite set of non-trivial elements $H_0 \subset H$) there exists $\phi \in {\mathcal F}$ such that $h^\phi \ne 1$ ($h^\phi \neq 1$ for every $h \in H_0$). In this case we say that $H$ is \index{separated@$G$-separated group}{\em $G$-separated} by $K$ or that $H$ is \index{residually@$G$-residually $G$ group}\emph{$G$-residually $K$} (\index{discriminated@$G$-discriminated group}{\em $G$-discriminated} by $K$ or that $H$ is \index{fully residually@$G$-fully residually $G$ group}\emph{$G$-fully residually $K$}). In the case that $G=1$, we simply say that $H$ is separated (discriminated) by $K$.

\begin{thm} \label{thm:coordgr}
Let $G$ be an equationally Noetherian {\rm(}$G$-{\rm)}group. Then the following classes coincide:
\begin{itemize}
    \item the class of all coordinate groups of algebraic sets over $G$ {\rm(}defined by systems of equations with coefficients in $G${\rm)};
    \item the class of all finitely generated {\rm(}$G$-{\rm)}groups that are {\rm(}$G$-{\rm)}separated by $G$;
    \item the class of all finitely generated {\rm(}$G$-{\rm)}groups that satisfy all the quasi-identities {\rm(}in the language $\cL_G$ {\rm(}or $\cL_A${\rm))} that are satisfied by $G$;
    \item the class of all finitely generated {\rm(}$G$-{\rm)}groups from the {\rm(}$G$-{\rm)}prevariety generated by $G$.
\end{itemize}
Furthermore, a coordinate group of an algebraic set $V_G(S)$ is {\rm($G$-)}separated by $G$ by homomorphisms $\pi_U$, $U\in V_G(S)$, corresponding to solutions.
\end{thm}

\begin{thm} \label{thm:ircoordgr}
Let $G$ be an equationally Noetherian {\rm(}$G$-{\rm)}group. Then the following classes coincide:
\begin{itemize}
    \item the class of all coordinate group of irreducible algebraic sets over $G$ {\rm(}defined by systems of equations with coefficients in $G${\rm)};
    \item the class of all finitely generated {\rm(}$G$-{\rm)}groups that are $G$-discriminated by $G$;
    \item the class of all finitely generated {\rm(}$G$-{\rm)}groups that satisfy all universal sentences {\rm(}in the language $\cL_G$ {\rm(}or $\cL_A${\rm))} that are satisfied by $G$.
\end{itemize}
Furthermore, a coordinate group of an irreducible algebraic set $V_G(S)$ is {\rm($G$-)}discriminated by $G$ by homomorphisms $\pi_U$, $U\in V_G(S)$, corresponding to solutions.
\end{thm}

\subsection{Partially commutative groups} \label{sec:pcgr}

Partially commutative groups are widely studied in different branches of mathematics and computer science, which explains the variety of names they were given: \index{graph group}\emph{graph groups}, \index{right-angled Artin group}\emph{right-angled Artin groups}, \emph{semifree groups}, etc. Without trying to give an account of the literature and results in the field we refer the reader to a recent survey \cite{charney} and to the introduction and references in \cite{DKRpar}.

Recall that a (free) \index{partially commutative group} \emph{partially commutative} group is defined as follows. Let \glossary{name={$\mathcal{G}$}, description={commutation graph of a partially commutative group}, sort=G}$\mathcal{G}$ be a finite, undirected, simplicial graph. Let $\cA = V(\mathcal{G}) = \{a_1, \dots , a_\rr\}$ be the set of vertices of $\mathcal{G}$ and let $F(\cA)$ be the free group on $\cA$. Let
$$
R = \{[a_i, a_j] \in F(\cA) \mid a_i, a_j \in \cA\hbox{ and there is an edge of $\mathcal{G}$ joining $a_i$ to $a_j$}\}.
$$
The partially commutative group corresponding to the \index{graph!commutation, of a partially commutative group} (commutation) graph $\mathcal{G}$ is the  group \glossary{name={$\GG(\mathcal{G})$}, description={partially commutative group with underlying commutation graph $\mathcal{G}$}, sort=G}$\GG(\mathcal{G})$ with presentation $\langle \cA \mid R\rangle$.  This means that the only relations imposed on the generators are commutation of some of the generators. When the underlying graph is clear from the context we write simply \glossary{name={$\GG$}, description={partially commutative group}, sort=G}$\GG$.

From now on \glossary{name={$\cA$}, description={a finite alphabet, the generating set of $\GG$}, sort=A}\glossary{name={$\rr$}, description={the cardinality of $\cA$}, sort=R}$\cA=  \left\{ a_1, \ldots , a_\rr \right\}$ always stands for a finite alphabet and its elements are called \index{letter}\emph{letters}. We reserve the term \index{occurrence}\emph{occurrence} to mean an occurrence of a letter or of its formal inverse in a word. In a more formal way, an occurrence is a pair (letter (or its inverse), its placeholder in the word).

Let $\mathcal{G}_1$ be a full subgraph of $\mathcal{G}$ and let $\mathcal{A}_1$ be its set of vertices. In \cite{EKR} it is shown that $\GG(\mathcal{G}_1)$ is the subgroup  of the group $\GG$ generated by $\mathcal{A}_1$, $\GG(\mathcal{G}_1)=\langle \mathcal{A}_1\rangle$. Following \cite{DKRpar}, we call $\GG(\mathcal{G}_1)= \GG(\mathcal{A}_1)$ a \index{canonical parabolic subgroup}\emph{canonical parabolic subgroup} of $\GG$.

We denote the \index{length}length of a word $w$ by \glossary{name={$|w|$}, description={length of the word $w\in \GG$}, sort=W}$|w|$. For a word $w \in \GG$, we denote by \glossary{name={$\ov{w}$}, description={a geodesic of a word $w$, $w\in \GG$}, sort=W} $\ov{w}$ a \index{geodesic}\emph{geodesic} of $w$. Naturally, $|\ov{w}|$ is called the length of an element $w\in \GG$. An element $w\in \GG$ is called \index{cyclically reduced element}\emph{cyclically reduced} if the length of $\ov{w^2}$ is twice the length of $\ov{w}$ or, equivalently, the length of $w$ is minimal in the conjugacy class of $w$.

Let $u$ be a geodesic word in $\GG$. We say that $u$ is a left\/ {\rm (}right{\rm )} \index{divisor of an element}\emph{divisor} of an element $w$, $w\in \GG$ if there exists a geodesic word $v\in \GG$ and a geodesic word $\ov w$, $\ov w=w$, such that $\ov{w}= u v$ {\rm (}$\ov w=  v u$, respectively{\rm )} and $|\ov{w}|=|u|+|v|$. In this case, we also say that $u$ \index{left-divide}\emph{left-divides} $w$ (\index{right-divide}\emph{right-divides} $w$, respectively). Let $u,v$ be geodesic words in $\GG$. If $|\ov{uv}|=|u|+|v|$ we sometimes write  \glossary{name={`$u\circ v$'}, description={no cancellation in the product of $u$ and $v$}, sort=Z}$u\circ v$ to stress that there is no cancellation between $u$ and $v$.

For a given word $w$ denote by \glossary{name={$\az(w)$}, description={the set of letters that occur in $w$}, sort=A} $\az(w)$ the set of letters occurring in $w$. For a  word $w\in \GG$ define \glossary{name={$\BA(w)$}, description={the subgroup of $\GG$ generated by all letters that do not occur in $\ov w$ and commute with $w$}, sort=A}$\BA(w)$ to be the subgroup of $\GG$ generated by all letters that do not occur in $\ov w$ and commute with $w$.  The subgroup $\BA(w)$ is well-defined (independent of the choice of a geodesic $\ov w$), see \cite{EKR}. Let $v,w\in \GG$ be so that $[v,w]=1$ and $\az(v)\cap \az(w)=1$, or, which is equivalent, $v\in \BA(w)$ and $w\in \BA(v)$. In this case we write \glossary{name={`$\lra$'}, description={disjoint commutation of elements or sets}, sort=Z}$v\lra w$.

Let $A, B\subseteq \GG$ be arbitrary subsets of $\GG$. We denote by $[A,B]$ the set \glossary{name={$[A,B]$}, description={the set of commutators of elements of sets $A$ and $B$}, sort=C} $[A,B]=\left\{[a,b]\mid a\in A, b\in B\right\}$ (not to confuse with the more usual notation $[A,B]$ for the subgroup generated by the set $\left\{[a,b]\mid a\in A, b\in B\right\}$). Naturally, the notation \glossary{name={$[A,B]=1$}, description={elements of sets $A$ and $B$ commute pairwise}, sort=C}$[A,B]=1$  means that $[a,b]=1$ for all $[a,b]\in [A,B]$. Analogously, given a set of words $W\subseteq \GG$ we denote by \glossary{name={$\az(W)$}, description={the set of letters that occur in a word $w\in W$}, sort=A}$\az(W)$ the set of letters that occur in a word $w\in W$, i.e. $\az(W)=\bigcup\limits_{w\in W} \az(w)$, and by \glossary{name={$\BA(W)$}, description={the subgroup of $\GG$ generated by all letters that do not occur in $\ov w$ and commute with $w$ for every $w\in W$}, sort=A}$\BA(W)$ the intersection of the subgroups $\BA(w)$ for all $w\in W$, i.e. $\BA(W)=\bigcup\limits_{w\in W} \BA(w)$. Similarly, we write $A\lra B$ whenever $\az(A)\cap\az(B)=1$ and $[A,B]=1$.

We say that a word $w\in \GG$ is written in the \index{normal form!lexicographical}\emph{lexicographical normal form} if the word $w$ is geodesic and minimal with respect to the lexicographical ordering induced by the ordering
$$
a_1<a_2\dots<a_\rr<a_\rr^{-1}<\dots<a_1^{-1}
$$
of the set $\cA\cup \cA^{-1}$. Note that if $w$ is written in the lexicographical normal form, the word $w^{-1}$ is not necessarily written in the lexicographical normal form, see Remark 32, \cite{DM}.

Henceforth, by the symbol \glossary{name={`$\doteq$'}, description={graphical equality of words}, sort=Z}`$\doteq$' we denote graphical equality of words, i.e. equality in the free monoid.

For a partially commutative  group $\GG$ consider its \index{graph!non-commutation, of a partially commutative group}non-commutation graph $\Delta$. The vertex set $V$ of $\Delta$ is the set of generators $\cA$ of $\GG$. There is an edge connecting $a_i$ and $a_j$ if and only if $\left[a_i, a_j \right] \ne 1$. Note that the graph $\Delta$ is the complement graph the graph $\Gamma$. The graph $\Delta$ is a union of its connected components $I_1, \ldots , I_k$. In the above notation
\begin{equation} \label{eq:decomp}
\GG= \GG(I_1) \times \cdots \times \GG(I_k).
\end{equation}

Consider $w \in \GG$ and the set $\az(w)$. For this set, just as above, consider the graph $\Delta (\az(w))$ (it is a full subgraph of $\Delta$). This graph can be either connected or not. If it is connected we will call $w$ a \index{block}\emph{block}. If $\Delta(\az(w))$ is not connected, then we can split $w$ into the product of commuting words
\begin{equation} \label{eq:bl}
w= w_{j_1} \cdot w_{j_2} \cdots w_{j_t};\ j_1, \dots, j_t \in J,
\end{equation}
where $|J|$ is the number of connected components of $\Delta(\az(w))$ and the word $w_{j_i}$ is a word in the letters from the $j_i$-th connected component. Clearly, the words $\{w_{j_1}, \dots, w_{j_t}\}$ pairwise commute. Each word $w_{j_i}$, $i \in {1, \dots,t}$ is a block and so we refer to presentation (\ref{eq:bl}) as the block decomposition of $w$.

An element $w\in \GG$ is called a least root (or simply, root) of $v\in \GG$ if there exists an integer $0\ne m\in \Z$ such that $v=w^m$ and there does not exists $w'\in \GG$ and $0\ne m'\in \Z$ such that $w={w'}^{m'}$. In this case we write $w=\sqrt{v}$. By a result from \cite{DK}, partially commutative groups have least roots, that is the root element of $v$ is defined uniquely.

The next result describes centralisers of elements in partially commutative groups.

\begin{thm}[Centraliser Theorem, Theorem 3.10, \cite{DK}, see also \cite{Servatius}] \label{thm:centr} \
Let $w\in \GG$ be a cyclically reduced word and $w=v_1\dots v_k$ be its block decomposition. Then, the centraliser of $w$ is the following subgroup of $\GG$:
\begin{equation} \notag
C(w)=\langle \sqrt{v_1}\rangle \times \cdots \times \langle \sqrt{v_k} \rangle\times \BA(w).
\end{equation}
\end{thm}

\begin{cor} \label{cor:centr}
  For any $w\in \GG$ the centraliser $C(w)$ of $w$ is an isolated subgroup of $\GG$, i.e. $C(w)=C(\sqrt{w})$.
\end{cor}

In Section \ref{se:4-1} we shall need an analogue of the definition of a cancellation scheme for a free group. We gave a description of cancellation schemes for partially commutative groups in \cite{CK1}. We shall use the following result.
\begin{prop}[Lemma 3.2, \cite{CK1}] \label{lem:prod}
Let $\GG$ be a partially commutative group and let $w_1,\dots w_k$ be geodesic words  in $\GG$ such that  $w_1\cdots w_k=1$. Then, there exist geodesic words $w_i^j$, $1\le i,j\le k$ such that for any $1\le l\le k$ there exists the following geodesic presentation for $w_l$:
$$
w_l=w_l^{l-1} \cdots w_l^1 w_l^k\cdots w_l^{l+1},
$$
where $w_l^i={w_i^l}^{-1}$.
\end{prop}

This result is illustrated in Figure \ref{pic:8}. In fact, Proposition \ref{lem:prod}  states that there exists a normal form for van Kampen diagrams over partially commutative groups and that, structurally, there are only finitely many possible normal forms of van Kampen diagrams for the product $w_1\cdots w_k=1$ corresponding to the different decompositions of the word $w_i$ as a product of the \emph{non-trivial} words $w_i^j$. We usually consider van Kampen diagrams in  normal forms and think of them in terms of their structure. It is, therefore, natural to refer to van Kampen diagrams in  normal forms as to \emph{cancellation schemes}. We refer the reader to \cite{CK1} for more details.

\begin{figure}[!h]
  \centering
   \includegraphics[keepaspectratio,width=6in]{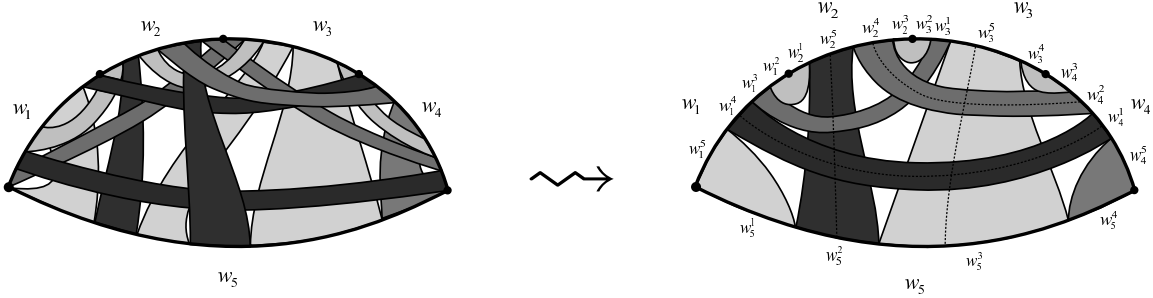}
\caption{Cancellation in a product of geodesic words $w_1w_2w_3w_4w_5=1$.} \label{pic:8}
\end{figure}

Figure \ref{pic:8} (on the right), shows a cancellation scheme for the product of words $w_1w_2w_3w_4w_5=1$. All cancellation schemes for $w_1w_2w_3w_4w_5=1$ can be constructed from the one shown on Figure \ref{pic:8} (on the right) by making some of the bands trivial.

In the last section of this paper we use Proposition \ref{prop:soltree} to prove that certain partially commutative $\GG$-groups are $\GG$-fully residually $\GG$. Proposition \ref{prop:soltree} is a generalisation of a result from \cite{CK1} that states that if $\GG$ is a non-abelian, directly indecomposable partially commutative group, then $\GG[X]$ is $\GG$-fully residually $\GG$. The next two lemmas were used in the proof of this result in \cite{CK1} and are necessary for our proof of Proposition \ref{prop:soltree}.

\begin{lem}[Lemma 4.11, \cite{CK1}] \label{lem:discr} \
There exists an integer $N=N(\GG)$ such that the following statements hold.
\begin{enumerate}
\item \label{it:lemdis1} Let $b\in \GG$ be a cyclically reduced block and let $z\in \GG$ be so that $b^{-1}$ does not left-divide and right-divide $z$. Then one has $b^{N+1}zb^{N+1}=b\circ b^{N} z b^{N} \circ b$.

\item \label{it:lemdis2}Let $b\in \GG$ be a cyclically reduced block and let $z=z_1^{-1}z_2z_1$, where $z_2$ is cyclically reduced. Suppose that $b$ does not left-divide $z$, $z^{-1}$, $z_2$ and $z_2^{-1}$, and $[b,z]\ne 1$. Then one has $z^{b^{N+1}}= b \circ z^{b^{N}} \circ b^{-1}$.
\end{enumerate}
\end{lem}

\begin{rem}
In \cite{CK1} we prove that $N$ is a linear function of the centraliser dimension of the group $\GG$, see \cite{DKR1} for definition.
\end{rem}

\begin{lem}[Lemma 4.17, \cite{CK1}] \label{lem:4.17}
Let $b\in \GG$ be a cyclically reduced block element and let $w_1,w_2\in \GG$ be geodesic words of the form
$$
w_1=b^{\delta_ 1} \circ g_1 \circ b^{\epsilon}, \quad w_2=b^{\epsilon} \circ g_2 \circ b^{\delta_2}, \hbox{ where } \epsilon, \delta_1, \delta_2= \pm 1.
$$
Then the geodesic word $\ov{w_1w_2}$  has the form $\ov{w_1w_2}=b^{\delta_1} \circ w_3\circ b^{\delta_2}$.
\end{lem}

\begin{cor} \label{cor:thm417}
Let $\GG$ be a partially commutative group and let $1\ne w\in \GG[x]$, $w=x^{k_1}g_1\cdots g_{l-1} x^{k_l}g_{l}$, where $g_i, \in \GG$, $g_1, \dots,g_l\ne 1$. Suppose that there exists a cyclically reduced block element $b\in \GG$ such that $[b,g_i]\ne 1$, $i=1,\dots, l-1$. Then there exists a positive integer $N$ such that for all $n>N$ the homomorphism $\varphi_{b,n}$ induced by the map $x\mapsto b^n$ maps the word $w$ to a non-trivial element of $\GG$.
\end{cor}
\begin{proof}
Let $g_i=g_{i,1}^{-1}g_{i,2}g_{i,1}$ be the cyclic decomposition of $g_i$. Taking $b'$ to be a large enough power of the block element $b$, we may assume that $b'$ does not left-divide the elements $g_{i,1}^{\pm 1}$, $g_{i,2}^{\pm 1}$, $i=1,\dots, l$.

Note that $b'$ satisfies the assumptions of Lemma \ref{lem:discr}.

Consider the homomorphism $\varphi_{b,2n}:\GG[x]\to \GG$, defined by $x\mapsto b^{2n}$.  Then
$$
\varphi_{b,2n}(w)=b^{k_1n}\left(b^{k_1n}g_1b^{k_2n}\right) \cdot \left(b^{k_2n}g_2b^{k_2n}\right) \cdots \left(b^{k_{l-1}n}g_{l-1}b^{k_ln}\right)\cdot b^{k_ln} g_l
$$
By Lemma \ref{lem:discr}, there exists $N$ such that if $n\ge N$, then every factor of $\varphi_{b,2n}(w)$ of the form $\left(b^{k_in}g_ib^{k_{i+1}n}\right)$ has the form $b^{\sign(k_i)}\circ \tilde{g_i}\circ b^{\sign(k_{i+1})}$. The statement now follows from Lemma \ref{lem:4.17}.
\end{proof}

\begin{prop} \label{prop:soltree}
Let $\GG$ be a partially commutative group and let $\HH$ be a directly indecomposable canonical parabolic subgroup of $\GG$. Then the group
$$
\GG'=\langle \GG, t\mid \rel(\GG), [t,C_\GG(\HH)]=1\rangle,
$$
where $C_\GG(\HH)$ is the centraliser of $\HH$ in $\GG$, is a $\GG$-discriminated by $\GG$ partially commutative group.
\end{prop}
\begin{proof}
Firstly, note that by Theorem \ref{thm:centr}, the subgroup  $C_\GG(\HH)$ is a canonical parabolic subgroup of the group $\GG$. It follows that the group $\GG'$ is a partially commutative group.

Consider an element $g\in \GG'$ written in the normal form induced by an ordering on $\cA^{\pm 1} \cup \{ t \}^{\pm 1}$, so that $t^{\pm 1}$ precedes any element $a \in \cA^{\pm 1}$,
$$
g=g_1t^{k_1}\cdots g_lt^{k_l} g_{l+1},
$$
where $[g_j,t]\ne 1$, $j=1,\dots, l$, $g_j\in\GG$. Since $[g_j,t]\ne 1$, we get that $g_i \notin C_{\GG}(\HH)$. Fix a block element $b$ from $\HH$, such that $[b,g_j]\ne 1$ for all $j=1,\dots,l$. Clearly, such $b$ exists, since $\HH$ is directly indecomposable. If we treat the word $g$ as an element of $\GG[t]$, then  by Corollary \ref{cor:thm417}, there exists a positive integer $M$ such that for all $n\ge M$, the family of homomorphisms $\varphi_{b,n}$ induced by the map $t\mapsto b^{n}$, maps $g$ to a non-trivial element of $\GG$.

Since, by the choice of $b$, the element $\varphi_{b,n}(t)$ belongs to $\HH$, it follows that the relations $[C_\GG(\HH),\varphi_{b,n}(t)]=1$ are satisfied. Therefore, the family of homomorphisms $\{\varphi_{b,n}\}$ induces a $\GG$-approximating family $\{\tilde{\varphi}_{b,n}\}$ of homomorphisms from the group $\GG'$ to $\GG$, see the diagram below.
$$
\xymatrix{
 \GG[t]  \ar[rd] \ar[dd]_{\varphi_{b,n}}   & \\
                                                     & \GG'\ar[ld]^{\tilde{\varphi}_{b,n}}\\
\GG &
}
$$
Considering, instead of one element $g$, a finite family of elements from $\GG'$ and choosing the block element $b$ in an analogous way, we get that the group $\GG'$ is $\GG$-discriminated by $\GG$.
\end{proof}

The property of a group to be equationally Noetherian plays an important role in algebraic geometry over groups, see Theorems \ref{thm:coordgr} and \ref{thm:ircoordgr}. It is known that every linear group (over a commutative, Noetherian, unitary ring) is equationally Noetherian (see \cite{Gub}, \cite{Br}, \cite{BMR1}). Partially commutative groups are linear, see \cite{Heqn}, hence equationally Noetherian.

\bigskip

In Section \ref{se:5.5} we shall use the notion of a  \index{graph product of groups}\emph{graph product of groups}. The idea of a graph product, introduced in \cite{graphpr}, is a generalisation of the concept of a partially commutative group. Let $G_1,\dots, G_k$, $G_i=\langle X_i\mid R_i\rangle$, $i=1,\dots, k$ be groups.  Let $\mathcal{G}'=(V(\mathcal{G}'), E(\mathcal{G}'))$ be a finite, undirected, simplicial graph, $V(\mathcal{G}')=\{v_1,\dots,v_k\}$.

A graph product $G=G_{\mathcal{G}'}=G_\Gamma(G_1,\dots,G_k)$ of the groups $G_1,\dots, G_k$ with respect to the graph $\mathcal{G}'$, is a group with a presentation of the form
$$
\langle X_1,\dots, X_k\mid R_1,\dots, R_k, \mathcal{R}\rangle,
$$
where $\mathcal{R}=\{[X_i,X_j]\mid \hbox{$v_i$ and $v_j$ are adjacent in $\mathcal{G}'$}\}$.

\subsection{Partially commutative monoids and DM-normal forms}

The aim of this section is to describe a normal form for elements of a partially commutative monoid. For our purposes, it is essential that the normal form be invariant with respect to inversion. As mentioned above, natural normal forms, such as the lexicographical normal form or the normal form arising from the bi-automatic structure on $\GG$, see \cite{VanWyk}, do not have this property and hence can not be used. In \cite{DM} V.~Diekert and A.~Muscholl specially designed a normal form that has this property. We now define this normal form and refer the reader to \cite{DM} for details.

Let $\GG$ be a partially commutative group given by the presentation $\langle \cA\mid R\rangle$. Let \glossary{name={$\FF$, $\FF(\cA^{\pm 1})$}, description={free monoid on the alphabet $\cA\cup \cA^{-1}$}, sort=F}$\FF=\FF(\cA^{\pm 1})$ be the free monoid on the alphabet $\cA\cup \cA^{-1}$ and let \glossary{name={$\Tr$, $\Tr(\cA^{\pm 1})$}, description={partially commutative monoid monoid on the alphabet $\cA\cup \cA^{-1}$ corresponding to the group $\GG$}, sort=T}$\Tr=\Tr(\cA^{\pm 1})$ be the partially commutative monoid with involution given by the presentation:
$$
\Tr(\cA^{\pm 1})=\langle \cA\cup \cA^{-1}\mid R_\Tr \rangle, \hbox{ where} \left[ a_i^{\epsilon}, a_j^\delta \right]\in R_\Tr \hbox{ if and only if } [a_i,a_j]\in R, \ \epsilon,\delta\in \{-1,1\}.
$$
The involution on $\Tr$ is induced by the operation of inversion in $\GG$ and does not have fixed points. We refer to it as to the \emph{inversion} in $\Tr$ and denote it by $^{-1}$.

Following \cite{DM}, we call the maximal subset $C=\mathcal{C}\cup\mathcal{C}^{-1}$ of $\cA \cup \cA^{-1}$ such that $[a,c]\notin R_\Tr$ if and only if $[b,c]\notin R_\Tr$ for all $a,b\in  C$ and $c\in \cA^{\pm 1}$, a \index{clan}\emph{clan}. A clan $C$ is called \index{clan!thin} \emph{thin} if there exist $a\in C$ and $b\in \cA^{\pm 1}\setminus C$ such that $[a,b]\in R_\Tr$ and is called \index{clan!thick} \emph{thick} otherwise. It follows that there is at most one thick clan and that the number of thin clans never equals 1.

It is convenient to encode an element of the partially commutative monoid as a finite labelled acyclic oriented graph $[V,E,\lambda]$, where $V$ is the set of vertices, $E$ is the set of edges and $\lambda:V\to \cA^{\pm 1}$ is the labelling. Such a graph induces a labelled partial order $[V, E^*,\lambda]$. For an element $w\in \Tr$, $w=b_{1}\cdots b_{n}$, $b_{i}\in \cA^{\pm 1}$, we introduce the graph $[V,E, \lambda]$ as follows. The set of vertices of $[V,E, \lambda]$ is in one-to-one correspondence with the letters of $w$, $V=\{1,\dots, n\}$. For the vertex $j$ we set $\lambda(j)=b_{j}$. We define an edge from $b_{i}$ to $b_{j}$ if and only if both $i<j$ and $[b_{i},b_{j}]\notin R_\Tr$. The graph $[V,E, \lambda]$ thereby obtained is called the \index{dependence graph}\emph{dependence graph} of $w$. Up to isomorphism, the dependence graph of $w$ is unique, and so is its induced labelled partial order, which we further denote by $[V, \le, \lambda]$.

Let $c_{1}<\dots<c_{q}$ be the linearly ordered subset of $[V,\le,\lambda]$ containing all vertices with label in the clan $C$. For the vertex $v\in V$, we define the \index{source point} \emph{source point} $s(v)$ and  and the \index{target point}\emph{target point} $t(v)$ as follows:
$$
s(v)=\sup\{i\mid c_i\le v\}, \quad t(v)=\inf \{i \mid v \le c_i\}.
$$
By convention $\sup\emptyset=0$ and $\inf\emptyset=q+1$. Thus, $0\le s(v)\le q$, $1\le t(v)\le q+1$ and $s(v)\le t(v)$ for all $v\in V$. Note that we have $s(v)=t(v)$ if and only if the label of $v$ belongs to $C$.

For $0\le s\le t\le q+1$, we define the \index{median position}\emph{median position} $m(s,t)$. For $s=t$ we let $m(s,t)=s$. For $s<t$, by Lemma 1 in \cite{DM}, there exist unique $l$  and $k$ such that $s\le l< t$, $k\ge 0$ and
$$
c_{s+1}\dots c_l\in \FF(\mathcal{C})(\mathcal{C}^{-1}\FF(\mathcal{C}))^k, \quad c_{l+1}\dots c_{t-1}\in (\FF(\mathcal{C}^{-1})\mathcal{C})^k \FF({\mathcal{C}^{-1}}).
$$
Then we define $m(s,t)=l+\frac{1}{2}$ and we call $m(s,t)$ the median position. Define the \index{global position} \emph{global position} of $v\in V$ to be $g(v)=m(s(v), t(v))$.

We define the normal form $\nf(w)$ of an element $w\in \Tr$ by introducing new edges into the dependence graph $[V, E, \lambda]$ of $w$. Let $u,v\in V$ be such that $\lambda(v)\in C$ and $[\lambda(u),\lambda(v)]\in R_\Tr$. We define a new edge from $u$ to $v$ if $g(u)<g(v)$, otherwise we define a new edge from $v$ to $u$. The new dependence graph $[V, \hat{E}, \lambda]$ defines a unique element of the trace monoid $\hat{\Tr}$, where $\hat{\Tr}$ is obtained from $\Tr$ by omitting the commutativity relations of the form $[c,a]$ for any $c\in C$ and any $a\in \cA^{\pm 1}$. Note that the number of thin clans of $\hat{\Tr}$ is strictly less than the number of thin clans of $\Tr$. We proceed by designating a thin clan in $\hat{\Tr}$ and introducing new edges in the dependence graph $[V, \hat{E}, \lambda]$.

It is proved in Lemma 4, \cite{DM}, that the normal form $\nf$ is a map from the trace monoid $\Tr$ to the free monoid $\FF(\cA\cup \cA^{-1})$, which is compatible with inversion, i.e.  it satisfies that $\pi(\nf(w))=w$ and $\nf(w^{-1})=\nf(w)^{-1}$, where $w\in \Tr$ and $\pi$ is the canonical epimorphism from $\FF(\cA\cup \cA^{-1})$ to $\Tr$.

We refer to this normal form as to the \index{normal form!DM-}\emph{DM-normal form} or simply as to the \index{normal form}\emph{normal form} of an element $w\in \Tr$.

\section{Reducing systems of equations over $\GG$ to constrained generalised equations over $\FF$} \label{sec:red}

The notion of a generalised equation was introduced by Makanin in \cite{Makanin}. A generalised equation is a combinatorial object which encodes a system of equations over a free monoid. In other words, given a system of equations $S$ over a free monoid, one can construct a generalised equation $\Upsilon(S)$. Conversely, to a generalised equation $\Upsilon$, one can canonically associate a system of equations $S_\Upsilon$ over the free monoid $\FF$. The correspondence described has the following property.  Given a system $S$ the system $S_{\Upsilon(S)}$ is equivalent to $S$, i.e. the set of solutions defined by $S$ and by $S_{\Upsilon(S)}$ are isomorphic; and vice-versa, given a generalised equation $\Upsilon$, one has that the generalised equations $\Upsilon(S_\Upsilon)$ and  $\Upsilon$ are equivalent,  see Definition \ref{defn:geequiv} and Lemma \ref{lem:cgege}.

The motivation for defining a generalised equation is two-fold. One the one hand, it gives an efficient way of encoding all the information about a system of equations and, on the other hand, elementary transformations, that are essential for Makanin's algorithm, see Section \ref{se:5.1}, have a cumbersome description in terms of systems of equations, but admit an intuitive one in terms of graphic representations of combinatorial generalised equations. In this sense graphic representations of generalised equations can be likened to matrices. In linear algebra there is a correspondence between systems of equations over a field $k$ and matrices with elements from $k$. To describe the set of solutions of a system of equations, one uses Gauss elimination which is usually applied to matrices, rather than systems of equations.

In \cite{Mak82}, Makanin reduced the decidability of equations over a free group to the decidability of finitely many systems of equations over a free monoid, in other words, he reduced the compatibility problem for a free group to the compatibility problem for generalised equations. In fact, Makanin essentially proved that the study of solutions of systems of equations over free groups reduces to the study of solutions of generalised equations in the following sense: every solution of the system of equations $S$ factors trough one of the solutions of one of the generalised equations and, conversely, every solution of the generalised equation extends to a solution of $S$.

A crucial fact for this reduction is that the set of solutions of a given system of equations $S$ over a free group, defines only finitely many   different cancellation schemes (cancellation trees).  By each of these cancellation trees, one can construct a generalised equation.

The goal of this section is to generalise this approach to systems of equations over a partially commutative group $\GG$.

In Section \ref{se:4-1}, we give the definition of a generalised equation over a monoid $\MM$. Then, along these lines, we define the constrained generalised equation over a monoid $\MM$. Informally, a constrained generalised equation is simply a system of equations over a monoid with some constrains imposed onto its variables. In our case, the monoid we work with is either a trace monoid (alias for a partially commutative monoid) or a free monoid and the constrains that we impose on the variables are $\lra$-commutation, see Section \ref{sec:pcgr}.

Our aim is to reduce the study of solutions of systems of equations over partially commutative groups to the study of solutions of constrained generalised equations over a free  monoid. We do this reduction in two steps.

In Section \ref{sec:redgrmon}, we show that to a system of equations over a partially commutative group one can associate a finite collection of (constrained) generalised equations over a partially commutative monoid. The family of solutions of the collection  of  generalised equations constructed describes all solutions of the initial system over a partially commutative group, see Lemma \ref{le:14}.

This reduction is performed using an analogue of the notion of a cancellation tree for free groups. Let $S(X)=1$ be an equation over $\GG$. Then for any solution $(g_1,\dots, g_n)\in \GG^n$ of $S$, the word $S(g_1,\dots, g_n)$ represents a trivial element in $\GG$. Thus, by van Kampen's lemma there exists a van Kampen diagram for this word. In the case of partially commutative groups van Kampen diagrams have a structure of a band complex, see \cite{CK1}. We show in Proposition \ref{lem:prod} that van Kampen diagrams over a partially commutative group can be taken to a ``standard form''. This standard form of van Kampen diagrams can be viewed as an analogue of the notion of a cancellation scheme. By Proposition \ref{lem:prod}, it follows that the set of solutions of a given system of equations $S$ over a partially commutative group defines only finitely many van Kampen diagrams in standard form, i.e. finitely many different cancellation schemes. For each of these cancellation schemes one can construct a constrained generalised equation over the partially commutative monoid $\Tr$.

In Section \ref{sec:pcmontofrmon} we show that for a given generalised equation over $\Tr$ one can associate a finite collection of (constrained) generalised equations over the free monoid $\FF$. The family of solutions of the generalised equations from this collection describes all solutions of the initial generalised equation over $\Tr$, see Lemma \ref{lem:R1}.

This reduction relies on the ideas of Yu.~Matiyasevich, see \cite{Mat} (see also \cite{DMM}) and V.~Diekert and A.~Muscholl, see \cite{DM}. Essentially, it states that there are finitely many ways to take the product of words (written in DM-normal form) in $\Tr$ to DM-normal form, see Proposition \ref{prop:DMM} and Corollary \ref{cor:DMM}. We apply these results to reduce the study of the solutions of  generalised equations over the trace monoid to the study of solutions of constrained generalised equations over a free monoid.

Finally, in Section \ref{sec:expl1} we give an example that follows the exposition of Section \ref{sec:red}. We advise the reader unfamiliar with the terminology, to read the example of Section \ref{sec:expl1} simultaneously with the rest of Section \ref{sec:red}.

We would like to mention that in \cite{DM} V.~Diekert and A.~Muscholl give a reduction of the compatibility problem of equations over a partially commutative group $\GG$ to the decidability of equations over a free monoid with constraints. The reduction given in \cite{DM} provides a solution only to the compatibility problem of systems of equations over $\GG$ and uses the theory of formal languages. In this section we employ the machinery of generalised equations in order to reduce the description of the set of solutions of a system over $\GG$ to the same problem for constrained generalised equations over a free monoid and obtain a convenient setting for a further development of the process.

\subsection{Definition of (constrained) generalised equations} \label{se:4-1}

Let $X = \{x_1, \ldots, x_n\}$ be a set of variables and let $\GG = \GG(\cA)$  be the partially commutative group generated by $\cA$ and  $\GG[X] = \GG \ast F(X)$.

Further by \glossary{name={$\MM$}, description={free or partially commutative monoid, $\MM=\FF(\cA^{\pm 1})$ or $\MM=\Tr(\cA^{\pm 1})$}, sort=M}$\MM$ we always mean either $\FF(\cA^{\pm 1})$ or $\Tr(\cA^{\pm 1})$.

\begin{defn}
A \index{generalised equation!combinatorial}\emph{combinatorial generalised equation} $\Upsilon$ over $\MM$ (with coefficients from $\cA^{\pm 1}$)  consists of the following objects:
\begin{enumerate}
    \item A finite set of \index{base!of a generalised equation}{\em bases} \glossary{name={$\BS$}, description={the set of bases of a generalised equation}, sort=B}$\BS = \BS(\Upsilon)$.  Every base is either a \index{base!constant}constant base or a \index{base!variable}variable base. Each constant base is associated with exactly one letter from $\cA^{\pm 1}$. The set of variable bases ${\mathcal M}$ consists of $2n$ elements ${\mathcal M} = \{\mu_1, \ldots, \mu_{2n}\}$. The set ${\mathcal M}$ comes equipped with two functions: a function \glossary{name={$\varepsilon$}, description={function from the set of variable bases of a generalised equation to $\{-1,1\}$ that determines orientation of the base}, sort=E}$\varepsilon: {\mathcal M} \rightarrow \{-1,1\}$ and an involution \glossary{name={$\Delta$}, description={involution on the set of variable bases of a generalised equation that determines the dual of a base}, sort=D}$\Delta: {\mathcal M} \rightarrow {\mathcal M}$ (i.e. $\Delta$ is a bijection such that $\Delta^2$ is the identity on  ${\mathcal M}$). Bases $\mu$ and $\Delta(\mu)$ are called \index{base!dual}{\em dual bases}.
    \item A set of \index{boundary!of a generalised equation}{\em boundaries} \glossary{name={$\BD$}, description={the set of boundaries of a generalised equation}, sort=B}$\BD = \BD(\Upsilon)$. The set $\BD$ is a finite initial segment of the set of positive integers  $\BD = \{1, 2, \ldots, \rho_\Upsilon+1\}$.
    \item Two functions \glossary{name={$\alpha$}, description={function from the set $\BS$ of bases to the set $\BD$ of boundaries of a generalised equation that determines the left-most boundary of a base}, sort=A}$\alpha : \BS \rightarrow \BD$ and \glossary{name={$\beta$}, description={function from the set $\BS$ of bases to the set $\BD$ of boundaries of a generalised equation that determines the right-most boundary of a base}, sort=B}$\beta : \BS \rightarrow \BD$.   These functions satisfy the following conditions: $\alpha(b) <  \beta(b)$  for every base $b \in \BS$; if $b$ is a constant base then $\beta(b) = \alpha(b) + 1$.
    \item A finite set of \emph{boundary connections} \glossary{name={$\BC$}, description={the set of boundary connections of a generalised equation}, sort=B} $\BC = \BC(\Upsilon)$. A \index{boundary connection}($\mu$-)boundary connection is a triple $(i,\mu,j)$ where $i, j \in \BD$, $\mu \in {\M}$ such that $\alpha(\mu) <  i < \beta(\mu)$, $\alpha(\Delta(\mu)) <  j < \beta(\Delta(\mu))$
        We assume that if $(i,\mu,j) \in \BC$ then $(j,\Delta(\mu),i) \in \BC$. This allows one to identify the boundary
        connections $(i,\mu,j)$ and $(j,\Delta(\mu),i)$.
\end{enumerate}
\end{defn}
Though, by the definition, a combinatorial generalised equation is a combinatorial object, it is not practical to work with combinatorial generalised equations describing its sets and functions. It is more convenient to encode all this information in its graphic representation. We refer the reader to Section \ref{sec:expl1} for the construction of a graphic representation of a generalised equation. All examples given in this paper use the graphic representation of generalised equations.

To a combinatorial generalised equation $\Upsilon$ over a monoid $\MM$, one can associate a system of equations \glossary{name={$S_\Upsilon$}, description={system of equations associated to the generealised equation}, sort=S}$S_\Upsilon$ in \index{variable of a generalised equation}{\em variables} $h_1, \ldots, h_\rho$, \glossary{name={$\rho$, $\rho_\Upsilon$, $\rho_\Omega$}, description={the number of items of a generalised equation}, sort=R}$\rho=\rho_\Upsilon$ and coefficients from $\cA^{\pm 1}$ (variables $h_i$ are sometimes  called \index{item!of a generalised equation}{\em items}). The system of equations  $S_\Upsilon$  consists of the following three types of equations.

\begin{enumerate}
    \item Each pair of dual variable bases $(\lambda, \Delta(\lambda))$ provides an equation over the monoid $\MM$:
$$
[h_{\alpha (\lambda )}h_{\alpha (\lambda )+1}\cdots h_{\beta (\lambda )-1}]^ {\varepsilon (\lambda)}= [h_{\alpha (\Delta (\lambda ))}h_{\alpha (\Delta (\lambda ))+1} \cdots h_{\beta (\Delta (\lambda ))-1}]^ {\varepsilon (\Delta (\lambda))}.
$$
These equations are called \index{equation of a generalised equation!basic}{\em basic equations}. In the case when $\beta(\lambda)=\alpha(\lambda)+1$ and $\beta(\Delta(\lambda))=\alpha(\Delta(\lambda))+1$, i.e. the corresponding basic equation takes the form:
$$
[h_{\alpha (\lambda )}]^ {\varepsilon (\lambda)}= [h_{\alpha (\Delta (\lambda ))}]^{\varepsilon (\Delta (\lambda))}.
$$
    \item For each constant base $b$ we write down a \index{equation of a generalised equation!coefficient}{\em coefficient equation} over $\MM$:
$$
h_{\alpha(b)} = a,
$$
where $a \in \cA^{\pm 1}$ is the constant associated to $b$.
\item Every boundary connection $(p,\lambda,q)$ gives rise to a \index{equation of a generalised equation!boundary}\emph{boundary equation} over $\MM$, either
            $$
            [h_{\alpha (\lambda )}h_{\alpha (\lambda)+1}\cdots h_{p-1}]= [h_{\alpha (\Delta (\lambda ))}h_{\alpha (\Delta (\lambda ))+1} \cdots h_{q-1}],
            $$
            if $\varepsilon (\lambda)= \varepsilon (\Delta(\lambda))$, or
            $$
            [h_{\alpha(\lambda )}h_{\alpha (\lambda )+1}\cdots h_{p-1}]= [h_{q}h_{q+1}\cdots h_{\beta (\Delta (\lambda))-1}]^{-1},
            $$
            if $\varepsilon(\lambda)= -\varepsilon (\Delta(\lambda))$.
\end{enumerate}

Conversely, given a system of equations $S(X,\cA)=S$ over a monoid $\MM$, one can construct a combinatorial generalised equation \glossary{name={$\Upsilon(S)$}, description={combinatorial generalised equation associated to a system of equations over a monoid}, sort=U}$\Upsilon(S)$ over $\MM$.

Let $S=\{L_1=R_1,\dots, L_\m=R_\m\}$ be a system of equations over a monoid $\MM$. Write $S$ as follows:
$$
\begin{array}{lll}
l_{11}\cdots l_{1i_1}&=&r_{11}\cdots r_{1j_1}\\
 &\cdots&\\
l_{\m 1}\cdots l_{\m i_\m}&=&r_{\m 1}\cdots r_{\m j_\m}
\end{array}
$$
where $l_{ij}, r_{ij}\in X^{\pm 1}\cup \cA^{\pm 1}$. The set of boundaries $\BD(\Upsilon(S))$ of the generalised equation $\Upsilon(S)$ is
$$
\BD(\Upsilon(S))=\left\{1,2,\dots, \sum\limits_{k=1}^\m (i_k+j_k) +1\right\}.
$$
For all $k=1,\dots, \m$, we introduce a pair of dual variable bases $\mu_k, \Delta(\mu_k)$, so that
$$
\begin{array}{lll}
\alpha(\mu_k)=\sum\limits_{n=1}^{k-1} (i_n+j_n) +1,& \beta(\mu_k)=\alpha(\mu_k)+i_k,& \varepsilon(\mu_k)=1;\\
\alpha(\Delta(\mu_k))=\alpha(\mu_k)+i_k,& \beta(\Delta(\mu_k))=\alpha(\Delta(\mu_k))+j_k, & \varepsilon(\Delta(\mu_k))=1.
\end{array}
$$
For any pair of distinct occurrences of a variable $x\in X$ as $l_{ij}=x^{\epsilon_{ij}}$, $l_{rs}=x^{\epsilon_{st}}$, $\epsilon_{ij},\epsilon_{st}\in \{\pm 1\}$, where $(i,j)$ precedes $(s,t)$ in left-lexicographical order, we introduce a pair of dual bases $\mu_{x,q}$, $\Delta(\mu_{x,q})$, where $q=(i,j,s,t)$ so that
$$
\begin{array}{lll}
\alpha(\mu_{x,q})=\sum\limits_{n=1}^{i-1} (i_n+j_n) +j,& \beta(\mu_{x,q})=\alpha(\mu_{x,q})+1, & \varepsilon(\mu_{x,q})=\epsilon_{ij};\\
\alpha(\Delta(\mu_{x,q}))=\sum\limits_{n=1}^{s-1} (i_n+j_n) +t,& \beta(\Delta(\mu_{x,q}))=\alpha(\mu_{x,q})+1, & \varepsilon(\Delta(\mu_{x,q}))=\epsilon_{st}.
\end{array}
$$
Analogously, for any two occurrences of a variable $x\in X$ in $S$ as $r_{ij}=x^{\epsilon_{ij}}$, $r_{st}=x^{\epsilon_{st}}$ or as $r_{ij}=x^{\epsilon_{ij}}$, $l_{st}=x^{\epsilon_{st}}$, we introduce the corresponding pair of dual bases.

For any occurrence of a constant $a\in \cA$ in $S$ as $l_{ij}=a^{\epsilon_{ij}}$ we introduce a constant base $\nu_a$ so that
$$
\alpha(\nu_a)=\sum\limits_{n=1}^{i-1} (i_n+j_n) +j,\quad \beta(\nu_a)=\alpha(\nu_a)+1.
$$
Similarly, for any occurrence of a constant $a\in \cA$ as $r_{st}=a^{\epsilon_{st}}$, we introduce a constant base $\nu_a$ so that
$$
\alpha(\nu_a)=\sum\limits_{n=1}^{s-1} (i_n+j_n) +i_s+t,\quad \beta(\nu_a)=\alpha(\nu_a)+1.
$$

The set of boundary connections $\BC$ is empty.

\begin{defn} \label{defn:geequiv}
Introduce an equivalence relation on the set of all combinatorial generalised equations over $\MM$ as follows. Two generalised equations $\Upsilon$ and $\Upsilon'$ are \index{equivalence!of generalised equations}\emph{equivalent}, in which case we write $\Upsilon\approx\Upsilon'$, if and only if the corresponding systems of equations $S_\Upsilon$  and $S_{\Upsilon'}$ are equivalent {\rm(}recall that two systems are called equivalent if their sets of solutions are isomorphic{\rm)}.
\end{defn}

\begin{lem} \label{lem:cgege}
There is a one-to-one correspondence between the set of $\approx$-equivalence classes of combinatorial generalised equations over $\MM$ and the set of equivalence classes of systems of equations over $\MM$. Furthermore, this correspondence is given effectively.
\end{lem}
\begin{proof}
Define the correspondence between the set of $\approx$-equivalence classes of combinatorial generalised equations over $\MM$ and the set of equivalence classes of systems of equations over $\MM$ as follows. To a combinatorial generalised equation $\Upsilon$ over $\MM$ we assign the system of equations $S_{\Upsilon}$ associated to $\Upsilon$ and to a system of equations $S$ over $\MM$ we assign the combinatorial generalised equation $\Upsilon(S)$ associated to it. We now prove that this correspondence is well-defined.

Let $\Upsilon$ and $\Upsilon'$ be two $\approx$-equivalent generalised equations and let $S_{\Upsilon}$ and $S_{\Upsilon'}$ be the corresponding associated systems of equations over $\MM$. Then, by definition of the equivalence relation `$\approx$', one has that $S_{\Upsilon}$ and $S_{\Upsilon'}$ are equivalent.

Conversely,  let $S$ and $S'$ be two equivalent systems of equations over $\MM$ and let $\Upsilon(S)$ and $\Upsilon(S')$ be the corresponding associated combinatorial generalised equations over $\MM$. By definition, $\Upsilon(S)$ and $\Upsilon(S')$ are equivalent if and only if the systems $S_{\Upsilon(S)}$ and $S_{\Upsilon(S')}$ are. By construction, it is easy to check, that  the  system of equations $S_{\Upsilon(S)}$ associated to $\Upsilon(S)$ is equivalent to $S$ and the system $S_{\Upsilon(S')}$ is equivalent to $S'$, thus, by transitivity, $S_{\Upsilon(S)}$ is equivalent to $S_{\Upsilon(S')}$.
\end{proof}

Abusing the language, we call the system $S_\Upsilon$ associated to a generalised equation $\Upsilon$ \index{generalised equation}{\em generalised equation} over $\MM$, and, abusing the notation, we further denote it by the same symbol $\Upsilon$.

\begin{defn}\label{defn:Re}
A \index{generalised equation!constrained}\emph{constrained generalised equation} $\Omega$ over $\MM$ is a pair \glossary{name={$\gpo$},description={constrained generalised equation}, sort=O}$\gpo$, where $\Upsilon$ is a generalised equation and $\Re_\Upsilon$ is a symmetric binary relation on the set of variables $h_1, \ldots, h_\rho$ of the generalised equation $\Upsilon$ that satisfies the following condition.
\begin{enumerate}
\item[($\star$)]
Let $\Re_\Upsilon(h_i)=\left\{ h_j\mid \Re_\Upsilon(h_i,h_j)\right\}$. If in $\Upsilon$ there is an equation of the form
$$
h_{i_1}^{\epsilon_{i_1}}\dots h_{i_k}^{\epsilon_{i_k}} =h_{j_1}^{\epsilon_{j_1}}\dots h_{j_l}^{\epsilon_{j_l}}, \ \epsilon_{i_n},\epsilon_{j_l},\in \{1,-1\}, \ n=1,\dots, k,\, t=1,\dots, l
$$
and there exists $h_m$ such that $\Re_{\Upsilon}(h_{i_n},h_m)$ for all $n=1,\dots, k$, then $\Re_{\Upsilon}(h_{j_t},h_m)$ for all $t=1,\dots, l$.
\end{enumerate}
\end{defn}

\begin{defn}
Let $\Upsilon(h) = \{L_1(h)=R_1(h), \ldots, L_s(h) = R_s(h)\}$ be a generalised equation over $\Tr$ in variables $h = (h_1, \ldots,h_{\rho})$ with coefficients from $\GG$. A tuple $H = (H_1, \ldots, H_{\rho})$ of non-empty words from $\GG$ in the normal form (see Section \ref{sec:pcgr}) is a \index{solution!of a generalised equation!over $\Tr$}{\em solution} of $\Upsilon$ if:
\begin{enumerate}
\item all words $L_i(H), R_i(H)$ are geodesic (treated as elements of $\GG$);
\item $L_i(H) =  R_i(H)$ in the monoid $\Tr$ for all $i = 1, \dots, s$.
\end{enumerate}
\end{defn}

\begin{defn}
Let $\Upsilon(h) = \{L_1(h)=R_1(h), \ldots, L_s(h) = R_s(h)\}$ be a generalised equation over $\FF$ in variables $h = (h_1, \ldots,h_{\rho})$ with coefficients from $\GG$. A tuple $H = (H_1, \ldots, H_{\rho})$ of non-empty geodesic words from $\GG$ is a \index{solution!of a generalised equation!over $\FF$}{\em solution} of $\Upsilon$ if:
\begin{enumerate}
\item all words $L_i(H), R_i(H)$ are geodesic (treated as elements of $\GG$);
\item $L_i(H) =  R_i(H)$ in $\FF$ for all $i = 1, \dots, s$.
\end{enumerate}
\end{defn}
The notation \glossary{name={$(\Upsilon, H)$}, description={$H$ is a solution of the generalised equation $\Upsilon$}, sort=U}$(\Upsilon, H)$ means that $H$ is a solution of the generalised equation $\Upsilon$.

\begin{defn}
Let $\Omega=\gpo$ be a constrained generalised equation over $\MM$ in variables $h = (h_1, \ldots,h_{\rho})$ with coefficients from $\GG$. A tuple $H = (H_1, \ldots, H_{\rho})$ of non-empty geodesic words in $\GG$ is a \index{solution!of a constrained generalised equation}{\em solution} of $\Omega$ if $H$ is a solution of the generalised equation $\Upsilon$ and $H_i\lra H_j$ if $\Re_\Upsilon(h_i,h_j)$.

The \index{length!of a solution of the generalised equation}\emph{length} of a solution $H$ is defined to \glossary{name={$|H|$}, description={length of a solution of a (constrained) generalised equation, $|H|=\sum\limits_{i=1}^\rho |H_i|$}, sort=H} be
$$
|H|=\sum\limits_{i=1}^\rho |H_i|.
$$
\end{defn}
The notation \glossary{name={$(\Omega, H)$}, description={$H$ is a solution of the constrained generalised equation $\Omega$}, sort=O}$(\Omega, H)$ means that $H$ is a solution of the constrained generalised equation $\Omega$.

The term ``constrained generalised equation over $\MM$'' can be misleading. We would like to stress that solutions of a constrained generalised equations are \emph{not} solutions of the system of equations (associated to the generalised equation) over $\MM$. They are tuples of non-empty geodesic words from $\GG$ such that substituting these words in to the equations of a generalised equation, one gets equalities in $\MM$.

\begin{NB}
Further we abuse the terminology and call $\Omega$ simply ``generalised equation''. However, we always use the symbol \glossary{name={$\Omega$}, description={constrained generalised equation}, sort=O}$\Omega$ for constrained generalised equations and $\Upsilon$\glossary{name={$\Upsilon$}, description={generalised equation}, sort=U} for generalised equations.
\end{NB}

We now introduce a number of notions that we use throughout the text. Let $\Omega$ be a generalised equation.

\begin{defn}[Glossary of terms]
A boundary $i$ \index{boundary!intersects a base}\emph{intersects} the base $\mu$ if $\alpha (\mu)<i<\beta (\mu)$. A boundary $i$ \index{boundary!touches a base}\emph{touches} the base $\mu$ if $i=\alpha (\mu)$ or $i=\beta (\mu)$. A boundary is said to be \index{boundary!open}\emph{open} if it intersects at least one base, otherwise it is called \index{boundary!closed}\emph{closed}. We say that a boundary $i$ is \index{boundary!tied}\emph{tied} in a base $\mu$ (or is \index{boundary!$\mu$-tied}\emph{$\mu$-tied}) if there exists a boundary connection $(p,\mu,q)$ such that $i = p$ or $i = q$. A boundary is \index{boundary!free}\emph{free} if it does not touch any base and it is not tied by a boundary connection.

An item $h_i$ \index{item!belongs to a base}\emph{belongs} to a base $\mu$ or, equivalently, $\mu$ \index{base!contains an item}\emph{contains} $h_i$, if $\alpha (\mu)\leq i\leq \beta (\mu)-1$ (in this case we sometimes write $h_i\in \mu$). An item $h_i$ is called a \index{item!constant}\emph{constant} item if it belongs to a constant base and $h_i$ is called a \index{item!free}\emph{free} item if it does not belong to any base. By \glossary{name={$\gamma(h_i)$, $\gamma_i$},description={the number of bases which contain $h_i$},sort=G}$\gamma(h_i) =\gamma_i$ we denote the number of bases which contain $h_i$, in this case we also say that $h_i$ is \index{item!covered $\gamma_i$ times}\emph{covered} $\gamma_i$ times. An item $h_i$ is called \index{item!linear}\emph{linear} if $\gamma_i=1$ and is called \index{item!quadratic}\emph{quadratic} if $\gamma_i=2$.

Let $\mu, \Delta(\mu)$ be a pair of dual bases such that $\alpha (\mu)=\alpha (\Delta(\mu))$ and  $\beta (\mu)=\beta (\Delta(\mu))$ in this case we say that bases $\mu$ and $\Delta(\mu)$ form \index{pair of matched bases}\emph{a pair of matched bases}. A base $\lambda$ is \index{base!contained in another base}\emph{contained} in a base $\mu$ if $\alpha(\mu) \leq \alpha(\lambda) < \beta(\lambda) \leq \beta(\mu)$. We say that two bases $\mu$ and $\nu$ \index{base!intersects another base}\emph{intersect} or \index{base!overlaps with another base}\emph{overlap}, if $[\alpha(\mu),\beta(\mu)]\cap [\alpha(\nu),\beta(\nu)]\ne \emptyset$. A base $\mu$ is called \index{base!linear}\emph{linear} if there exists an item $h_i\in\mu$ so that $h_i$ is linear.

A set of consecutive items \glossary{name={$[i,j]$}, description={section $\{h_i,\ldots, h_{j-1}\}$ of a generalised equation}, sort=S}$[i,j] = \{h_i,\ldots, h_{j-1}\}$ is called a \index{section}\emph{section}. A section is said to be \index{section!closed}\emph{closed} if the boundaries $i$ and $j$ are closed and all the boundaries between them are open. If $\mu$ is a base then by \glossary{name={$\sigma(\mu)$}, description={the section $[\alpha(\mu),\beta(\mu)]$}, sort=S}$\sigma(\mu)$ we denote the section $[\alpha(\mu),\beta(\mu)]$ and by \glossary{name={$h(\mu)$}, description={the product of items $h_{\alpha(\mu)}\ldots h_{\beta(\mu)-1}$}, sort=H}$h(\mu)$ we denote the product of items $h_{\alpha(\mu)}\ldots h_{\beta(\mu)-1}$. In general for a section $[i,j]$ by \glossary{name={$h[i,j]$}, description={the product of items $h_i \ldots h_{j-1}$}, sort=H}$h[i,j]$ we denote the product  $h_i \ldots h_{j-1}$. A base $\mu$ \index{base!belongs to a section}\emph{belongs} to a section $[i,j]$ if $i\le\alpha(\mu)<\beta(\mu)\le j$. Similarly an item $h_k$ \index{item!belongs to a section}\emph{belongs} to a section $[i,j]$ if $i\le k<j$. In these cases we write $\mu\in [i,j]$ or $h_k\in [i,j]$.

Let $H=(H_1,\dots,H_{\rho})$ be a solution of a generalised equation $\Omega$ in variables $h=\{h_1,\dots, h_\rho\}$. We use the following notation. For any word $W(h)$ in $\GG[h]$ set \glossary{name={$W(H)$}, description={for any word $W(h)$ in $\GG[h]$ and any solution $H$ of a generalised equation, $W(H)=H(W(h))$}, sort=W}$W(H)=H(W(h))$. In particular, for any base $\mu$ (section $\sigma=[i,j]$) of $\Omega$,  we have \glossary{name={$H(\mu)$}, description={if $H$ is a solution  of a generalised equation, $H(h(\mu))=H_{\alpha(\mu)} \cdots H_{\beta(\mu)-1}$}, sort=H} $H(\mu)=H(h(\mu))=H_{\alpha(\mu)} \cdots H_{\beta(\mu)-1}$ (\glossary{name={$H[i,j]$, $H(\sigma)$},description={if $H$ is a solution  of a generalised equation, $H(\sigma)=H_i\cdots H_{j-1}$}, sort=H}$H[i,j]=H(\sigma)=H(h(\sigma))=H_i\cdots H_{j-1}$, respectively).
\end{defn}

We now formulate some necessary conditions for a generalised equation to have a solution.

\begin{defn} \label{def:forcon}
A generalised equation $\Omega=\gpo$ is called \index{generalised equation!formally consistent}{\em formally consistent} if it satisfies the following conditions.
\begin{enumerate}
    \item \label{it:forcon1} If $\varepsilon (\mu)=-\varepsilon (\Delta (\mu))$, then the bases $\mu$ and $\Delta (\mu )$ do not intersect, i.e. none of the items $h_{\alpha(\mu)}, h_{\beta(\mu)-1}$ is contained in $\Delta (\mu )$.
    \item \label{it:forcon2}  Given two boundary connections $(p,\lambda ,q)$ and $(p_1,\lambda ,q_1)$, if $p\le p_1$, then $q\le q_1$ in the case when $\varepsilon (\lambda)\varepsilon (\Delta (\lambda))=1$, and $q\ge q_1$ in the case when $\varepsilon (\lambda)\varepsilon (\Delta (\lambda))=-1$. In particular, if $p=p_1$ then $q = q_1$.
    \item  Let $\mu$ be a base such that  $\alpha (\mu)=\alpha (\Delta (\mu))$, in other words, let $\mu$ and $\Delta(\mu)$ be a pair of matched bases. If $(p,\mu ,q)$ is a $\mu$-boundary connection  then  $p=q$.
    \item A variable cannot occur in two distinct coefficient equations, i.e., any two constant bases with the same left end-point are labelled by the same letter from $\cA^{\pm 1}$.
    \item If $h_i$ is a variable from some coefficient equation and $(i,\mu,q_1), (i+1,\mu,q_2)$ are boundary connections, then $|q_1- q_2|=1$.
    \item If $\Re_\Upsilon(h_i,h_j)$ then $i\ne j$.
\end{enumerate}
\end{defn}

\begin{lem}\label{le:4.1} \
\begin{enumerate}
    \item If a generalised equation $\Omega$ over a monoid $\MM$ has a solution, then $\Omega$ is formally consistent;
    \item There is an algorithm to check whether or not a given generalised equation is formally consistent.
\end{enumerate}
\end{lem}
\begin{proof}
We show that condition (\ref{it:forcon1}) of Definition \ref{def:forcon} holds for the generalised equation $\Omega$ in the case $\MM=\Tr$. Assume the contrary, i.e. $\varepsilon (\mu)=-\varepsilon (\Delta (\mu))$ and the bases $\mu$ and $\Delta (\mu )$ intersect. Let $H$ be a solution of $\Omega$. Without loss of generality we may assume that $\varepsilon (\mu)=1$ and $\alpha(\mu)\le\alpha(\Delta(\mu))$. Then $\Omega$ has the following basic equation:
$$
h_{\alpha(\mu)}\cdots h_{\alpha(\Delta(\mu))}=\left( h_{\beta(\mu)}\cdots h_{\beta(\Delta(\mu))}\right)^{-1}.
$$
Since the bases $\mu$ and $\Delta (\mu )$ intersect, the words $H[\alpha(\Delta(\mu)),\beta(\mu)]$ and $H[\alpha(\Delta(\mu)),\beta(\mu)]^{-1}$ right-divide the word $H(\mu)$.  This derives a contradiction with the fact that $H(\mu)$ is a geodesic word in $\GG$, see \cite{EKR}.

Proof follows by straightforward verification of the conditions in Definition \ref{def:forcon}.
\end{proof}

\begin{rem}
We further consider only formally consistent generalised equations.
\end{rem}

\subsection{Reduction to generalised equations: from partially commutative groups to monoids} \label{sec:redgrmon}

In this section we show that to a given finite system of equations $S(X,\cA) = 1$ over a partially commutative group $\GG$ one can associate a finite collection of (constrained) generalised equations $\GE'(S)$ over $\Tr$ with coefficients from $\cA^{\pm 1}$. The family of solutions of the generalised equations from $\GE'(S)$ describes all solutions of the system $S(X,\cA) = 1$, see Lemma \ref{le:14}.

\subsubsection{$\GG$-partition tables} \label{sec:Gparttab}
Write the system $\{S(X,\cA)=1\}= \{S_ 1 = 1, \ldots, S_\m = 1\}$ in the form
\begin{equation}\label{*}
\begin{array}{c}
 r_{11}r_{12}\ldots r_{1l_1}=1,\\
 r_{21}r_{22}\ldots r_{2l_2}=1,\\
 \ldots \\
 r_{\m 1}r_{\m 2}\ldots r_{\m l_\m}=1,\\
\end{array}
\end{equation}
where $r_{ij}$ are letters of the alphabet $ X^{\pm 1}\cup \cA^{\pm 1}$.

We aim to define a combinatorial object called a $\GG$-partition table, that encodes a particular type of cancellation that happens when one substitutes a solution $W(\cA) \in \GG^n$ into $S(X,\cA) = 1$ and then reduces the words in $S(W(\cA),\cA)$ to the empty word.

Informally, Proposition \ref{lem:prod} describes all possible cancellation schemes for the set of all solutions of the system $S(X,\cA)$ in the following way: the cancellation scheme corresponding to a particular solution, can be obtained from the one described in Proposition \ref{lem:prod} by setting some of the words $w_i^j$'s (and the corresponding bands) to be trivial. Therefore, every $\GG$-partition table (to be defined below) corresponds to one of the cancellation schemes obtained from the general one by setting some of the words $w_i^j$'s to be trivial. Every non-trivial word $w_i^j$ corresponds to a variable $z_k$ and the word $w_j^i$ to the variable $z_k^{-1}$.  If a variable $x$ that occurs in the system $S(X,A) = 1$ is subdivided into a product of some words $w_i^j$'s, i.e. the variable $x$ is a word in the $w_i^j$'s, then the word $V_{ij}$ from the definition of a partition table is this word in the corresponding $z_k$'s. If the bands corresponding to the words $w_i^j$ and $w_k^l$ cross, then the corresponding variables $z_r$ and $z_s$ commute in the group $\HH$.

The definition of a $\GG$-partition table is rather technical, we refer the reader to Section \ref{sec:expl1} for an example.

A pair $\TB$ (a finite set of geodesic words from $\GG*F(Z)$, a $\GG$-partially commutative group), $\TB=(V,\HH)$ of the form:
$$
V = \{V_{ij}(z_1, \ldots ,z_p)\} \subset \GG[Z]=\GG*F(Z) \ \ (1\leq i\leq m, 1\leq j\leq l_i), \quad \HH=\GG(\cA\cup Z),
$$
is called a  \index{partition@$\GG$-partition table}$\GG$-{\em partition table} of the system $S(X,\cA)$ if the following conditions are satisfied:
\begin{enumerate}
    \item\label{it:pt0} Every element $z\in Z\cup Z^{-1}$ occurs in the words $V_{ij}$ only once;
    \item\label{it:pt1} The equality $V_{i1}V_{i2} \cdots V_{il_i}=1, 1\leq i\leq \m,$ holds in $\HH$;
    \item\label{it:pt2} $|V_{ij}|\leq l_i-1$;
    \item\label{it:pt3} if $r_{ij} = a \in \cA^{\pm 1}$, then $|V_{ij}|=1$.
\end{enumerate}
Here the designated copy of $\GG$ in $\HH$ is the natural one and is generated by $\cA$.

Since $|V_{ij}|\leq l_i - 1$ then at most $\sum\limits_{i = 1}^\m (l_i - 1)l_i$ different letters $z_i$ can occur in a partition table of $S(X,\cA) = 1$. Therefore we always assume that $p \leq \sum\limits_{i = 1}^\m (l_i - 1)l_i$.

\begin{rem} \label{rem:frgr1}
In the case that $\GG=F(\cA)$ and $\HH=F(Z)$ are free groups, the notion of a $\GG$-partition table coincides with the notion of a partition table in the sense of Makanin, see \cite{Mak82}.
\end{rem}

\begin{lem}\label{le:4.2}
Let $S(X,\cA) = 1$ be a finite system of equations over $\GG$. Then
\begin{enumerate}
    \item the set $\PT(S)$ of all $\GG$-partition tables  of $S(X,\cA) = 1$ is finite, and its cardinality is bounded by a number which depends only on the system $S(X,\cA)$;
    \item one can effectively enumerate the set $\PT(S)$.
\end{enumerate}
\end{lem}
\begin{proof}   Since the words $V_{ij}$ have  bounded length, one can effectively enumerate the finite set  of all collections of words $\{V_{ij}\}$ in $F(Z)$ which satisfy the conditions (\ref{it:pt0}), (\ref{it:pt2}), (\ref{it:pt3}) above.  Now for each such collection $\{V_{ij}\}$,  one can
effectively check  whether the  equalities $V_{i1}V_{i2} \cdots V_{il_i}=1, 1\leq i\leq \m$ hold in one of the finitely many (since $|Z|<\infty$) partially commutative groups $\HH$ or not. This allows one to list effectively all partition tables for $S(X,\cA) = 1$.
\end{proof}

\subsubsection{Associating a generalised equation over $\Tr$ to a $\GG$-partition table} \label{sec:geneqT}
By a partition table $\TB=(\{V_{ij}\},\HH)$ we construct a generalised equation $\Omega _\TB'=\langle\Upsilon_\TB',\Re_{\Upsilon_\TB'}\rangle$ in the following way. We refer the reader to Section \ref{sec:expl1} for an example.

Consider the following word $\mathcal{V}$ in $\FF(Z^{\pm 1})$:
$$
\mathcal{V}\doteq V_{11}V_{12}\cdots V_{1l_1}\cdots V_{\m 1}V_{\m 2} \cdots V_{\m l_\m} = y_1 \cdots y_{\rho'},
$$
where $\FF(Z^{\pm 1})$ is the free monoid on the alphabet $Z^{\pm 1}$; $y_i \in Z^{\pm 1}$ and $\rho' =|\mathcal{V}|$ is the length of $\mathcal{V}$. Then the generalised equation $\Omega_\TB' = \Omega_\TB'(h')$ has $\rho' + 1$ boundaries and $\rho'$ variables $h'_1,\ldots ,h'_{\rho'}$ which are denoted by $h' = (h'_1,\dots, h'_{\rho'})$.

Now we define the bases of $\Omega_\TB'$ and the functions $\alpha, \beta, \varepsilon$.

Let $z \in Z$. For the (unique) pair of distinct occurrences of $z$ in $\mathcal{V}$:
$$
y_i = z^{\epsilon _i}, \quad y_j = z^{\epsilon _j} \quad \epsilon _i, \epsilon _j \in \{1,-1\}, \quad i<j,
$$
we introduce a pair of dual variable bases $\mu_{z,i}, \mu_{z,j}$ such that $\Delta(\mu_{z,i}) = \mu_{z,j}$. Put
\begin{gather}\notag
\begin{split}
\alpha(\mu_{z,i}) = i, \quad &\beta(\mu_{z,i}) = i+1, \quad \varepsilon(\mu_{z,i}) = \epsilon _i,\\
&\alpha(\mu_{z,j}) = j, \quad \beta(\mu_{z,j}) = j+1, \quad \varepsilon(\mu_{z,j}) = \epsilon _j.
\end{split}
\end{gather}
The basic equation that corresponds to this pair of dual bases is ${h'_{i}}^{\epsilon_i}\doteq {h'_{j}}^{\epsilon _j}$.

Let $x \in X$. For any two distinct occurrences of $x$ in $S(X,\cA) = 1$:
$$
r_{i,j} = x^{\epsilon_{ij}}, \quad r_{s,t} = x^{\epsilon_{st}} \quad (\epsilon _{ij}, \epsilon _{st} \in \{1,-1\})
$$
so that $(i,j)$ precedes $(s,t)$ in left-lexicographical order, we introduce a pair of dual bases $\mu_{x,q}$ and $\Delta(\mu_{x,q})$, $q=(i,j,s,t)$.  Now suppose that $V_{ij}$ and $V_{st}$ occur in the word $\mathcal{V}$ as subwords
$$
V_{ij} = y_{c_1} \ldots y_{d_1}, \quad V_{st} = y_{c_2} \ldots y_{d_2} \hbox{ correspondingly}.
$$
Then we put
\begin{gather}\notag
\begin{split}
\alpha(\mu_{x,q}) = {c_1}, \quad &\beta(\mu_{x,q}) = d_1+1, \quad \varepsilon(\mu_{x,q}) = \epsilon_{ij},\\
&\alpha(\Delta(\mu_{x,q})) = {c_2}, \quad \beta(\Delta(\mu_{x,q})) = d_2+1, \quad \varepsilon(\Delta(\mu_{x,q})) = \epsilon_{st}.
\end{split}
\end{gather}
The basic equation over $\Tr$ which corresponds to this pair of dual bases can be written in the form
$$
\left(h'_{\alpha(\mu_{x,q})} \cdots h'_{\beta(\mu_{x,q})-1}\right)^{\epsilon _{ij}}=\left(h'_{\alpha(\Delta(\mu_{x,q}))}\ldots h'_{\beta(\Delta(\mu_{x,q}))-1}\right)^{\epsilon_{st}}.
$$

Let $r_{ij} = a \in \cA^{\pm 1}$. In this case we introduce a constant base $\mu_{ij}$ with the label $a$. If $V_{ij}$ occurs in $\mathcal{V}$ as $V_{ij} = y_c$, then we put
$$
\alpha(\mu_{ij}) = c,\quad  \beta(\mu_{ij}) = c+1.
$$
The corresponding coefficient equation is  $h_{c}'=a$. The set of boundary connections of the generalised equation $\Upsilon_\TB'$ is empty. This defines the generalised equation $\Upsilon_\TB'$.

We define the binary relation $\Re_{\Upsilon_\TB'}\subseteq h'\times h'$ to be the minimal subset of $h'\times h'$ which contains the pairs $(h'_i,h'_j)$ such that $[y_i, y_j]=1$\ in $\HH$ and $y_i\ne y_j$, is symmetric and satisfies condition ($\star$) of Definition \ref{defn:Re}. This defines the constrained generalised equation $\Omega'_\TB=\langle\Upsilon_\TB', \Re_{\Upsilon_\TB'}\rangle$.

Put
\glossary{name={$\GE'(S)$}, description={the set of all constrained generalised equations $\Omega'_\TB$ over $\Tr$ constructed by $\GG$-partition tables $\TB$ for the system $S(X,\cA)$}, sort=G}
$$
\GE'(S) = \{\Omega'_\TB \mid \TB \hbox{ is a $\GG$-partition table for } S(X,\cA)= 1 \}.
$$
Then $\GE'(S)$ is a finite collection of generalised equations over $\Tr$ which can be effectively constructed for a given system of equations $S(X,\cA) = 1$ over $\GG$.

\begin{rem} \label{rem:frgr2}
As in Remark \ref{rem:frgr1}, if $\GG=F(\cA)$ and $\HH=F(Z)$ are free groups and the set $\Re_{\Upsilon_\TB'}$ is empty, then the basic equations are considered over a free monoid and the collection of generalised equations $\GE'$ coincides with the one used by Makanin, see \cite{Mak82}.
\end{rem}

\subsubsection{Coordinate groups of systems of equations and coordinate groups of generalised equations} \label{sec:relcoordgrgrtr}

\begin{defn}
For a generalised equation $\Upsilon$ in variables $h$ over a monoid $\MM$ we can consider the same system of equations over the partially commutative group (not in the monoid). We denote this system by \glossary{name={${\Upsilon}^*$}, description={generalised equation $\Upsilon$ treated as a system of equations over the partially commutative group $\GG$}, sort=U}
${\Upsilon}^*$. In other words, if
$$
\Upsilon = \{L_1(h)=R_1(h), \ldots, L_s(h) = R_s(h)\}
$$
is an arbitrary system of equations over $\MM$ with coefficients from $\cA^{\pm 1}$, then by $\Upsilon^*$ we denote the system of equations
$$
\Upsilon^*= \{L_1(h)R_1(h)^{-1} = 1, \ldots, L_s(h)R_s(h)^{-1} = 1\}
$$
over the group $\GG$.

Similarly, for a given constrained generalised equation $\Omega=\gpo$ in variables $h$ over a monoid $\MM$ we can consider the system of equations \glossary{name={${\Omega}^*$}, description={the generalised equation $\Omega$ treated as a system of equations over the partially commutative group $\GG$}, sort=O}${\Omega}^*={\Upsilon}^*\cup \{[h_i,h_j]\mid \Re_\Upsilon(h_i,h_j)\}$ over the partially commutative group (not in the monoid). \glossary{name={$\GG_{R(\Omega^\ast)}$}, description={coordinate group of the generalised equation $\Omega$}, sort=G} Let
$$
\GG_{R(\Omega^\ast)}= \factor{\GG\ast F(h)}{R({\Upsilon}^*\cup \{[h_i,h_j]\mid \Re_\Upsilon(h_i,h_j)\})}\, ,
$$
where $F(h)$ is the free group with basis $h$. We call $\GG_{R(\Omega^\ast)}$ the \index{coordinate group!of the (constrained) generalised equation}\emph{coordinate group} of the constrained generalised equation $\Omega$.
\end{defn}
Note that the definition of the coordinate group of a constrained generalised equation over $\MM$ is independent of the monoid $\MM$.

Obviously, each solution $H$ of $\Omega$ over $\MM$ gives rise to a solution of $\Omega^*$ in the partially commutative group $\GG$. The converse does not hold in general even in the case when $\MM=\Tr$: it may happen that $H$ is a solution of $\Omega^*$ in $\GG$ but not in $\Tr$, i.e. some equalities $L_i(H) = R_i(H)$ hold only after a reduction in $\GG$ (and, certainly, the equation $[h_i,h_j]$ does not imply the $\lra$-commutation for the solution).

Now we explain the relation between the coordinate groups $\GG_{R({\Omega_\TB'}^\ast)}$  and $\GG_{R(S(X,\cA))}$.

For a letter $x$ in $X$ we choose an arbitrary occurrence of $x$ in $S(X,\cA) = 1$ as
$$
r_{ij} = x^{\epsilon_{ij}}.
$$
Let $\mu = \mu_{x,q}$, $q=(i,j,s,t)$ be a base that corresponds to this occurrence of $x$. Then $V_{ij}$ occurs in $\mathcal{V}$ as the subword
$$
V_{ij} = y_{\alpha(\mu)} \ldots y_{\beta(\mu) -1}.
$$
Notice that the word $V_{ij}$ does not depend on the choice of the base $\mu_{x,q}$ corresponding to the occurrence $r_{ij}$.
Define a word $P_x(h') \in \GG[h']$ (where $h' = \{h'_1, \ldots,h'_{\rho'}\}$) as follows
$$
P_x(h',\cA) = {\left(h'_{\alpha(\mu)} \cdots h'_{\beta(\mu)-1}\right)}^{\epsilon_{ij}},
$$
and put
$$
P(h') = (P_{x_1}, \ldots, P_{x_n}).
$$
The word $P_x(h')$ depends on the choice of occurrence $r_{ij} = x^{\epsilon_{ij}}$ in $\mathcal{V}$.

It follows from the construction above that the map $X \rightarrow \GG[h']$ defined by $x \mapsto  P_x(h',\cA)$ gives rise to a $\GG$-homomorphism
\begin{equation}\label{eq:hompt}
\pi : \GG_{R(S)}\rightarrow \GG_{R({\Omega_\TB'}^\ast)}.
\end{equation}
Indeed,  if $f(X) \in R(S)$ then $\pi(f(X))= f(P(h))$. It follows from condition (\ref{it:pt1}) of the definition of partition table that $f(P(h)) = 1$ in $\GG_{R({\Omega_\TB'}^\ast)}$, thus $f(P(h))\in R({\Omega_\TB'}^*)$. Therefore $R(f(S))\subseteq R({\Omega_\TB'}^*)$ and $\pi$ is a homomorphism.

Observe that the image  $\pi (x)$ in $\GG_{R({\Omega_\TB'}^\ast)}$ does not depend on a particular choice of the occurrence of $x$ in $S(X,\cA)$ (the basic equations of $\Omega_\TB'$ make these images equal). Hence $\pi$ depends only on $\Omega_\TB'$. Thus, every solution $H$ of a generalised equation gives rise to a solution $U$ of $S$ so that $\pi_U=\pi\pi_H$.

\subsubsection{Solutions of systems of equations over $\GG$ and solutions of generalised equations over $\Tr$}
Our goal in this section is to prove that every solution of the system of equations $S(X,\cA)$ over $\GG$ factors through one of the solutions of one of the finitely many generalised equations from $\GE'(S)$, i.e. for a solution $W$ of $S$ there exists a generalised equation $\Omega'_\TB\in \GE'(S)$ and a solution $H$ of $\Omega'_\TB$ so that $\pi_W=\pi\pi_H$, where $\pi$ is defined in (\ref{eq:hompt}). In order to do so, we need an analogue of the notion of a cancellation scheme for free groups. The results of this section rely on the techniques developed in Section 3 of \cite{CK1}.

Let $W(\cA)$ be a solution of $S(X,\cA) = 1$ in $\GG$. If in the system (\ref{*}) we make the substitution  $\sigma : X \rightarrow W(\cA)$, then
$$
(r_{i1}r_{i2}\ldots r_{il_i})^{\sigma} = r_{i1}^{\sigma}r_{i2}^{\sigma}\ldots r_{il_i}^{\sigma} = 1
$$
in $\GG$ for every $i = 1, \ldots, \m$.

Since every product $R_i^\sigma = r_{i1}^{\sigma}r_{i2}^{\sigma}\ldots r_{il_i}^{\sigma}$ is trivial, we can construct a van Kampen diagram ${\mathcal D}_{R_i^\sigma}$ for $R_i^\sigma$. Denote by ${\tilde z}_{i,1}, \ldots, {\tilde z}_{i,p_i}$ the subwords $w_j^k$, $1\le j<k\le l_i$ of $r_{ij}^{\sigma}$, where $w_j^k$ are defined as in Proposition \ref{lem:prod}. As, by Proposition \ref{lem:prod}, $w_j^k= {w_k^j}^{-1}$, so every $r_{ij}$ can be written as a freely reduced word $V_{ij}(Z_i)$ in variables $Z_i = \{z_{i,1}, \ldots, z_{i,p_i}\}$ so that if we set ${z}_{i,j}^\sigma={\tilde z}_{i,j}$, then we have that the following equality holds in $\Tr$:
$$
r_{ij}^\sigma = V_{ij}({\tilde z}_{i,1}, \ldots, {\tilde z}_{i,p_i}).
$$
Observe that if $r_{ij} = a \in \cA^{\pm 1}$ then $r_{ij}^{\sigma} = a$ and we have $|V_{ij}| = 1$. By Proposition \ref{lem:prod}, $r_{ij}^\sigma$ is a product of at most $l_i-1$ words $w_j^k$, thus we have that $|V_{ij}| \leq l_i - 1$. Denote by $Z=\bigcup\limits_{i=1}^{\m} Z_i= \{z_1, \ldots, z_p\}$. Take the partially commutative group $\HH=\GG(\cA\cup Z)$ whose underlying commutation graph is defined as follows:
\begin{itemize}
    \item two elements $a_i,a_j$ in $\cA$ commute whenever they commute in $\GG$;
    \item an element $a\in \cA$ commutes with $z_i$ whenever $\tilde z_i\lra a$;
    \item two elements $z_i, z_j\in Z$ commute whenever $\tilde z_i\lra \tilde z_j$.
\end{itemize}

By construction, the set $\{V_{ij}\}$ along with the group $\HH$ is a partition table $\TB_{W(\cA)}$ for the system $S(X,\cA) = 1$ and the solution ${W(\cA)}$. Obviously,
$$
H' = ({\tilde z}_1, \ldots, {\tilde z}_p)
$$
is the solution of the generalised equation $\Omega_{\TB_{W(\cA)}}$ induced by $W(\cA)$. From the construction of the map $P(h')$ we
deduce that $W(\cA) = P(H')$.

The following lemma shows that to describe the set of solutions of a system of equations over $\GG$ is equivalent to describe the set of solutions of constrained generalised equations over the trace monoid $\Tr$. This lemma can be viewed as the first ``divide'' step of the process.

\begin{lem}\label{le:14}
For a given system of equations $S(X,\cA)=1$ over $\GG$, one can effectively construct a finite set
$$
\GE'(S) = \{\Omega_\TB' \mid \TB \hbox{ is a $\GG$-partition table for } S(X,\cA)= 1 \}
$$
of generalised equations over $\Tr$ such that
\begin{enumerate}
    \item if the set $\GE'(S)$ is empty, then $S(X,\cA)= 1$ has no solutions in $\GG$;
    \item for each $\Omega' (h') \in\GE'(S)$ and for each $x \in X$ one can effectively find a word $P_x(h',\cA) \in \GG[h']$ of length at most $|h'|$ such that the map $x \mapsto P_x(h',\cA)$ gives rise to a $\GG$-homomorphism $\pi_{\Omega'}: \GG_{R(S)}\rightarrow \GG_{R({\Omega'}^\ast)}$ (in particular, for every solution $H'$ of the generalised equation $\Omega'$ one has that $P(H')$ is a solution of the system $S(X, \cA)$, where $P(h') = (P_{x_1}, \ldots, P_{x_n})$);
    \item for any solution $W(\cA) \in \GG^n$ of the system $S(X,\cA)=1$ there exists $\Omega' (h') \in \GE'(S)$ and a solution $H'$ of $\Omega'(h')$ such that $W(\cA) = P(H')$, where $P(h') = (P_{x_1}, \ldots, P_{x_n})$, and this equality holds in the partially commutative monoid $\Tr(\cA^{\pm 1})$;
\end{enumerate}
\end{lem}

\begin{cor} \label{co:14}
In the notation of {\rm Lemma \ref{le:14}}  for any solution $W(\cA)  \in \GG^n=\GG(\cA)^n$ of the system $S(X,\cA)=1$ there exist a generalised equation $\Omega' (h') \in \GE'(S)$  and a solution $H'$ of $\Omega'(h')$ such that the following diagram commutes
$$
\xymatrix@C3em{
 \GG_{R(S)}  \ar[rd]_{\pi_W} \ar[rr]^{\pi_{\Omega'}}  &  &\GG_{R({\Omega'}^\ast)} \ar[ld]^{\pi_{H'}}
                                                                             \\
                               &  \GG &
}
$$
Conversely, for every generalised equation $\Omega' (h') \in \GE'(S)$  and a solution $H'$ of $\Omega'(h')$ there exists a solution $W(\cA)  \in \GG^n=\GG(\cA)^n$ of the system $S(X,\cA)=1$ such that the above diagram commutes.
\end{cor}

\subsection{Reduction to generalised equations: from partially commutative monoids to free monoids}\label{sec:pcmontofrmon}

In this section we show that to a given generalised equation $\Omega'$ over $\Tr$ one can associate a finite collection of (constrained) generalised equations $\GE(\Omega')$ over $\FF$ with coefficients from $\cA^{\pm 1}$. The family of solutions of the generalised equations from $\GE(\Omega')$ describes all solutions of the generalised equation $\Omega'$, see Lemma \ref{lem:R1}.

We shall make use of the proposition below, which is due to V.~Diekert and A.~Muscholl. Essentially, it states that for a given partially commutative monoid $\Tr$, there are finitely many ways to take the product of two words (written in DM-normal form) in $\Tr$ to DM-normal form. More precisely, given a trace monoid $\Tr$ there exists a global bound $k$ such that the words $u$ and $v$ can be written as a product of $k$ subwords in such a way that the product $w=uv$ written in the normal form is obtained by concatenating these subwords in some order.

\begin{prop}[Theorem 8, \cite{DM}] \label{prop:DMM}
Let $u,v,w$ be words in DM-normal form in the trace monoid $\Tr$, such that $uv=w$ in $\Tr$. Then there exists a positive integer $k$ bounded above by a computable function of $\rr$ {\rm(}recall that $\rr=|\cA|${\rm)} such that
$$
\begin{array}{lll}
  u & \doteq & u_1\cdots  u_k \\
  v & \doteq & v_{j_1}\cdots v_{j_k} \\
  w & \doteq & u_1v_1\cdots u_kv_{k}
\end{array}
$$
for some permutation $(j_1,\dots,j_k)$ of $\{1,\dots, k\}$ and words $u_i,v_j$ written in DM-normal form satisfying
$$
 u_i\lra v_j  \hbox{ for all } i>j \hbox{ and } v_{j_p}\lra v_{j_q} \hbox{ for all $p, q$ such that } (j_p-j_q)(p-q)<0
$$
\end{prop}

\begin{cor} \label{cor:DMM}
Let $w,w_1\dots w_l$ be words in DM-normal form in the trace monoid $\Tr$, such that $w_1\dots w_l=w$ in $\Tr$. Then there exists a positive integer $\kk=\kk(l,\rr)$ bounded above by a computable function of $l$ and $\rr$ such that
$$
\begin{array}{lll}
  w_i & \doteq & w_{i,\varsigma_i(1)}\cdots  w_{i,\varsigma_i(\kk)},\  i=1,\dots, l; \\
  w & \doteq & w_{1,1}w_{2,1}\dots w_{l,1}\dots w_{1,\kk}\dots w_{l,\kk}
\end{array}
$$
for some permutations $\varsigma_i$, $i=1,\dots, l$ of $\{1,\dots, \kk\}$ and some words $w_{i,j}$ written in DM-normal form satisfying
\begin{gather}\notag
\begin{split}
& w_{i_1,j_1}\lra w_{i_2,j_2}  \hbox{ for all  $i_1>i_2$ and $j_1<j_2$, and } \\
           & w_{i,\varsigma_i(j_1)}\lra w_{i,\varsigma_i(j_2)} \hbox{ for all $j_1, j_2$ such that } (j_1-j_2)(\varsigma_i(j_1)-\varsigma_i(j_2))<0.
\end{split}
\end{gather}
\end{cor}
\begin{proof}
The proof is by induction on $l$ and is left to the reader.
\end{proof}

\subsubsection{Generalised equations over $\FF$ and generalised equations over $\Tr$}\label{sec:fromTtoF}
Let $\Omega'=\langle\Upsilon',\Re_{\Upsilon'}\rangle$ be a generalised equation over $\Tr$ and let $\Upsilon'=\{L_i=R_i\mid i=1,\dots,m\}$. To every word $L_i$ (correspondingly, $R_i$) (in the alphabet $h'$) we associate a set $C(L_i)$ (correspondingly, $C(R_i)$) constructed below.

We explain the idea behind this construction. By Corollary \ref{cor:DMM}, for every solution $H$ the word $L_i(H)$  (correspondingly, $R_i(H)$) can be taken to the DM-normal form by using appropriate permutations $\varsigma_i$ and setting some of the words $w_{i,\varsigma_i(j)}$ to be trivial. The system of equations in $c$, $c\in C(L_i)$ (correspondingly, $d\in C(R_i)$), see below for definition, is obtained from the equations
$$
w_i \doteq  w_{i,\varsigma_i(1)}\cdots  w_{i,\varsigma_i(k)},
$$
see  Corollary \ref{cor:DMM}, setting some of the words $w_{i,\varsigma_i(j)}$ to be trivial and applying the appropriate permutations $\varsigma_i$.

The word $NF_c(L_i) \in c$, $c\in C(L_i)$ (correspondingly, $NF_d(R_i)$, $d\in C(R_i)$) is the DM-normal form of the word $L_i$ (correspondingly, $R_i$) obtained from the equation
$$
w  \doteq  w_{1,1}w_{2,1}\dots w_{l,1}\dots w_{1,\kk}\dots w_{l,\kk},
$$
see  Corollary \ref{cor:DMM}, setting some of the words $w_{i,\varsigma_i(j)}$ to be trivial and applying the appropriate permutations $\varsigma_i$.

The relations $\mathcal{R}_c(L_i)$ (correspondingly, $\mathcal{R}_d(R_i)$) are the ones induced by Corollary \ref{cor:DMM}.

We treat the constant items $h_{\const}'$ (see below) separately, since, for every solution every item in $h_{\const}'$ corresponds to an element from $\cA^{\pm 1}$ and, therefore does not need to be subdivided to be taken to the normal form.

The formalisation of the above ideas is rather involved and technical. We refer the reader to Section \ref{sec:expl1} for an example.

\bigskip

Let
$$
h_{\const}'=\{h_i'\in h'\mid \hbox{ there is a coefficient equation of the form } h_i'=a \hbox{ in $\Upsilon'$ for some $a\in \cA$}\}.
$$
Given $L_i$, an element $c$ of $C(L_i)$ is a triple:
\begin{gather}\notag
\begin{split}
\hbox{(system of equations in $h^{(L_i)}\cup h'$;}&\hbox{ a word $NF_c(L_i)$ in $h^{(L_i)}$;} \\
&\hbox{a symmetric subset $\mathcal{R}_c(L_i)$ of $h^{(L_i)}\times h^{(L_i)}$).}
\end{split}
\end{gather}
Suppose first that the length of $L_i$ (treated as a word in $h'$) is greater than $1$. Then the system of equations is defined as follows:
$$
\left\{
\begin{array}{ll}
{h_j'}^\epsilon=h^{(L_i)}_{j,\varsigma_{c,j}(i_{c,j,1})}\cdots h^{(L_i)}_{j,\varsigma_{c,j}(i_{c,j,k_{c,j}})},& \hbox{ for every occurrence ${h_j'}^\epsilon$ in $L_i$, $h_j' \notin h_{\const}'$, $\epsilon=\pm 1$,}\\
{h_j'}^\epsilon=h^{(L_i)}_{j,1}, &\hbox{ for every occurrence ${h_j'}^\epsilon$ in $L_i$, $h_j'\in h_{\const}'$, $\epsilon=\pm 1$}
\end{array}
\right\},
$$
where $1\le i_{c,j,1}<i_{c,j,2}<\dots<i_{c,j,k_{c,j}}\le \kk$, $\kk=\kk(|L_i|, \rr)$, $\varsigma_{c,j}$ is a permutation on $k_{c,j}$ symbols, $h^{(L_i)}_{j,\varsigma_{c,j}(i_{c,j,1})},\ldots,h^{(L_i)}_{j,\varsigma_{c,j}(i_{c,j,k_{c,j}})}\in h^{(L_i)}$, see Corollary \ref{cor:DMM}.

The word $NF_c(L_i)$ is a product of all the variables $h^{(L_i)}$. The variable $h^{(L_i)}_{j_1,k_1}$ is to the left of $h^{(L_i)}_{j_2,k_2}$ in $NF_c(L_i)$ if and only if $(j_1,k_1)$ precedes $(j_2,k_2)$ in the right lexicographical order.

The symmetric subset $\mathcal{R}_c(L_i)$ of $h^{(L_i)}\times h^{(L_i)}$ is defined as follows (cf. Corollary \ref{cor:DMM}):
$$
\begin{array}{ll}
\left(h^{(L_i)}_{j_1,k_1},h^{(L_i)}_{j_2,k_2}\right)\in \mathcal{R}_c(L_i) & \hbox{ for all $j_1>j_2$ and $k_1<k_2$, and} \\
&\\
\left(h^{(L_i)}_{j,\varsigma_{c,j}(k_1)},h^{(L_i)}_{j,\varsigma_{c,j}(k_2)}\right)\in \mathcal{R}_c(L_i)& \hbox{ for all $k_1, k_2$ such that } (k_1-k_2)(\varsigma_{c,j}(k_1)-\varsigma_{c,j}(k_2))<0.
\end{array}
$$

Suppose now that $L_i=h_j'$, then the system of equations is: $\{h_j'=h^{(L_i)}_{j,1}\}$, the word $NF_c(L_i)=h^{(L_i)}_{j,1}$ and $\mathcal{R}_c(L_i)=\emptyset$.

If $L_i$ is just a constant ${a_{i_1}}^{\pm 1}$, we define $C(L_i)=\{c_i\}$, $c_i=(\emptyset,a_{i_1}^{\pm 1},\emptyset)$.

The construction of the set $C(R_i)$ is analogous.

\begin{rem}
Notice that the sets $C(L_i)$ and of $C(R_i)$ can be effectively constructed and that their cardinality is bounded above by a computable function of $|L_i|$, $|R_i|$ and $\rr$.
\end{rem}

Given a generalised equation $\Omega'=\langle \Upsilon', \Re_{\Upsilon'}\rangle$ over $\Tr$, where $\Upsilon'=\{L_1=R_1,\dots, L_m=R_m\}$  we construct a finite set of generalised equations $\GE(\Omega')$ over $\FF$
\glossary{name={$\GE(\Omega')$}, description={the set of all constrained generalised equations $\Omega_T$ over $\FF$ constructed for the constrained generalised equation $\Omega'$} over $\Tr$, sort=G}
$$
\GE (\Omega')=\left\{ \Omega_T\mid T=(c_1,\dots,c_{m},d_1,\dots,d_{m}), \hbox{ where } c_i\in C(L_i), d_i\in C(R_i)\right\},
$$
and $\Omega_T=\gpof{T}$ is constructed as follows. The generalised equation $\Upsilon_T$ in variables $h=\bigcup\limits_{i=1}^m\left(h^{(L_i)}\cup h^{(R_i)}\right)$ consists of the following equations:
\begin{enumerate}
    \item[(a)] Equating the normal forms of $L_i$ and $R_i$:
    $$
    NF_{c_i}(L_i)=NF_{d_i}(R_i), \hbox{ for }i=1,\dots, m.
    $$
    \item[(b)]  Equating different decompositions (as words in the $h$'s) of the same variable $h_j'$ of $\Upsilon'$:
        \begin{itemize}
    \item   $\left(h^{(L_{i_1})}_{j,\varsigma_{c_{i_1},j}(i_{c_{i_1},j,1})}\cdots h^{(L_{i_1})}_{j,\varsigma_{c_{i_1},j} (i_{ c_{i_1},j, k_{c_{i_1},j}})}\right)^\epsilon= \left(h^{(L_{i_2})}_{j,\varsigma_{c_{i_2},j}(i_{c_{i_2},j,1})}\cdots h^{(L_{i_2})}_{j,\varsigma_{c_{i_2},j} (i_{ c_{i_2},j, k_{c_{i_2},j}})}\right)^\delta$, for any two occurrences ${h_j'}^\epsilon$ and ${h_j'}^\delta$ of $h_j'$ in equations of $c_{i_1}$ and $c_{i_2}$, respectively, $\epsilon,\delta\in\{-1,1\}$;
    \item  $\left(h^{(L_{i_1})}_{j,\varsigma_{c_{i_1},j}(i_{c_{i_1},j,1})}\cdots h^{(L_{i_1})}_{j,\varsigma_{c_{i_1},j} (i_{ c_{i_1},j, k_{c_{i_1},j}})}\right)^\epsilon= \left(h^{(R_{i_2})}_{j,\varsigma_{d_{i_2},j}(i_{d_{i_2},j,1})}\cdots h^{(R_{i_2})}_{j,\varsigma_{d_{i_2},j} (i_{ d_{i_2},j, k_{d_{i_2},j}})}\right)^\delta$,
            for any two occurrences ${h_j'}^\epsilon$ and ${h_j'}^\delta$ of $h_j'$ in equations of $c_{i_1}$ and $d_{i_2}$ respectively, $\epsilon,\delta\in\{-1,1\}$;
    \item $\left(h^{(R_{i_1})}_{j,\varsigma_{d_{i_1},j}(i_{d_{i_1},j,1})}\cdots h^{(R_{i_1})}_{j,\varsigma_{d_{i_1},j} (i_{ d_{i_1},j, k_{d_{i_1},j}})}\right)^\epsilon=
            \left(h^{(R_{i_2})}_{j,\varsigma_{d_{i_2},j}(i_{d_{i_2},j,1})}\cdots h^{(R_{i_2})}_{j,\varsigma_{d_{i_2},j} (i_{ d_{i_2},j, k_{d_{i_2},j}})}\right)^\delta$,
            for any two occurrences ${h_j'}^\epsilon$ and ${h_j'}^\delta$ of $h_j'$ in equations of $d_{i_1}$ and $d_{i_2}$ respectively, $\epsilon,\delta\in\{-1,1\}$;
        \end{itemize}
\end{enumerate}
The set of boundary connections of $\Upsilon_T$ is empty.

We now define $\Re_{\Upsilon_T}$. Set
\begin{itemize}
\item $\Re_{\Upsilon_T}\left(h^{(L_{i_1})}_{j_1,k_1},h^{(L_{i_2})}_{j_2,k_2}\right)$, $\Re_{\Upsilon_T}\left(h^{(R_{i_1})}_{j_1,k_1},h^{(L_{i_2})}_{j_2,k_2}\right)$, $\Re_{\Upsilon_T}\left(h^{(R_{i_1})}_{j_1,k_1},h^{(R_{i_2})}_{j_2,k_2}\right)$ if $\Re_{\Upsilon'}(h_{j_1}',h_{j_2}')$;
\item $\Re_{\Upsilon_T}\left(h^{(L_{i_1})}_{j_1,k_1},h^{(L_{i_2})}_{j_2,k_2}\right)$ if $\left(h^{(L_{i_1})}_{j_1,k_1},h^{(L_{i_2})}_{j_2,k_2}\right)\in \mathcal{R}_{c_i}(L_i)$ for some $i$;
\item $\Re_{\Upsilon_T}\left(h^{(R_{i_1})}_{j_1,k_1},h^{(R_{i_2})}_{j_2,k_2}\right)$ if $\left(h^{(R_{i_1})}_{j_1,k_1},h^{(R_{i_2})}_{j_2,k_2}\right)\in \mathcal{R}_{d_i}(R_i)$ for some $i$;
\end{itemize}
We define $\Re_{\Upsilon_T}$ to be the minimal subset of $h\times h$ that is symmetric and satisfies condition ($\star$) of Definition \ref{defn:Re}.
This defines a combinatorial generalised equation, see Lemma \ref{lem:cgege}.

\subsubsection{Coordinate groups of generalised equations over $\Tr$ and over $\FF$}

We now explain the relation between the coordinate group of the generalised equation $\Omega'$ over $\Tr$ and the coordinate group of the generalised equation $\Omega_T$ over $\FF$. For any variable $h_j'\in h'$ we choose an arbitrary equation in $c_i$, $c_i\in T$ (or $d_i$, $d_i\in T$) of the form
${h_j'}^\epsilon=h^{(L_i)}_{j,\varsigma_{c_i,j}(i_{c,j,1})}\cdots h^{(L_i)}_{j,\varsigma_{c_i,j}(i_{c_i,j,k_{c_i,j}})}$ and define a word $P_{h_j'}(h,\cA)\in \GG[h]$ as follows:
$$
P_{h_j'}(h,\cA)=\left(h^{(L_i)}_{j,\varsigma_{c_i,j}(i_{c_i,j,1})}\cdots h^{(L_i)}_{j,\varsigma_{c_i,j}(i_{c_i,j,k_{c_i,j}})}\right)^\epsilon.
$$
The word $P_{h_j'}(h,\cA)$ depends on the choice of the equation in $c_i$ and on $i$. It follows from the construction above, that the map
$$
h'\to \GG[h] \hbox{ defined by } h_j'\mapsto P_{h_j'}(h,\cA)
$$
gives rise to a $\GG$-homomorphism $\pi_{\Omega_T}:\GG_{R({\Omega'}^\ast)}\to \GG_{R({\Omega_T}^\ast)}$. Observe that the image $\pi_{\Omega_T}(h_j')\in \GG_{R({\Omega_T}^\ast)}$ does not depend on a particular choice of $i$ and of the equation in $c_i$ (or in $d_i$), since equations in $\Omega$ make all of them identical. Hence, $\pi_{\Omega_T}$ depends only on $\Omega_T$.

If $H$ is an arbitrary solution of the generalised equation $\Omega_T$, then $H'=(P_{h_1'}(H,\cA),\dots,P_{h_{\rho'}'}(H,\cA))$ is a solution of $\Omega'$.

The converse also holds. From Corollary \ref{cor:DMM}, the fact that DM-normal form is invariant with respect to inversion and definition of solution of a generalised equation, it follows that any solution $H'$ of $\Omega'$ induces a solution $H$ of $\Omega_T$. Furthermore, if $H'=(H_1',\dots,H_{\rho'}')$ then $H_j'=P_{h_j'}(H,\cA)$.

The following lemma shows that to describe the set of solutions of a constrained generalised equation over $\Tr$ is equivalent to describe the set of solutions of constrained generalised equations over $\FF$. This lemma can be viewed as the second ``divide'' step of the process.

\begin{lem} \label{lem:R1}
Let $\Omega'$ be a generalised equation in variables $h'=\{h_1',\dots, h_{\rho'}'\}$ over $\Tr$, one can effectively construct a finite set $\GE(\Omega')$ of constrained generalised equations over $\FF$ such that
\begin{enumerate}
    \item if the set $\GE(\Omega')$ is empty, then $\Omega'$ has no solutions;
    \item for each $\Omega (h) \in\GE(\Omega')$ and for each $h_i' \in h'$ one can effectively find a word $P_{h_i'}(h,\cA) \in \GG[h]$ of length at most $|h|$ such that the map $h_i' \mapsto P_{h_i'}(h,\cA)$ {\rm(}$h_i'\in h'${\rm)} gives rise to a $\GG$-homomorphism $\pi_{\Omega}: \GG_{R({\Omega'}^\ast)}\to\GG_{R(\Omega^\ast)}$ {\rm(}in particular, for every solution $H$ of the generalised equation $\Omega$ one has that $P(H)$ is a solution of the generalised equation $\Omega'$, where $P(h) = (P_{h_1'}, \ldots, P_{h_{\rho'}'})${\rm)};
    \item for any solution $H' \in \GG^n$ of $\Omega'$ there exists $\Omega (h) \in \GE(\Omega')$ and a solution $H$ of $\Omega(h)$ such that $H' = P(H)$, where $P(h) = (P_{h_1'}, \ldots, P_{h_{\rho'}'})$, and this equality holds in the free monoid $\FF$.
\end{enumerate}
\end{lem}

Combining Lemma \ref{le:14} and Lemma \ref{lem:R1} we get the following

\begin{cor} \label{co:R1}
In the notation of {\rm Lemmas \ref{le:14}} and {\rm \ref{lem:R1}}  for any solution $W(\cA)  \in \GG^n=\GG(\cA)^n$ of the system $S(X,\cA)=1$ there exist a generalised equation $\Omega' (h') \in \GE'(S)$, a solution $H'(\cA)$ of $\Omega'(h')$, a generalised equation $\Omega (h) \in \GE(S)$, and a solution $H(\cA)$ of $\Omega(h)$ such that the following diagram commutes
$$
\xymatrix@C3em{
 \GG_{R(S)}  \ar[rd]_{\pi_W} \ar[r]^{\pi_{\Omega'}} &  \GG_{R({\Omega'}^\ast)}\ar[d]^{\pi_{H'}} \ar[r]^{\pi_{\Omega}} & \GG_{R(\Omega^\ast)} \ar[ld]^{\pi_H}\\
                                             &     \GG                                           &
}
$$
\end{cor}

\subsection{Example} \label{sec:expl1}
The aim of Section \ref{sec:red} is to show that to a given finite system of equations $S=S(X,\cA) = 1$ over a partially commutative group $\GG$ one can associate a finite collection of (constrained) generalised equations $\GE(S)$ over $\FF$ with coefficients from $\cA^{\pm 1}$ such that
for any solution $W(\cA) \in \GG^n$ of $S$ there exists a generalised equation $\Omega (h)=\Omega$ over $\FF$ and a solution $H(\cA)$ of $\Omega$ such that $$
\xymatrix@C3em{
 \GG_{R(S)}  \ar[rd]_{\pi_W}  \ar[rr]^{\pi_{\Omega}} & & \GG_{R(\Omega^\ast)} \ar[ld]^{\pi_H}\\
     &\GG                                           &
}
$$

The family of solutions of the generalised equations from $\GE'(S)$ describes all solutions of the system $S(X,\cA) = 1$, see Lemma \ref{le:14}.

In this section for a particular equation over a given partially commutative group $\GG$ and one of its solutions, following the exposition of Section \ref{sec:red}, we first construct the constrained generalised equation over $\Tr$ and then the generalised equation over $\FF$ such that the above diagram commutes.

\begin{figure}[!h]
  \centering
   \includegraphics[keepaspectratio,height=1.5in]{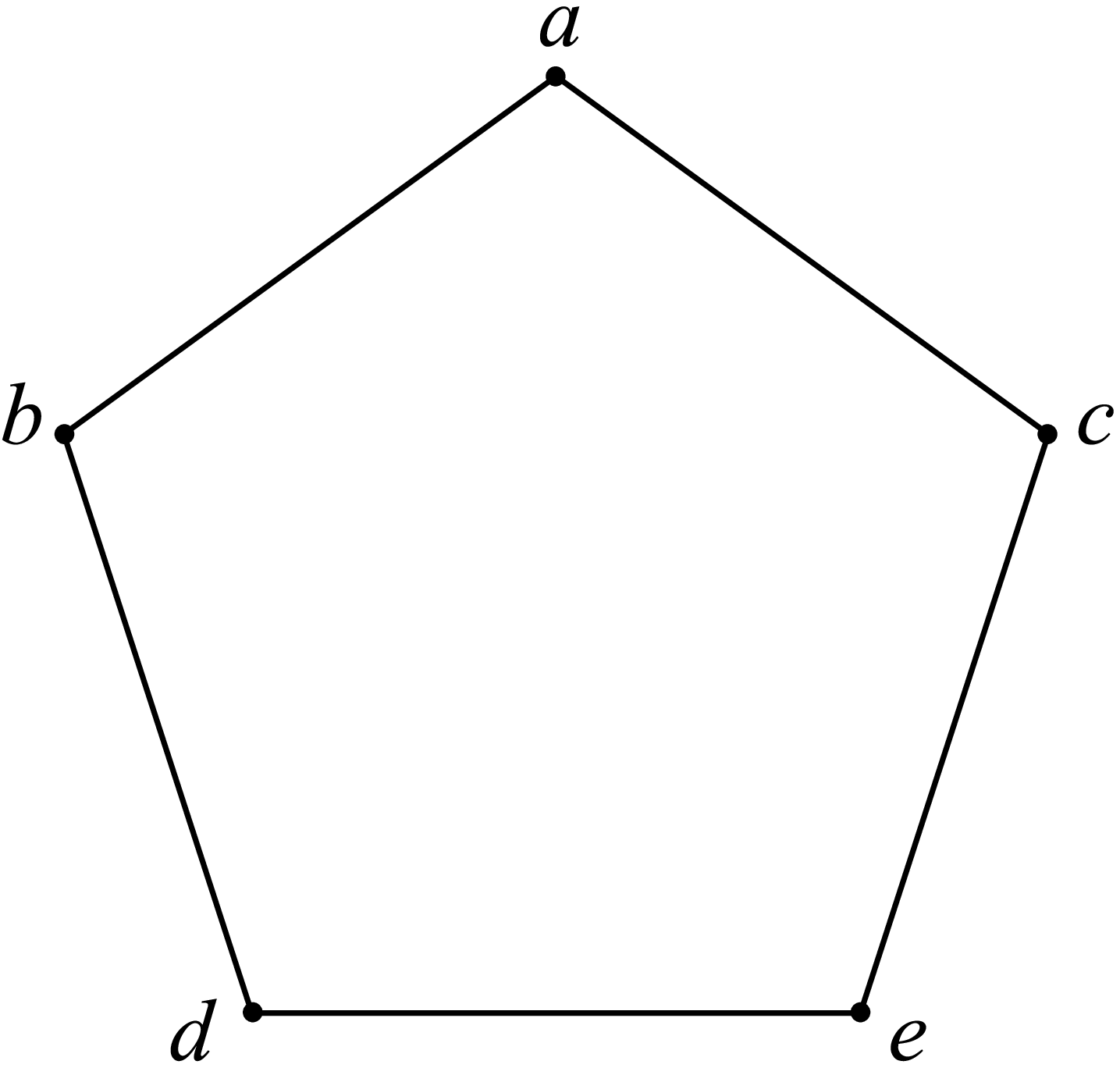}
\caption{The commutation graph of $\GG$.} \label{fig:pentagon}
\end{figure}

Let $\GG$ be the free partially commutative group whose underlying commutation graph is a pentagon, see Figure \ref{fig:pentagon}. Let $S(x,y,z,\cA)=\{xyzy^{-1}x^{-1}z^{-1}eb e^{-1}b^{-1}=1\}$ be a system consisting of a single equation over $\GG$ in variables $x$, $y$ and $z$. Since the word
$$
S(bac,\, c^{-1}a^{-1}d,\,e)=bac \, c^{-1}a^{-1}d \, e \, d^{-1}ac \, c^{-1}a^{-1}b^{-1} \, e^{-1}\, e b e^{-1}b^{-1}
$$
is trivial in $\GG$, the tuple $x=bac$, $y=c^{-1}a^{-1}d$, $z=e$ is a solution of $S$. Construct a van Kampen diagram $\mathcal{D}$ for the word $S(bac,\, c^{-1}a^{-1}d,\,e)=1$ and consider the underlying cancellation scheme as in Proposition \ref{lem:prod}, see Figure \ref{fig:example}.

\begin{figure}[!h]
  \centering
   \includegraphics[keepaspectratio,width=5in]{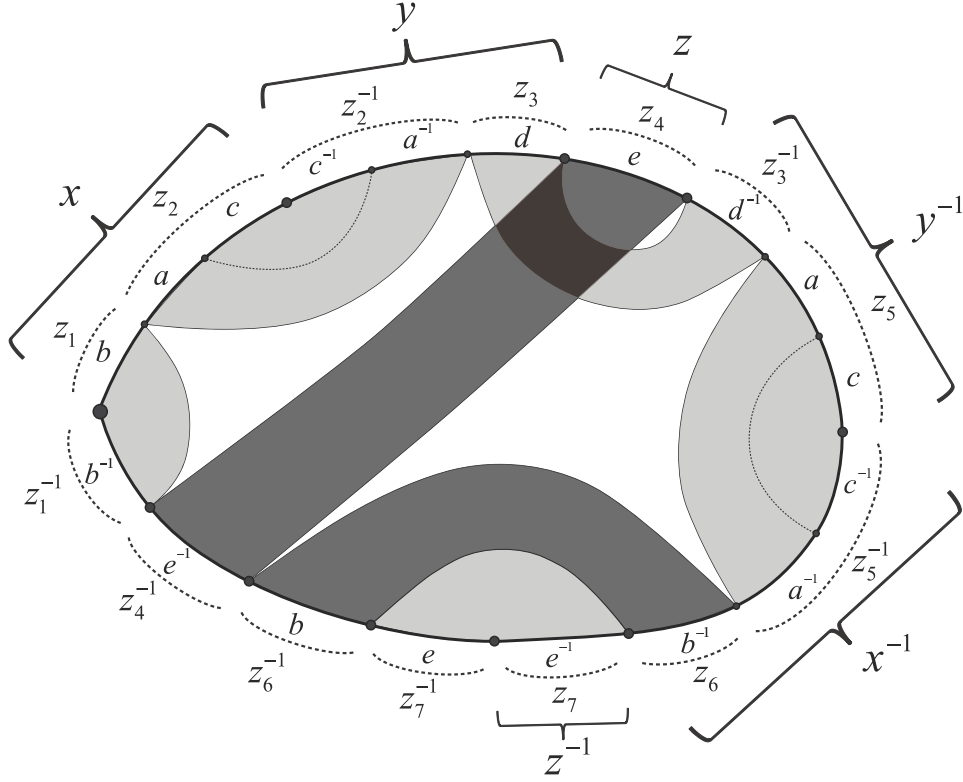}
\caption{Cancellation scheme for the solution $x=bac$, $y=c^{-1}a^{-1}d$, $z=e$ of $S$.} \label{fig:example}
\end{figure}

\subsubsection*{Content of Section \ref{sec:Gparttab}}

Write  the equation $S(x,y,z,\cA)$ in the form $xyzy^{-1}x^{-1}z^{-1}eb e^{-1}b^{-1}=r_{1,1}\dots r_{1,10}=1$, see Equation (\ref{*}).

By the cancellation scheme shown on Figure \ref{fig:example}, we construct a partition table $\TB$. To every end of a  band on Figure \ref{fig:example}, we associate a variable $z_i$ or its inverse $z_i^{-1}$. We can write the variables $x,y,z$ as words in the variables $z_i$'s. Two variables $z_i$ and $z_j$ commute if and only if the corresponding bands cross. Thus, the partition table associated to the equation $S(x,y,z,\cA)$ and its solution $x=bac$, $y=c^{-1}a^{-1}d$, $z=e$ is given below
\begin{gather}\notag
\begin{split}
&\TB=(\{V_{1i}(z_1,\dots,z_7)\},\quad \HH=\langle \GG, z_1,\dots,z_7\mid [z_3,z_4]=1\rangle, \quad i=1,\dots,10\\
&
\begin{array}{llll}
V_{1,1}=z_1 z_2;         &    V_{1,4}=z_3^{-1}z_5;     &    V_{1,7}=z_7^{-1}&V_{1,9}=z_4^{-1};\\
V_{1,2}=z_2^{-1} z_3;    &    V_{1,5}=z_5^{-1}z_6 ;    &    V_{1,8}=z_6^{-1}&V_{1,10}=z_1^{-1}.\\
V_{1,3}=z_4'             &    V_{1,6}=z_7;             &    &
\end{array}
\end{split}
\end{gather}
In the above notation, it is easy to see that every variable $z_i^{\pm 1}$, $i=1,\dots, 7$ occurs in the words $V_{1,j}$ precisely once, and that $V_1=V_{1,1}\cdots V_{1,10}=_\HH 1$.

\subsubsection*{Content of Section \ref{sec:geneqT}}
We now construct the generalised equation $\Omega'_\TB$ associated to the partition table $\TB$. Consider the word $\mathcal{V}$ in the free monoid $\FF(z_1,\dots,z_7, z_1^{-1},\dots,z_7^{-1})$:
$$
\mathcal{V}=V_1=z_1 z_2 z_2^{-1} z_3 z_4 z_3^{-1} z_5 z_5^{-1} z_6 z_7 z_7^{-1} z_6^{-1} z_4^{-1} z_1^{-1}=y_1\cdots y_{14}.
$$
The variables of the generalised equation $\Omega'_\TB$ are $h_1',\dots, h_{14}'$ (equivalently, the set of boundaries of the combinatorial generalised equation is $\{1,\dots, 15\}$).

For every pair of distinct occurrences of $z_i$ in $\mathcal{V}$ we construct a pair of dual variable bases and the corresponding basic equation. For example, $z_1=y_1$ and $z_1^{-1}=y_{14}$.  We introduce a pair of dual bases $\mu_{z_1,1}$, $\mu_{z_1,14}$ and $\Delta(\mu_{z_1,1})=\mu_{z_1,14}$ so that
$$
\alpha(\mu_{z_1,1})=1,\ \beta(\mu_{z_1,1})=2,\quad \alpha(\Delta(\mu_{z_1,1}))=14,\  \beta(\Delta(\mu_{z_1,1}))=15, \quad \varepsilon(\mu_{z_1,1})=1, \ \varepsilon(\mu_{z_1,14})=-1
$$
The corresponding basic equation is $h_1'={h_{14}'}^{-1}$, see Figure \ref{fig:geneq}. The following table describes all basic equations thus constructed.
\begin{center}
\begin{tabular}{c|c|c}
pair of occurrences&  pair of dual bases & corresponding basic equation\ \\
    \hline
$y_1=z_1, y_{14}=z_1^{-1}$ & $\Delta(\mu_{z_1,1})=\mu_{z_1,14}$, $\varepsilon(\mu_{z_1,1})=-\varepsilon(\mu_{z_1,14})=1$ & $\{h_1'={h_{14}'}^{-1}\}=\{L_1=R_1\}$   \\
$y_2=z_2, y_{3}=z_2^{-1} $     &   $  \Delta(\mu_{z_2,2})=\mu_{z_2,3}   $, $\varepsilon(\mu_{z_2,2})=-\varepsilon(\mu_{z_2,3})=1$  & $\{h_2'={h_{3}'}^{-1}\}=\{L_2=R_2\}$\\
$y_4=z_3, y_{6}=z_3^{-1} $     &   $  \Delta(\mu_{z_3,4})=\mu_{z_3,6}   $, $\varepsilon(\mu_{z_3,4})=-\varepsilon(\mu_{z_3,6})=1$    & $\{h_{4}'={h_{6}'}^{-1}\}=\{L_3=R_3\}$\\
$y_5=z_4, y_{13}=z_4^{-1}$     &   $   \Delta(\mu_{z_4,5})=\mu_{z_4,13} $, $\varepsilon(\mu_{z_4,5})=-\varepsilon(\mu_{z_4,13})=1$    & $\{h_{5}'={h_{13}'}^{-1}\}=\{L_4=R_4\}$\\
$y_7=z_5, y_{8}=z_5^{-1} $     &   $  \Delta(\mu_{z_5,7})=\mu_{z_5,8}   $, $\varepsilon(\mu_{z_5,7})=-\varepsilon(\mu_{z_5,8})=1$    & $\{h_{7}'={h_{8}'}^{-1}\}=\{L_5=R_5\}$\\
$y_9=z_6, y_{12}=z_6^{-1}$     &   $  \Delta(\mu_{z_6,9})=\mu_{z_6,12}  $, $\varepsilon(\mu_{z_6,9})=-\varepsilon(\mu_{z_6,12})=1$ & $\{h_{9}'={h_{12}'}^{-1}\}=\{L_6=R_6\}$\\
$y_{10}=z_7, y_{11}=z_7^{-1}$  &   $  \Delta(\mu_{z_7,10})=\mu_{z_7,11} $, $\varepsilon(\mu_{z_7,10})=-\varepsilon(\mu_{z_7,11})=1$    & $\{h_{10}'={h_{11}'}^{-1}\}=\{L_7=R_7\}$
\end{tabular}
\end{center}

For every pair of distinct occurrences of $x$ (correspondingly, of $y$, or of $z$) in $r_{1,1}\cdots r_{1,10}=1$, we construct a pair of dual variable bases and the corresponding basic equation. For example, $r_{1,1}=x$, $r_{1,5}=x^{-1}$. We introduce a pair of dual bases $\mu_{x,q_1}$, $\Delta(\mu_{x,q_1})$, where $q_1=(1,1,1,5)$ so that
$$
\alpha(\mu_{x,q_1})=1,\  \beta(\mu_{x,q_1})=3,\quad  \alpha(\Delta(\mu_{x,q_1}))=8, \ \beta(\Delta(\mu_{x,q_1}))=10,\quad  \varepsilon(\mu_{x,q_1})=1,\  \varepsilon(\Delta(\mu_{x,q_1}))=-1.
$$
The corresponding basic equation is $h_1'h_2'={\left(h_{8}'h_{9}'\right)}^{-1}$, see Figure \ref{fig:geneq}. The following table describes all basic equations thus constructed.
\begin{center}
\begin{tabular}{c|c|c}
 pair of occurrences  & pair of dual bases & corresponding basic equation\\
    \hline
$r_{1,1}=x, r_{1,5}=x^{-1}$ &   $\Delta(\mu_{x,q_1}),\mu_{x,q_1}$, $\varepsilon(\mu_{x,q_1})=-\varepsilon(\mu_{x,1,5})=1$&   $\{h_1'h_2'={\left(h_{8}'h_{9}'\right)}^{-1}\}=\{L_8=R_8\}$\\
$r_{1,2}=y, r_{1,4}=y^{-1}$ &   $\Delta(\mu_{y,q_2}),\mu_{y,q_2}$, $\varepsilon(\mu_{y,q_2})=-\varepsilon(\mu_{y,q_2})=1$&    $\{h_3'h_4'={\left(h_{6}'h_{7}'\right)}^{-1}\}=\{L_9=R_9\}$\\
$r_{1,3}=z, r_{1,6}=z^{-1}$ &   $\Delta(\mu_{z,q_3}),\mu_{z,q_3}$, $\varepsilon(\mu_{z,q_3})=-\varepsilon(\mu_{z,q_3})=1$ &   $\{h_5'={h_{10}'}^{-1}\}=\{L_{10}=R_{10}\}$
\end{tabular}
\end{center}
Here $q_1=(1,1,1,5)$, $q_2=(1,2,1,4)$, $q_3=(1,3,1,6)$.

For $r_{1j}=a\in \cA^{\pm 1}$, we construct a constant base and a coefficient equation. For example, $r_{1,7}=e$. We introduce a constant base $\mu_{1,11}$ so that
$$
\alpha(\mu_{1,11})=11, \ \beta(\mu_{1,11})=12.
$$
The corresponding coefficient equation is ${h_{11}'}=e$, see Figure \ref{fig:geneq}. The following table describes all coefficient equations.
\begin{center}
\begin{tabular}{c|c|c}
\phantom{aa} an occurrence \phantom{aa} & constant base  &\phantom{aa} corresponding coefficient equation \phantom{aa}\\
    \hline
$r_{1,11}=e$        & $\mu_{1,11}=e$  &   $\{{h_{11}'}=e\}=\{L_{11}=R_{11}\}$\\
$r_{1,12}=b$        & $\mu_{1,12}=b$ &   $\{{h_{12}'}=b\}=\{L_{12}=R_{12}\}$\\
$r_{1,13}=e^{-1}$   & $\mu_{1,13}={e}^{-1}$&   $\{{h_{13}'}^{-1}=e\}=\{L_{13}=R_{13}\}$\\
$r_{1,14}=b^{-1}$   & $\mu_{1,14}=b^{-1}$ &   $\{{h_{14}'}^{-1}=b\}=\{L_{14}=R_{14}\}$
\end{tabular}
\end{center}
The set of boundary connections of $\Upsilon_\TB'$ is empty. This defines the generalised equation $\Upsilon_\TB'$ shown on Figure \ref{fig:geneq}.

The set of pairs $(h_i',h_j')$ such that $[y_i, y_j]=1$\ in $\HH$ and $y_i\ne y_j$ consists of:
$$
\{(h'_4,h'_5), (h'_5,h'_4), (h'_5,h'_6), (h'_6,h'_5), (h'_4,h'_{13}), (h'_{13},h'_4), (h'_6,h'_{13}), (h'_{13},h'_6) \}.
$$

To define  the relation $\Re_{\Upsilon_{\TB}'}$ we have to make sure that the set $\Re_{\Upsilon_{\TB}'}\subseteq h'\times h'$ is symmetric and satisfies condition ($\star$) of Definition \ref{defn:Re}. Since we have the equations $h_5'={h_{10}'}^{-1}$, $h_{10}'={h_{11}'}^{-1}$, we get that
$$
\Re_{\Upsilon_{\TB}'}=\{(h'_4,h'_5), (h'_5,h'_6), (h'_4,h'_{10}), (h'_6,h'_{10}), (h'_4,h'_{11}), (h'_6,h'_{11}), (h'_4,h'_{13}), (h'_6,h'_{13})\}.
$$
Clearly, the relation $\Re_{\Upsilon_{\TB}'}$ has to be symmetric, but, in this section, we further write only one of the two symmetric pairs. This defines the constrained generalised equation $\Omega_\TB'=\langle \Upsilon_\TB', \Re_{\Upsilon_{\TB}'}\rangle$.

We  represent generalised equations graphically in the way shown on Figure \ref{fig:geneq}.
\begin{figure}[!h]
  \centering
   \includegraphics[keepaspectratio,width=6in]{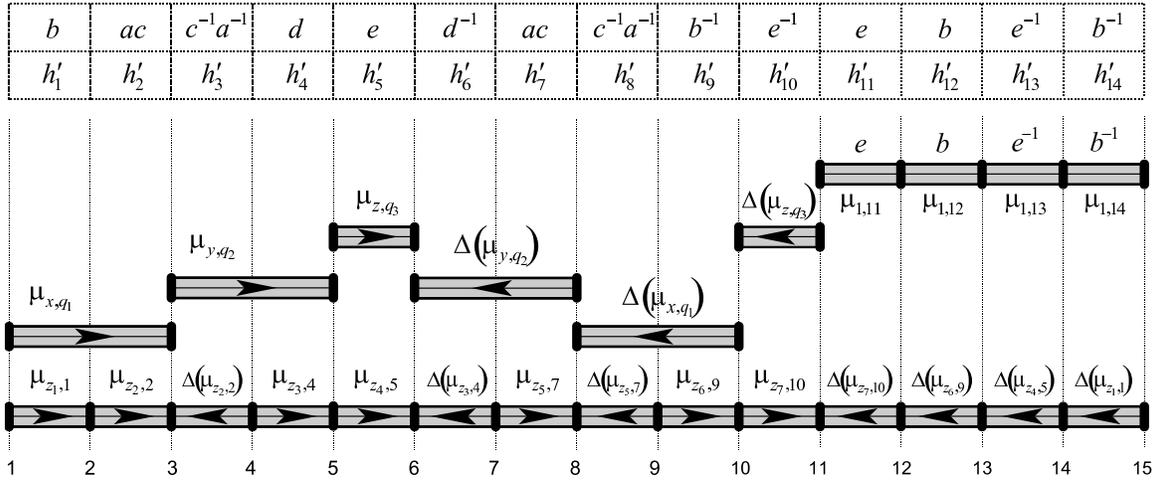}
\caption{The generalised equation $\Omega_\TB'$.} \label{fig:geneq}
\end{figure}

\subsubsection*{Content of Section \ref{sec:fromTtoF}}

By Corollary \ref{co:R1} for every pair $(\Omega'_\TB, H')$ there exists a pair $(\Omega_T,H)$ such that the diagram given in Corollary \ref{co:R1} is commutative.  We now construct the set $T=(c_1,\dots,c_{14}, d_1,\dots, d_{10})$ that defines a generalised equation $\Omega_T$ over $\FF$. We construct the elements $c_i$ of the sets $C(L_i)$ that correspond to the solution $x=bac$, $y=c^{-1}a^{-1}d$, $z=e$.

Recall that the partially commutative group $\GG$ is defined by the commutation graph shown on Figure \ref{fig:pentagon}. The definition of DM-normal form in the partially commutative monoid $\Tr$ is given by induction on the number of thin clans. In our example thin clans of $\Tr$ are $\{a,a^{-1}\}$, $\{b,b^{-1}\}$, $\{c,c^{-1}\}$, $\{d,d^{-1}\}$ and $\{e,e^{-1}\}$.

Consider the DM-normal form in $\Tr$ with respect to the thin clan $\{a,a^{-1}\}$. Note that for $L_8=h_1'h_2'$, we have that $L_8(bac,c^{-1}a^{-1}d,e)=bac$ is not written in the DM-normal form.  We introduce a new variable and write $h_1'=h^{(L_8)}_{1,2}$, $h_2'=h^{(L_8)}_{2,1}h^{(L_8)}_{2,2}$  (here the corresponding permutation $\varsigma_{c,2}$ on $\{1,2\}$ is trivial) and $NF_{c_8}(L_8)=h^{(L_8)}_{2,1}h^{(L_8)}_{1,2}h^{(L_8)}_{2,2}$. Notice that
$NF_{c_8}(L_8)(bac,c^{-1}a^{-1}d,e)=abc$ is in the DM-normal form.

Below we describe all the elements $c_i$ for the solution $x=bac$, $y=c^{-1}a^{-1}d$, $z=e$.
\begin{gather}\notag
\begin{split}
&
\begin{array}{ll}
C(L_1)\ni c_1=\left(\left\{h_1'=h^{(L_1)}_{1,1}\right\};\ h^{(L_1)}_{1,1};\ \emptyset\right); &    C(L_7)\ni c_7=\left(\left\{h_{10}'=h^{(L_7)}_{10,1}\right\};\ h^{(L_7)}_{10,1};\ \emptyset\right);\\
C(L_2)\ni c_2=\left(\left\{h_2'=h^{(L_2)}_{2,1}\right\};\ h^{(L_2)}_{2,1};\ \emptyset\right);&    C(L_{10})\ni c_{10}=\left(\left\{h_5'=h^{(L_{10})}_{5,1}\right\};\ h^{(L_{10})}_{5,1};\ \emptyset\right);\\
C(L_3)\ni c_3=\left(\left\{h_4'=h^{(L_3)}_{4,1}\right\};\ h^{(L_3)}_{4,1};\ \emptyset\right);&    C(L_{11})\ni c_{11}=\left(\left\{h_{11}'=h^{(L_{11})}_{11,1}\right\};\ h^{(L_{11})}_{11,1};\ \emptyset\right);\\
C(L_4)\ni c_4=\left(\left\{h_5'=h^{(L_4)}_{5,1}\right\};\ h^{(L_4)}_{5,1};\ \emptyset\right);&    C(L_{12})\ni c_{12}=\left(\left\{h_{12}'=h^{(L_{12})}_{12,1}\right\};\ h^{(L_{12})}_{12,1};\ \emptyset\right);\\
C(L_5)\ni c_5=\left(\left\{h_7'=h^{(L_5)}_{7,1}\right\};\ h^{(L_5)}_{7,1};\ \emptyset\right);&    C(L_{13})\ni c_{13}=\left(\left\{{h_{13}'}^{-1}=h^{(L_{13})}_{13,1}\right\};\ h^{(L_{13})}_{13,1};\ \emptyset\right)\\
C(L_6)\ni c_6 =\left(\left\{h_9'=h^{(L_6)}_{9,1}\right\};\ h^{(L_6)}_{9,1};\ \emptyset\right);&    C(L_{14})\ni c_{14}=\left(\left\{{h_{14}'}^{-1}=h^{(L_1)}_{14,1}\right\};\ h^{(L_{14})}_{14,1};\ \emptyset\right);
\end{array}
\\
&\ C(L_8)\ni c_8=\left(\left\{h_1'=h^{(L_8)}_{1,2},h_2'=h^{(L_8)}_{2,1}h^{(L_8)}_{2,2}\right\};\  h^{(L_8)}_{2,1}h^{(L_8)}_{1,2}h^{(L_8)}_{2,2};\
\mathcal{R}_c(L_8)=\left\{\left(h^{(L_8)}_{1,2},h^{(L_8)}_{2,1}\right)\right\}\right); \\
&\ C(L_9)\ni c_9=\left(\left\{h_3'=h^{(L_9)}_{3,1}, h_4'=h^{(L_9)}_{4,1}\right\};\ h^{(L_9)}_{3,1}h^{(L_9)}_{4,1};\ \emptyset \right).
\end{split}
\end{gather}

Analogously, we describe all the elements $d_i$ for the solution $x=bac$, $y=c^{-1}a^{-1}d$, $z=e$.
\begin{gather}\notag
\begin{split}
&
\begin{array}{ll}
C(R_1)\ni d_1=\left(\left\{h_{14}'=h^{(R_{1})}_{14,1}\right\};\   h^{(R_{1})}_{14,1};\ \emptyset\right);& C(R_7)\ni d_7=\left(\left\{h_{11}'=h^{(R_{7})}_{11,1}\right\};\   h^{(R_{7})}_{11,1};\ \emptyset\right);\\
C(R_2)\ni d_2=\left(\left\{h_{3}'=h^{(R_{2})}_{3,1}\right\};\     h^{(R_{2})}_{3,1};\ \emptyset\right);&  C(R_{10})\ni d_{10}=\left(\left\{{h_{10}'}^{-1}=h^{(R_{10})}_{10,1}\right\};\ h^{(R_{10})}_{10,1};\ \emptyset\right);  \\
C(R_3)\ni d_3=\left(\left\{h_6'=h^{(R_{3})}_{6,1}\right\};\       h^{(R_{3})}_{6,1};\ \emptyset\right);&  C(R_{11})\ni d_{11}=\left(\emptyset;\ e;\ \emptyset\right);   \\
C(R_4)\ni d_4=\left(\left\{h_{13}'=h^{(R_{4})}_{13,1}\right\};\   h^{(R_{4})}_{13,1};\ \emptyset\right);&  C(R_{12})\ni d_{12}=\left(\emptyset;\ b;\ \emptyset\right);\\
C(R_5)\ni d_5=\left(\left\{h_8'=h^{(R_{5})}_{8,1}\right\};\       h^{(R_{5})}_{8,1};\ \emptyset\right);&  C(R_{13})\ni d_{13}=\left(\emptyset;\ e^{-1};\ \emptyset\right);\\
C(R_6)\ni d_6=\left(\left\{h_{12}'=h^{(R_{6})}_{12,1}\right\};\   h^{(R_{6})}_{12,1};\ \emptyset\right);&C(R_{14})\ni d_{14}=\left(\emptyset;\ b^{-1};\ \emptyset\right);\\
\end{array}
\\
&\ C(R_8)\ni d_8=\left(\left\{{h_9'}^{-1}=h^{(R_{8})}_{9,1},{h_8'}^{-1}=h^{(R_{8})}_{8,1}h^{(R_{8})}_{8,2}\right\};\ h^{(R_{8})}_{8,1}h^{(R_{8})}_{9,1}h^{(R_{8})}_{8,2};\ \mathcal{R}_c(R_8)=\left\{\left(h^{(R_{8})}_{8,1},h^{(R_8)}_{9,1}\right)\right\}\right);\\
&\ C(R_9)\ni d_9=\left(\left\{{h_7'}^{-1}=h^{(R_{9})}_{7,1}, {h_6'}^{-1}=h^{(R_{9})}_{6,1}\right\};\ h^{(R_{9})}_{6,1}h^{(R_{9})}_{7,1};\ \emptyset \right).
\end{split}
\end{gather}

For the tuple $T=(c_1,\dots, c_{14}, d_1,\dots, d_{14})$ we construct the generalised equation $\Omega_T$.
Equations corresponding to $NF_{c_i}(L_i)=NF_{d_i}(R_i)$, $i=1,\dots, 14$ are:
$$
\begin{array}{lll}
 h^{(L_1)}_{1,1} = h^{(R_{1})}_{14,1};   &\quad h^{(L_6)}_{9,1} = h^{(R_{6})}_{12,1};&      h^{(L_{11})}_{11,1}  =e;\\
 h^{(L_2)}_{2,1} = h^{(R_{2})}_{3,1};  & \quad  h^{(L_7)}_{10,1}  =  h^{(R_{7})}_{11,1};  &h^{(L_{12})}_{12,1}  =b;\\
 h^{(L_3)}_{4,1}=h^{(R_{3})}_{6,1};& \quad h^{(L_8)}_{2,1}h^{(L_8)}_{1,2}h^{(L_8)}_{2,2} = h^{(R_{8})}_{8,1}h^{(R_{8})}_{9,1}h^{(R_{8})}_{8,2}; \ &
 h^{(L_{13})}_{13,1}  =e^{-1};\\
 h^{(L_4)}_{5,1} = h^{(R_{4})}_{13,1};& \quad h^{(L_9)}_{3,1}h^{(L_9)}_{4,1}  = h^{(R_{9})}_{6,1}h^{(R_{9})}_{7,1};& h^{(L_{14})}_{14,1}  =b^{-1}.\\
 h^{(L_5)}_{7,1} = h^{(R_{5})}_{8,1};  &\quad h^{(L_{10})}_{5,1}=h^{(R_{10})}_{10,10,1};&
 \end{array}
$$

Equations that equate different decompositions of the same variable $h_j'$ are obtained as follows:
$$
\begin{array}{ll}
h_1': \hbox{ from $L_1,L_8$, we get } h^{(L_1)}_{1,1}= h^{(L_8)}_{1,2};&                       h_8': \hbox{ from $R_5,R_8$, we get } h^{(R_{5})}_{8,1}={h^{(R_{8})}_{8,2}}^{-1}{h^{(R_{8})}_{8,1}}^{-1};\\
h_2': \hbox{ from $L_2, L_8$, we get }h^{(L_2)}_{2,1}=h^{(L_8)}_{2,1}h^{(L_8)}_{2,2};&                        h_9':    \hbox{ from $L_6, R_8$, we get } h^{(L_6)}_{9,1}={h^{(R_{8})}_{9,1}}^{-1};\\
h_3': \hbox{ from $L_9, R_2$, we get } h^{(L_9)}_{3,1} ={h^{(R_{2})}_{3,1}}^{-1};&     h_{10}': \hbox{ from $L_7, R_{10}$, we get } h^{(L_7)}_{10,1}={h^{(R_{10})}_{10,1}}^{-1};\\
h_4': \hbox{ from $L_3,L_9$, we get  }h^{(L_3)}_{4,1}=h^{(L_9)}_{4,1};&                        h_{11}': \hbox{ from $L_{11}, R_7$, we get } h^{(L_{11})}_{11,1}=h^{(R_{7})}_{11,1};\\
h_5': \hbox{ from $L_4,L_{10}$, we get } h^{(L_4)}_{5,1}=h^{(L_{10})}_{5,1}; &                  h_{12}': \hbox{ from $L_{12}, R_6$, we get } h^{(L_{12})}_{12,1}=h^{(R_{6})}_{12,1};\\
h_6': \hbox{ from $R_3,R_9$, we get } h^{(R_{3})}_{6,1}={h^{(R_{9})}_{6,1}}^{-1};&                h_{13}': \hbox{ from $L_{13}, R_4$, we get } {h^{(L_{13})}_{13,1}}^{-1}=h^{(R_{4})}_{13,1};\\
h_7': \hbox{ from $L_5, R_9$, we get } h^{(L_5)}_{7,1}={h^{(R_{9})}_{7,1}}^{-1};&                  h_{14}': \hbox{ from $L_{14}, R_1$, we get } {h^{(L_{14})}_{14,1}}{-1}=h^{(R_{1})}_{14,1}.
\end{array}
$$
This defines the generalised equation $\Upsilon_T$ (by the definition, the set of boundary connections of $\Upsilon_T$ is empty).

Set
\begin{itemize}
\item $\Re_{\Upsilon_T}\left(h^{(L_{i_1})}_{j_1,k_1},h^{(L_{i_2})}_{j_2,k_2}\right)$, $\Re_{\Upsilon_T}\left(h^{(R_{i_1})}_{j_1,k_1},h^{(L_{i_2})}_{j_2,k_2}\right)$, $\Re_{\Upsilon_T}\left(h^{(R_{i_1})}_{j_1,k_1},h^{(R_{i_2})}_{j_2,k_2}\right)$ if $\Re_{\Upsilon'}(h_{j_1}',h_{j_2}')$;
\item $\Re_{\Upsilon_T}\left(h^{(L_{i_1})}_{j_1,k_1},h^{(L_{i_2})}_{j_2,k_2}\right)$ if $\left(h^{(L_{i_1})}_{j_1,k_1},h^{(L_{i_1})}_{j_2,k_2}\right)\in \mathcal{R}_{c_i}(L_i)$ for some $i$;
\item $\Re_{\Upsilon_T}\left(h^{(R_{i_1})}_{j_1,k_1},h^{(L_{i_2})}_{j_2,k_2}\right)$ if $\left(h^{(R_{i_1})}_{j_1,k_1},h^{(L_{i_1})}_{j_2,k_2}\right)\in \mathcal{R}_{d_i}(R_i)$ for some $i$;
\end{itemize}
We define $\Re_{\Upsilon_T}$ to be the minimal subset of $h\times h$ that is symmetric and satisfies condition ($\star$) of Definition \ref{defn:Re}.

\begin{rem}
Observe, that the initial system of equations $S(x,y,z,\cA)$ contains a single equation in three variables. The generalised equation $\Omega'_\TB$ is a system of equations in $14$ variables. The generalised equation $\Omega_T$ is a system of equations in at least $30$ variables. (Note that for $\Omega_T$, we constructed the system of equations over the monoid $\FF$. We still should have constructed a combinatorial generalised equation associated to this system, see Lemma \ref{lem:cgege}.)

Notice that, the above considered example describes only one generalised equation (which was traced by a particular solution of $S(x,y,z,\cA)$). Consideration of all the \emph{finite} collection of the generalised equations corresponding to the ``simple'' system of equations $S(x,y,z,\cA)$, can only be done on a computer.

This picture is, in fact, general. The size of the systems grows dramatically and makes difficult to work with examples.
\end{rem}

\section{The process: construction of the tree $T$} \label{se:5}
\begin{flushright}
\parbox{2in}{\textit{``Es werden aber bei uns in der Regel keine aussichtslosen Prozesse gef\"{u}hrt.''}\\ Franz Kafka, ``Der Prozess''}
\end{flushright}

In the previous section we reduced the study of the set of solutions of a system of equations over a partially commutative group to the study of solutions of constrained generalised equations over a free monoid. In order to describe the solutions of constrained generalised equations over a free monoid, in this section we describe a branching rewriting process for constrained generalised equations.

In his important papers \cite{Makanin} and \cite{Mak82}, G.~Makanin devised a process for proving that the compatibility problem of systems of equations over a free monoid (over a free group) is decidable. This process was later developed by A.~Razborov. In his work \cite{Razborov1}, \cite{Razborov3}, gave a complete description of all solutions of a system of equations. A further step was made by O.~Kharlampovich and A.~Miasnikov. In \cite{IFT}, in particular, the authors extend Razborov's result to systems of equations with parameters. In \cite{EJSJ} the authors establish a correspondence between the process and the JSJ decomposition of fully residually free groups.

In another direction, Makanin's result (on decidability of equations over a free monoid) was developed by K.Schulz, see \cite{Schulz}, who proved that the compatibility problem of equations with regular constraints over a free monoid is decidable.

It turns out to be that the compatibility problem of equations over numerous groups and monoids can be reduced to the compatibility problem of equations over a free monoid with constraints. This technique turned out to be rather fruitful. G.~Makanin was the first to do such a reduction. In \cite{Mak82}, in particular, he reduced the compatibility problem of equations over a free group to the decidability of the compatibility problem for free monoids. Later V.~Diekert, C.~Guti\'{e}rrez and C.~Hagenah, see \cite{DGH}, reduced the compatibility problem of systems of equations over a free group with rational constraints to compatibility problem of equations with regular constraints over a free monoid.

The reduction of compatibility problem for hyperbolic groups to free group was made  in \cite{RipsSela} by E.~Rips and Z.~Sela; for relatively hyperbolic groups with virtually abelian parabolic subgroups in \cite{Dahmani} by F.~Dahmani, for HNN-extensions with finite associated subgroups and for amalgamated products with finite amalgamated subgroups in \cite{LohrSeni} by M.~Lohrey and G.~S\'{e}nizergues, for partially commutative monoids in \cite{Mat} by Yu.~Matiasevich, for partially commutative groups in \cite{DM} by V.~Diekert and A.~Muscholl, for graph product of groups in \cite{DL} by V.~Diekert and M.~Lohrey.

The complexity of Makanin's algorithm has received a great deal of attention. The best result about arbitrary systems of equations over monoids is due to W.~Plandowski. In a series of two papers \cite{Pl1, Pl2} he gave a new approach to the compatibility problem of systems of equations over a free monoid and showed that this problem is in PSPACE. An important ingredient of Plandowski's method is data compression in terms of exponential expressions. This approach was further extended by Diekert, Guti\'{e}rrez and Hagenah, see \cite{DGH} to systems of equations over free groups. Recently, O.~Kharlampovich, I.~Lys\"{e}nok, A.~Myasnikov and N.~Touikan have shown that solving quadratic equations over free groups is NP-complete, \cite{KhLMT}.

Another important development of the ideas of Makanin is due to E.~Rips and is now known as the Rips' machine. In his work Rips interprets Makanin's algorithm in terms of partial isometries of real intervals, which leads him to a classification theorem of finitely generated groups that act freely on $\BR$-trees. A complete proof of Rips' theorem was given by D.~Gaboriau, G.~Levitt, and F.~Paulin, see \cite{GLP}, and, independently, by M.~Bestvina and M.~Feighn, see \cite{BF}, who also generalised Rips' result to give a classification theorem of groups that have a stable action on $\BR$-trees.

\bigskip

The process we describe is a rewriting system based on the ``divide and conquer'' algorithm design paradigm (more precisely, ``divide and marriage before conquest'' technique, \cite{Bl}).

For a given generalised equation $\Omega_{v_0}$, this branching process results in a locally finite and possibly infinite oriented rooted at $v_0$ tree $T$, $T=T(\Omega_{v_0})$. The vertices of the tree $T$ are labelled by (constrained) generalised equations $\Omega_{v_i}$ over $\FF$. The edges of the tree $T$ are labelled by epimorphisms of the corresponding coordinate groups. Moreover, for every solution $H$ of $\Omega_{v_0}$, there exists a path in the tree $T$ from the root vertex to a vertex $v_l$ and a solution $H^{(l)}$ of $\Omega_{v_l}$ such that the solution $H$ is a composition of the epimorphisms corresponding to the edges in the tree and the solution  $H^{(l)}$.  Conversely, every path from the root to a vertex $v_l$ in $T$ and any solution $H^{(l)}$ of $\Omega_{v_l}$ give rise to a solution of $\Omega_{v_0}$, see Proposition \ref{prop:TO}.

The tree is constructed by induction on the height. Let $v$ be a vertex of height $n$. One can check under the assumptions of which of the 15 cases described in Section \ref{se:5.2} the generalised equation $\Omega_v$ falls. If $\Omega_v$ falls under the assumptions of Case 1 or Case 2, then $v$ is a leaf of the tree $T$. Otherwise, using the combination of elementary and derived transformations (defined in Sections \ref{se:5.1} and \ref{se:5.2half}) given in the description of the corresponding (to $v$) case, one constructs finitely many generalised equations and epimorphisms from the coordinate group of $\Omega_v$ to the coordinate groups of the generalised equations constructed.

We finish this section (see Lemma \ref{3.2}) by proving that infinite branches of the tree $T$, as in the case of free groups, correspond to one of the following three cases:
\begin{enumerate}
\item[(A)] Case 7-10: Linear case (Levitt type, thin type);
\item[(B)] Case 12: Quadratic case (surface type, interval exchange type);
\item[(C)] Case 15: General case (toral type, axial type).
\end{enumerate}

\subsection{Preliminary definitions}
In this section we give some definitions that we use throughout the text.
\begin{defn} \label{de:gepar}
Let $\Omega$ be a generalised equation. We partition the set \glossary{name={$\Sigma$, $\Sigma(\Omega)$}, description={the set of all closed sections of a generalised equation $\Omega$}, sort=S}$\Sigma = \Sigma(\Omega)$ of all closed sections of $\Omega$ into a disjoint union of two subsets
\glossary{name={$V\Sigma$}, description={the set of all variable sections of a generalised equation}, sort=V}
\glossary{name={$C\Sigma$}, description={the set of all constant sections of a generalised equation}, sort=C}
$$
\Sigma(\Omega) = V\Sigma\cup  C\Sigma.
$$
The sections from $V\Sigma$ and $C\Sigma$, are called correspondingly, \index{section!variable}\emph{variable}, and \index{section!constant}\emph{constant} sections.

To organise the process properly, we partition the closed sections of $\Omega$ in another way into a disjoint union of two sets, which we refer to as \index{part of a generalised equation!non-active}\index{part of a generalised equation!active}\emph{parts}:
\glossary{name={$A\Sigma$}, description={the active part of a generalised equation}, sort=A}
\glossary{name={$NA\Sigma$}, description={the non-active part of a generalised equation}, sort=N}
$$
\Sigma(\Omega) = A\Sigma\cup NA\Sigma
$$
Sections from $A\Sigma$ are called \index{section!active}{\em active}, and sections from $NA\Sigma$ are called \index{section!non-active}{\em non-active}. We set
$$
C\Sigma \subseteq NA\Sigma.
$$
If not stated otherwise, we assume that all sections from  $V\Sigma$ belong to the active part $A\Sigma$.

If $\sigma \in \Sigma$, then every item (or base) from $\sigma$ is called \index{item!active}\index{base!active}active or \index{item!non-active}\index{base!non-active}non-active, depending on the type of $\sigma$.
\end{defn}

\begin{defn}
We say that a generalised equation $\Omega$ is in the \index{standard form of a generalised equation}{\em standard  form} if the following conditions hold.
\begin{enumerate}
    \item All non-active sections are located to the right of all active sections. Formally, there are numbers \glossary{name={$\rho_A$}, description={the boundary between active and non-active parts of a generalised equation}, sort=R}$1 \leq \rho_A  \leq  \rho_\Omega+1$ such that $[1,\rho_A]$ and $[\rho_{A}, \rho_\Omega+1]$ are, correspondingly, unions of all active and all non-active sections.
    \item All constant bases belong to $C\Sigma$, and for every letter $a \in \cA^{\pm 1}$ there is at most one constant base in $\Omega$ labelled by $a$.
    \item Every free variable $h_i$ of $\Omega$ belongs to a section from $C\Sigma$.
 \end{enumerate}
 \end{defn}

We will show in Lemma \ref{lem:stform} that every generalised equation can be taken to the standard form.

\subsection{Elementary transformations} \label{se:5.1}

In this section we describe \index{transformation of a generalised equation!elementary}{\em elementary transformations} of generalised equations. Recall that we consider only formally consistent generalised equations. In general, an elementary transformation $\ET$ associates  to a generalised equation $\Omega=\gpo$ a finite set of generalised equations $\ET(\Omega) = \left\{\Omega_1, \dots, \Omega_r\right\}$, $\Omega_i=\gpof{i}$  and a collection of surjective homomorphisms $\theta_i: \GG_{R(\Omega^\ast)}\rightarrow \GG_{R({\Omega _i}^\ast)}$ such that for every pair $(\Omega,H)$ there exists a unique pair $(\Omega_i, H^{(i)})$ such that the following diagram commutes.
$$
\xymatrix@C3em{
 \GG_{R(\Omega^\ast)}  \ar[rd]_{\pi_H} \ar[rr]^{\theta_i}  &  &\GG_{R(\Omega_i^\ast )} \ar[ld]^{\pi_{H^{(i)}}}
                                                                             \\
                               &  \GG &
}
$$
Since the pair $(\Omega_i,H^{(i)})$ is defined uniquely, we have a well-defined map $\ET:(\Omega,H) \to (\Omega_i,H^{(i)})$.

Every elementary transformation is first described formally and then we give an example using graphic representations of generalised equations, the latter being much more intuitive. It is a good exercise to understand what how the elementary transformations change the system of equations corresponding to the generalised equation.

\subsubsection*{\glossary{name={$\ET 1$}, description={elementary transformation $\ET 1$}, sort=E}$\ET 1$: Cutting a base}
Suppose that $\Omega$ contains a boundary connection $(p,\lambda ,q )$.

The transformation $\ET 1$ carries $\Omega$ into a single generalised equation $\Omega_1=\gpof{1}$ which is obtained from $\Omega$ as follows. To obtain $\Upsilon_1$ from $\Upsilon$ we
\begin{itemize}
\item replace (cut in $p$) the base $\lambda$ by two new bases  $\lambda _1$ and $\lambda _2$, and
\item replace (cut in $q$) $\Delta(\lambda)$ by two new bases $\Delta (\lambda _1)$ and $\Delta (\lambda _2)$,
\end{itemize}
so that the following conditions hold.

If $\varepsilon(\lambda) =  \varepsilon(\Delta(\lambda))$, then
$$
\alpha(\lambda_1) = \alpha(\lambda), \ \beta(\lambda_1) = p, \quad \alpha(\lambda_2) = p, \  \beta(\lambda_2) = \beta(\lambda); $$
$$
\alpha(\Delta(\lambda_1))= \alpha(\Delta(\lambda)),\ \beta(\Delta(\lambda_1))= q, \quad \alpha(\Delta(\lambda_2)) = q, \ \beta(\Delta(\lambda_2))= \beta(\Delta(\lambda));
$$

If  $\varepsilon(\lambda) = - \varepsilon(\Delta(\lambda))$, then
$$
\alpha(\lambda_1) = \alpha(\lambda), \ \beta(\lambda_1) = p, \quad \alpha(\lambda_2) = p, \  \beta(\lambda_2) = \beta(\lambda); $$
$$
\alpha(\Delta(\lambda_1)) = q,  \ \beta(\Delta(\lambda_1)) =\beta(\Delta(\lambda)), \quad \alpha(\Delta(\lambda_2)) =\alpha(\Delta(\lambda)),  \  \beta(\Delta(\lambda_2)) = q;
$$

Put
$$
\varepsilon(\lambda_i) = \varepsilon(\lambda), \ \varepsilon(\Delta(\lambda_i)) = \varepsilon(\Delta(\lambda)), i= 1,2.
$$

Let $(p', \lambda, q')$ be a boundary connection in $\Omega$. If $p' < p$,  then we replace $(p', \lambda, q')$ by  $(p',\lambda_1, q')$. If $p' > p$, then we replace   $(p', \lambda, q')$ by  $(p',\lambda_2, q')$. Notice that from property (\ref{it:forcon2}) of Definition \ref{def:forcon}, it follows that $(p',\lambda_1, q')$ (or  $(p',\lambda_2, q')$) is a boundary connection in the new generalised equation.

We define the new generalised equation $\Omega_1=\gpof{1}$, by setting $\Re_{\Upsilon_1}=\Re_{\Upsilon}$.  The resulting generalised equation $\Omega_1$ is also formally consistent. Put $\ET 1(\Omega) = \{\Omega_1\}$, see   Figure \ref{fig:ET1}.

\begin{figure}[!h]
  \centering
   \includegraphics[keepaspectratio,width=6in]{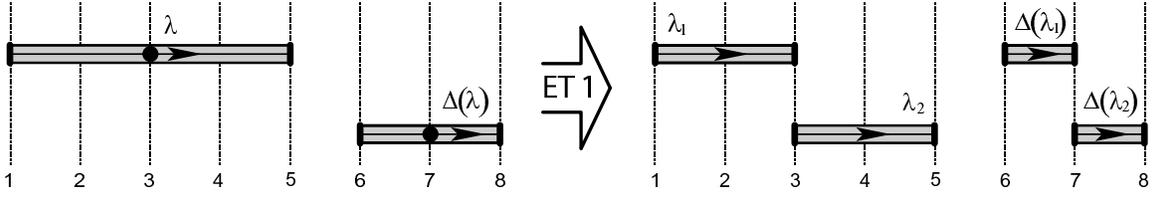}
\caption{Elementary transformation $\ET 1$: Cutting a base.} \label{fig:ET1}
\end{figure}

\subsubsection*{\glossary{name={$\ET 2$}, description={elementary transformation $\ET 2$}, sort=E}$\ET 2$: Transferring a base}
Let a base $\lambda$ of a generalised equation $\Omega$  be contained in the base $\mu$, i.e., $\alpha (\mu)\leq \alpha (\lambda)<\beta (\lambda)\leq\beta(\mu)$. Suppose that the boundaries $\alpha(\lambda)$ and $\beta(\lambda)$ are $\mu$-tied, i.e.  there are boundary connections of the form $(\alpha (\lambda),\mu, q_1)$ and $(\beta (\lambda),\mu,q_2)$.  Suppose also that every $\lambda$-tied boundary is $\mu$-tied.

The transformation $\ET 2$ carries $\Omega$ into a single generalised equation $\Omega_1=\gpof{1}$ which is obtained from $\Omega$ as follows. To obtain $\Upsilon_1$ from $\Upsilon$ we transfer $\lambda$ from the base $\mu$ to the base $\Delta (\mu)$ and adjust all the basic and boundary equations (see Figure \ref{fig:ET2}). Formally, we replace $\lambda$ by a new base $\lambda^\prime$ such that $\alpha(\lambda^\prime) = q_1, \beta(\lambda^\prime) = q_2$ and replace each $\lambda$-boundary connection $(p,\lambda,q)$ with a new one $(p^\prime,\lambda^\prime,q)$ where $p$ and $p^\prime$ are related by a $\mu$-boundary connection $(p,\mu, p^\prime)$.

By definition, set $\Re_{\Upsilon_1}=\Re_{\Upsilon}$. We therefore defined a generalised equation $\Omega=\gpof{1}$, set $\ET 2(\Omega)=\{\Omega_1\}$.

\begin{figure}[!h]
  \centering
   \includegraphics[keepaspectratio,width=6in]{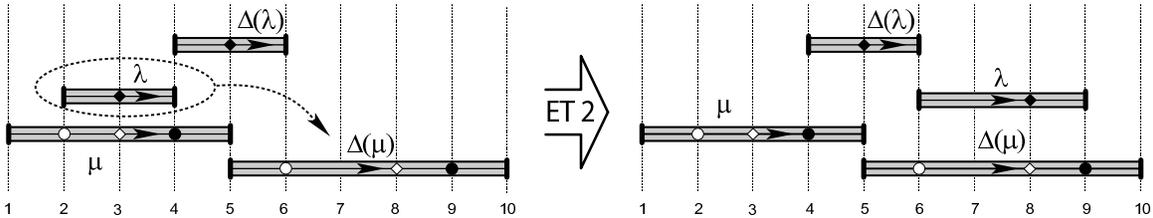}
\caption{Elementary transformation $\ET 2$: Transferring a base.} \label{fig:ET2}
\end{figure}

\subsubsection*{\glossary{name={$\ET 3$}, description={elementary transformation $\ET 3$}, sort=E}$\ET 3$: Removing a pair of matched bases}

Let $\mu$ and $\Delta(\mu)$ be a pair of matched bases in $\Omega$. Since $\Omega$ is formally consistent, one has $\varepsilon(\mu) = \varepsilon(\Delta(\mu))$, $\beta(\mu) = \beta(\Delta(\mu))$ and every $\mu$-boundary connection is of the form $(p,\mu,p)$.

The transformation $\ET 3$ applied to $\Omega$ results in a single generalised equation $\Omega_1=\gpof{1}$ which is obtained from $\Omega$ by
removing the pair of bases  $\mu, \Delta(\mu)$ with all the $\mu$-boundary connections and setting $\Re_{\Upsilon_1}=\Re_{\Upsilon}$, see Figure \ref{fig:ET3}.

\begin{figure}[!h]
  \centering
   \includegraphics[keepaspectratio,width=6in]{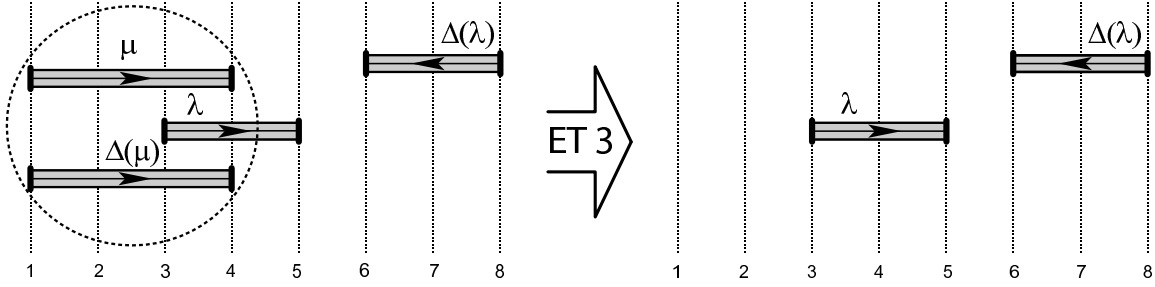}
\caption{Elementary transformation $\ET 3$: Removing a pair of matched bases.} \label{fig:ET3}
\end{figure}

\begin{rem} \label{rem:et123}
Observe that for $i = 1,2,3$, the set $\ET i(\Omega)$ consists of a single generalised equation $\Omega_1$,  such that  $\Omega$ and $\Omega_1$ have the same set of variables $h$ and $\Re_\Upsilon=\Re_{\Upsilon_1}$. The identity isomorphism $\tilde\theta_1:F[h]\to F[h]$, where $F$ is the free group on $\cA$, trivially induces a $\GG$-isomorphism $\theta_1$ from $\GG_{R(\Omega^\ast)}$ to $\GG_{R({\Omega_1}^\ast)}$ and an $F$-isomorphism $\theta_1'$ from $F_{R(\Upsilon^\ast)}$ to $F_{R({\Upsilon_1}^\ast)}$.

Moreover, if $H$ is a solution of $\Omega$, then the tuple $H$ is a solution of $\Omega_1$, since the substitution of $H$ into the equations of $\Omega_1$ result in graphical equalities.
\end{rem}

\subsubsection*{\glossary{name={$\ET 4$}, description={elementary transformation $\ET 4$}, sort=E}$\ET 4$: Removing a linear base}

Suppose that in $\Omega$ a variable base $\mu$ does not intersect any other variable base, i.e. the items $h_{\alpha(\mu)}, \ldots, h_{\beta(\mu)-1}$ are contained only in one variable base $\mu$. Moreover, suppose that all boundaries that intersect $\mu$ are $\mu$-tied, i.e. for every $i$, $\alpha(\mu)< i< \beta(\mu)$ there exists a boundary $\t(i)$ such that $(i,\mu ,\t(i))$ is a boundary connection in $\Omega$. Since $\Omega$ is formally consistent, we have $\t(\alpha(\mu)) = \alpha(\Delta(\mu))$ and  $\t(\beta(\mu)) = \beta(\Delta(\mu))$ if $\varepsilon(\mu)\varepsilon(\Delta(\mu)) = 1$, and $\t(\alpha(\mu)) = \beta(\Delta(\mu))$ and $\t(\beta(\mu)) = \alpha(\Delta(\mu))$ if $\varepsilon(\mu)\varepsilon(\Delta(\mu)) = -1$.

The transformation $\ET 4$ carries $\Omega$ into a single generalised equation $\Omega_1=\gpof{1}$ which is obtained from $\Omega$ by deleting the pair of bases $\mu$ and $\Delta(\mu)$; deleting all the boundaries $\alpha(\mu)+1, \ldots, \beta(\mu)-1$, deleting all $\mu$-boundary connections, re-enumerating the remaining boundaries and setting $\Re_{\Upsilon_1}=\Re_{\Upsilon}$.

We define the epimorphism $\tilde\theta_1: F[h] \to F[h^{(1)}]$, where $F$ is the free group on $\cA$ and $h^{(1)}$ is the set of variables of $\Omega_1$, as follows:
$$
\tilde\theta_1(h_j)=
\left\{
  \begin{array}{ll}
h_j^{(1)},                            & \hbox{if $j<\alpha (\mu)$ or $j\ge \beta (\mu)$;} \\
h^{(1)}_{\t(j)} \dots h^{(1)}_{\t(j+1)-1},   & \hbox{if $\alpha +1 \leq j\leq\beta (\mu)-1$ and  $\varepsilon (\mu)=\varepsilon (\Delta(\mu))$;} \\
h^{(1)}_{\t(j-1)}\dots h^{(1)}_{\t(j+1)},  & \hbox{if $\alpha +1 \leq j\leq\beta (\mu)-1$ and $\varepsilon (\mu)=-\varepsilon (\Delta(\mu))$.}
\end{array}
\right.
$$
It is not hard to see that $\tilde\theta_1$ induces a $\GG$-isomorphism $\theta_1: \GG_{R(\Omega^\ast )} \to \GG_{R({\Omega_1}^\ast)}$ and an $F$-isomorphism $\theta_1'$ from $F_{R(\Upsilon^\ast)}$ to $F_{R({\Upsilon_1}^\ast)}$. Furthermore, if $H$ is a solution of $\Omega$ the tuple $H^{(1)}$ (obtained from $H$ in the same way as $h^{(1)}$ is obtained from $h$) is a solution of $\Omega_1$.

\begin{figure}[!h]
  \centering
   \includegraphics[keepaspectratio,width=6in]{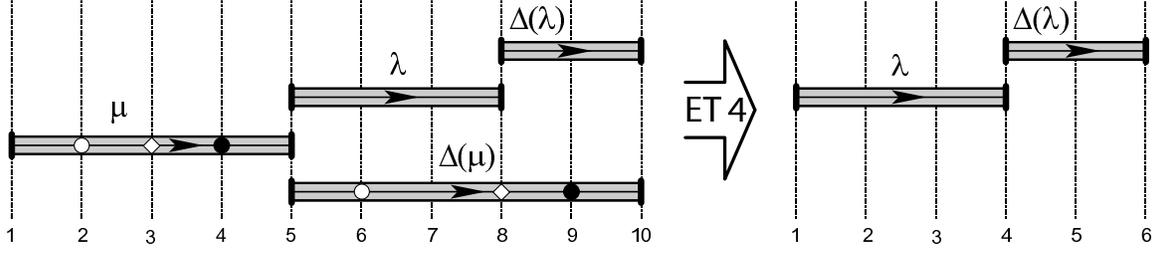}
\caption{Elementary transformation $\ET 4$: Removing a linear base.} \label{fig:ET4}
\end{figure}

\begin{rem} \label{rem:boundterm}
Every time when we transport or delete a closed section $[i,j]$ (see $\ET 4$, $\D 2$, $\D 5$, $\D 6$), we re-enumerate the boundaries as follows.

The re-enumeration of the boundaries is defined by the correspondence $\mathcal{C}:\BD(\Omega)\to \BD(\Omega_1)$, $\Omega_1\in\ET 4(\Omega)$:
$$
\mathcal{C}:
k\mapsto \left\{
           \begin{array}{ll}
             k, & \hbox{if $k\le i$;} \\
             k-(j-i), & \hbox{if $k\ge j$.}
           \end{array}
         \right.
$$
Naturally, in this case we write $\Re_{\Upsilon_1}=\Re_{\Upsilon}$ meaning that
$$
\Re_{\Omega_1}(h_k)=
\left\{
           \begin{array}{ll}
            \Re_\Upsilon(h_k), & \hbox{if $k\le i$;} \\
             \Re_\Upsilon(h_{k-(j-i)}), & \hbox{if $k\ge j$.}
           \end{array}
         \right.
$$
\end{rem}

\subsubsection*{\glossary{name={$\ET 5$}, description={elementary transformation $\ET 5$}, sort=E}$\ET 5$: Introducing a new boundary}

Suppose that a boundary $p$ intersects a base $\mu$ and $p$ is not $\mu$-tied.

The transformation $\ET 5$ $\mu$-ties the boundary $p$ in all possible ways, producing  finitely many different generalised equations. To this end, let $q$ be a boundary intersecting $\Delta (\mu)$. Then we perform one of the following two transformations (see Figure \ref{fig:ET5}):

\begin{figure}[!h]
  \centering
   \includegraphics[keepaspectratio,width=6in]{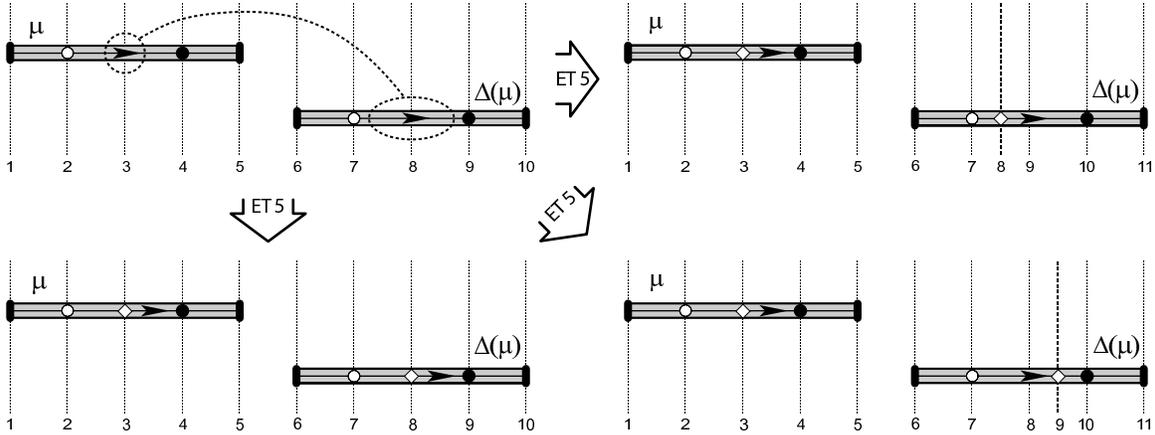}
\caption{Elementary transformation $\ET 5$: Introducing a new boundary.} \label{fig:ET5}
\end{figure}

\begin{enumerate}
\item Introduce the boundary connection $(p,\mu ,q)$ in $\Upsilon$, provided that the resulting generalised equation $\Upsilon_q$ is formally consistent. Define the set $\Re_{\Upsilon_q}\subseteq h\times h$ to be the minimal subset that contains $\Re_\Upsilon$, is symmetric and satisfies condition ($\star$), see Definition \ref{defn:Re}.    This defines the generalised equation $\Omega_{q}=\gpof{q}$.

The identity isomorphism $\tilde\theta_q:F[h]\to F[h]$, where $F$ is the free group on $\cA$, trivially induces a $\GG$-epimorphism $\theta_q$ from $\GG_{R(\Omega^\ast)}$ to $\GG_{R({\Omega_q}^\ast)}$ and an $F$-epimorphism $\theta_q'$ from $F_{R(\Upsilon^\ast)}$ to $F_{R({\Upsilon_q}^\ast)}$.

Observe that $\theta_q$ is not necessarily an isomorphism. More precisely, $\theta_q$ is not an isomorphism whenever the boundary equation corresponding to the boundary connection $(p,\mu ,q)$ does not belong to $R(\Omega^\ast)$. In this case $\theta_q$ is a proper epimorphism. Moreover, if $H$ is a solution of $\Omega$, then the tuple $H$ is a solution of $\Omega_q$.

\item Introduce a new boundary $q^\prime$ between $q$ and $q+1$; introduce a new boundary connection $(p,\mu,q^\prime)$ in $\Upsilon$. Denote the resulting generalised equation by $\Upsilon_{q'}$. Set (below we assume that the boundaries of $\Upsilon_{q'}$ are labelled $1,2,\dots,q,q',q+1,\dots, \rho_{\Upsilon_{q^\prime}}+1$):
    $$
    \Re_{\Upsilon_{q'}}(h_i)=\left\{
    \begin{array}{ll}
    \Re_{\Upsilon}(h_q), &\hbox{ if } i=q,q';\\
    \Re_{\Upsilon}(h_i), &\hbox{ otherwise.}
    \end{array}
    \right.
    $$
    We then make the set $\Re_{\Upsilon_{q'}}$ symmetric and complete it in such a way that it satisfies condition ($\star$), see Definition \ref{defn:Re}.  This defines the generalised equation $\Omega_{q^\prime}=\gpof{q'}$.

We define the $F$-monomorphism  $\tilde\theta_{q^\prime}: F[h]\hookrightarrow F[h^{(q^{\prime})}]$, where $F$ is the free group on $\cA$ and $h^{(q^{\prime})}$ is the set of variables of $\Omega_{q^\prime}$, as follows:
$$
\tilde\theta (h_i)=
\left\{
  \begin{array}{ll}
    h_i, & \hbox{ if $i\neq q$;} \\
    h_{q'-1}h_{q'}, & \hbox{ if $i=q$.}
  \end{array}
\right.
$$
Observe that the $F$-monomorphism  $\tilde\theta_{q^\prime}$  induces a $\GG$-isomorphism $\theta_{q^\prime}: \GG_{R(\Omega^\ast)}\to \GG_{R({\Omega_{q^{\prime}}}^\ast)}$ and an $F$-isomorphism $\theta_{q'}'$ from $F_{R(\Upsilon^\ast)}$ to $F_{R({\Upsilon_{q'}}^\ast)}$. Moreover, if $H$ is a solution of $\Omega$, then the tuple $H^{(q')}$ (obtained from $H$ in the same way as $h^{h^{(q^{\prime})}}$ is obtained from $h$) is a solution of $\Omega_{q'}$.
\end{enumerate}

\bigskip

\begin{lem}\label{le:hom-check}
Let $\Omega_1\in \{\Omega_i\}=\ET(\Omega)$ be a generalised equation obtained from $\Omega$ by an elementary transformation $\ET$ and let $\theta_1: \GG_{R(\Omega^*)}\to \GG_{R(\Omega_1^*)}$ be the corresponding epimorphism. There exists an algorithm which determines whether or not the epimorphism $\theta_1$ is a proper epimorphism.
\end{lem}
\begin{proof}
The only non-trivial case is when $\ET = \ET 5$ and no new boundaries were introduced. In this case $\Omega_1$ is obtained from $\Omega$ by adding a new boundary equation $s = 1$, which is effectively determined by $\Omega$ and $\Omega_1$.
In this event, the coordinate group
$$
\GG_{R({\Omega_1}^\ast)} = \GG_{R({\{\Omega \cup \{s\}\}}^\ast)}
$$
is a quotient of the group $\GG_{R(\Omega^\ast)}$. The homomorphism $\theta_1$ is an isomorphism if and only if $R(\Omega^\ast) = R({\{\Omega \cup \{s\}\}}^\ast)$, or, equivalently, $s \in R(\Omega^\ast)$. The latter condition holds if and only if $s$ vanishes on all solutions of the system of equations $\Omega^\ast = 1$ in $\GG$, i.e. if the following universal formula (quasi identity) holds in $\GG$:
$$
\forall x_1 \ldots \forall x_\rho (\Omega^\ast (x_1, \ldots, x_\rho) = 1 \rightarrow s(x_1, \ldots, x_\rho) = 1).
$$
This can be verified effectively, since the universal theory of $\GG$  is decidable, see \cite{DL}.
\end{proof}

\begin{lem}\label{lem:indETepi}
Let $\Omega_1\in \{\Omega_i\}=\ET(\Omega)$ be a generalised equation obtained from $\Omega$ by an elementary transformation $\ET$ and let $\theta_1: \GG_{R(\Omega^*)}\to \GG_{R(\Omega_1^*)}$ be the corresponding epimorphism. Then there exists a homomorphism
$$
\tilde\theta_1: {F[h_1,\dots, h_{\rho_\Omega}]}\to {F[h_1,\dots, h_{\rho_{\Omega_1}}]}
$$
such that $\tilde\theta_1$ induces an epimorphism
$$
\theta_1':F_{R(\Upsilon^*)}\to F_{R(\Upsilon_1^*)}
$$
and the epimorphism $\theta_1$. In other words, the following diagram commutes:
$$
\CD
 \GG_{R(\Omega^*)} @<<<  F[h_1,\dots, h_{\rho_\Omega}]           @>>> F_{R(\Upsilon^*)} \\
 @V\theta_1 VV       @V \tilde\theta VV   @VV \theta' V  \\
 \GG_{R(\Omega_1^*)} @<<<   F[h_1,\dots, h_{\rho_{\Omega_1}}]           @>>> F_{R(\Upsilon_1^*)}
\endCD
$$
\end{lem}
\begin{proof}
Follows by examining the definition of $\theta_1$ for every elementary transformation $\ET 1-\ET 5$.
\end{proof}

\subsection{Derived transformations} \label{se:5.2half}

In this section we describe several useful transformations of generalised equations. Some of them are finite sequences of elementary transformations, others result in equivalent generalised equations but cannot be presented as a composition of finitely many elementary transformations.

In general, a \index{transformation of a generalised equation!derived}\emph{derived transformation} $\D$ associates  to a generalised equation $\Omega=\gpo$ a finite set of generalised equations $\D(\Omega) = \left\{\Omega_{1}, \dots, \Omega_{r}\right\}$, $\Omega_i=\gpof{i}$  and a collection of surjective homomorphisms $\theta_i: \GG_{R(\Omega^\ast)}\to \GG_{R({\Omega _i}^\ast)}$ such that for every pair $(\Omega, H)$ there exists a unique pair $(\Omega_i, H^{(i)})$ such that the following diagram commutes.
$$
\xymatrix@C3em{
 \GG_{R(\Omega^*)}  \ar[rd]_{\pi_H} \ar[rr]^{\theta_i}  &  &\GG_{R(\Omega_i^*)} \ar[ld]^{\pi_{H^{(i)}}}
                                                                             \\
                               &  \GG &
}
$$
Since the pair $(\Omega_i,H^{(i)})$ is defined uniquely, we have a well-defined map $\D:(\Omega,H) \to (\Omega_i,H^{(i)})$.

\subsubsection*{\glossary{name={$\D 1$}, description={derived transformation $\D 1$}, sort=D}$\D 1:$ Closing a section}

Let $\sigma=[i,j]$ be a section of $\Omega $. The transformation $\D 1$ makes the section $\sigma$ closed. To perform $\D 1$, using  $\ET 5$, we $\mu$-tie the boundary $i$ (the boundary $j$) in every base $\mu$ containing $i$ ($j$, respectively). Using $\ET 1$, we cut all the bases containing $i$ (or $j$) in the boundary $i$ (or in $j$), see Figure \ref{fig:D1}.

\begin{figure}[!h]
  \centering
   \includegraphics[keepaspectratio,width=6in]{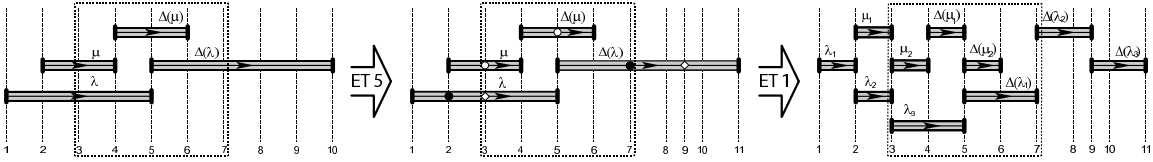}
\caption{Derived transformation $\D 1$: Closing a section.} \label{fig:D1}
\end{figure}

\subsubsection*{\glossary{name={$\D 2$}, description={derived transformation $\D 2$}, sort=D}$\D 2:$ Transporting a closed section}

Let $\sigma$ be a closed section of a generalised equation $\Omega$. The derived transformation $\D 2$ takes $\Omega$ to a single generalised equation $\Omega_1$ obtained from $\Omega$ by cutting $\sigma$ out from the interval $[1,\rho_\Omega+1]$ together with all the bases, and boundary connections on $\sigma$ and moving $\sigma$ to the end of the interval or between two consecutive closed sections of $\Omega$, see Figure \ref{fig:D2}. Then we re-enumerate all the items and boundaries of the generalised equation obtained as appropriate, see Remark \ref{rem:boundterm}. Clearly, the original equation $\Omega$ and the new one $\Omega_1$ have the same solution sets and their coordinate groups are isomorphic (the isomorphism is induced by a permutation of the variables $h$).

\begin{figure}[!h]
  \centering
   \includegraphics[keepaspectratio,width=6in]{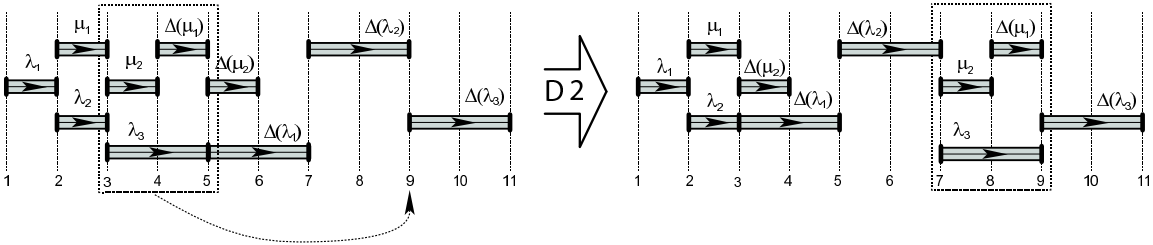}
 \caption{Derived transformation $\D 2$: Transporting a closed section.} \label{fig:D2}
\end{figure}

\subsubsection*{\glossary{name={$\D 3$}, description={derived transformation $\D 3$}, sort=D}$\D 3:$ Completing a cut}

Let $\Omega $ be constrained a generalised equation. The derived transformation $\D 3$ carries $\Omega$ into a single generalised equation \glossary{name={$\widetilde{\Omega}$, $\widetilde{\Upsilon}$}, description={the generalised equation obtained from $\Omega$ (or $\Upsilon$) by $\D 3$}, sort=O}$\widetilde{\Omega}=\Omega_1$ by applying $\ET 1$ to all boundary connections in $\Omega$. Formally, for every boundary connection $(p,\mu,q)$ in $\Omega$ we cut the base $\mu$ in $p$  applying $\ET 1$. Clearly, the generalised equation $\widetilde\Omega$ does not depend on the choice of a sequence of transformations $\ET 1$. Since, by Remark \ref{rem:et123}, $\ET 1$ induces the identity isomorphism between the coordinate groups, equations $\Omega$ and $\widetilde\Omega$ have isomorphic coordinate groups and the isomorphism is induced by the identity map $\GG[h] \to \GG[h]$.

\subsubsection*{\glossary{name={$\D 4$}, description={derived transformation $\D 4$}, sort=D}$\D 4:$ \index{kernel of a generalised equation}Kernel of a generalised equation}
The derived transformation $\D 4$ applied to a generalised equation $\Omega$ results in a single generalised equation \glossary{name={$\Ker(\Omega)$, $\Ker(\Upsilon)$}, description={kernel of a (constrained) generalised equation}, sort=K}$\Ker(\Omega)=\D 4(\Omega)$ constructed below.

Applying $\D 3$, if necessary, one can assume that a generalised equation $\Omega=\gpo$ does not contain boundary connections. An active base $\mu  \in A\Sigma(\Upsilon)$ is called \index{base!eliminable}\emph{eliminable} if at least one of the following holds:
\begin{enumerate}
    \item [a)] $\mu$ contains an  item $h_i$ with $\gamma(h_i)=1$;
    \item [b)] at least one of the boundaries $\alpha (\mu),\beta (\mu)$ is different from $1,\rho +1 $ and it does not touch any other base (except  $\mu$).
\end{enumerate}

An \index{elimination process for a generalised equation}\emph{elimination  process} for $\Upsilon$ consists of consecutive removals (\index{elimination of a base}\emph{eliminations}) of  eliminable bases until there are no eliminable bases left in the generalised equation. The resulting generalised equation is called the \emph{kernel} of $\Upsilon$ and we denote it by $\Ker(\Upsilon)$. It is easy to see that $\Ker(\Upsilon)$ does not depend on the choice of the elimination process. Indeed, if $\Upsilon$ has two different eliminable bases $\mu_1$, $\mu_2$, and elimination of  $\mu_i$ results in a generalised equation $\Upsilon_i$ then  by induction on the number of eliminations $\Ker(\Upsilon_i)$ is defined uniquely for $i = 1,2$. Obviously, $\mu_1$ is still eliminable in $\Upsilon_2$, as well as $\mu_2$ is eliminable in $\Upsilon_1$. Now eliminating $\mu_1$ and $\mu_2$ from $\Upsilon_2$ and $\Upsilon_1$ respectively, we get the same generalised equation $\Upsilon_0$. By induction $\Ker(\Upsilon_1) = \Ker(\Upsilon_0) = \Ker(\Upsilon_2)$ hence the result.

We say that a variable $h_i$ \index{item!belongs to the kernel}{\em belongs to the kernel}, if $h_i$ either belongs to at least one base in the kernel, or is constant.

For a generalised equation $\Omega$ we denote by $\overline{\Omega}=\langle \overline{\Upsilon}, \Re_{\overline{\Upsilon}}\rangle$  the generalised equation which is obtained from $\Omega$ by deleting all free variables in $\Upsilon$ and setting $\Re_{\overline{\Upsilon}}$ to be the restriction of $\Re_\Upsilon$ to the set of variables in $\overline{\Upsilon}$. Obviously,
$$
\GG_{R(\Omega^\ast)} =\factor{\GG[h_1,\dots, h_{\rho_{\ov{\Omega}}}, y_1,\dots,y_k]}{R({{\overline {\Omega}}}^\ast\cup\{[y_i,h_j]\mid \Re_{\Upsilon}(y_i,h_j)\})},
$$
where $\{y_1,\dots y_{k}\}$ is the set of free variables in $\Upsilon$ and $h_j$ is a variable of $\Omega$ (in particular, $h_j$ may be one of the free variables $y_1,\dots, y_k$).

Let us consider what happens in the elimination process on the level of coordinate groups.

We start with the case  when just one base is eliminated. Let $\mu$ be an eliminable base in $\Upsilon = \Upsilon(h_1, \ldots, h_\rho)$. Denote by $\Upsilon_1$  the generalised equation obtained from $\Upsilon$ by eliminating $\mu$ and let $\Omega_1=\gpof{1}$, where $\Re_{\Upsilon_1}=\Re_{\Upsilon}$.

\begin{enumerate}
\item \label{item:7-10c1} Suppose  $h_i \in \mu$ and $\gamma(h_i) = 1$, i.e. $\mu$ falls under the assumption a) of the definition of an eliminable base. Then the variable $h_i$ occurs only once in $\Upsilon$ (in the equation $h(\mu)^{\varepsilon(\mu)}=h(\Delta(\mu))^{\varepsilon(\Delta(\mu))}$ corresponding to the base $\mu$, denote this equation by $s_\mu$).  Therefore, in the coordinate group $\GG_{R(\Omega^\ast)}$ the relation $s_\mu$ can be written as $h_i = w_\mu$, where $w_\mu$ does not contain $h_i$. Consider the generalised equation $\Upsilon'$ obtained from $\Upsilon$ by deleting the equation $s_\mu$ and the item $h_i$. The presentation of the coordinate group of $\Upsilon'$ is obtained from the presentation of the coordinate group of $\Upsilon$ using a sequence of Tietze transformations, thus these groups are isomorphic. We define the relation $\Re_{\Upsilon'}$ as the restriction of $\Re_\Upsilon$ to the set $h\setminus\{h_i\}$, and $\Omega'=\langle\Upsilon',\Re_{\Upsilon'}\rangle$.

It follows that
$$
\GG_{R(\Omega^\ast)}\simeq \factor{\GG[h_1,\dots,h_{i-1},h_{i+1},\dots,h_\rho]}{R({\Omega'}^*\cup\{[h_j,w_\mu]\mid \Re_\Upsilon(h_j,h_i)\})}.
$$
We therefore get that
$$
   \GG_{R({{\Omega_1}^\ast})} \simeq  \factor{\GG[h_1,\dots,h_{\rho}]}{R({\Omega'}^*\cup\{[h_i,h_j]\mid\Re_\Upsilon(h_i,h_j)\})},
$$
and
\begin{gather} \label{eq:ker1}
\begin{split}
\GG_{R(\Omega^\ast)}& \simeq  \factor{\GG[h_1,\dots,h_{i-1},h_{i+1},\dots,h_{\rho}]}{R({\Omega^\prime}^\ast\cup \{[h_j,w_\mu]\mid \Re_\Upsilon(h_j,h_i)\})}\\
&\simeq  \factor{\GG[h_1,\dots, h_{\rho_{\ov{\Omega}_1}}, z_1,\dots,z_l]}{R\left(\begin{array}{l}{\ov{\Omega}_1}^* \cup\{[h_j,w_\mu]\mid \Re_\Upsilon(h_j,h_i)\}\\ \cup \{[z_k,h_j]\mid\Re_\Upsilon(z_k,h_j)\}\cup \{[z_k,w_\mu]\mid\Re_\Upsilon(z_k,h_i)\}
\end{array}\right)},
\end{split}
\end{gather}
where $\{z_1,\dots,z_l\}$ are the free variables of $\Omega_1$ distinct from $h_i$. Note that all the groups and equations which occur above can be found effectively.

\item \label{item:7-10c2} Suppose now that $\mu$ falls under the assumptions of case b) of the definition of an eliminable base, with respect to a boundary $i$. Then in all the equations of the generalised equation $\Omega$ but $s_\mu$, the variable $h_{i-1}$ occurs as a part of the subword $(h_{i-1}h_i)^{\pm 1}$. In the equation $s_\mu$ the variable $h_{i-1}$ either occurs once or it occurs precisely twice and in this event the second occurrence of $h_{i-1}$ (in $\Delta(\mu)$) is a part of the subword $(h_{i-1}h_i)^{\pm 1}$.
It is obvious that the tuple
    $$
    (h_1, \ldots, h_{i-2},h_{i-1}, h_{i-1}h_i, h_{i+1}, \ldots,h_\rho)=(h_1, \ldots, h_{i-2},h_{i-1}, h_i', h_{i+1}, \ldots,h_\rho)
    $$
forms a basis of the ambient free group generated by $(h_1,\ldots, h_\rho)$. We further assume that $\Omega$ is a system of equations in these variables. Notice, that in the new basis, the variable $h_{i-1}$ occurs only once in the generalised equation $\Omega$, in the equation $s_\mu$. Hence, as in case (\ref{item:7-10c1}), the relation $s_\mu$ can be written as $h_{i-1} = w_\mu$, where $w_\mu$ does not contain $h_{i-1}$.

Therefore, if we eliminate the relation $s_\mu$, then the set $Y = (h_1, \ldots, h_{i-2}, h_i', h_{i+1}, \ldots,h_{\rho})$ is a generating set of $\GG_{R(\Omega^\ast)}$. In the coordinate group $\GG_{R(\Omega^\ast)}$ the commutativity relations $\{[h_j,h_k]\mid \Re_\Upsilon(h_j,h_k)\}$ rewrite as follows:
$$
\left\{
\begin{array}{ll}
  \mathstrut [h_j,h_k], & \hbox{ if } \Re_\Upsilon (h_j,h_k) \hbox{ and } j,k \notin \{i,i-1\}; \\
   \mathstrut [w_{\mu} , h_j ], & \hbox{ if } \Re_\Upsilon (h_{i-1},h_j);\\
  \mathstrut [{w_\mu}^{-1}h_i',h_j], & \hbox{ if }\Re_\Upsilon(h_i,h_j)
\end{array}
\right\}.
$$
This shows that
    $$
    \GG_{R(\Omega^\ast)} \simeq \factor{\GG[h_1, \dots, h_{i-2}, h_i', h_{i+1},\dots,h_{\rho}]}{R\left( \begin{array}{l}{\Omega^\prime}^\ast\cup \{[w_\mu,h_j]\mid \Re_\Upsilon(h_{i-1},h_j)\}\\ \cup \{[{w_\mu}^{-1}h_i',h_j] \mid \Re_\Upsilon(h_i,h_j)\}\end{array}\right)},
    $$
where $\Omega^\prime=\langle \Upsilon^\prime,\Re_{\Upsilon^\prime}\rangle $ is the generalised equation obtained from $\Omega_1$ by deleting the boundary $i$ and $\Re_{\Upsilon^\prime}$ is the restriction of $\Re_{\Upsilon_1}$ to the set $h\setminus \{h_{i-1},h_i\}$. Denote by $\Omega^{\prime \prime}$ the generalised equation obtained from $\Omega^\prime$ by adding a free variable $z$. It now follows that
\begin{gather}\notag
\begin{split}
\GG_{R({\Omega_1}^\ast)} &\simeq  \factor{\GG[h_1, \dots, h_{i-2}, h_{i+1},\dots,h_{\rho},z]}{R
\left(
\begin{array}{l}
{\Omega^{\prime \prime}}^\ast \cup \{[z,h_j]\mid \Re_{\Upsilon_1}(h_{i-1},h_j)\}\\
\cup\{[z^{-1}h_i',h_j]\mid \Re_{\Upsilon_1}(h_{i},h_j)\}
\end{array}
\right)}
\\
&\simeq\factor{\GG[h_1, \dots, h_{i-2}, h_{i+1},\dots,h_{\rho},z]}{R
\left(
\begin{array}{l}
{\Omega^{\prime}}^\ast \cup \{[z,h_j]\mid \Re_{\Upsilon_1}(h_{i-1},h_j)\}\\
\cup\{[z^{-1}h_i',h_j]\mid \Re_{\Upsilon_1}(h_{i},h_j)\}
\end{array}
\right)} \ \hbox{ and }
\end{split}
\end{gather}
\begin{gather}\label{eq:ker2}
\begin{split}
& \GG_{R(\Omega^\ast)} \simeq \factor{\GG[h_1, \dots, h_{i-2}, h_i', h_{i+1},\dots,h_{\rho}]}{R
 \left(
 \begin{array}{l}
 {\Omega^\prime}^\ast\cup \{[w_\mu,h_j]\mid \Re_\Upsilon(h_{i-1},h_j)\}\\
 \cup \{[{w_\mu}^{-1}h_i',h_j] \mid \Re_\Upsilon(h_i,h_j)\}
 \end{array}
 \right)} \\
& \simeq \factor{\GG[h_1,\dots, h_{\rho_{\ov{\Omega^{\prime\prime}}}}, z_1,\dots, z_l]}
{R\left(
\begin{array}{l}
\overline{\Omega^{\prime\prime}}^\ast \cup \{[{w_{\mu}}^{-1}h_i',h_j]\mid \Re_{\Upsilon_1}(h_{i},h_j)\}\\
\cup \{[{w_{\mu}}^{-1}h_i',z_j]\mid \Re_{\Upsilon_1}(h_{i},z_j)\} \cup \{[w_\mu,h_j]\mid \Re_{\Upsilon_1}(h_{i-1},h_j)\}\\
\cup \{[w_\mu,z_j]\mid \Re_{\Upsilon_1}(h_{i-1},z_j)\}\cup \{[z_j,h_k]\mid\Re_{\Upsilon'}(z_j,h_k)\}
\end{array}
\right)}
\end{split}
\end{gather}
where $\{z_1,\dots, z_l\}$ are the free variables of $\Omega^{\prime\prime}$ distinct from $z$. Notice that all the groups and equations which occur above can be found effectively.
\end{enumerate}

By induction on the number of steps in the elimination process we obtain the following lemma.
\begin{lem}\label{7-10}
In the above notation,
$$
\GG_{R(\Omega^*)}\simeq  \factor{\GG[h_1,\dots,h_{\rho_{\ov{\Ker(\Omega)}}},z_1,\dots, z_l]}{R({\overline{\Ker(\Omega)}}^\ast \cup \mathcal{K})},
$$
where $\{z_1,\dots, z_l\}$ is a set of free variables of ${\Ker(\Omega)}$, and $\mathcal{K}$ is a certain computable set of commutators of words in $h\cup Z$.
\end{lem}
\begin{proof} Let
$$
\Omega = \Omega_0 \rightarrow \Omega_1 \rightarrow \ldots \rightarrow \Omega_k = \Ker(\Omega),
$$
where $\Omega_{j+1}$ is obtained from $\Omega_j$ by eliminating an eliminable base, $j = 0, \ldots,k-1$. It is easy to see (by induction on $k$)  that for every $j = 0, \ldots,k-1$
$$
\overline{\Ker(\Omega_j)} = \overline{\Ker (\overline{\Omega_j})}.
$$
Moreover, if $\Omega_{j+1}$ is obtained from $\Omega_j$ as in case (\ref{item:7-10c2}) above, then (in the above notation)
$$
\overline{\left(\Ker (\Omega_j)\right)_1} = \overline{\Ker(\Omega_j^{\prime\prime})}.
$$
The statement of the lemma now follows from the remarks above and Equations (\ref{eq:ker1}) and (\ref{eq:ker2}). Note that the cardinality $l$ of the set $\{z_1,\dots, z_l\}$ is the number of free variables of ${\Ker(\Omega)}$ minus $k$.
\end{proof}

\subsubsection*{\glossary{name={$\D 5$}, description={derived transformation $\D 5$}, sort=D}$\D 5:$ \index{transformation of a generalised equation!entire}Entire transformation}

In order to define the derived transformation $\D 5$ we need to introduce several notions.  A base $\mu$ of the generalised equation $\Omega$ is called a \index{base!leading}{\em leading} base if $\alpha(\mu)=1$. A leading base is called a \index{base!carrier}\index{carrier}{\em carrier} if for any other leading base $\lambda$ we have $\beta (\lambda)\leq \beta (\mu)$. Let $\mu $ be a carrier base of $\Omega$. Any active base $\lambda \neq \mu$ with $\beta(\lambda )\leq \beta (\mu )$ is called a \index{base!transfer}{\em transfer} base (with respect to $\mu$).

Suppose now that $\Omega$ is a generalised equation with $\gamma(h_i)\geq 2$ for each $h_i$ in the active part of $\Omega$. The \emph{entire transformation} is the following sequence of elementary transformations.

We fix a carrier base $\mu$ of $\Omega$. For any transfer base $\lambda$ we $\mu$-tie (applying $\ET 5$) all boundaries that intersect $\lambda$. Using $\ET 2$ we transfer all transfer bases from $\mu$ onto $\Delta (\mu)$. Now, there exists some $k<\beta (\mu)$ such that $h_1,\ldots ,h_k$ belong to only one base $\mu,$ while  $h_{k+1}$ belongs to at least two bases. Applying $\ET 1$ we cut $\mu$ in the boundary $k+1$. Finally, applying $\ET 4$, we delete the section $[1,k+1]$, see Figure \ref{quadratic}.

Let $\rho_A$ be the boundary between the active and non-active parts of $\Omega$, i.e. $[1,\rho_A]$ is the active part $A\Sigma$ of $\Omega$ and $[\rho_A,\rho+1]$ is the non-active  part $NA\Sigma$ of $\Omega$.

\begin{defn}\label{defn:excess}
For a pair $(\Omega,H)$, we introduce the following \glossary{name={$d_{A\Sigma}(H)$}, description={length of the active part of the solution, $d_{A\Sigma}(H)=\sum \limits_{i=1}^{\rho_A-1}|H_i|$}, sort=D} \glossary{name={$\psi_{A\Sigma}(H)$}, description={excess of the solution, $\psi_{A\Sigma}(H)=\sum \limits_{\mu\in \omega_1}|H({\mu})|-2d_{A\Sigma}(H)$}, sort=P}
notation
\begin{equation}\label{3.12}
d_{A\Sigma}(H)=\sum \limits_{i=1}^{\rho_A-1}|H_i|,
\end{equation}
\begin{equation}\label{3.13}
\psi_{A\Sigma}(H)=\sum \limits_{\mu\in \omega_1}|H({\mu})|-2d_{A\Sigma}(H),
\end{equation}
where \glossary{name={$\omega_1$}, description={the set of all variable bases $\nu $ for which either $\nu$ or $\Delta (\nu)$ belongs to the active part of a generalised equation}, sort=O}$\omega _1$ is the set of all variable bases $\nu$ for which either $\nu$ or $\Delta (\nu)$ belongs to the active part $[1,\rho_A]$ of $\Omega$.

We call the number $\psi_{A\Sigma}(H)$ the \index{excess of the solution}\emph{excess} of the solution $H$ of the generalised equation $\Omega$.
\end{defn}

Every item $h_i$ of the section $[1,\rho_A]$ belongs to at least two bases, each of these bases belongs to $A\Sigma$, hence $\psi_{A\Sigma}(H)\ge 0$.

Notice, that if for every item $h_i$ of the section $[1,\rho_A]$ one has $\gamma(h_i)=2$, for every solution $H$ of $\Omega$ the excess $\psi_{A\Sigma}(H)=0$.  Informally, in some sense, the excess of $H$ measures how far the generalised equation $\Omega$ is from being \index{generalised equation!quadratic}quadratic (every item in the active part is covered twice).

\begin{lem}\label{lem:excess}
Let $\Omega$ be a generalised equation such that  every item $h_i$ from the active part $A\Sigma=[1,\rho_A]$ of $\Omega$ is covered at least twice. Suppose that $\D 5(\Omega,H)=(\Omega_i,H^{(i)})$ and the carrier base $\mu$ of $\Omega$ and its dual $\Delta(\mu)$ belong to the active part of $\Omega$.
Then the excess of the solution $H$ equals the excess of the solution $H^{(i)}$, i.e. $\psi_{A\Sigma}(H)=\psi_{A\Sigma}(H^{(i)})$.
\end{lem}
\begin{proof}
By the definition of the entire transformation $\D 5$, it follows that the word $H^{(i)}[1,\rho_{\Omega_i}+1]$ is a terminal subword of the word $H[1,\rho+1]$, i.e.
$$
H[1,\rho+1]\doteq U_iH^{(i)}[1,\rho_{\Omega_i}+1], \hbox{ where } \rho=\rho_\Omega.
$$
On the other hand, since $\mu, \Delta(\mu)\in A\Sigma$, the non-active parts of $\Omega$ and $\Omega_i$ coincide. Therefore, $H[\rho_A,\rho+1]$ is the terminal subword of the word $H^{(i)}[ 1,\rho _{\Omega_i}+1]$, i.e. the following graphical equality holds:
$$
H^{(i)}[1,\rho_{\Omega_i}+1]\doteq V_i H[\rho_A,\rho +1].
$$
So we have
\begin{equation}\label{3.22}
d_{A\Sigma}(H)-d_{A\Sigma}(H ^{(i)})= |H[1,\rho_A]|-|V_{i}|=|U_{i}|= |H({\mu})|-|H^{(i)}({\mu})|,
\end{equation}
where, in the above notation, if $k=\beta(\mu)-1$ (and thus $\mu$ has been completely eliminated in $\Omega^{(i)}$), then  $H^{(i)}({\mu})=1$; otherwise, $H^{(i)}({\mu})=H[k+1,\beta(\mu)]$.

From (\ref{3.13}) and (\ref{3.22}) it follows that $\psi _{A\Sigma}(H)=\psi _{A\Sigma}(H^{(i)})$.
\end{proof}

\subsubsection*{\glossary{name={$\D 6$}, description={derived transformation $\D 6$}, sort=D}$\D 6:$ Identifying closed constant sections}

Let $\lambda$ and $\mu$ be two constant bases in $\Omega$ with labels $a^{\epsilon_\lambda}$ and $a^{\epsilon_\mu}$, where $a \in \cA^{\pm 1}$ and $\epsilon_\lambda=\epsilon_\mu$, $\epsilon_\mu \in \{1,-1\}$. Suppose that the sections $\sigma(\lambda) = [i,i+1]$ and $\sigma(\mu) = [j,j+1]$ are closed.

The transformation $\D 6$ applied to $\Omega$ results in a single generalised equation $\Omega_1$ which is obtained from $\Omega$ in the following way, see Figure \ref{constants}. Introduce a new variable base $\eta$ with its dual $\Delta(\eta)$ such that
$$
\sigma(\eta) = [i,i+1], \ \sigma(\Delta(\eta)) = [j,j+1],\  \varepsilon(\eta) = \epsilon_\lambda, \ \varepsilon(\Delta(\eta)) = \epsilon_\mu.
$$
Then we transfer all bases from $\eta$ onto $\Delta(\eta)$ using $\ET 2$, remove the bases $\eta$ and $\Delta(\eta)$, remove the item $h_i$, re-enumerate the remaining items  and adjust the relation $\Re_\Upsilon$ as appropriate, see Remark \ref{rem:boundterm}.

The corresponding homomorphism $\theta_1:\GG_{R(\Omega^*)}\to \GG_{R(\Omega_1^*)}$ is induced by the composition of the homomorphisms defined by the respective elementary transformations. Obviously, $\theta_1$ is an isomorphism.

\bigskip

\begin{lem} \label{lem:solleng}
Let $\Omega_i$ be a generalised equation and $H^{(i)}$ be a solution of $\Omega_i$ so that
$$
\ET:(\Omega,H)\to (\Omega_i,H^{(i)}) \hbox{ or } \D:(\Omega,H)\to (\Omega_i,H^{(i)}).
$$
Then one has $|H^{(i)}|\le |H|$. Furthermore, in the case that $\ET=\ET 4$ or $\D=\D 5$  this inequality is strict $|H^{(i)}|<|H|$.
\end{lem}
\begin{proof}
Proof is by straightforward examination of descriptions of elementary and derived transformations.
\end{proof}

\begin{lem}\label{lem:indDepi}
Let $\Omega_1\in \{\Omega_i\}=\ET(\Omega)$ be a generalised equation obtained from $\Omega$ by a derived transformation $\D$ and let $\theta_1: \GG_{R(\Omega^*)}\to \GG_{R(\Omega_1^*)}$ be the corresponding epimorphism. Then there exists a homomorphism
$$
\tilde\theta_1: {F[h_1,\dots, h_{\rho_\Omega}]}\to {F[h_1,\dots, h_{\rho_{\Omega_1}}]}
$$
such that $\tilde\theta_1$ induces an epimorphism
$$
\theta_1':F_{R(\Upsilon^*)}\to F_{R(\Upsilon_1^*)}
$$
and the epimorphism $\theta_1:\GG_{R(\Omega^*)}\to \GG_{R(\Omega_1^*)}$. In other words, the following diagram commutes:
$$
\CD
 \GG_{R(\Omega^*)} @<<<  F[h_1,\dots, h_{\rho_\Omega}]           @>>> F_{R(\Upsilon^*)} \\
 @V\theta_1 VV       @V \tilde\theta VV   @VV \theta' V  \\
 \GG_{R(\Omega_1^*)} @<<<   F[h_1,\dots, h_{\rho_{\Omega_1}}]           @>>> F_{R(\Upsilon_1^*)}
\endCD
$$
\end{lem}
\begin{proof}
Follows by examining the definition of $\theta_1$ for every derived transformation $\D 1-\D 6$ and Lemma \ref{lem:indETepi}.
\end{proof}

\begin{lem} \label{lem:stform}
Using derived transformations, every generalised equation can be taken to the standard form.

Let $\Omega_1$ be obtained from $\Omega$ by an elementary or a derived transformation and let $\Omega$ be in the standard form. Then $\Omega_1$ is in the standard form.
\end{lem}
\begin{proof}
Proof is straightforward.
\end{proof}

\subsection{Construction of the tree  $T(\Omega)$} \label{se:5.2}

In this section we describe a branching process for rewriting a generalised equation $\Omega$. This process results in a locally finite and possibly infinite tree \glossary{name={$T(\Omega)$}, description={the infinite, locally finite tree of the process}, sort=T}
$T(\Omega)$. In the end of the section we describe infinite paths in $T(\Omega)$. We summarise the results of this section in the proposition below.

\begin{prop} \label{prop:TO}
For a {\rm(}constrained{\rm)} generalised equation $\Omega=\Omega_{v_0}$ over $\FF$, one can effectively construct a locally finite, possibly infinite, oriented rooted at $v_0$ tree $T$, $T=T(\Omega_{v_0})$, such that:
\begin{enumerate}
\item The vertices $v_i$ of $T$ are labelled by generalised equations $\Omega_{v_i}$ over $\FF$.
\item The edges $v_i\to v_{i+1}$ of $T$ are labelled by epimorphisms
$$
\pi(v_i,v_{i+1}):\GG_{R(\Omega_{v_i}^\ast)}\to \GG_{R(\Omega_{v_{i+1}}^\ast)}.
$$
The edges $v_k\to v_{k+1}$, where $v_{k+1}$ is a leaf of $T$ and $\tp(v_{k+1})=1$, are labelled by proper epimorphisms. All the other epimorphisms $\pi(v_i,v_{i+1})$ are isomorphisms, in particular, edges that belong to infinite branches of the tree $T$ are labelled by isomorphisms.
\item Given a solution $H$ of $\Omega_{v_0}$, there exists a path $v_0\to v_1\to \dots \to v_l$ and a solution $H^{(l)}$ of $\Omega_{v_l}$ such that $$
\pi_H=\pi(v_0,v_1)\cdots \pi(v_{l-1},v_l)\pi_{H^{(l)}}.
    $$
    Conversely, for every path $v_0\to v_1\to \dots \to v_l$ in $T$ and every solution $H^{(l)}$ of $\Omega_{v_l}$ the homomorphism
    $$
    \pi(v_0,v_1)\cdots \pi(v_{l-1},v_l)\pi_{H^{(l)}}
    $$
    gives rise to a solution of $\Omega_{v_0}$.
\end{enumerate}
\end{prop}

\begin{defn}
Denote by \glossary{name={$n_A$, $n_A(\Omega)$}, description={the number of bases in the active sections of a generalised equation}, sort=N}$n_A=n_A(\Omega)$ the number of bases in the active sections of $\Omega$ and by \glossary{name={$\xi$}, description={the number of open boundaries in the active sections of a generalised equation}, sort=X}$\xi$ the number of open boundaries in the active sections.

For a closed section $\sigma \in \Sigma(\Omega)$ denote by $n(\sigma)$\glossary{name={$n(\sigma)$}, description={the number of bases in the section $\sigma$}, sort=N} the number of bases in $\sigma$. The \index{complexity of the generalised equation}{\em complexity} of a generalised equation $\Omega$ is defined as follows
\glossary{name={$\comp$, $\comp (\Omega)$}, description={complexity of a generalised equation}, sort=C}
$$
\comp = \comp (\Omega) = \sum\limits_{\sigma \in A\Sigma(\Omega)} \max\{0, n(\sigma)-2\}.
$$
\end{defn}

\begin{rem}
We use the following convention. Let $\Omega$ be a generalised equation. By a function of a generalised equation $f(\Omega)$ we mean a function of the parameters $n_A(\Omega)$, $\xi(\Omega)$, $\rho_\Omega$ and $\comp(\Omega)$.
\end{rem}

We begin with a general description of the tree $T(\Omega)$ and then construct it using induction on its height.  For each vertex $v$ in $T(\Omega)$ there exists a generalised equation $\Omega_v$ associated to $v$. Recall, that we consider only formally consistent generalised equations. The initial generalised equation $\Omega$ is associated to the root $v_0$, $\Omega_{v_0} = \Omega$. For each edge $v\to v'$ there exists a unique surjective homomorphism $\pi(v,v'):\GG_{R(\Omega _v^* )}\rightarrow \GG_{R(\Omega _{v'}^* )}$ associated to $v\rightarrow v'$.

If
$$
v\rightarrow v_1\rightarrow\ldots\rightarrow v_s\rightarrow u
$$
is a path in $T(\Omega )$, then by $\pi (v,u)$ we denote the composition of corresponding homomorphisms
$$
\pi (v,u) = \pi (v,v_1)   \cdots \pi (v_s,u).
$$
We call this epimorphism the \index{canonical homomorphism of coordinate groups of generalised equations}\emph{canonical homomorphism from $\GG_{R(\Omega _v^*)}$ to $\GG_{R(\Omega _{u}^*)}$}

There are two kinds of edges in $T(\Omega)$: \emph{principal} and \emph{auxiliary}. Every edge constructed is principal, if not stated otherwise.
Let $v \to v^\prime$ be an edge of $T(\Omega)$, we assume that active (non-active) sections in $\Omega_{v^\prime}$ are naturally inherited from $\Omega_v$. If $v \to v^\prime$ is a principal edge, then there exists a finite sequence of elementary  and derived transformations from $\Omega_v$ to $\Omega_{v^\prime}$ and the homomorphism $\pi(v,v')$ is a composition of the homomorphisms corresponding to these transformations.

Since both elementary and derived transformations uniquely associate a pair $(\Omega_i,H^{(i)})$ to $(\Omega,H)$:
$$
\ET:(\Omega,H)\to (\Omega_i,H^{(i)}) \hbox{ and } \D:(\Omega,H)\to (\Omega_i,H^{(i)}),
$$
a solution $H$ of $\Omega$ defines a path in the tree $T(\Omega)$. We call such a path \index{path!defined by a solution in the tree $T(\Omega)$}\emph{the path defined by a solution $H$ in $T$}.

Let $\Omega $ be a generalised equation. We construct an oriented rooted tree $T(\Omega)$. We start from the root $v_0$ and proceed by induction on the height of the tree.

Suppose, by induction, that the tree $T(\Omega)$ is constructed  up to  height $n$, and let $v$ be a vertex of height $n$. We now describe how to extend the tree from $v$.  The construction of the  outgoing edges from $v$ depends on which of the case described below takes place at the vertex $v$. We always assume that:
\begin{center}
\emph{If the generalised equation $\Omega_v$ satisfies the assumptions of Case $i$, then $\Omega_v$ does not satisfy the assumptions of all the Cases $j$, with $j < i$.}
\end{center}

The general guideline of the process is to use the entire transformation to transfer bases to the right of the interval. Before applying the entire transformation one has to make sure that the generalised equation is ``clean''. The stratification of the process into 15 cases, though seemingly unnecessary, is convenient in the proofs since each of the cases has a different behaviour with respect to ``complexity'' of the generalised equation, see Lemma \ref{3.1}.

\subsubsection*{Preprocessing}

In $\Omega_{v_0}$ we transport closed sections using $\D 2$ in such a way that all active sections are at the left end of the interval (the active part of the generalised equation), then come all non-active sections.

\subsubsection*{Termination conditions: Cases 1 and 2}

\subsubsection*{Case 1: The homomorphism $\pi (v_0,v)$ is not an isomorphism, or, equivalently, the canonical homomorphism $\pi(v_1,v)$, where $v_1\to v$ is an edge of $T(\Omega)$, is not an isomorphism.}

The vertex $v$, in this case, is a leaf of the tree $T(\Omega)$. There are no outgoing edges from $v$.

\subsubsection*{Case 2: The generalised equation $\Omega _v$ does not contain active sections.}

Then the vertex $v$ is a leaf of the tree $T(\Omega)$. There are no outgoing edges from $v$.

\subsubsection*{Moving constant bases to the non-active part: Cases 3 and 4}

\subsubsection*{Case 3: The generalised equation $\Omega _v$ contains a constant base $\lambda$ in an active section such that the section $\sigma(\lambda)$ is not closed.}

In this case, we make the section $\sigma(\lambda)$ closed using the derived transformation $\D 1$.

\subsubsection*{Case 4: The generalised equation $\Omega _v$ contains a constant base $\lambda$ such that the section $\sigma(\lambda)$ is closed.}

In this case, we transport the section $\sigma(\lambda)$ to $C\Sigma$ using the derived transformation $\D 2$. Suppose that the base $\lambda$ is labelled by a letter $a \in \cA^{\pm 1}$. Then we identify all closed sections of the type $[i,i+1]$, which contain a constant base with the label $a^{\pm 1}$, with the transported section $\sigma(\lambda)$, using the derived transformation $\D 6$. In the resulting generalised equation $\Omega_{v^\prime}$ the section $\sigma(\lambda)$ becomes a constant section, and the corresponding edge $(v,v^\prime)$ is auxiliary, see Figure \ref{constants} (note that $\D 1$ produces a finite set of generalised equations, Figure \ref{constants} shows only one of them).

\begin{figure}[!h]
  \centering
   \includegraphics[keepaspectratio,width=6in]{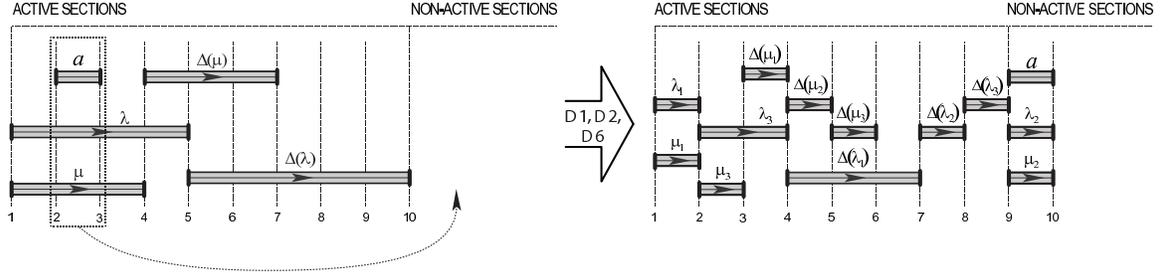}
\caption{Cases 3-4: Moving constant bases.} \label{constants}
\end{figure}

\subsubsection*{Moving free variables to the non-active part: Cases 5 and 6}

\subsubsection*{Case 5: The generalised equation $\Omega _v$ contains a free variable $h_q$ in an  active section.}

Using $\D 2$, we transport the section $[q,q+1]$ to the very end of the interval behind all the items of $\Omega_v$. In the resulting generalised equation $\Omega_{v^\prime}$ the transported section becomes a constant section, and the corresponding edge $(v,v^\prime)$ is auxiliary.

\subsubsection*{Case 6: The generalised equation $\Omega _v$ contains a pair of matched bases in an active section.}

In this case, we perform $\ET 3$ and delete it, see Figure \ref{useless}.

\begin{figure}[!h]
  \centering
   \includegraphics[keepaspectratio,width=6in]{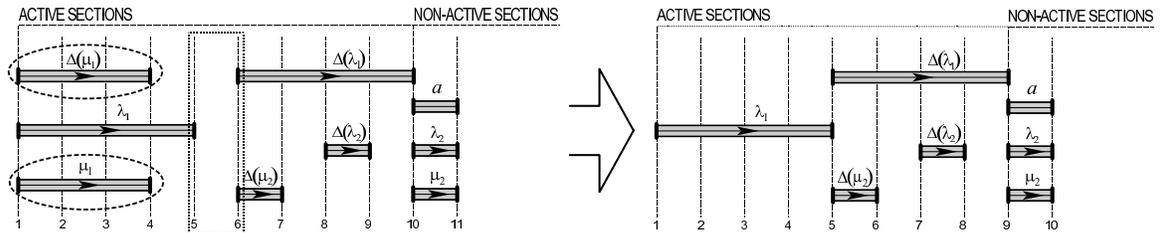}
 \caption{Cases 5-6: Removing a pair of matched bases and free variables.}
 \label{useless}
 \end{figure}

\begin{rem} \label{rem:stform}
If $\Omega_v$ does not satisfy the conditions of any of the Cases 1-6, then the generalised equation $\Omega_v$ is in the standard form.
\end{rem}

\subsubsection*{Eliminating linear variables: Cases 7-10}

\subsubsection*{Case 7: There exists an item $h_i$ in an active section of $\Omega_v$ such that $\gamma _i=1$ and  such that both boundaries $i$ and $i+1$ are closed.}

Then, we remove the closed section $[i,i+1]$ together with the linear base using $\ET 4$.

\subsubsection*{Case 8: There exists an item $h_i$ in an active section of $\Omega_v$ such that $\gamma _i=1$ and such that  one of the boundaries $i$, $i+1$ is open, say $i+1$, and the other is closed.}

In this case, we first perform $\ET 5$ and $\mu$-tie $i+1$ by the only base $\mu$ it intersects; then using $\ET 1$ we cut $\mu$ in $i+1$; and then we delete the closed section $[i,i+1]$ using $\ET 4$,  see Figure \ref{linear'}.

\begin{figure}[!h]
  \centering
   \includegraphics[keepaspectratio,width=6in]{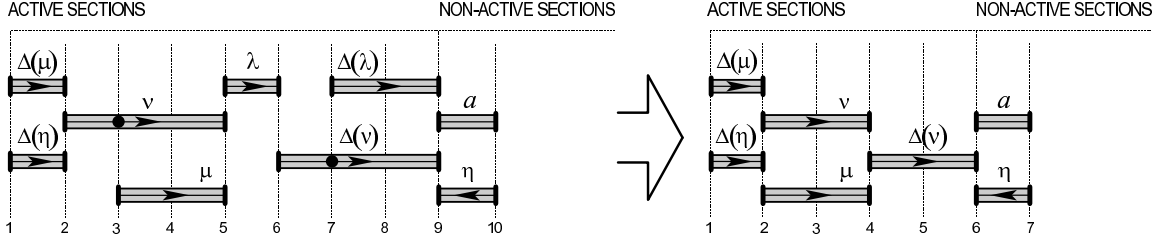}
 \caption{Cases 7-8: Linear variables.} \label{linear'}
\end{figure}

\subsubsection*{Case 9: There exists an item $h_i$ in an active section of $\Omega_v$ such that $\gamma _i=1$ and  such that both boundaries $i$ and $i+1$ are open. In addition, there is a closed section $\sigma$ such that $\sigma$ contains exactly two bases $\mu _1$ and $\mu _2$, $\sigma = \sigma(\mu_1) = \sigma(\mu_2)$ and $\mu_1, \mu_2$ is not a pair of matched bases, i.e. $\mu_1\ne \Delta(\mu_2)$; moreover, in the generalised equation $\widetilde {\Omega}_v= \D 3(\Omega)$ all the bases obtained from $\mu _1,\mu _2$  by $\ET 1$ when constructing $\widetilde {\Omega}_v$ from $\Omega_v$, do not belong to the kernel of $\widetilde{\Omega}_v$.}

In this case, using $\ET 5$ we  $\mu _1$-tie  all the boundaries that intersect $\mu_1$; using $\ET 2$, we transfer $\mu _2$ onto $\Delta (\mu_1)$; and  remove $\mu _1$ together with the closed section $\sigma$  using  $\ET 4$,  see Figure \ref{linear''}.

\subsubsection*{Case 10: There exists an item $h_i$ in an active section of $\Omega_v$ such that $\gamma _i=1$ and  such that both boundaries $i$ and $i+1$ are open.}

In this event we close the section $[i,i+1]$ using $\D 1$ and remove it using $\ET 4$, see Figure \ref{linear''}.

\begin{figure}[!h]
  \centering
   \includegraphics[keepaspectratio,width=6in]{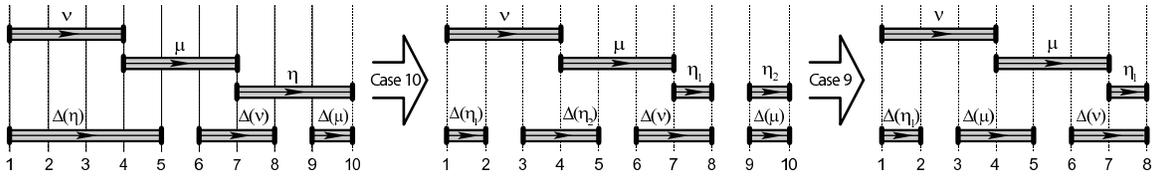}
  \caption{Cases 9-10: Linear case.} \label{linear''}
\end{figure}

\subsubsection*{Tying a free boundary: Case 11}

\subsubsection*{Case 11: Some boundary $i$ in the active part of $\Omega_v$ is free.}

Since the assumptions of Case 5 are not satisfied, the boundary $i$ intersects at least one base, say, $\mu$.

In this case, we $\mu$-tie $i$  using $\ET 5$.

\subsubsection*{Quadratic case: Case 12}

\subsubsection*{Case 12: For every item $h_i$ in the active part of $\Omega_v$ we have $\gamma _i = 2$.}

We apply the entire transformation $\D 5$, see Figure \ref{quadratic}.

\begin{figure}[!h]
  \centering
   \includegraphics[keepaspectratio,width=6in]{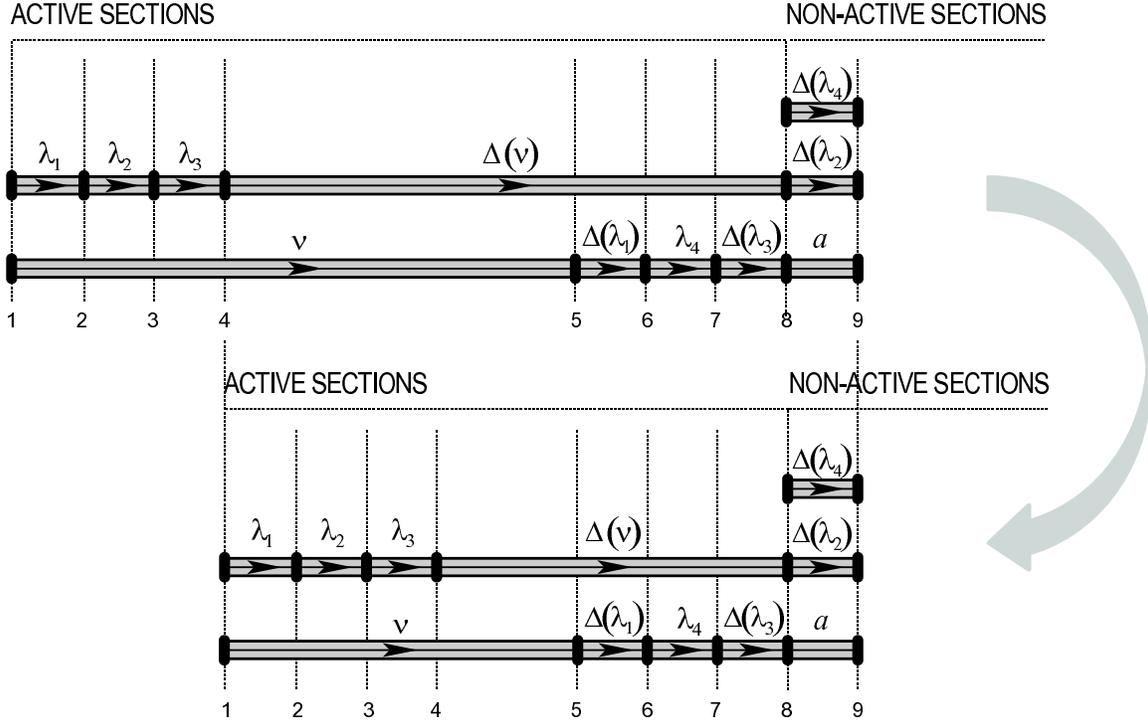}
  \caption{Case 12: Quadratic case, entire transformation.} \label{quadratic}
\end{figure}

\subsubsection*{Removing a closed section: Case 13}

\subsubsection*{Case 13: For every item $h_i$ in the active part of $\Omega_v$ we have $\gamma _i\geq 2$, and $\gamma_{i'}>2$ for at least one item $h_{i'}$ in the active part. Moreover, for some active base $\mu$ the section $\sigma(\mu)$ is closed.}

In this case, using $\D 2$, we transport the section $\sigma(\mu)$ to the beginning of the interval (this makes $\mu$ a carrier base). We then apply the entire transformation $\D 5$ (note that in this case the whole section $\sigma(\mu)$ is removed).

\subsubsection*{Tying a boundary: Case 14}

\subsubsection*{Case 14:  For every item $h_i$ in the active part of $\Omega_v$ we have $\gamma _i\geq 2$, and $\gamma _{i'}>   2$ for at least one item $h_{i'}$ in the active part. Moreover,  some boundary $j$ in the active part touches some base $\lambda$, intersects some base $\mu$,  and $j$ is not $\mu$-tied.}

In this case, using $\ET 5$, we $\mu$-tie $j$.

\subsubsection*{General case: Case 15}

\subsubsection*{Case 15: For every item $h_i$ in the active part of $\Omega_v$ we have $\gamma _i\geq 2$, and $\gamma _{i'}>   2$ for at least one item $h_{i'}$ in the active part. Moreover, every active section $\sigma(\mu)$ is not closed and every boundary $j$ is $\mu$-tied in every base it intersects.}

We first apply the entire transformation $\D 5$. Then, using $\ET 5$, we $\mu$-tie every boundary $j$ in the active part that intersects a base $\mu$ and touches at least one base. This results in finitely many new vertices connected to $v$ by principal edges.

If, in addition,  $\Omega_v$ satisfies the assumptions of Case 15.1 below, then besides the principal edges already constructed, we construct a few more auxiliary edges outgoing from the vertex $v$.

\subsubsection*{Case 15.1: The carrier base $\mu$ of the generalised equation $\Omega _v$ intersects with its dual $\Delta (\mu)$.}

We first construct an auxiliary generalised equation ${\widehat{\Omega}_{ v}}$ (which does \emph{not} appear in the tree $T(\Omega)$) as follows. Firstly, we add a new constant section $[\rho _{\Omega_v}+1,\rho_{\Omega_v}+2]$ to the right of all the sections in $\Omega_v$ (in particular, $h_{\rho _{\Omega_v}+1}$ is a new free variable). Secondly, we introduce a new pair of dual variable bases $\lambda ,\Delta (\lambda)$ so that
$$
\alpha(\lambda)=1, \ \beta(\lambda)=\beta(\Delta (\mu)),\  \alpha(\Delta(\lambda))=\rho_{\Omega_v}+1,\  \beta(\Delta(\lambda))=\rho_{\Omega_v}+2.
$$
Notice that $\Omega _v$ can be obtained from ${\widehat{\Omega}_{v}}$ if we apply $\ET 4$ to ${\widehat{\Omega}_{v}}$ and delete the base $\Delta(\lambda)$ together with the closed section $[\rho _v+1,\rho _v+2]$. Let
$$
{\hat \pi}_{v}: \GG_{R(\Omega _v^\ast)}\rightarrow \GG_{R({{{\widehat{\Omega}}_v}}^{\ast})}
$$
be the isomorphism induced by $\ET 4$. The assumptions of Case 15 still hold for $\widehat {\Omega} _{ v}$. Note that the carrier base of  $\widehat {\Omega} _{ v}$ is the base $\lambda$. Applying the transformations described in Case 15 to ${\widehat \Omega _{v}}$ (first $\D 5$ and then $\ET 5$), we obtain a set of new generalised equations $\{\Omega_{v^\prime_i}\mid i=1,\dots, \nn\}$ and the set of corresponding epimorphisms of the form:
$$
{\phi}_{v'_i}: \GG_{R({\widehat{\Omega}_{ v}}^\ast)} \rightarrow \GG_{R(\Omega_{v_i'}^\ast)}.
$$

Now for each generalised equation $\Omega_{v'_i}$ we add a vertex $v'_i$ and an auxiliary edge $v\to v'_i$ in the tree $T(\Omega)$. The edge $v\to v'_i$ is labelled by the homomorphism $\pi(v,v'_i)$ which is the composition of homomorphisms $\phi_{v'_i}$ and ${\hat \pi}_{v}$, $\pi(v,v'_i) = {\hat \pi}_{v}\phi_{v'_i}$. We associate the generalised equation $\Omega_{v'_i}$ to the vertex $v'_i$.

\bigskip

The tree $T(\Omega)$ is therefore constructed. Observe that, in general, $T(\Omega)$ is an infinite locally finite tree.

If the generalised equation $\Omega_v$ satisfies the assumptions of Case $i\ (1\leq i\leq 15)$, then we say that the vertex $v$ has \index{type!of a vertex}\index{type!of a generalised equation}\emph{type} $i$ and write \glossary{name={$\tp(v)$}, description={type of the vertex $v$ of the tree $T$, $T_0$, $T_{\dec}$, $T_{\ext}$ or $T_{\sol}$}, sort=T}$\tp(v)=i$.

\begin{lem}
For any vertex $v$ of the tree $T(\Omega)$ there exists an algorithm to decide whether or not $v$ is of type $k$, $k=1,\dots, 15$.
\end{lem}
\begin{proof}
The statement of the lemma is obvious for all cases, except for Case 1.

To decide whether or not the vertex $v$ is of type 1 it suffices to show that there exists an algorithm to check whether or not the canonical epimorphism $\pi(v,u)$ associated to an edge $v \rightarrow u$ in $T(\Omega)$ is a proper epimorphism. This can be done effectively by Lemma \ref{le:hom-check}.
\end{proof}

\begin{lem}[cf. Lemma 3.1, \cite{Razborov3}] \label{3.1}
Let $u\rightarrow v$ be  a principal edge of the tree $T(\Omega)$. Then the following statements hold.
\begin{enumerate}
    \item If $\tp(u)\ne 3,10$, then $n_{A}(\Omega_v) \leq n_{A}(\Omega_u)$, moreover if $\tp(u)=6,7,9,13$, then this inequality is strict;
    \item If $\tp(u)=10$, then $n_{A}(\Omega_v) \leq n_{A}(\Omega_u) + 2$;
    \item If $\tp(u)\leq 13$ and $\tp (u)\ne 3,11$, then $\xi(\Omega_v) \leq \xi(\Omega_u)$;
    \item If $\tp (u) \ne 3$, then $\comp(\Omega_v)  \leq \comp(\Omega_u)$.
\end{enumerate}
\end{lem}
\begin{proof}
Straightforward verification.
\end{proof}

The following lemma gives a description of the infinite branches in the tree $T(\Omega)$.

\begin{lem}[cf. Lemma 3.2, \cite{Razborov3}] \label{3.2} Let
\begin{equation}\label{eq:path}
v_0\rightarrow v_1\rightarrow\ldots \rightarrow v_r \rightarrow \ldots
\end{equation}
be an infinite path in the tree $T(\Omega )$. Then there exists a natural number $N$ such that all the edges
$v_n \rightarrow v_{n+1}$ of this path with $n \geq N$ are principal edges, and one of the following conditions holds:
$$
\begin{array}{crl}
  (A) & \hbox{linear case: } & 7\leq \tp(v_n)\leq 10 \hbox{ for all } n\ge N;\\
  (B) & \hbox{quadratic case: } & \tp(v_n)=12 \hbox{ for all } n\ge N; \\
  (C) & \hbox{general case: } & \tp(v_n)=15 \hbox{ for all } n\ge N.
\end{array}
$$
\end{lem}
\begin{proof}
Firstly, note that Cases 1 and 2 can only occur once in the path (\ref{eq:path}).

Secondly, note Cases 3 and 4 can occur only finitely many times in the path (\ref{eq:path}), namely, at most $2t$ times where $t$ is the number of constant bases in the original generalised equation $\Omega$. Therefore, there exists a natural number $N_1$ such that $\tp(v_i)\geq 5$ for all $i \geq N_1$.

Now we show that the number of vertices $v_i$ ($i \geq N$) for which $\tp (v_i)=5$ is bounded above by the minimal number of generators of the group $\GG_{R(\Upsilon^\ast)}$, in particular, it cannot be greater than $\rho +1 +|\cA|$, where $\rho =  \rho_\Upsilon$. Indeed, if a path from the root $v_0$ to a vertex $v$ contains $k$ vertices of type 5, then $\Upsilon_v$ has at least $k$ free variables in the non-active part. This implies that the coordinate group $\GG_{R(\Upsilon _v^\ast)}$ has the free group of rank $k$ as a free factor, hence it cannot be generated by less than $k$ elements. Since $\pi'(v_0,v): \GG_{R(\Upsilon^\ast )} \to \GG_{R(\Upsilon _v^\ast)}$ is a surjective homomorphism, the group $\GG_{R(\Upsilon^\ast)}$ cannot be generated by less then $k$ elements. This shows that $k \leq \rho +1 +|\cA|$.  It follows that there exists  a number $N_2 \geq N_1$ such that $\tp(v_i) >5$ for every $i \geq N_2$.

Since the path (\ref{eq:path}) is infinite, by Lemma \ref{3.1}, we may assume that for every $i,j\ge N_3$  one has $\comp(\Omega_{v_i})=\comp(\Omega_{v_j})$, i.e. the complexity stabilises.

If $v_i \rightarrow v_{i+1}$ is  an auxiliary edge, where $i \geq N_3$, then $\tp(v_i) = 15$ and $\Omega_{v_i}$ satisfies the assumption of Case 15.1. In the notation of Case 15.1, we analyse the complexity of the generalised equation $\Omega_{v_{i+1}}$ obtained from ${\widehat{\Omega}}_{v_i}$. As both bases $\mu$ and $\Delta (\mu)$ are transferred from the carrier base $\lambda$ of ${\widehat{\Omega}}_{v_i}$ to the  non-active part, so the complexity decreases by at least two, i.e. $\comp(\Omega_{v_{i+1}}) \leq \comp(\widehat{\Omega}_{v_i}) - 2$. Observe also that $\comp ({\widehat{\Omega}_{v_i}}) = \comp(\Omega_{v_i}) + 1$. Hence $\comp(\Omega_{v_{i+1}}) <  \comp(\Omega_{v_i})$, which derives a contradiction. Hence, if $i \geq N_3$ the edge $v_i\to v_{i+1}$ is principal.

Suppose that $i \ge N_3$. If $\tp(v_i)=6$, then the closed section containing the matched bases $\mu, \Delta (\mu)$, does not contain any other bases (otherwise $\comp(\Omega_{v_{i+1}})<\comp(\Omega_{v_{i}})$). But then $\tp(v_{i+1})=5$ deriving a contradiction. We therefore get that for all $i \ge N_3$ one has $\tp(v_i) >6$.

If $\tp(v_i)=12$, then it is easy to see that $\tp(v_{i+1})=6$ or $\tp(v_{i+1})=12$. Therefore, if $i\ge N_3$ we get that $\tp(v_{i+1})=\tp(v_{i+2})= \dots = \tp(v_{i+j}) = 12$ for every $j>0$ and we have case (B) of the lemma.

Suppose that $\tp(v_{i})\neq 12$ for all $i \geq N_3$.

Notice that the only elementary (derived) transformation that may produce free boundaries is $\ET 3$ and that $\ET 3$ is applied only in Case 6. Since for $i \geq N_3$ we have $\tp(v_i)\ge 7$, we see that there are no new free boundaries in the generalised equations $\Omega_{v_i}$ for $i \geq N_3$. It follows that there exists a number $N_4 \geq N_3$ such that $\tp(v_i)\neq 11$ for every $i \geq N_4$.

Suppose now that for some $i \geq N_4$,  $13\leq \tp(v_i)\leq 15$. It is easy to see from the description of Cases 13, 14 and 15 that $\tp(v_{i+1})\in\{6,13,14,15\}$. Since $\tp(v_{i+1})\neq 6$, this implies that $13\leq \tp(v_j)\leq 15$ for every $j \geq i$. Then, by Lemma \ref{3.1}, we have that $n_A(\Omega_{v_j})\le n_A(\Omega_{v_{N_4}})$, for every $j \geq N_4$. Furthermore, if $\tp (v_j)=13$, then $n_A(\Omega_{v_{j+1}}) <  n_A(\Omega_{v_j})$. Hence there exists a number $N_5 \geq N_4$ such that $\tp (v_j)\ne 13$ for all $j \geq N_5$.

Suppose that $i \geq N_5$. Note that there can be at most $8(n_A(\Omega_{v_i}))^2$ vertices of type 14 in a row starting at the vertex $v_i$, hence there exists $j\geq i$ such that $\tp(v_j)=15$. From the description of Case 15 it follows that $\tp (v_{j+1})\ne 14$, thus $\tp(v_{j+1})=15$, and we have case (C) of the lemma.

Finally, we are left with the case when $7\le \tp (v_i)\leq 10$ for all the vertices of the path. We then have case (A) of the lemma.
\end{proof}

\begin{rem}\label{rem:leng<}
Let
$$
(\Omega_{v_1},H^{(1)})\to (\Omega_{v_2},H^{(2)})\to \dots \to (\Omega_{v_l},H^{(l)})
$$
be the path defined by the solution $H^{(1)}$. If $ \tp(v_i)\in \{7,8,9,10,12,15\}$, then by {\rm Lemma \ref{lem:solleng}},  $|H^{(i+1)}|<|H^{(i)}|$.
\end{rem}

\section{Minimal solutions} \label{sec:minsol}

In this section we introduce a reflexive, transitive relation on the set of solutions of a generalised equation. We use this relation to introduce the notion of a minimal solution with respect to a group of automorphisms of the coordinate group of the generalised equation. In the end of the section we describe the behaviour of minimal solutions with respect to the elementary transformations.

Let $u$ and $v$ be two geodesic words from $\GG$. Consider the product $uv$ of $u$ and $v$. The geodesic $\ov{uv}$ of the element $uv$ can be written as follows $\ov{uv}\doteq u_1 \cdot v_2$, where $u=u_1d^{-1}$, $v= dv_2$ and $d$ is the greatest (left) common divisor of $u^{-1}$ and $v$ (in the sense of \cite{EKR}). Notice that if the word $uv$ is not geodesic, then it follows that $d$ is non-trivial. We call $d$ the \index{cancellation divisor}\emph{cancellation divisor of $u$ and $v$} and we denote it by \glossary{name={$\nfd(u,v)$}, description={the cancellation divisor of $u$ and $v$}, sort=C}$\nfd(u,v)$.

\begin{defn}\label{defn:sol<}
Let $\GG(\cA\cup B)$ be a partially commutative $\GG$-group, let $\Omega$  be a generalised equation with coefficients from $(\cA\cup B)^{\pm 1}$ over the free monoid $\FF(\cA^{\pm 1}\cup B^{\pm 1})$. Let $\BB(\Omega)$ be an arbitrary  group of $\GG$-automorphisms of $\factor{\GG(\cA \cup B)[h]}{R(\Omega^*)}$. For solutions  $H^{(1)}$ and $H^{(2)}$ of the generalised equation $\Omega$ we write \glossary{name={`$<_{\BB(\Omega)}$'}, description={reflexive, transitive relation on the set of solutions of a generalised equation}, sort=Z}$H^{(1)}<_{\BB(\Omega)} H^{(2)}$ if there exists a $\GG$-endomorphism $\pi$ of the group $\GG(\cA \cup B)$ and an automorphism $\sigma\in \BB(\Omega )$  such that the following conditions hold:
\begin{enumerate}
\item \label{it:minsol1} $\pi_{H^{(2)}}=\sigma\pi_{ H^{(1)}}\pi$;
\item \label{it:minsol3} for any $k,j$, $1\le k,j\le \rho$, if the word ${(H_k^{(2)})}^{\epsilon}{(H_j^{(2)})}^{\delta}$, $\epsilon,\delta \in \{1, -1\}$, is geodesic as written, then the word ${(H_k^{(1)})}^{\epsilon}{(H_j^{(1)})}^{\delta}$ is geodesic as written;
\item \label{it:minsol4} for any $k,j$, $1\le k,j\le \rho$ such that the word ${(H_k^{(2)})}^{\epsilon}{(H_j^{(2)})}^{\delta}$, $\epsilon,\delta \in \{1, -1\}$, is not geodesic, if
\begin{equation}\label{eq:minsol1}
    \BA(J)\cap \left\{a\in\cA^{\pm 1}\mid a \hbox{ is a left divisor of } \nfd\left({{(H_k^{(2)})}^{\epsilon}},{(H_j^{(2)})}^{\delta}\right)\right\}=\emptyset,
\end{equation}
for some $J=\{H_{i_1}^{(2)},\dots, H_{i_t}^{(2)}\}$, then either the word ${(H_k^{(1)})}^{\epsilon}{(H_j^{(1)})}^{\delta}$ is geodesic as written or
\begin{equation}\label{eq:minsol2}
    \BA(J')\cap \left\{a\in\cA^{\pm 1}\mid a \hbox{ is a left divisor of } \nfd\left({{(H_k^{(1)})}^{\epsilon}}, {(H_j^{(1)})}^{\delta}\right)\right\}=\emptyset,
\end{equation}
where $J'=\{H_{i_1}^{(1)},\dots, H_{i_t}^{(1)}\}$.
\end{enumerate}
\end{defn}
Obviously, the relation `$<_{\BB(\Omega )}$' is transitive. We would like to draw the reader's attention to the fact that the relation $H<_{\BB(\Omega)}H'$ does not imply relations on the lengths of the solutions $H$ and $H'$.

The motivation for property (\ref{it:minsol4}) in the above definition is the following lemma.

\begin{lem}\label{lem:minsol}
Let $\GG(\cA\cup B)$ be a partially commutative $\GG$-group, let $\Omega$  be a generalised equation with coefficients from $(\cA\cup B)^{\pm 1}$ over the free monoid $\FF(\cA^{\pm 1}\cup B^{\pm 1})$. Let $\BB(\Omega)$ be an arbitrary  group of $\GG$-automorphisms of $\factor{\GG(\cA \cup B)}{R(\Omega^*)}$ and let  $H^{(1)}$ and $H^{(2)}$ be two solutions of the generalised equation $\Omega$ such that $H^{(1)}<_{\BB(\Omega)} H^{(2)}$. Then for any word $W(x_1,\dots,x_\rho)\in F(x_1,\dots,x_\rho)$ such that the $W(H^{(2)}_1,\dots,H^{(2)}_\rho)$ is geodesic as written {\rm(}treated as an element of $\GG(\cA\cup B)${\rm)}, the word $W(H^{(1)}_1,\dots,H^{(1)}_\rho)$ is geodesic as written {\rm(}treated as an element of $\GG(\cA\cup B)${\rm)}.
\end{lem}
\begin{proof}
We use induction on the length of $W$. If $|W|=1$, then the statement follows, since by the definition of a solution of the generalised equation, the words $(H^{(1)}_j)^{\pm 1}$ and $(H^{(2)}_j)^{\pm 1}$ are both geodesic.

Let $W(x_1, \dots, x_\rho)= x_{l_1}^{\epsilon_1}\cdots x_{l_{n+1}}^{\epsilon_{n+1}}$, $\epsilon_i\in \{-1,1\}$, $1\le i\le n+1$, be a word such that $W(H^{(2)}_1,\dots,H^{(2)}_\rho)$ is geodesic and assume that $W(H^{(1)}_1,\dots,H^{(1)}_\rho)$ is not geodesic. By induction hypothesis, the word $W_n(H^{(1)})={H_{l_1}^{(1)}}^{\epsilon_1}\cdots {H_{l_{n}}^{(1)}}^{\epsilon_{n}}$ is geodesic. Consider the cancellation divisor $D=\nfd\left(W_n(H^{(1)}),{H_{l_{n+1}}^{(1)}}^{\epsilon_{n+1}}\right)$. Analogously, by induction hypothesis, the word ${H_{l_2}^{(1)}}^{\epsilon_2}\cdots {H_{l_{n+1}}^{(1)}}^{\epsilon_{n+1}}$ is geodesic. Using the fact that the word $W_n(H^{(1)})$ is geodesic, we get that the cancellation divisor $\nfd\left({H_{l_1}^{(1)}}^{\epsilon_1},{H_{l_2}^{(1)}}^{\epsilon_2}\cdots {H_{l_{n+1}}^{(1)}}^{\epsilon_{n+1}}\right)$ is, in fact, $D$, and, in particular, no left divisor of the word ${H_{l_2}^{(1)}}^{\epsilon_2}\cdots {H_{l_{n}}^{(1)}}^{\epsilon_{n}}$ left-divides $D$. We thereby have that
\begin{enumerate}
    \item[(a)] on one hand, the word ${H_{l_1}^{(1)}}^{\epsilon_1}D$ is not geodesic and,
    \item[(b)] on the other hand, by Proposition 3.18 of \cite{EKR}, that $D \lra {H_{l_{2}}^{(1)}}^{\epsilon_{2}}\cdots{H_{l_{n}}^{(1)}}^{\epsilon_{n}}$.
\end{enumerate}

From (a) we get that the word ${H_{l_1}^{(1)}}^{\epsilon_1}{H_{l_{n+1}}^{(1)}}^{\epsilon_{n+1}}$ is not geodesic, and thus by property (\ref{it:minsol3}) from Definition \ref{def:minsol}, it follows that the word ${H_{l_1}^{(2)}}^{\epsilon_1}{H_{l_{n+1}}^{(2)}}^{\epsilon_{n+1}}$ is not geodesic.

From (b) and the fact that the word ${H_{l_{2}}^{(1)}}^{\epsilon_{2}}\cdots{H_{l_{n}}^{(1)}}^{\epsilon_{n}}$ is geodesic, we have that $D\in \BA({H_{l_i}^{(1)}})$ for every $i=2,\dots, n$. It follows that Equation (\ref{eq:minsol2}) fails for the set $J'=\{H_{l_2}^{(1)},\dots, {H_{l_{n}}^{(1)}}\}$ and thus Equation (\ref{eq:minsol1}) fails for the set $J=\{H_{l_2}^{(2)},\dots, {H_{l_{n}}^{(2)}}\}$. This derives a contradiction with the fact that ${H_{l_1}^{(2)}}^{\epsilon_1}\cdots {H_{l_{n+1}}^{(2)}}^{\epsilon_{n+1}}$ is geodesic, since there exists a letter $a\in \cA^{\pm 1}$ such that $a$ is a divisor of $\nfd\left({H_{l_1}^{(2)}}^{\epsilon_1},{H_{l_{n+1}}^{(2)}}^{\epsilon_{n+1}}\right)$ and $a \in \BA(J)$.
\end{proof}

\begin{lem}\label{lem:liftsol}
Let the generalised equation $\Omega_1$ be obtained from the generalised equation $\Omega$ by one of the elementary transformation $\ET 1-\ET 5$, i.e. $\Omega_1\in \ET i(\Omega)$ for some $i=1,\dots, 5$. Let $H$ be a solution of $\Omega$ and $H^{(1)}$ be a solution of $\Omega_1$ so that the following diagram commutes
$$
\xymatrix@C3em{
 \GG_{R(\Omega^\ast)}  \ar[rd]_{\pi_H} \ar[rr]^{\theta}  &   &\GG_{R(\Omega_1^\ast )} \ar[ld]^{\pi_{H^{(1)}}} \\
                               &  \GG &}
$$
Let ${H^{(1)}}^+$ be another solution of $\Omega_1$ so that ${H^{(1)}}^+<_{\BB(\Omega)}{H}^{(1)}$, where ${\BB(\Omega)}$ is arbitrary. The $\rho$-tuple $H^+$ of geodesic words of $\GG$ defined by the homomorphism $\pi_{H^+}=\theta \pi_{{H^{(1)}}^+}$ is a solution of the generalised equation $\Omega$.
\end{lem}
\begin{proof}
Proof is by examination of the definitions of elementary transformations. We further use the notation introduced in the definitions of elementary transformations.

Suppose first that $\Omega_1$ is obtained from $\Omega$ by $\ET 1$.  Since, in this case, $H^+={H^{(1)}}^+$, we have that $H^+$ is a tuple of non-trivial geodesic words in $\GG$. The elementary transformation $\ET 1$ is invariant on all the bases but the pair $\lambda$, $\Delta(\lambda)$. Therefore, we are left to prove that the words $H^+(\lambda)$ and $H^+(\Delta(\lambda))$ are geodesic and that the equality ${H^+(\lambda)}^{\varepsilon(\lambda)}= {H^+(\Delta(\lambda))}^{\varepsilon(\Delta(\lambda))}$ is graphical.

Since $H(\lambda)=H^{(1)}[\alpha(\lambda_1), \beta(\lambda_2)]$ is a geodesic word and ${H^{(1)}}^+<_{\BB(\Omega)}{H}^{(1)}$, by Lemma \ref{lem:minsol} it follows that $H^+(\lambda)={H^{(1)}}^+[\alpha(\lambda_1), \beta(\lambda_2)]$ is geodesic. Similarly, $H^+(\Delta(\lambda))$ is also geodesic. Since
${H^{(1)}}^+$ is a solution of $\Omega_1$, we have that
$$
{{H^{(1)}}^+(\lambda_i)}^{\varepsilon(\lambda_i)}\doteq {{H^{(1)}}^+(\Delta(\lambda_i))}^{\varepsilon(\Delta(\lambda_i))},\ \hbox{  where $i=1,2$.}
$$
From the equalities
$$
H^+(\lambda)={H^{(1)}}^+(\lambda_1){H^{(1)}}^+(\lambda_2) \ \hbox{ and } \ H^+(\Delta(\lambda))={H^{(1)}}^+(\Delta(\lambda_1)){H^{(1)}}^+(\Delta(\lambda_2)),
$$
and the fact that the words $H^+(\lambda)$ and $H^+(\Delta(\lambda))$ are geodesic, we obtain that the equality ${H^+(\lambda)}^{\varepsilon(\lambda)}= {H^+(\Delta(\lambda))}^{\varepsilon(\Delta(\lambda))}$ is graphical.

The other cases are similar and left to the reader.
\end{proof}

\begin{defn}\label{defn:nfmatrix}
Let $\Omega$ be a generalised equation in $\rho$ variables and let $H$ be a solution of $\Omega$.  Consider a $2\rho \times 2\rho$ matrix $(m_{i_1,i_2})$, $1\le i_1,i_2\le 2\rho$ constructed by the solution $H$ in the following way. The elements $m_{i_1,i_2}$ of the matrix are $2^{\rho}+1$-vectors with entries from the set $\{0,1\}$. We enumerate (in an arbitrary way) the set  $\chi(\{H_1, \dots, H_\rho\})$ of all subsets of $\{H_1, \dots, H_\rho\}$. Abusing the notation, in this definition if $i_l>\rho$, then by $H_{i_l}$ we mean ${H_{i_l-\rho}}^{-1}$.

The vector $m_{i_1,i_2}$ has all of its components equal to $0$ if and only if the word $H_{i_1}H_{i_2}$ is geodesic as written. The first component of the vector $m_{i_1,i_2}$ equals $1$ if and only if the word $H_{i_1}H_{i_2}$ is not geodesic. The $l$-th component of $m_{i_1,i_2}$ equals $1$, $l>1$, if for the $l-1$-th set $J$ of $\chi(\{H_1, \dots, H_\rho\})$ Equation (\ref{eq:minsol1}) is satisfied. Otherwise the $l$-th component of $m_{i_1,i_2}$ equals $0$.

We call the matrix $(m_{i_1,i_2})$, $1\le i_1,i_2\le 2\rho$ the \index{cancellation matrix}\emph{cancellation matrix} of the solution $H$.
\end{defn}

\begin{defn}\label{def:minsol}
A solution $H$ of $\Omega$ is called \index{solution!of a generalised equation!minimal with respect to the group of automorphisms}\emph{minimal with respect to the group of automorphisms $\BB(\Omega)$} if there exist no solutions $H'$ of the generalised equation $\Omega$ so that $H' <_{\BB(\Omega)} H$ and $|H_k'|\leq |H_k|$  for all $k$, $k=1,\dots,\rho$ and $|H_j'|< |H_j'|$ for at least one $j$, $1\le j\le \rho$.

Since the length of a solution $H$ is a positive integer, every strictly descending chain of solutions
$$
H>_{\BB(\Omega)} H^{(1)} >_{\BB(\Omega)}\dots >_{\BB(\Omega)} H^{(k)} >_{\BB(\Omega)}\ldots
$$
is finite. It follows that for every solution $H$ of $\Omega$ there exists a minimal solution $H^+$ such that $H^+ < _{\BB(\Omega)} H$.
\end{defn}

\begin{rem} \label{rem:ms}
Note that given a solution $H$, there exists a minimal solution $H'$ (perhaps more than one) so that $H'<_{\BB(\Omega)} H$. Furthermore, there may exist minimal solutions $H'$ and $H^+$ so that $H'\not <_{\BB(\Omega)} H^+$ and $H^+\not <_{\BB(\Omega)} H'$. In this case one has $|H_k'|< |H_k^+|$  for some $1\le k\le\rho$ and $|H_j'|> |H_j^+|$ for some $1\le j\le \rho$.
\end{rem}

\begin{rem} \label{rem:minsol}
Note that every generalised equation $\Omega$ with coefficients from $\GG$ can be considered as a generalised equation $\Omega'$ with coefficients from $\GG(\cA \cup B)$ for \emph{any} partially commutative $\GG$-group $\GG(\cA \cup B)$. Therefore, any solution $H$ of $\Omega$ induces a solution $H'$ of $\Omega'$ such that the following diagram commutes:
$$
\xymatrix@C3em{
 \GG_{R(\Omega^*)}  \, \ar@{^{(}->}[r] \ar[d]_{\pi_H}  & \,\factor{\GG(\cA \cup B)}{R(\Omega^\ast)} \ar[d]^{\pi_{H'}}
                                                                             \\
                   \GG \,   \ar@{^{(}->}[r] & \,\GG(\cA \cup B)
}
$$
A solution $H$ of $\Omega$ is minimal if  so is \emph{any} induced solution $H'$.
\end{rem}
The reason for extending the generating set from $\cA$ to $\cA\cup B$ in the definition above, becomes clear in the proof of Lemma \ref{lem:23-1.5}

\bigskip

\begin{lem} \label{lem:2.1}
Let the generalised equation $\Omega_1$ be obtained from the generalised equation $\Omega$ by one of the elementary transformation $\ET 1-\ET 5$, i.e. $\Omega_1\in \ET i(\Omega)$ for some $i=1,\dots, 5$ and assume that $\GG_{R(\Omega_1^*)}$ is isomorphic to $\GG_{R(\Omega^*)}$. Let $H$ be a solution of $\Omega$ and $H^{(1)}$ be a solution of $\Omega_1$ so that the following diagram commutes
$$
\xymatrix@C3em{
 \GG_{R(\Omega^\ast)}  \ar[rd]_{\pi_H} \ar[rr]^{\theta}  &   &\GG_{R(\Omega_1^\ast )} \ar[ld]^{\pi_{H^{(1)}}} \\
                               &  \GG &}
$$
If $H$ is a minimal solution of $\Omega$ with respect to a group of automorphisms $\BB$ of $\GG_{R(\Omega^*)}$, then $H^{(1)}$ is a minimal solution of $\Omega_1$ with respect to the group of automorphisms $\theta^{-1}\BB\theta$.
\end{lem}
\begin{proof}
Assume the converse, i.e. $H^{(1)}$ is not minimal with respect to $\theta^{-1}\BB\theta$. Then there exists a solution ${H^{(1)}}^+$ of $\Omega_1$ so that ${H^{(1)}}^+<_{\theta^{-1}\BB\theta} H^{(1)}$,  $|{H^{(1)}}^+_k|\le |H^{(1)}_k|$ for all $k$ and $|{H^{(1)}}^+_j|< |H^{(1)}_j|$ for some $j$.

Let $H^+$ be a solution of the system of equations $\Omega^*$ so that $\pi_{H^+}=\theta\pi_{{H^{(1)}}^+}$. By Lemma \ref{lem:liftsol}, $H^+$ is a solution of $\Omega$.

By Lemma \ref{lem:minsol}, condition (\ref{it:minsol3}) from Definition \ref{defn:sol<} holds for the pair $H^+$ and $H$. We now show that $H^+<_{\BB} H$ and derive a contradiction with the minimality of $H$.

We now show that condition (\ref{it:minsol4}) from Definition \ref{defn:sol<} holds. Assume that for the pair $H_k^+$, $H_j^+$ the word ${H_k^+}^\epsilon {H_j^+}^\delta$ is not geodesic and that there exists a set $J'=\{H_{i_1}^{+},\dots, H_{i_t}^{+}\}$ such that Equation (\ref{eq:minsol2}) fails, i.e. there exists a set $J'$ and a letter $a\in \cA^{\pm 1}$ such that $a\in \BA(J')$  and $a$ is a left divisor of $\nfd({H_k^+}^\epsilon, {H_j^+}^\delta)$. Without loss of generality, we further assume that $\epsilon, \delta=1$. Write $H_k^+$ and $H_j^+$ as words in the ${H_{i}^{(1)}}^+$'s:
$$
H_k^+\doteq {{H_{k_1}^{(1)}}^+}^{\epsilon_1}\cdots {{H_{k_s}^{(1)}}^+}^{\epsilon_s} \hbox{ and } H_j^+\doteq {{H_{j_1}^{(1)}}^+}^{\delta_1}\cdots {{H_{j_r}^{(1)}}^+}^{\delta_r}.
$$
It follows that $a$ is a left divisor of $\nfd\left({{H_{k_m}^{(1)}}^+}^{\epsilon_m}, {{H_{j_n}^{(1)}}^+}^{\delta_n}\right)$ for some $m$ and $n$, and that
$$
a\lra {{H_{j_1}^{(1)}}^+}^{\delta_1}\cdots {{H_{j_{n-1}}^{(1)}}^+}^{\delta_{n-1}} \hbox{ and } a\lra {{H_{k_{m+1}}^{(1)}}^+}^{\epsilon_{m+1}}\cdots {{H_{k_s}^{(1)}}^+}^{\epsilon_s}.
$$
Therefore, Equation (\ref{eq:minsol2}) fails for the solution ${H^{(1)}}^+$ and the set ${J^{(1)}}'$, where  the set ${J^{(1)}}'$ is a union of  words ${{H_{i}^{(1)}}^+}$ that appear in the decomposition of a word $H_{i_l}^{+}\in J'$ and  the set $\{{{H_{j_1}^{(1)}}^+},\dots, {{H_{j_{n-1}}^{(1)}}^+},{{H_{k_{m+1}}^{(1)}}^+}, \dots, {{H_{k_s}^{(1)}}^+}\}$. Since ${H^{(1)}}^+<_{\theta^{-1}\BB\theta} H^{(1)}$, by condition (\ref{it:minsol4}) from Definition \ref{defn:sol<}, Equation (\ref{eq:minsol1}) fails for the solution ${H^{(1)}}$ and the corresponding set ${J^{(1)}}$, i.e. there exists a letter $b\in \cA^{\pm 1}$ such that $b\in \BA({J^{(1)}})$ and $b$ is a left divisor of $\nfd\left({{H_{k_m}^{(1)}}}^{\epsilon_m}, {{H_{j_n}^{(1)}}}^{\delta_n}\right)$.

Since $b\lra \{{{H_{j_1}^{(1)}}},\dots, {{H_{j_{n-1}}^{(1)}}},{{H_{k_{m+1}}^{(1)}}}, \dots, {{H_{k_s}^{(1)}}}\}$, we get that $b$ is a left divisor of $\nfd(H_k^\epsilon, H_j^\delta)$. Furthermore, since $b$ $\lra$-commutes with the words ${{H_{i}^{(1)}}}$'s that appear in the decomposition of a word $H_{i_l}\in J$ (where $J$ is the set corresponding to $J'$), we get that $b\in \BA(J)$. It follows that Equation (\ref{eq:minsol1}) fails for the solution $H$ and the set $J$. Hence, condition (\ref{it:minsol4}) from Definition \ref{defn:sol<} holds for the solutions $H$ and $H^+$.

Furthermore, by condition  (\ref{it:minsol1}) from Definition \ref{defn:sol<}, we have $\pi_{H^{(1)}}=\theta^{-1}\psi\theta\pi_{{H^{(1)}}^+}$, where $\psi \in \BB$, hence $\pi_H=\psi\pi_{H^+}$, and thus condition (\ref{it:minsol1}) from Definition \ref{defn:sol<} holds for the pair $H^+$ and $H$. We thereby have proven that $H^+<_{\BB}H$.

Finally, since $H_i\doteq w_i(H^{(1)})$, $H_i^+\doteq w_i({H^{(1)}}^+)$ and $|{H^{(1)}}^+_k|\le |H^{(1)}_k|$ for all $k$ and $|{H^{(1)}}^+_j|< |H^{(1)}_j|$ for some $j$, we get that $|H_k^+|\le |H_k|$ for all $k$ and $|H_l^+|<|H_l|$ for some $l$, contradicting the minimality of $H$.
\end{proof}

\section{Periodic structures}\label{sec:periodstr}

Informally, the aim of this section is to prove the following strong version of the so-called Bulitko's Lemma, \cite{Bul}:
\begin{center}
\parbox{5.5in}{\textit{Applying automorphisms from a finitely generated subgroup $\AA(\Omega)$ of the group of automorphisms of the coordinate group $\GG_{R(\Omega^\ast)}$ to a periodic solution either one can bound the exponent of periodicity of the solution {\rm(}regular case, see {\rm Lemma \ref{lem:23-2})}, or one can get a solution of a proper equation  {\rm(}strongly singular case, see {\rm Lemma \ref{lem:23-1ss}}, and singular case, see {\rm Lemma \ref{lem:23-1})}.}}
\end{center}
Above, by a solution of a proper equation we mean a homomorphism from a proper quotient of the coordinate group of $\Omega$ to $\GG$.

This approach for free groups was introduced by A.~Razborov. In \cite{Razborov1}, he defines a combinatorial object, called a periodic structure on a generalised equation $\Omega$ and constructs a finite set of generators for the group $\AA(\Omega)$.

In Section \ref{sec:explperstr} we give an example that follows the exposition of Section \ref{sec:perstr}. We advise the reader unfamiliar with the definitions, to consult this example while reading Section \ref{sec:perstr}.

\subsection{Periodic  structures}\label{sec:perstr}

We fix till the end of this section a generalised equation $\Omega=\gpo$ in the standard form. Suppose that some boundary $k$ (between $h_{k-1}$ and $h_k$) in the active part of $\Omega$ does not touch bases. Since the generalised equation $\Omega$ is in the standard form, the boundary $k$ intersects at least one base $\mu$. Using $\ET 5$ we $\mu$-tie the boundary $k$. Applying $\D 3$, if necessary, we may assume that the set of boundary connections in $\Omega$ is empty and that each boundary of $\Omega$ touches a base.

A cyclically reduced word $P$ in $\GG$ is called a \index{period}\emph{period} if $P$ is geodesic (in $\GG$) and is not a proper power treated as an element of the ambient free monoid. A word $w$ is called \index{periodic@($P$-)periodic word}$P$-\emph{periodic} if $w$ is geodesic (treated as an element of $\GG$), $|w| \ge |P|$ and, $w$ is a subword of $P^n$ for some $n$. Every $P$-periodic word $w$ can be presented in the form
\begin{equation}\label{2.50}
w \doteq Q^rQ_1
\end{equation}
where $Q$ is a cyclic permutation of $P^{\pm 1}$,  $r \geq 1$, $Q \doteq Q_1 Q_2$ and $Q_2 \neq 1$. The number $r$ is called the \index{exponent!of a word}\emph{exponent} of $w$. A maximal exponent of a $P$-periodic subword in a word $w$ is called the \index{exponent!of ($P$-)periodicity}\emph{exponent of $P$-periodicity of $w$}. We denote it by \glossary{name={$\e(w)$}, description={exponent of periodicity of the word $w$}, sort=E}$\e(w)$.

\begin{defn}\label{11'}
Let $\Omega$ be a generalised equation in the standard form. A solution $H=(H_1,\dots, H_{\rho})$ of $\Omega$ is called \index{solution!of a generalised equation!periodic with respect to a period}\emph{periodic with respect to a period $P$}, if for every closed variable section $\sigma$ of $\Omega$ one of the following conditions holds:
\begin{enumerate}
 \item  \label{it:per1} $H(\sigma)$ is $P$-periodic with exponent $r \ge 2$;
 \item  \label{it:per2} $|H(\sigma)| \le |P|$;
 \item  \label{it:per3} $H(\sigma)$ is $A$-periodic and $|A| \le |P|$;
\end{enumerate}
Moreover, condition (\ref{it:per1}) holds for at least one closed variable section $\sigma$ of $\Omega$.
\end{defn}

Let $H$ be a $P$-periodic solution of $\Omega$. Then a section $\sigma$ satisfying condition (\ref{it:per1}) of the above definition is called \index{section!($P$-)periodic}\emph{$P$-periodic} (with respect to $H$).

The following lemma gives an intuition about the kind of generalised equations that have periodic solutions.

\begin{lem} \label{lem:case2}
Let $\Omega$ be a generalised equation such that every closed section $\sigma_i$ of $\Omega$ is either constant or there exists a pair of dual bases $\mu_i$, $\Delta(\mu_i)$ such that $\mu_i$ and $\Delta(\mu_i)$ intersect but do not form a pair of matched bases, and $\sigma_i=[\alpha(\mu_i),\beta(\Delta(\mu_i))]$. Let $H$ be a solution of $\Omega$. Then there exists a period $P$ such that $H$ is $P$-periodic.
\end{lem}
\begin{proof}
Consider a section $\sigma=[\alpha(\mu),\beta(\Delta(\mu))]$. The boundary $i_1=\alpha(\Delta(\mu))$ intersects the base $\mu$, since the bases $\mu$ and $\Delta(\mu)$ overlap. We $\mu$-tie the boundary $i_1$, i.e. we introduce a boundary connection $(i_1,\mu,i_2)$ in such a way that $H$ is a solution of the generalised equation obtained. It follows that $i_1<i_2$.

Repeating this argument, we obtain a finite set of boundaries $i_1<\dots<i_{k+1}$, $k\ge 1$ such that $i_1,\dots,i_{k}$ do and $i_{k+1}$ does not intersect $\mu$, there is a boundary connection $(i_j, \mu, i_{j+1})$ for all $j=1,\dots, k$ and $H$ induces a solution of the generalised equation obtained. This set of boundaries is finite, since the length of the solution $H$ is finite.

Let $H[\alpha(\mu),i_1]=w$, $w=A^l$, where $l\ge 1$ and $A$ is a period. Then the section $\sigma$ is $A$-periodic. Indeed,
$$
\sigma=[\alpha(\mu),i_1]\cup[i_1,i_2]\cup\dots\cup[i_k,i_{k+1}]\cup[i_{k+1},\beta(\Delta(\mu))].
$$
Using the boundary equations, we get that
$$
h[\alpha(\mu),i_1]=h[i_1,i_2]=\dots=h[i_k,i_{k+1}] \hbox{ and } h[i_{k+1},\beta(\Delta(\mu))]=h[i_k,\beta(\mu)],
$$
thus $H(\sigma)\doteq A^{l\cdot(k+1)}\cdot A_1$, where $A\doteq A_1A_2$.

Set $P=A_j$, where $|A_j|=\max\limits_i\{|A_i|\mid\sigma_i \hbox{ is $A_i$-periodic}\}$. By definition, $H$ is $P$-periodic.
\end{proof}

Below we introduce the notion of a periodic structure. The idea of considering periodic structures on $\Omega$ is to subdivide the set of periodic solutions into subsets so that any two solutions from the same subset have the same set of ``long items'', i.e. $P$-periodic solutions that factor through the generalised equation $\Omega$ and a periodic structure $\langle \P, R\rangle$ on $\Omega$ satisfy the following property:
$$
h_i\in \P \hbox{ if and only if } |H_i|\ge 2|P|.
$$
One can regard Lemma \ref{le:PP} below as a motivation for the definition of a periodic structure.

\begin{defn} \label{above}
Let  $\Omega$ be a generalised equation in the standard form without boundary connections. A \index{periodic structure}\emph{periodic structure} on $\Omega$ is a pair \glossary{name={$\langle {\P}, R \rangle$}, description={periodic structure}, sort=P}$\langle {\P}, R \rangle$, where
\begin{enumerate}
\item \label{it:ps1} \glossary{name={$\P$},description={non-empty set of variables, variable bases, and closed sections that belong to the periodic structure $\langle {\P}, R \rangle$}, sort=P}${\P}$ is a non-empty set consisting of some variables $h_i$, some variable bases $\mu$, and some closed sections $\sigma$ from $V\Sigma$ such that the following conditions are satisfied:
\begin{itemize}
    \item[(a)] if $h_i \in {\P}$ and $h_i \in \mu$, then $\mu \in {\P}$;

    \item[(b)] if $\mu \in {\P}$, then $\Delta(\mu) \in {\P}$;

    \item[(c)] if $\mu \in {\P}$ and $\mu \in \sigma$, then $\sigma \in {\P}$;

    \item[(d)] there exists a function ${\mathcal X}$ mapping the set of closed sections from ${\P}$ into $\{-1,+1\}$ such that for every $\mu, \sigma_1, \sigma_2 \in {\P}$, the condition that $\mu \in \sigma_1$ and $\Delta(\mu) \in \sigma_2$ implies $\varepsilon(\mu) \cdot \varepsilon(\Delta(\mu)) = {\mathcal X}(\sigma_1) \cdot {\mathcal X}(\sigma_2)$;
\end{itemize}
\item \label{it:ps2}  $R$ is an equivalence relation on a certain set ${\mathcal B}$ (defined in (e)) such that condition (f) is satisfied.
\begin{itemize}
 \item[(e)] Notice, that for every boundary $l$ belonging to a closed section in $\P$ either there exists a unique closed section $\sigma(l)$ in ${\P}$ containing $l$, or there exist precisely two closed sections $\sigma_{\lef}(l) = [i,l], \sigma_{\rig} =  [l,j]$ in ${\P}$ containing $l$. The set of boundaries of the first type we denote  by ${\B}_1$, and of the second type by  ${\B}_2$. Put
$$
{\B} = {\B}_1  \cup \{l_{\lef}, l_{\rig}  \mid l \in {\B}_2\}
$$
here $l_{\lef}, l_{\rig}$ are two ``formal copies'' of $l$. We  will use the following agreement: for any base $\mu$ if $\alpha(\mu) \in {\B}_2$ then by $\alpha(\mu)$ we mean $\alpha(\mu)_{\rig}$ and, similarly, if $\beta(\mu) \in {\B}_2$ then by $\beta(\mu)$ we mean $\beta(\mu)_{\lef}$.

\item[(f)]  If $\mu \in {\P}$ then
$$
\begin{array}{lll}
\alpha(\mu) \sim_R \alpha(\Delta(\mu)), &  \beta(\mu) \sim_R \beta(\Delta(\mu)), & \hbox{ if } \varepsilon(\mu) =
\varepsilon(\Delta(\mu)); \\
\alpha(\mu) \sim_R \beta(\Delta(\mu)), & \beta(\mu)\sim_R  \alpha(\Delta(\mu)), & \hbox{ if }\varepsilon(\mu) = -
 \varepsilon(\Delta(\mu)).
\end{array}
$$
\end{itemize}
\end{enumerate}
\end{defn}

\begin{rem}
For a given generalised equation $\Omega$, there exists only finitely many periodic structures on $\Omega$, and all of them can be constructed effectively. Indeed, every periodic structure $\langle\P,R\rangle$ is uniquely defined by the subset of items, bases and sections of $\Omega$ that belong to $\P$ and the relation $\sim_R$.
\end{rem}

Now we will show how to a $P$-periodic solution $H$ of  $\Omega$ one can associate a periodic structure \glossary{name={${\P}(H, P)$}, description={periodic structure associated to a $P$-periodic solution $H$}, sort=P}${\P}(H, P) = \langle {\P}, R \rangle$ on $\Omega$. We define ${\P}$ as follows. A closed section $\sigma$ is in ${\P}$ if and only if $\sigma$ is $P$-periodic. A variable $h_i$ is in ${\P}$ if and only if $h_i \in \sigma$ for some $\sigma \in {\P}$ and $|H_i| \geq 2 |P|$. A variable base $\mu$ is in ${\P}$ if and only if either $\mu$ or $\Delta(\mu)$
contains an item $h_i$ from ${\P}$.

Put ${\mathcal X}([i,j]) = \pm 1$ depending on whether in (\ref{2.50}) the word $Q$ is conjugate to $P$ or to $P^{-1}$.

Now let $[i,j]\in {\P}$ and $ i \leq l \leq j$. Then one can write $P \doteq P_1P_2$ in such a way that if ${\mathcal X} ([i,j]) =1$, then the word $H[i,l]$ is the terminal subword of the word $(P^\infty)P_1$, where $P^\infty$ is the infinite word obtained by concatenating the powers of $P$, and $H[l,j]$ is the initial subword of the word $P_2(P^\infty)$; and if ${\mathcal X}([i,j])= -1$, then the word $H[i,l]$ is the terminal subword of the word $(P^{-1})^\infty P_2^{-1}$ and $H[l,j]$ is the initial subword of $P_1^{-1}(P^{-1})^\infty$. By Lemma 1.2.9 \cite{1}, the decomposition $P\doteq P_1P_2$ with these properties is unique;  denote this decomposition by \glossary{name={$\delta(l)$}, description={decomposition of the period defined by the boundary $l$}, sort=D}$\delta(l)$. We define the relation $R$ as follows:
$$
l_1 \sim_R l_2 \hbox{ if and only if } \delta(l_1) =  \delta(l_2).
$$

Every periodic solution $H$ of $\Omega$ induces a periodic structure $\P(H,P)$ on $\Omega$.
\begin{lem}\label{le:PP}
Let $H$ be a periodic solution of $\Omega$. Then ${\P}(H, P)$ is a periodic structure on $\Omega$.
\end{lem}
\begin{proof}
Let ${\P}(H, P) = \langle {\P}, R \rangle$.  Obviously,  ${\P}$ satisfies conditions (a) and (b) from Definition \ref{above}.

We now prove that $\P$ satisfies condition (c) from Definition \ref{above}. Let $\mu \in {\P}$ and $\mu \in [i,j]$. There exists an item $h_k \in {\P}$ such that $h_k \in \mu$ or $h_k \in \Delta({\mu})$. If $h_k \in \mu$, then, by construction, $[i,j] \in {\P}$. If $h_k \in \Delta(\mu)$ and $\Delta(\mu) \in [i',j']$, then $[i', j'] \in {\P}$, and hence, the word $H(\Delta(\mu))$ can be written in the form $Q^{r'} Q_1$, where $Q\doteq Q_1 Q_2$ is a cyclic permutation of the word $P^{\pm 1}$ and $r' \geq 2$. Since $|H[i,j]| \ge |H(\mu)| = |H(\Delta(\mu))|\ge 2 |P|$ and from Definition \ref{11'}, it follows that $[i,j]$ is an $A$-periodic section, where $|A|\le |P|$. Then $H(\mu) \doteq B^s B_1$, where $B$ is a cyclic permutation of the word $A^{\pm 1}$, $|B| \leq |P|$, $B \doteq B_1 B_2$, and $s \ge 0$. From the equality ${H(\mu)}^{\varepsilon(\mu)} \doteq {H(\Delta(\mu))}^{\varepsilon(\Delta(\mu))}$ and Lemma 1.2.9 \cite{1} it follows that $B$ is a cyclic permutation of the word $Q^{\pm 1}$. Consequently, $A$ is a cyclic permutation of the word $P^{\pm 1}$. Therefore, $[i,j]$ is a $P$-periodic section of $\Omega$ with respect to $H$, in other words, the length of $H[i,j]$ is greater than or equal to $2|P|$ and so $[i,j]\in \P$.

If $\mu \in [i_1, j_1]$, $\Delta(\mu) \in [i_2, j_2]$ and $ \mu \in {\P}$, then the equality $\varepsilon(\mu) \cdot \varepsilon(\Delta(\mu))$ = ${\mathcal X}([i_1, j_1]) \cdot {\mathcal X}([i_2, j_2])$ follows from the fact that given $A^r A_1 \doteq B^s B_1$ and $r,s \geq 2$, the word $A$ cannot be a cyclic permutation of the word $B^{-1}$, hence condition (d) of Definition \ref{11'} holds.

Since  ${H(\mu)}^{\varepsilon(\mu)}\doteq {H(\Delta(\mu))}^{\varepsilon(\Delta(\mu))}$, from Lemma 1.2.9 in \cite{1} it follows that condition (f) also holds for the relation $R$.
\end{proof}

\begin{rem} \label{rem:subword}
Now let us fix a non-empty periodic structure $\langle {\P}, R \rangle$ on a generalised equation $\Omega$. Item (d) of Definition \ref{above} allows us to assume (after replacing the variables $h_i, \ldots, h_{j-1}$ by $h_{j-1}^{-1}, \ldots, h_i^{-1}$ on those closed sections $[i,j] \in {\P}$ for which ${\mathcal X}([i,j])=-1$) that $\varepsilon(\mu)=1$ for all $\mu \in {\P}$. Therefore, we may assume that for every item $h_i$, the word $H_i$ is a subword of the word $P^\infty$.
\end{rem}

The rest of this section is devoted to defining the group of automorphisms $\AA(\Omega)$. The idea is as follows. We change the set of generators $h$ of the coordinate group $\GG_{R(\Omega^*)}$ by $\bar x$ and we use the new set of generators to give an explicit description of the generating set of the group of automorphisms $\AA(\Omega)$.

To construct the set of generators $\bar x$ of $\GG_{R(\Omega^*)}$, the following definitions are in order. We refer the reader to Section \ref{sec:explperstr} for an example.

\begin{defn}\label{defn:graphperstr}
We construct the \index{graph!of a periodic structure}\emph{graph \glossary{name=$\Gamma$, description={graph of a periodic structure}, sort=G}$\Gamma$ of a periodic structure $\langle \P,R\rangle$}. The set of vertices $V(\Gamma)$ of the graph $\Gamma$ is the set of $R$-equivalence classes. For a boundary $k$, denote by $(k)$ the equivalence class of the relation $R$ to which it belongs. For each variable $h_k$ that belongs to a certain closed section from ${\P}$, we introduce an oriented edge $e\in E(\Gamma)$ from $(k)$ to $(k+1)$, $e:(k)\to(k+1)$ and an inverse edge $e^{-1}:(k+1)\to (k)$. We label the edge $e$ by $h(e) = h_k$ (correspondingly, $h(e^{-1}) = h_k^{-1}$). For every path $\p=e_1^{\epsilon_1} \ldots e_j^{\epsilon_j}$ in the graph $\Gamma$, we denote its label by $h(\p)$,  $h(\p)=h(e_1^{\epsilon_1}) \ldots h(e_j^{\epsilon_j})$, $\epsilon_1,\dots,\epsilon_j\in \{1,-1\}$.
\end{defn}

The periodic structure $\langle {\P}, R \rangle$ is called \index{periodic structure!connected}\emph{connected}, if its graph $\Gamma$ is connected.

Let $\langle {\P}, R \rangle$ be an arbitrary periodic structure of a generalised equation $\Omega$. Let $\Gamma_1, \ldots, \Gamma_r$ be the connected components of the graph $\Gamma$. The labels of edges of the component $\Gamma_i$ form (in the generalised equation $\Omega$) a union of labels of closed sections from ${\P}$. Moreover, if a base $\mu \in {\P}$ belongs to a section from $\P$, then its dual $\Delta(\mu)$, by condition (f) of Definition \ref{above}, also belongs to a section from $\P$. Define ${\P}_i$ to be the union of the following three sets: the set of labels of edges from $\Gamma_i$ that belong to ${\P}$, closed sections to which these labels belong, and bases $\mu \in {\P}$ that belong to these sections. Define $R_i$ to be the restriction of the relation $R$ to $\P_i$. We thereby obtain a connected periodic structure $\langle {\P}_i, R_i \rangle$ whose graph is  $\Gamma_i$.

We further assume that the periodic structure $\langle {\P}, R \rangle$ is connected.

Let $\Gamma$ be the graph of a periodic structure $\langle {\P}, R \rangle = {\P}(H, P)$ and let $\p=e_1^{\epsilon_1} \ldots e_j^{\epsilon_j}$ be a path in $\Gamma$, $h(\p)$ be its label, $h(\p)=h(e_1^{\epsilon_1}) \ldots h(e_j^{\epsilon_j})$, $\epsilon_1,\dots,\epsilon_j\in \{1,-1\}$. To simplify the notation we write $H(\p)$ instead of $H(h(\p))$.

\begin{lem} \label{2.9}
Let $H$ be a $P$-periodic solution of a generalised equation $\Omega$, let $\langle {\P}, R \rangle = {\P}(H, P)$ be a periodic structure on $\Omega$ and let $\cc$ be a cycle in the graph $\Gamma$ at the vertex $(l)$, $\delta(l)=P_1P_2$. Then there exists $n \in \Z$ such that $H(\cc) = (P_2P_1)^n$.
\end{lem}
\begin{proof}
If $e$ is an edge $v\to v'$ in the graph $\Gamma$, and $P = P_1P_2$,  $P = P_1' P_2'$ are two decompositions corresponding to the boundaries from $(v)$ and $(v')$ respectively. Then, obviously, $H(e) = P_2 P^{n_k}P_1'$, $n_k \in \Z$. The statement follows if we multiply the values $H(e)$ for all the edges $e$ in the cycle $\cc$.
\end{proof}

\begin{defn} \label{2.51}
A generalised equation $\Omega$ is called \index{generalised equation!periodised}\emph{periodised} with respect to a given periodic structure $\langle {\P}, R \rangle$, if for every two cycles $\cc_1$ and $\cc_2$ based at the same vertex in the graph $\Gamma$, there is a relation $[h(\cc_1), h(\cc_2)]=1$ in $\GG_{R(\Upsilon^\ast)}$. Note that, in particular $[h(\cc_1), h(\cc_2)]=1$ in $\GG_{R(\Omega^\ast)}$.
\end{defn}

Let \glossary{name={$\Sh$, $\Sh(\Gamma)$}, description={the set of ``short'' edges of the graph $\Gamma$ of a periodic structure}, sort=S}\glossary{name={$h(\Sh)$}, description={labels of the ``short'' edges of $\Gamma$, $h(\Sh)=\{h(e) \mid e\in \Sh \}$}, sort=H}$\Sh=\Sh(\Gamma)=\{e \in E(\Gamma) \mid h(e)\notin \P \}$ and $h(\Sh)=\{h(e) \mid e\in \Sh \}$.

Let \glossary{name={$\Gamma_0$}, description={the subgraph of $\Gamma$ all of whose edges are ``short'', $\Gamma_0=(V(\Gamma),\Sh(\Gamma))$}, sort=G}$\Gamma_0=(V(\Gamma),\Sh(\Gamma))$ be the subgraph of the graph $\Gamma$ having the same set of vertices as $\Gamma$ and the set of edges  $E(\Gamma_0)=\Sh$. Choose a maximal subforest $T_0$ in the graph $\Gamma_0$ and extend it to a maximal subforest $T$ of the graph $\Gamma$. Since the periodic structure $\langle {\P}, R \rangle$ is connected by assumption, it follows that $T$ is a tree. Fix an arbitrary vertex $v_\Gamma$ of the graph $\Gamma$ and denote by $\p(v_\Gamma, v)$ the (unique) path in $T$ from $v_\Gamma$ to $v$. For every edge $e: v \to v'$ not lying in $T$, we introduce a cycle $\cc_e = \p(v_\Gamma, v) e (\p(v_\Gamma, v'))^{-1}$. Then the fundamental group $\pi_1(\Gamma, v_\Gamma)$ is generated by the cycles $\cc_e$ (see, for example, the proof of Proposition III.2.1, \cite{LS}). This and the decidability of the universal theory of $\GG$ imply that the property of a generalised equation ``to be periodised with respect to a given periodic structure'' is algorithmically decidable. Indeed, it suffices to check if the following universal formula (quasi-identity) holds in $\GG$ (for every pair of cycles $\cc_{e_1}$, $\cc_{e_2}$):
$$
\forall H_1, \dots, H_\rho \left(\left(\bigwedge (\Upsilon^*(H)=1)\right)\to \left([H(\cc_{e_1}),H(\cc_{e_2})]=1\right)\right).
$$

Furthermore, the set of elements
\begin{equation} \label{2.52}
\{h(e) \mid e \in T \} \cup \{h(\cc_e) \mid e \not \in T \}
\end{equation}
forms a basis of the free group generated by
$$
\{h_k \mid h_k\in \sigma, \sigma\in {\P} \}.
$$
If $\mu \in {\P}$, then $(\beta(\mu)) = (\beta(\Delta(\mu)))$, $(\alpha(\mu)) = (\alpha(\Delta(\mu)))$ by property (f) from Definition \ref{above} and, consequently, the word $h(\mu) {h(\Delta(\mu))}^{-1}$ is the label of a cycle $\cc'_\mu$ from $\pi_1 (\Gamma, (\alpha(\mu)))$. Let $\cc_\mu = \p(v_\Gamma, (\alpha(\mu)))\cc'_\mu \p(v_\Gamma, (\alpha(\mu)))^{-1}$. Then
\begin{equation} \label{2.53}
h(\cc_\mu) = uh(\mu) {h(\Delta(\mu))}^{-1} u^{-1},
\end{equation}
where $u$ is the label of the path $\p(v_\Gamma, (\alpha(\mu)))$. Since $\cc_\mu \in \pi_1(\Gamma, v_\Gamma)$, it follows that $\cc_\mu = b_\mu (\{\cc_e \mid e \not \in T \})$, where $b_\mu$ is a certain word in the indicated generators that can be  constructed effectively (see Proposition III.2.1, \cite{LS}).

Let $\tilde{b}_\mu$ denote the image of the word $b_\mu$ in the abelianisation of $\pi_1(\Gamma ,v_\Gamma)$. Denote by $\widetilde{Z}$ the free abelian group consisting of formal linear combinations $\sum\limits_{e \not \in T} n_e \tilde{\cc}_e$, $n_e \in {\Z}$, and by $\widetilde{B}$ its subgroup generated by the elements $\tilde{b}_\mu$, $\mu \in {\P}$ and the elements $\tilde{\cc}_e$, $e \not \in T$, $e \in \Sh$.

By the classification theorem of finitely generated abelian groups, one can effectively construct a basis $\{\widetilde{C}^{(1)}, \widetilde{C}^{(2)}\}$ of $\widetilde{Z}$ such that
\begin{equation} \label{2.54}
\widetilde{Z} = \widetilde{Z}_1 \oplus \widetilde{Z}_2, \ \widetilde{B} \subseteq \widetilde{Z}_1, \ [\widetilde{Z}_1 : \widetilde{B}] < \infty,
\end{equation}
where $\widetilde{C}^{(1)}$ is a basis of $\widetilde{Z}_1$ and $\widetilde{C}^{(2)}$ is a basis of $\widetilde{Z}_2$.

By Proposition I.4.4 in \cite{LS}, one can effectively construct a basis \glossary{name={$C^{(1)}$}, description={``short'' cycles in $\Gamma$}, sort=C}$C^{(1)}$, \glossary{name={$C^{(2)}$}, description={``free'' cycles in $\Gamma$}, sort=C}$C^{(2)}$ of the free (non-abelian) group $\pi_1(\Gamma, v_\Gamma)$ such that $\widetilde{C}^{(1)}$, $\widetilde{C}^{(2)}$ are the natural images of the elements $C^{(1)}$, $C^{(2)}$ in $\widetilde{Z}$.

\begin{rem} \label{rem:hab}
Notice that any equation in $\widetilde{Z}$ of the form
$$
\tilde{ \cc}= \sum\limits_{\tilde \cc_i\in \widetilde{Z}} n_i \tilde {\cc}_i
$$
lifts to an equation in $\pi_1(\Gamma,v_\Gamma)$ of the form
$$
\cc =\prod\limits_{\cc_i\in \pi_1(\Gamma,v_\Gamma)} {\cc}_i^{n_i} V,
$$
where $V$ is an element of the derived subgroup of $\pi_1(\Gamma,v_\Gamma)$. If the generalised equation is periodised, for any two cycles $\cc'_1, \cc'_2 \in \pi_1(\Gamma,v_\Gamma)$, we have that $[h(\cc'_1),h(\cc'_2)]=1$ in $\GG_{R(\Omega^*)}$. Hence, $h(\cc)=\prod\limits_{\cc_i\in \pi_1(\Gamma,v_\Gamma)} {h(\cc_i)}^{n_i}$ in $\GG_{R(\Omega^*)}$.
\end{rem}

\begin{lem} \label{cl:propc1}
Let $\Omega$ be a periodised generalised equation. Then the basis $\widetilde{C}^{(1)}$ can be chosen in such a way that for every $\cc\in C^{(1)}$ either $\cc=\cc_e$, where $e\notin T$, $e\in \Sh$, or for any solution $H$ we have $H(\cc)=1$.
\end{lem}
\begin{proof}
The set $\{\tilde{\cc}_e \mid e\notin T, e\in \Sh\}$ is a subset of the set of generators of $\widetilde{Z}$ contained in $\widetilde{Z}_1$. Thus, this set can be extended to a basis of $\widetilde{Z}_1$. Since, by (\ref{2.54}),  $[\widetilde{Z}_1 : \widetilde{B}] < \infty$, for every $\tilde \cc\in \widetilde{Z}_1$ there exists $n_\cc\in \N$ such that
$$
n_\cc \tilde \cc=\sum\limits_{e \not \in T,  e\in \Sh} n_e\tilde{\cc}_e +\sum\limits_{\mu \in {\P}} n_\mu \tilde {b}_\mu.
$$
It follows that the set $\{\tilde{\cc}_1,\dots, \tilde{\cc}_k\}$ which completes the set $\{\tilde{\cc}_e \mid e\notin T,  e\in \Sh\}$ to a basis of $\widetilde{Z}_1$, can be chosen so that the following equality holds:
$$
n_{\cc_i} \tilde{ \cc}_i= \sum\limits_{\mu \in {\P}} n_\mu \tilde {b}_\mu.
$$

Hence, by Remark \ref{rem:hab}, for any solution $H$ we have ${H(\cc_i)}^{n_{\cc_i}}=\prod\limits_{\mu \in {\P}} {H(b_\mu)}^{n_\mu}=1$. Since $\GG$ is torsion-free, we have $H(\cc_i)=1$.
\end{proof}

Let $\{e_1, \ldots, e_m\}$ be the set of edges of $T \setminus T_0$. Since $T_0$ is the spanning forest of the graph $\Gamma_0$, it follows that $h(e_1), \ldots, h(e_m) \notin h({\Sh})$; in particular, $h(e_1), \ldots, h(e_m) \in \P$.

Let $F(\Omega)$ be the free group generated by the variables of $\Omega$. Consider in the group $F(\Omega )$ a new set of generators $\bar x$ defined below (we prove that the set $\bar x$ is in fact a set of generators of $F(\Omega )$ in part (\ref{spl3}) of Lemma \ref{2.10''}).

Let $v_i$ be the origin of the edge $e_i$. We introduce new variables
\begin{equation} \label{eq:uie}
\bar u^{(i)}=\{u_{ie}\mid e\not\in T,\ e\in \Sh\}, \ \bar z^{(i)}=\{z_{ie}\mid e\not\in T, e\in \Sh\}, \hbox{ for } 1\leq i\leq m,
\end{equation}
as follows
\begin{equation}\label{2.59}
u_{ie}={h(\p(v_\Gamma,v_i))}^{-1}h(\cc_e)h(\p(v_\Gamma,v_i)),\ z_{ie}=h(e_i)^{-1}u_{ie}h(e_i).
\end{equation}
We denote by \glossary{name={$\bar u$}, description={a subset of the set $\bar x$ of generators of the coordinate group of a periodised generalised equation}, sort=U}$\bar u$ the union $\bigcup\limits_{i=1}^m \bar u^{(i)}$ and by \glossary{name={$\bar z$}, description={a subset of the set $\bar x$ of generators for the coordinate group of a periodised generalised equation}, sort=Z}$\bar z$ the union $\bigcup\limits_{i=1}^m \bar z^{(i)}$. Denote by $\bar t$ the family of variables that do not belong to closed sections from $\P$. Let
\glossary{name={$\bar x$}, description={a set of generators of the coordinate group of a periodised generalised equation}, sort=X}
$$
\bar x = \bar t \cup \{h(e) \mid e\in T_0\} \cup \bar u \cup \bar z\cup \{h(e_1),\dots,h(e_m)\}\cup h(C^{(1)})\cup h(C^{(2)})
$$
\begin{rem}
Note that without loss of generality we may assume that $v_\Gamma$ corresponds to the beginning of the period $P$. Indeed, it follows from the definition of the periodic structure that for any cyclic permutation $P'=P_2P_1$ of $P$, we have $\P(H,P)=\P(H,P')$.
\end{rem}

\begin{lem} \label{lem:spl1}
Let $\Omega $ be a generalised equation periodised with respect to a periodic structure $\langle {\P},R\rangle$. Then for any cycle $\cc_{e_0}$ such that $h(e_0)\notin \P$ and for any solution $H$ of $\Omega$ periodic with respect to a period $P$, such that ${\P}(H,P)=\langle {\P},R\rangle $ one has $H(\cc_{e_0})=P^n$, where $|n|\le 2\rho$. In particular, one can choose a basis $C^{(1)}$ in such a way that for any $\cc\in C^{(1)}$ one has $H(\cc)=P^n$, where $|n|\le 2\rho$.
\end{lem}
\begin{proof}
Let $\cc_{e_0}$ be a cycle such that the edges of $\cc_{e_0}$ are labelled by variables $h_k$, $h_k\notin \P$. Observe that $e_0 = \p_1 \cc_{e_0} \p_2$, where $\p_1$ and $\p_2$ are paths in the tree $T$. Since $e_0 \in {\Gamma_0}$, it follows that the origin and the terminus of the edge $e_0$ lie in the same connected component of the graph $\Gamma_0$ and, consequently, are connected by a path $\s$ in the forest $T_0$. Furthermore, $\p_1$ and $\s\p_2^{-1}$ are paths in the tree $T$ connecting the same vertices; therefore, $\p_1 = \s \p_2^{-1}$. Hence, $\cc_{e_0} = \p_2 \cc'_{e_0} \p_2^{-1}$, where $\cc'_{e_0}$ is a certain cycle based at the vertex $v_\Gamma'$ in the graph $\Gamma_0$.

From the equality $H(\cc_{e_0}) = H(\p_2) H(\cc'_{e_0}) {H(\p_2)}^{-1}$ and by Lemma \ref{2.9}, we get that $P^{n_{e_0}}=P^{n_2}P_1(P_2P_1)^{n_{e_0}'}P_1^{-1}P^{-n_2}$, where $\delta(v_\Gamma')=P_1P_2$, $n_{e_0}=\e(H(\cc_{e_0}))$, $n_{e_0}'=\e(H(\cc'_{e_0}))$, $n_2=\e(H(\p_2))$. Hence $n_{e_0}=n_{e_0}'$ and thus $|H(\cc_{e_0})| = |H(\cc'_{e_0})|$.  From the construction of ${\P}(H, P)$, it follows that the inequality $|H_k| \le 2 |P|$ holds for every item $h_k \not \in {\P}$. Since the cycle $\cc'_{e_0}$ is simple, we have that $|H(\cc_{e_0})| = |H(\cc'_{e_0})|\le 2 \rho |P|$.

In particular, one has that $|\e(H(\cc_e))| \leq 2 \rho$ for every $e \notin T$, $e \in \Sh$. By Lemma \ref{cl:propc1} one can choose a basis $C^{(1)}$ such that for every $\cc\in C^{(1)}$ either $\cc=\cc_e$ and $|\e(H(\cc_e))| \leq 2 \rho$, or $H(\cc)=1$.
\end{proof}

The following lemma describes a generating set and the group of automorphisms $\AA(\Omega)$ of the coordinate group $\GG_{R(\Upsilon^*)}$.

\begin{lem} \label{2.10''}
Let $\Omega $ be a generalised equation periodised with respect to a periodic structure $\langle {\P},R\rangle $. Then the following statements hold.
\begin{enumerate}
    \item \label{spl3} The system $\Upsilon ^\ast$ is equivalent to the union of the following two systems of equations:
            $$
                \left\{
                \begin{array}{ll}
                u_{ie}^{h(e_i)}=z_{ie},&\hbox{where } e\in T,\, e\in \Sh; 1\leq i\leq m \\
                \left[u_{ie_1},u_{ie_2}\right]=1,&\hbox{where } e_j\in T, \,e_j\in \Sh, \, j=1,2; 1\leq i\leq m \\
                \left[h(\cc_1),h(\cc_2)\right]=1, & \hbox{where } \cc_1,\cc_2\in C^{(1)}\cup C^{(2)}
                \end{array}
                \right.
            $$
        and a system:
            $$
            \Psi \left( \{h(e)\,\mid\, e\in T, e\in \Sh\}, h(C^{(1)}), \bar t,\bar u,\bar z,\cA\right)=1,
            $$
        such that neither $h(e_i)$, $1\le i\le m$, nor $h(C^{(2)})$ occurs in $\Psi$.

    \item \label{spl4} If $\cc$ is a cycle based at the origin of $e_i$, then the transformation $h(e_i)\rightarrow h(\cc)h(e_i)$ which is identical on all the other elements from $\cA\cup\bar x$, extends to a $\GG$-automorphism of $\GG_{R(\Upsilon^\ast)}$.

    \item \label{spl5} If $\cc\in C^{(2)}$ and $\cc'\in C^{(1)} \cup C^{(2)}$, $\cc'\ne \cc$, then the transformation defined by $h(\cc)\rightarrow h(\cc')h(\cc)$, which is identical on all the other elements from $\cA\cup\bar x$, extends to a $\GG$-automorphism of $\GG_{R(\Upsilon^\ast)}$.
\end{enumerate}
\end{lem}
\begin{proof}
We first prove (\ref{spl3}). It is easy to check that (in the above notation):
\begin{gather}\notag
\begin{split}
\{ h_1,\dots, h_\rho\}=& \{h(e)\, \mid \, h(e)\in \sigma, \sigma \notin \P\}\, \cup\\
&\{h(e) \,\mid\, e\in T_0\} \cup  \{h(e) \,\mid\, e\notin T_0, e\in \Sh\}\, \cup \\
&\{h(e) \,\mid\, e\in T\setminus T_0\} \cup \{h(e) \,\mid\, e\notin T, e\notin \Sh\}.
\end{split}
\end{gather}
From Equation (\ref{2.52}) and the discussion above it follows that
\begin{equation} \label{eq:genset}
\langle h_1,\dots, h_\rho\rangle= \left< \bar t \cup \{h(e) \mid e\in T_0\} \cup \bar u\cup \bar z\cup \{h(e_1),\dots,h(e_m)\}\cup h(C^{(1)})\cup h(C^{(2)})\right>=\langle \bar x\rangle,
\end{equation}
i.e. the sets $h$ and $\bar x$ generate the same free group $F(\Omega)$.

We now rewrite the equations from $\Upsilon$ in terms of the new set of generators $\bar x$.

We first consider the relations induced by bases which do not belong to $\P$, i.e. basic equations of the form $h(\mu)=h(\Delta(\mu))$, where $\mu$ is a variable base, $\mu\notin \P$. By construction of the generating set $\bar x$, all the items $h_k$ which appear in these relations belong to the set
$$
\bar t \cup \{h(e) \,\mid\, e\in T_0\} \cup  \{h(e) \,\mid\, e\notin T_0, e\in \Sh\}= \bar t \cup M_1\cup M_2.
$$
The elements of the sets $\bar t$ and $M_1$ are generators in both basis. We now study how elements of $M_2$ rewrite in the new basis.
Let $h(e) = h_k$, $e\notin T_0$, $e\in \Sh$. We have $e=\s\p_2^{-1}\cc_e\p_2$ and
\begin{equation} \label{eq:obt}
h_k = h(\s) h(\p_2)^{-1} h(\cc_e)h(\p_2),
\end{equation}
where $\s$ is a path in $T_0$ and $\p_2$ is a path in $T$ (see proof of Lemma \ref{lem:spl1}). The variables $h(e_i)$, $1 \leq i \leq m$ can occur in the right-hand side of Equation (\ref{eq:obt}) (written in the basis $\bar{x}$) only in $h(\p_2)^{\pm 1}$ and at most once. Moreover, the sign of this occurrence (if it exists) depends only on the orientation of the edge $e_i$ with respect to the root $v_\Gamma$ of the tree $T$. If $\p_2 = \p_2' e_i^{\pm 1}\p_2 ''$, then all the occurrences of the variable $h(e_i)$ in the words $h_k$ written in the basis $\bar{x}$, with $h_k \not\in {\P}$, are contained in the subwords of the form $h(e_i)^{\mp 1} h((\p_2')^{-1}\cc_e\p_2')h(e_i)^{\pm 1}$, i.e. in the subwords of the form $h(e_i)^{\mp 1} h(\cc) h(e_i)^{\pm 1}$, where $\cc$ is a certain cycle in the graph $\Gamma$ based at the origin of the edge $e_i^{\pm 1}$. So the variable $h_k$ rewrites as a word in the generators $\bar u$, $\bar z$, $\{h(e) \,\mid\, e\in T_0 \}$.

Summarising, the basic equations corresponding to variable bases which do not belong to $\P$ rewrite in the new basis as words in $\bar t$, $\{ h(e) \,\mid\, e\in T_0\}$,  $\bar u$ and $\bar z$.

In the new basis, the relations of the form $h(\mu)=h(\Delta(\mu))$, where $\mu\in \P$, are, modulo commutators, words in $h(C^{(1)})$, see Remark \ref{rem:hab}.

Since $\Omega$ is periodised with respect to $\langle \P, R\rangle$, we have
\begin{equation}\label{2.61}
[u_{ie_1},u_{ie_2}]=1 \hbox{ and } [h(\cc_1),h(\cc_2)]=1, \ \cc_1,\cc_2\in C^{(1)}\cup C^{(2)}.
\end{equation}

It follows that the system $\Upsilon ^\ast$ is equivalent to the union of the following two systems of equations in the new variables, a system
$$
\mathcal{O}=\left\{
\begin{array}{ll}
u_{ie}^{h(e_i)}=z_{ie},&\hbox{where } e\in T,\, e\in \Sh; 1\leq i\leq m \\
\left[u_{ie_1},u_{ie_2}\right]=1,&\hbox{where } e_j\in T, \,e_j\in \Sh, \, j=1,2; 1\leq i\leq m  \\
\left[h(\cc_1),h(\cc_2)\right]=1, & \hbox{where } \cc_1,\cc_2\in C^{(1)}\cup C^{(2)}
\end{array}
\right.
$$
and a system (defined by the equations from $\Upsilon$):
$$
\Psi \left( \{h(e)\,\mid\, e\in T, h(e)\notin {\P}\}, h(C^{(1)}), \bar t,\bar u,\bar z,\cA\right)=1,
$$
such that neither $h(e_i)$, $1\le i\le m$, nor $h(C^{(2)})$ occurs in $\Psi$.

The transformations from statements (\ref{spl4}) and (\ref{spl5}) extend to an automorphism $\varphi$ of $\GG_{R(\Upsilon^*)}$. Indeed, by the universal property of the quotient, the following diagram commutes
$$
\xymatrix{
  \GG[\bar x] \ar[d] \ar[r]^{\varphi} & \GG[\bar x] \ar[d]  \\
 \GG_{R(\Psi\cup \mathcal{O})}  \ar@{-->>}[r]^{\tilde\varphi}  &\GG_{R(\Psi\cup \mathcal{O}\cup \varphi(\mathcal{O}))}
 }
$$
It is easy to check that $\varphi(\Psi)=\Psi$ and that $\varphi(\mathcal{O}) \subseteq R(\Psi\cup \mathcal{O})$. Therefore, since by statement (\ref{spl3}) of the lemma the system $\Psi\cup \mathcal{O}$ is equivalent to $\Upsilon^*$, we get that $\tilde\varphi$ is an automorphism of $\GG_{R(\Upsilon^*)}$.
\end{proof}

\begin{defn} \label{defn:singreg}
Let $\Omega=\gpo$ be a generalised equation and let $\langle \P,R\rangle$ be a connected periodic structure on $\Omega$. We say that the generalised equation $\Omega$ is \index{generalised equation!strongly singular with respect to a periodic structure}\emph{strongly singular with respect to the periodic structure $\langle \P,R\rangle$} if one of the following conditions holds.
\begin{itemize}
\item[(a)] The generalised equation $\Omega$ is not periodised with respect to the periodic structure $\langle \P,R\rangle$.
\item[(b)] The generalised equation $\Omega$ is periodised with respect to the periodic structure $\langle \P,R\rangle$ and there exists an automorphism $\varphi$ of the coordinate group $\GG_{R(\Upsilon^*)}$ of the form described in parts (\ref{spl4}) or (\ref{spl5}) of Lemma \ref{2.10''}, such that $\varphi$ does not induce an automorphism of $\GG_{R(\Omega^*)}$.
\end{itemize}

We say that the generalised equation $\Omega$ is \index{generalised equation!singular with respect to a periodic structure}\emph{singular with respect to the periodic structure $\langle \P,R\rangle$} if $\Omega$ is not strongly singular with respect to the periodic structure $\langle \P,R\rangle$ and one of the following conditions holds
\begin{itemize}
\item[(a)] The set $C^{(2)}$ has more than one element.
\item[(b)] The set $C^{(2)}$ has exactly one element, and (in the above notation) there exists a cycle $\cc_{e_0}\in \langle C^{(1)}\rangle$, $h(e_0)\notin \P$ such that $h(\cc_{e_0})\ne 1$ in $\GG_{R(\Omega^\ast)}$.
\end{itemize}

Otherwise, we say that $\Omega$ is \index{generalised equation!regular with respect to a periodic structure}\emph{regular with respect to the periodic structure $\langle \P,R\rangle$}. In particular if $\Omega$ is singular or regular with respect to the periodic structure  $\langle \P,R\rangle$ then $\Omega$ is periodised.

When no confusion arises, instead of saying that $\Omega$ is (strongly) singular (or regular) with respect to the periodic structure $\langle \P,R\rangle$ we say that the periodic structure $\langle \P,R\rangle$ is \index{periodic structure!strongly singular!of type (a)}\index{periodic structure!strongly singular!of type (b)}(\emph{strongly}) \index{periodic structure!singular}\emph{singular} (or \index{periodic structure!regular}\emph{regular}).
\end{defn}

\begin{defn} \label{defn:AA}
Let $\Omega=\gpo$ be a generalised equation and let  $\langle \P, R\rangle $ be a periodic structure on $\Omega$. If $\Omega$ is strongly singular  with respect to $\langle \P, R\rangle$, then we define the group \glossary{name={$\AA(\Omega)$}, description={finitely generated group of automorphisms of $\GG_{R(\Omega^\ast)}$ associated with a periodic structure on $\Omega$}, sort=A} $\AA(\Omega)$ of automorphisms of $\GG_{R(\Omega^\ast)}$ to be trivial.

Otherwise, i.e. if $\langle \P, R\rangle $ is singular or regular, we set  $\AA(\Omega)$ to be the group of automorphisms of $\GG_{R(\Omega^\ast)}$ generated by the automorphisms induced by the automorphisms of $\GG_{R(\Upsilon^*)}$ described in statements (\ref{spl4}) and (\ref{spl5}) of Lemma \ref{2.10''}. Note that by definition of singular and regular periodic structures, every such automorphism of $\GG_{R(\Upsilon^*)}$ induces an automorphism of $\GG_{R(\Omega^\ast)}$  and, therefore the group $\AA(\Omega)$ is finitely generated.
\end{defn}

\subsection{Example} \label{sec:explperstr}

In this section we give an example of a generalised equation $\Omega$, a periodic structure on the generalised equation $\Omega$ and some of the constructions used in Section \ref{sec:perstr}.

Let $\Omega$ be a generalised equation shown on Figure \ref{quadratic} (the one to which the entire transformation is applied). Let $\langle\P,R\rangle$ be a periodic structure on $\Omega$ defined as follows. Set $h_1,h_3,h_4,h_5,h_7\in \P$. From part (a) of Definition \ref{above} it follows that
$$
\lambda_1,\lambda_3, \nu, \Delta(\lambda_1), \Delta(\lambda_3), \Delta(\nu) \in \P.
$$
From  part (c) of Definition \ref{above} we get that then $\sigma=[1,8]\in \P$. The set $\B$ in part (e) of Definition \ref{above} is defined to be $\{1,\dots,8\}$. For the bases $\lambda_1,\lambda_3, \nu$ and their duals, we have
\begin{gather}\notag
\begin{split}
&[1=\alpha(\lambda_1)]\,\sim_R\,[\alpha(\Delta(\lambda_1))=5=\beta(\mu)]\, \sim_R \, [ \beta(\Delta(\mu))=8=\beta(\Delta(\lambda_3))]\, \sim_R\, [\beta(\lambda_3)=4];\\
&[2=\beta(\lambda_1)]\, \sim_R\, [\beta(\Delta(\lambda_1))=6];\\
&[3=\alpha(\lambda_3)]\, \sim_R\, [\alpha(\Delta(\lambda_3))=7].
\end{split}
\end{gather}
Therefore, there are three $R$-equivalence  classes: $\{1,4,5,8\}, \{2,6\},\{3,7\}$.

Suppose that $\Omega$ is not strongly singular with respect to the periodic structure $\langle\P, R\rangle$. Note that, using the decidability of the universal theory of $\GG$, one could check if $\langle\P, R\rangle$ is strongly singular or not.

We construct the graph $\Gamma$ of the periodic structure $\langle\P,R\rangle$ as in Definition \ref{defn:graphperstr}. The graph $\Gamma$ is shown on Figure \ref{fig:perstgraph}. The edges $e_i$ of the graph $\Gamma$ are labelled by the corresponding items $h_i$. In this example the set $\Sh(\Gamma)$ equals $\{e_2,e_6\}$. We take the subgraph $\Gamma_0$ of $\Gamma$ whose edges are elements of $\Sh$ and choose a maximal subtree $T_0$ of $\Gamma_0$ as shown on Figure \ref{fig:perstgraph}. We extend the tree $T_0$ to a maximal subtree $T$ of $\Gamma$, which is also shown on Figure \ref{fig:perstgraph}.

\begin{figure}[!h]
  \centering
   \includegraphics[keepaspectratio,width=5in]{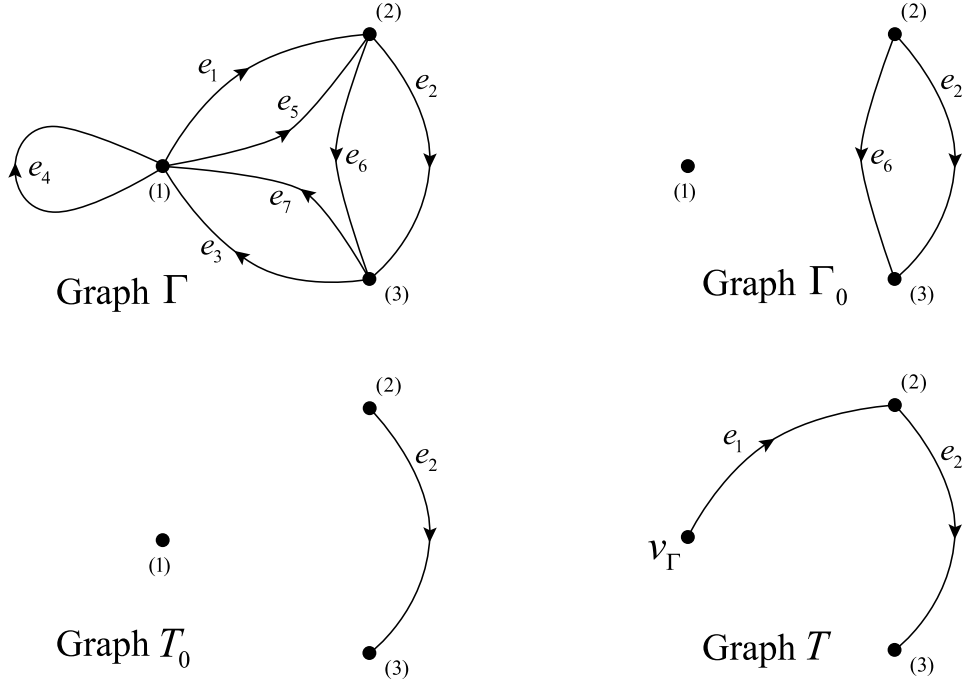}
\caption{The graphs   $\Gamma$, $\Gamma_0$, $T_0$ and $T$.} \label{fig:perstgraph}
\end{figure}

We fix a vertex  $v_\Gamma$ in $\Gamma$. Then the fundamental group $\pi_1(\Gamma,v_\Gamma)$ of $\Gamma$ with respect to $v_\Gamma$ is generated by the cycles $\cc_{e_3},\cc_{e_4},\cc_{e_5},\cc_{e_6},\cc_{e_7}$:
$$
\pi_1(\Gamma,v_\Gamma)=\langle \cc_{e_3},\cc_{e_4},\cc_{e_5},\cc_{e_6},\cc_{e_7}\rangle, \hbox{ where}
$$
$$
\begin{array}{lll}
\cc_{e_3}=e_1e_2e_3; & \cc_{e_5}=e_5e_1^{-1};& \cc_{e_4}=e_4;\\
\cc_{e_6}=e_1e_6e_2^{-1}e_1^{-1};&\cc_{e_7}=e_1e_2e_7.&
\end{array}
$$

Since, by definition, for any base $\mu\in \P$ one has that $\beta(\mu)\, \sim_R\, \beta(\Delta(\mu))$, the word $h(\mu)h(\Delta(\mu))^{-1}$ labels a cycle $\cc_\mu$ in the graph $\Gamma$. For the bases ${\lambda_1},{\lambda_3},\nu \in \P$, we construct the corresponding cycles $\cc_{\lambda_1}, \cc_{\lambda_3},\cc_\nu$ as in Equation (\ref{2.53}):
$$
\cc_{\lambda_1}=e_1e_5^{-1},\quad \cc_{\lambda_3}=e_1e_2e_3e_7^{-1}e_2^{-1}e_1^{-1},\quad \cc_\nu=e_1e_2e_3e_4e_7^{-1}e_6^{-1}e_5^{-1}e_4^{-1}.
$$
Since the cycles $\cc_{\lambda_1}, \cc_{\lambda_3},\cc_\nu$ belong to $\pi_1(\Gamma,v_\Gamma)$, these cycles can be written as words $b_{\lambda_1}, b_{\lambda_3},b_\nu$ in the generators $\cc_{e_i}$, $i=3,4,5,6,7$ as follows:
$$
b_{\lambda_1}=\cc_{e_5}^{-1}, \quad b_{\lambda_3}=\cc_{e_3}\cc_{e_7}^{-1}, \quad b_\nu=\cc_{e_3}\cc_{e_4}\cc_{e_7}^{-1}\cc_{e_6}^{-1}\cc_{e_5}^{-1}\cc_{e_4}^{-1}.
$$
Let $\widetilde{Z}$ be the free abelian group generated by the images of the cycles $\cc_{e_i}$, $i=3,4,5,6,7$ in the abelianisation of $\pi_1(\Gamma,v_\Gamma)$:
$$
\widetilde{Z}=\langle\tilde{\cc}_{e_3}\rangle\oplus\langle\tilde{\cc}_{e_4}\rangle\oplus\langle\tilde{\cc}_{e_5} \rangle\oplus\langle\tilde{\cc}_{e_6}\rangle\oplus\langle\tilde{\cc}_{e_7}\rangle=\langle e_1+e_2+e_3,\, e_4,\,e_5-e_1,\,e_6-e_2,\, e_1+e_2+e_7\rangle.
$$
Consider the subgroup $\widetilde{B}$ of $\widetilde{Z}$ generated by the images of  $b_{\lambda_1}, b_{\lambda_3},b_\nu$ and $\cc_{e_6}$ in the abelianisation of $\pi_1(\Gamma,v_\Gamma)$:
$$
\widetilde{B}=\langle \tilde{b}_{\lambda_1}\rangle\oplus\langle\tilde{b}_{\lambda_3}\rangle\oplus\langle\tilde{b}_{\nu}\rangle\oplus\langle\tilde{\cc}_{e_6}\rangle=\langle e_1-e_5,\, e_3-e_7,\, e_1+e_2+e_3-e_7-e_6-e_5,\, e_6-e_5\rangle.
$$
By the classification theorem of finitely generated abelian groups, there exist free abelian groups $\widetilde{Z}_1$ and $\widetilde{Z}_2$ such that
$\widetilde{Z}=\widetilde{Z}_1 \oplus \widetilde{Z}_2$ and $[\widetilde{Z}_1:\widetilde{B}]<\infty$. To construct  this decomposition (see Equation (\ref{2.54})), we construct a matrix with rows corresponding to the decomposition of the generators of $\widetilde{B}$ in terms of the generators of $\widetilde{Z}$ and then apply the elementary transformations:
$$
\left(
  \begin{array}{ccccc}
    0 & 0 & -1 & 0 & 0 \\
    1 & 0 & 0 & 0 & -1 \\
    1 & 0 & -1& -1& -1\\
    0 & 0 & 0 & 1 & 0 \\
  \end{array}
\right)
\rightsquigarrow
\left(
  \begin{array}{ccccc}
    0 & 0 & -1 & 0 & 0 \\
    1 & 0 & 0 & 0 & 0 \\
    0 & 0 & 0& 0& 0\\
    0 & 0 & 0 & 1 & 0 \\
  \end{array}
\right)
$$
It follows that
$$
\widetilde{B}=\widetilde{Z}_1, \quad \widetilde{Z}_1=\langle e_3-e_7,\, e_5-e_1,\,e_6-e_2\rangle, \quad \widetilde{Z}_2=\langle e_4,\, e_1+e_2+e_7\rangle.
$$
The basis of $\widetilde{Z}_1$ and $\widetilde{Z}_2$ are $\widetilde{C}^{(1)}=\{e_3-e_7,\, e_5-e_1,\,e_6-e_2\}$ and  $\widetilde{C}^{(2)}=\{ e_4,\, e_1+e_2+e_7\}$, respectively.

It is easy to see that the sets
$$
C^{(1)}=\{\cc_{e_3}\cc_{e_7}^{-1}, \cc_{e_5},\cc_{e_6}\},\quad  C^{(2)}=\{ \cc_{e_4}, \cc_{e_7}\}
$$
form a base of $\pi_1(\Gamma,v_\Gamma)$ and their images in $\widetilde{Z}$ are the sets $\widetilde{C}^{(1)}=\{e_3-e_7,\, e_5-e_1,\,e_6-e_2\}$ and $\widetilde{C}^{(2)}=\{ e_4,\, e_1+e_2+e_7\}$, correspondingly. Since $|C^(2)|=2$, $\Omega$ is singular with respect to the periodic structure $\langle \P, R\rangle$.

Note that for every solution $H$ of $\Omega$ we have $H(C^{(1)})=1$. Indeed, since, by assumption, $\Omega$ is periodised with respect to $\langle \P, R\rangle$, the following equalities hold in the coordinate group $\GG_{R(\Omega^*)}$:
$$
\cc_{e_3}\cc_{e_7}^{-1}=b_{\lambda_3}, \quad \cc_{e_5}=b_{\lambda_1}^{-1},\quad \cc_{e_6}=b_{\lambda_3}^{-1}b_{\nu}b_{\lambda_1}^{-1}.
$$
The statement follows, since for any cycle $\cc_\mu$, one has $H(\cc_\mu)=1$ for any solution $H$.

We construct the set of generators $\bar x$ of $\GG_{R(\Omega^*)}$ used in Lemma \ref{2.10''}. The set $\bar x$ is a union of $\bar t=\{h_8\}$, $\left\{h_2\right\}$ (since $e_2\in T_0$), $\{u_{1e_6}=h(e_1e_6e_2^{-1}e_1^{-1})\}$, $\{z_{1e_6}=h(e_6e_2^{-1})\}$ (since $T\setminus T_0=\{e_1\}$), $\{h(e_1)\}$, $h(C^{(1)})$ and $h(C^{(2)})$.

\subsection{Strongly singular and singular cases} \label{sec:singcase}

The next lemma states that  if a generalised equation $\Omega$ is strongly singular with respect to a periodic structure, then every periodic solution of $\Omega$  is in fact a solution  of a proper quotient, which can be effectively constructed. In other words, every periodic solution of the generalised equation can be obtained as a composition of a proper epimorphism and a solution of a proper generalised equation.

\begin{lem} \label{lem:23-1ss}
Let $\Omega=\gpo$ be a generalised equation without boundary connections, strongly singular with respect to the periodic structure $\langle \P, R\rangle$. Then there exists a finite family of elements $\{g_i\}$, $g_i\in \GG_{R(\Omega^\ast)}$ so that for any solution $H$ of $\Omega$ periodic with respect to a period $P$, such that ${\P}(H,P)=\langle{\P},R\rangle$, one has that $g_i\ne 1$ in $\GG_{R(\Omega^\ast)}$ for some $i$ and $H(g_j)=1$ for all $g_j\in\{g_i\}$.
\end{lem}
\begin{proof}
Suppose that $\Omega$ is strongly singular of type (a) with respect to $\langle \P, R\rangle$. Since $\Omega$ is not periodised with respect to $\langle \P, R\rangle$ there exist two cycles $\cc_1,\cc_2\in C^{(1)}\cup C^{(2)}$ such that $g=[h(\cc_1),h(\cc_2)]\ne 1$ in $\GG_{R(\Omega^*)}$. Then, by Lemma \ref{2.9}, for every $P$-periodic solution $H$ of $\Omega$ one has $H(g)=1$. In this case we set the family of elements $\{g_i\}$ to be the finite set of commutators of the form $[h(\cc_1),h(\cc_2)]$, where $\cc_1,\cc_2\in C^{(1)}\cup C^{(2)}$.

Suppose that $\Omega$ is strongly singular of type (b) with respect to $\langle \P, R\rangle$. Notice that by the universal property of the quotient, for every automorphism $\varphi$ of the coordinate group $\GG_{R(\Upsilon^*)}$, the following diagram commutes
$$
\xymatrix{
  \GG_{R(\Upsilon^*)} \ar[d] \ar[r]^{\varphi} & \GG_{R(\Upsilon^*)} \ar[d]  \\
 \GG_{R(\Upsilon^*\cup \mathcal{C})}  \ar@{->>}[r]^{\tilde\varphi}  &\GG_{R(\Upsilon^*\cup \mathcal{C}\cup \varphi(\mathcal{C}))}
 }
$$
where $\mathcal{C}=\{[h_i,h_j]\mid \Re_\Upsilon(h_i,h_j)\}$. Note that $\tilde\varphi$ is an automorphism of $\GG_{R(\Upsilon^*\cup \mathcal{C})} =\GG_{R(\Omega^*)}$ if and only if $\varphi(\mathcal{C})\subset R(\Upsilon^*\cup \mathcal{C})$. Since $\Omega$ is strongly singular of type (b) with respect to $\langle \P, R\rangle$, there exists an automorphism $\varphi$ of $\GG_{R(\Upsilon^*)}$ described in statements (\ref{spl4}) or (\ref{spl5}) of Lemma \ref{2.10''} and a commutator $[h_i,h_j]\in \mathcal{C}$ such that $\varphi([h_i,h_j])\notin R(\Upsilon^*\cup \mathcal{C})$, thus $g=\varphi([h_i,h_j])$ is non-trivial in $\GG_{R(\Omega^*)}$. In this case, the finite set of elements of the form $\varphi([h_i,h_j])$, where  $\Re_\Upsilon(h_i,h_j)$ and $\varphi$ is a generator of the group $\AA(\Omega)$, satisfies the required properties.

Since the automorphism $\tilde \varphi$ is identical on $\bar t$, $h(e)$, where $e\in T_0$, $\bar u^{(i)}$, $\bar z^{(i)}$, $i=1,\dots, m$, and on $h(C^{(1)})$, it follows that $\tilde\varphi$ is identical on the set $\{h_k\mid h_k\notin \P\}$. Therefore, from $[h_i,h_j]\in \mathcal{C}$ and $\varphi([h_i,h_j])\notin R(\Upsilon^*\cup\mathcal{C})$, we get that $h_i \in \P$ or $h_j \in \P$. Without loss of generality, we may assume that $h_i\in \P$.

By part (\ref{spl3}) of Lemma \ref{2.10''}, $\bar x$ is a generating set of $\GG_{R(\Omega^*)}$. The item $h_i$ can be written in the generators $\bar x$ as a word $w(h(\Sh), \bar u, \bar z, h(C^{(1)}), h(C^{(2)}), h(e_1), \dots, h(e_m))$. From the above discussion, it follows that $\tilde\varphi(h_i)=w(h(\Sh), \bar u, \bar z, h(C^{(1)}), \tilde\varphi(h(C^{(2)})), \tilde\varphi(h(e_1)), \dots, \tilde\varphi(h(e_m)))$.

Let $H$ be a solution of $\Omega$ periodic with respect to a period $P$ such that ${\P}(H,P)=\langle{\P},R\rangle$. Since
\begin{gather}\label{eq:23ss1}
\begin{split}
\az(H(h(\Sh))), &\az(H(\bar u)), \az(H(\bar z)), \\
&\az(H(C^{(1)})), \az(H(C^{(2)})), \az(H(h(e_1))), \dots, \az(H(h(e_m)))\subseteq\az(P),
\end{split}
\end{gather}
by the definition of the automorphism $\tilde \varphi$, we get that
\begin{gather}\label{eq:23ss2}
\begin{split}
\az(H(h(\Sh))), & \az(H(\bar u)), \az(H(\bar z)), \az(H(C^{(1)})), \az(H(\tilde\varphi(h(C^{(2)})))), \\
&\az(H(\tilde\varphi(h(e_1)))), \dots,\az(H(\tilde\varphi(h(e_m)))) \subseteq\az(P).
\end{split}
\end{gather}
Therefore, $\az(H(\tilde\varphi(h_i)))\subseteq \az(P)$.

Since $h_i\in \P$, we get that  $\az(H_i)=\az(P)$. As $H$ is a solution of $\Omega$ and $\Re_\Upsilon(h_i,h_j)$, so $H_j\lra H_i$, i.e. $H_j\lra \az(P)$. Furthermore, from Equation (\ref{eq:23ss2}), it follows that $H(\tilde\varphi(h_i))\lra H_j$. Since $\bar x$ generates $\GG_{R(\Omega^*)}$ and from Equation (\ref{eq:23ss1}), it follows that $h_j$ can be written as a word in $\bar t$. Then, since, as mentioned above $\tilde \varphi$ is identical on $\bar t$, we have that $H(g)=H(\tilde\varphi([h_i,h_j]))=H([\tilde\varphi(h_i),h_j])=1$.
\end{proof}

The next lemma states that  if a generalised equation $\Omega$ is singular with respect to a periodic structure, then one can construct finitely many proper quotients of the coordinate group $\GG_{R(\Omega^*)}$, such that for every periodic solution there exists an $\AA(\Omega)$-automorphic image such that this image is in fact a solution of one of the quotients constructed. In other words, every solution of the generalised equation can be obtained as a composition of an automorphism from $\AA(\Omega)$ and a solution of a proper generalised equation.

\begin{lem} \label{lem:23-1}
Let $\Omega $ be a generalised equation without boundary connections, singular with respect to the periodic structure $\langle \P, R\rangle $. Then there exists a finite family of cycles $\cc_1,\ldots,\cc_r$ in the graph $\Gamma$ such that the following conditions hold:
\begin{enumerate}
\item \label{it:23-11} $h(\cc_i)\ne 1$, $i=1,\ldots, r$ in $\GG_{R(\Omega^\ast)}$;
\item \label{it:23-12} for any solution $H$ of $\Omega$ periodic with respect to a period $P$, such that ${\P}(H,P)=\langle{\P},R\rangle$, there exists an $\AA(\Omega)$-automorphic image $H^+$ of $H$ such that $H^+(\cc_i)=1$ for some $1\le i\le r$;
\item\label{it:23-13} for any $h_k \notin \P$, $H_k\doteq H^+_k$; and for any $h_k \in \P$ if $H_k \doteq P_1P^{n_k}P_2$, then $H^+_k\doteq P_1P^{n^+_k}P_2$, where $\delta(k)=P_1P_2$, $n_k, n^+_k \in \Z$.
\end{enumerate}
\end{lem}
\begin{proof}
Suppose that $|C^{(2)}|\ge 2$. Let $H$ be a solution of a generalised equation $\Omega $ periodised with respect to the period $P$ such that ${\P}(H,P)=\langle {\P}, R\rangle$. Let $\cc_1,\cc_2\in C^{(2)}$. By Lemma \ref{2.9}, $H(\cc_1)=P^{n_1}$, $H(\cc_2)=P^{n_2}$.

Without loss of generality we may assume that $n_1\le n_2$. Write $n_2=tn_1+r$, where either $r=0$ or $0<|r|<|n_1|$. Applying the automorphism $\varphi:h(\cc_2)\mapsto {h({\cc_1})}^{-t}h(\cc_2)$, $\varphi\in \AA(\Omega)$, we get $H(\varphi(h(\cc_2)))=P^r$, hence the exponent of periodicity $\e(H(\varphi(h(\cc_2))))$ is reduced.

Applying the Euclidean algorithm, we get that there exists an automorphism $\psi$ from $\AA(\Omega)$ such that $H(\psi(h(\cc_2)))=1$ or $H(\psi(h(\cc_1)))=1$. Set $H^+=\psi(H)$. Then the set $\{\cc_1,\cc_2\}$ satisfies conditions (\ref{it:23-11}) and (\ref{it:23-12}) of the lemma.

In the notation of Lemma \ref{2.10''}, the automorphism $\psi$ is identical on all the elements of $\bar x$ except for $h(C^{(2)})$, hence, in particular, it is identical on $h(e)$ such that $e \in T$ or $e\notin\P$, and on $h(C^{(1)})$, i.e. $H^+(e)=H(e)$, for any $e\in T$ or $e\notin \P$ and $H^+(\cc^{(1)})=H(\cc^{(1)})$ for any $\cc^{(1)}\in C^{(1)}$.

If $h_k=h(e)$, $e \notin T$, $h(e) \in \P$, then $h_k=h(\p_1)h(\cc_e)h(\p_2)$, where $\p_1,\p_2$ are paths in $T$. Therefore, $H_k^+=H^+(e)=H^+(\p_1)H^+(\cc_e)H^+(\p_2)=H(\p_1)H^+(\cc_e)H(\p_2)$. Since $h(\cc_e)$ lies in the subgroup generated by $h(C^{(1)})$ and $h(C^{(2)})$, then $H^+(\cc_e)$ and $H(\cc_e)$ lie in the cyclic group generated by $P$. This proves statement (\ref{it:23-13}) of the lemma.

Suppose that $C^{(2)}=\{\cc^{(2)}\}$. Since the periodic structure is singular, there exists a cycle $\cc_{e_0}$ such that $h(e_0)\notin \P$, $\cc_{e_0}\in \langle C^{(1)}\rangle$ and $h(\cc_{e_0})\ne 1$ in $\GG_{R(\Omega^\ast)}$.

We define the set of cycles $\{\cc_1,\ldots, \cc_r\}$  to be
$$
\left\{(\cc_{e_0})^i(\cc^{(2)})^j\,\mid\,  \hbox{$i$ and $j$ are not simultaneously zero}, \, |i|,|j|\leq 2\rho\right\}.
$$

We want to show that the elements of the set just defined are non-trivial. In order to do so we show that if ${h(\cc_{e_0})}^i {h(\cc^{(2)})}^j=1$ in $\GG_{R(\Omega ^\ast)}$, then $i=j=0$. Suppose ${h(\cc_{e_0})}^i {h(\cc^{(2)})}^j=1$. Let $\sigma _0$ be an automorphism from $\AA(\Omega)$ such that $\sigma _0(h(\cc_{e_0}))=h(\cc_{e_0})$ and $\sigma_0(h(\cc^{(2)}))=h(\cc_{e_0})h(\cc^{(2)})$. Hence ${h(\cc_{e_0})}^{i+j}{h(\cc^{(2)})}^j=1$ in $\GG_{R(\Omega ^\ast)}$, and so ${h(\cc_{e_0})}^{j}=1$. Since $\GG$ is torsion-free, this implies that either $h(\cc_{e_0})=1$ or $j=0$. Since $h(\cc_{e_0})\ne 1$, we have $j=0$. Then $h(\cc_{e_0})^i=1$ and, since $\GG$ is torsion-free, so $i=0$.

Let $H$ be a $P$-periodic solution of the generalised equation $\Omega$ such that ${\P}(H, P)=\langle {\P}, R \rangle$. By Lemma \ref{lem:spl1}, we have $H(\cc_{e_0})=P^{n_0}$, $|n_0| \leq 2 \rho$. Let $H(\cc^{(2)})= P^n$, $\cc^{(2)}\in C^{(2)}$.

If $n_0 = 0$, we can take $\sigma = 1$, $H^+ = H$ and  the set of cycles $\{\cc_{e_0}\}$.

Let $n_0 \neq 0$, $n = tn_0 + n'$, and $|n'| \leq 2 \rho$. Let $\sigma= \sigma _0^{-t}$ and define $H^+$ to be the image of $H$ under $\sigma$. Set $\cc=(\cc_{e_0})^{-n'} (\cc^{(2)})^{n_0}$, then $H^+ (\cc_{e_0}) = P^{n_0}$, $H^+(\cc^{(2)}) = P^{n'}$ and $H^+(\cc)=1$.
An analogous argument to the one given in the case $|C^{(2)}| \geq 2$ shows that $H^+_k=H_k$ if $h_k \notin \P$, and $H^+_k=P_1P^{n^+_k}P_2$ if $H_k=P_1P^{n_k}P_2$ and $h_k \in \P$.
\end{proof}

\subsection{Example} \label{sec:explsing}
Let $\Omega$ be the generalised equation shown on Figure \ref{quadratic}. The associated system of equations is as follows:
$$
h_1h_2h_3h_4=h_4h_5h_6h_7;\  h_1=h_5;\  h_3=h_7;\  h_2=h_8;\  h_6=h_8;\  h_8=a.
$$
Consider the periodic structure $\langle \P, R\rangle$ defined in Section \ref{sec:explperstr}. We have shown in the example given in Section \ref{sec:explperstr}, that this periodic structure is singular. Let $H=(H_1,\dots, H_8)$ be defined as follows:
$$
\begin{array}{llll}
H_1=(bac)^2b; &  H_3=(cba)^2c; &H_5=(bac)^2b;&  H_7=(cba)^2c;\\
H_2=a; &H_4=(bac)^6;&H_6=a;&H_8=a.
\end{array}
$$
It is easy to see that for $P=bac$ the tuple $H$ is a $P$-periodic solution of $\Omega$.

In the notation of Lemma \ref{2.10''}, from the example given in Section \ref{sec:explperstr}, we have
$$
\bar x=\left\{h_8, h_2,h_1h_6h_2^{-1}h_1^{-1},h_6h_2^{-1}, h(C^{(1)}),h(C^{(2)})\right\},
$$
where $C^{(1)}=\{\cc_{e_3}\cc_{e_7}^{-1}, \cc_{e_5},\cc_{e_6}\}$, $C^{(2)}=\{ \cc_{e_4}, \cc_{e_7}\}$. Notice that, in particular
$$
H(C^{(2)})=\left\{H(\cc_{e_4}), H(\cc_{e_7})\right\}= \left\{P^6,P^5\right\}.
$$
Consider the automorphism $\varphi\in \AA(\Omega)$ induced by the transformation:
$$
\cc_{e_4}\mapsto\cc_{e_7}^{-1}\cc_{e_4};\quad\cc_{e_7}\mapsto(\cc_{e_7}^{-1}\cc_{e_4})^{-5}\cc_{e_7}
$$
and identical on the other elements of $\bar x$.

Set $H^+=\varphi(H)$. Then $H^+$ satisfies the conditions (\ref{it:23-12}) and (\ref{it:23-13}) of Lemma \ref{lem:23-1}. Indeed,
$$
H^+(\cc_{e_4})=H(\cc_{e_7}^{-1}\cc_{e_4})=P \quad H^+(\cc_{e_7})=H((\cc_{e_7}^{-1}\cc_{e_4})^{-5}\cc_{e_7})=1.
$$
Furthermore, since $\varphi$ is the identity on the other elements of $\bar x$, we get that
\begin{gather}\notag
\begin{split}
&H_1^+=H_1;\ H_2^+=H_2;\  H_8^+=H_8;\\
&H_6^+ {H_2^+}^{-1}=H_6 {H_2}^{-1}\hbox{ and so } H_6^+=H_6;\\
&H_5^+ {H_1^+}^{-1}=H_5 {H_1}^{-1}\hbox{ and so } H_5^+=H_5;\\
&H_4^+=H^+ (\cc_{e_4})=P;\\
&1=H^+ (\cc_{e_7})=H_1^+H_2^+H_7^+\hbox{ and so } H_7^+= H_2^{-1}H_1^{-1};\\
&H_3^+{H_7^+}^{-1}=H_3{H_7}^{-1}=1 \hbox{ and so }   H_3^+=H_7^+= H_2^{-1}H_1^{-1}.
\end{split}
\end{gather}

\subsection{Regular case}

The aim of this section is to prove that if a generalised equation $\Omega$ is regular with respect to a periodic structure, then periodic solutions of $\Omega$ minimal with respect to the group of automorphisms $\AA(\Omega)$ have bounded exponent of periodicity. In other words, the length of such a solution is bounded above by a function of $\Omega$ and the length $|P|$ of its period $P$.

\begin{lem}\label{lem:23-1.5}
Let $\Omega $ be a generalised equation with no boundary connections, periodised with respect to a connected periodic structure $\langle \P, R\rangle$. Let $H$ be a periodic solution of $\Omega$ such that $\P(H,P)=\langle \P, R\rangle$ and minimal with respect to the trivial group of automorphisms. Then, either for all $k$, $1\le k\le \rho$ we have
\begin{equation} \label{2.62}
|H_k| \le 2 \rho |P|,
\end{equation}
or there exists a cycle $\cc\in \pi_1(\Gamma,v_\Gamma)$  so that $H(\cc)=P^n$, where $1\le n\le 2\rho$.
\end{lem}
\begin{proof}
Let $H$ be a $P$-periodic solution of the generalised equation $\Omega$ minimal with respect to the trivial group of automorphisms.

Suppose that there exists a variable $h_k \in {\P}$ such that $|H_k| > 2 \rho |P|$. We then prove that there exists a cycle $\cc\in \pi_1(\Gamma,v_\Gamma)$  so that $H(\cc)=P^n$, where $1\le n\le 2\rho$.

Construct a chain
\begin{equation} \label{2.63}
(\Omega, {H})=(\Omega_{v_0}, {H}^{(0)}) \to (\Omega_{v_1}, {H}^{(1)}) \to \cdots \to (\Omega_{v_t}, {H}^{(t)}),
\end{equation}
in which for all $i$, $\Omega_{v_{i+1}}$ is obtained from $\Omega_{v_{i}}$ using $\ET 5$: by $\mu$-tying a free boundary that intersects a certain base $\mu \in {\P}$. Chain (\ref{2.63}) is constructed once all the boundaries intersecting bases $\mu$ from ${\P}$ are $\mu$-tied. This chain is finite, since, by the definition, boundaries that are introduced when applying $\ET 5$ are not free.

By construction, the generalised equations $\Omega_{v_{i}}$'s in (\ref{2.63}) have boundary connections. Our definition of periodic structure is given for generalised equations without boundary connections, see Definition \ref{above}. For this reason we define $\Omega_{v_{i}'}$ to be the generalised equation obtained from $\Omega_{v_i}$ by omitting all boundary connections (here we do not apply $\D 3$). It is obvious that the solution ${H}^{(i)}$ of  $\Omega_{v_i}$ is also a solution of the generalised equation $\Omega_{v_{i}'}$ and is periodic with respect to the period $P$. Denote by $\langle {\P}_i, R_i \rangle$ the periodic structure ${\P}({H}^{(i)}, P)$ on the generalised equation $\Omega_{v_i'}$ restricted to the closed sections from ${\P}$, and by $\Gamma^{(i)}$ the corresponding graph.

If $(p, \mu, q)$, $\mu \in {\P}$, is a boundary connection of the generalised equation $\Omega_{v_i}$, $i=1,\dots,t$, then $\delta(p) = \delta(q)$. Therefore, all the graphs $\Gamma^{(0)}$, $\Gamma^{(1)}, \ldots, \Gamma^{(t)}$ have the same set of vertices, whose cardinality does not exceed $\rho$. By Lemma \ref{lem:2.1}, the solution ${H}^{(t)}$ of the generalised equation $\Omega_{v_t}$ is also minimal with respect to the trivial group of automorphisms.

Suppose that for some variable $h_l$ belonging to a section from ${\P}$ the inequality $|H_l^{(t)}| > 2 |P|$ holds.

Let $\mathcal{H}=\left\{h_i\in \sigma \mid \sigma \in \P\hbox{ and } H_i^{(t)}\doteq {{H}_l^{(t)}}^{\pm 1}\right\}$.  Consider the partially commutative $\GG$-group $\GG(\cA\cup\{u\})$, where $u$ is a new letter and $[u,a_k]=1$, $a_k\in \cA$ for all $a_k\in \BA(H_j)$ (see Section \ref{sec:pcgr} for definition), where $\Re_\Upsilon(h_i,h_j)$ for some $h_i\in \mathcal{H}$.

In the solution ${H}^{(t)}$, replace all the components ${H_i}^{(t)}$ such that $H_i^{(t)}\doteq {{H}_l^{(t)}}$ or $H_i^{(t)}\doteq {{H}_l^{(t)}}^{- 1}$ by the letter $u$ or $u^{-1}$, correspondingly (see Definition \ref{def:minsol}). Denote the resulting $\rho_{\Omega_{v_t}}$-tuple of words by ${H^{(t)}}'$.

We show that ${H^{(t)}}'$ is in fact a solution of $\Omega_{v_t}$ (considered as a generalised equation with coefficients from $\cA \cup \{u\}$, see Remark \ref{rem:minsol}). Obviously, by definition, every component of ${H^{(t)}}'$ is non-empty and geodesic. Since in the generalised equation $\Omega_{v_t}$ all the boundaries from ${\P}$ are $\mu$-tied, $\mu\in \P$, there is a one-to-one correspondence between the items that belong to $\mu$ and the items that belong to $\Delta(\mu)$. Therefore, ${H^{(t)}}'$ satisfies all basic equations $h(\mu)=h(\Delta(\mu))$, $\mu \in \P$ and all the boundary equations of the generalised equation $\Omega_{v_t}$.

If, on the other hand, $\mu \not \in {\P}$, then for every item $h_k \in \mu$ of the generalised equation $\Omega$ lying on a closed section from ${\P}$, we have $h_k \not \in {\P}$ and, consequently, $|H_k| \leq 2 |P|$. In particular, this inequality holds for every item $h_k \in \mu$, $\mu \notin \P$ of the generalised equation $\Omega_{v_t}$.  Therefore, such items have not been replaced in the solution ${H}^{(t)}$, thus ${H^{(t)}}'$ satisfies all basic equations $h(\mu)=h(\Delta(\mu))$, $\mu \notin \P$. From the construction of $\GG(\cA\cup\{u\})$ it follows that ${H^{(t)}}'$ satisfies the constraints from $\Re_{\Upsilon_{v_t}}$.

Let $\pi:\GG(\cA\cup\{u\})\to \GG(\cA\cup\{u\})$ be a map that sends $u$ to ${H}_l^{(t)}$ and fixes $\GG$. Since ${H^{(t)}}$ is a solution of $\Omega_{v_t}$, one has ${H_l^{(t)}}\lra H^{(t)}_j$, where $\Re_{\Upsilon_{v_t}}(h_i,h_j)$, $h_i\in \mathcal{H}$. Therefore, $[H_l^{(t)}, a_k]=1$, where $a_k\in \az(H^{(t)}_j)$ and $\Re_{\Upsilon_{v_t}}(h_i,h_j)$, $h_i\in \mathcal{H}$. These are the relations that, by construction, are satisfied by $u$, hence $\pi$ is  a $\GG$-endomorphism. This contradicts the minimality of the solution ${H^{(t)}}$. Indeed, one has $\pi_{H^{(t)}}=\pi_{{H^{(t)}}'}\pi$. By construction, it is easy to check that conditions (\ref{it:minsol3}) and (\ref{it:minsol4}) of Definition \ref{defn:sol<} hold. Therefore, ${H^{(t)}}'<_{\{1\}}{H^{(t)}}$. Obviously, $|{H_i^{(t)}}'|=|{H_i^{(t)}}|$ for all $i\ne l$ and $1=|u|=|{H_l^{(t)}}'|<|{H_l^{(t)}}|$. We thereby have shown that $|H_l^{(t)}| \leq 2 |P|$, if $h_l$ belongs to a closed section from ${\P}$.

In the construction of chain (\ref{2.63}) we introduced new boundaries, so every item from $\Omega_{v_0}$ can be expressed as a product of items from $\Omega_{v_t'}$. Consequently, since $|H_l^{(t)}| \leq 2 |P|$ for every $l$, if the component $H_k$ of the solution $H$ of the generalised equation $\Omega$ does not satisfy inequality (\ref{2.62}), then $h_k$ is a product of at least $\rho+1$ distinct items of $\Omega_{v_t'}$, $h_k=h^{(t)}_{s}\cdots h^{(t)}_{s+\varrho}$, where $\varrho\ge \rho+1$.

Since the graph $\Gamma^{(t)}$ contains at most $\rho$ vertices, there exist boundaries $l,l'\in \{s,\dots, s+\rho+1\}$, $l<l'$ so that  $\delta({l})=\delta({l'})$. The word $h[l,l']$ is a label of a cycle $\cc_t$ of the graph $\Gamma^{(t)}$ for which $0<|H(\cc_t)|\le 2 \rho |P|$.

Recall that by $\pi(v_i,v_j)$ we denote the homomorphism $\GG_{R({\Omega_{v_i}}^\ast)} \to \GG_{R({\Omega_{v_j}}^\ast)}$ induced by the elementary transformations. It remains to prove the existence of a cycle $\cc_0$ of the graph $\Gamma$ for which $\pi(v_\Gamma,v_t) (h(\cc_0))=h(\cc_t)$. To do this, it suffices to show that for every path $\p_{i+1} : v \rightarrow v'$ in the graph $\Gamma^{(i+1)}$ there exists a path $\p_i : v \rightarrow v'$ in the graph $\Gamma^{(i)}$ such that $\pi(v_i, v_{i+1}) (h(\p_i)) = h(\p_{i+1})$. In turn, it suffices to verify the latter statement in the case when $\p_{i+1}$ is an edge $e$.

The generalised equation $\Omega_{v_{i+1}}$ is obtained from $\Omega_{v_i}$ by $\mu$-tying a boundary $p$. Below we use the notation from the definition of the elementary transformation $\ET 5$.
\begin{enumerate}
    \item Either we introduce a boundary connection $(p,\mu,q)$. In this case every variable $h(e)$ of the generalised equation $\Omega_{v_{i+1}}$ is also a variable of the generalised equation $\Omega_{v_i}$ and the statement is obvious;
    \item Or we introduce a new boundary $q'$ between the boundaries $q$ and $q+1$, and a boundary connection $(p,\mu, q')$. Using the boundary equations we get
\begin{gather}\label{eq:theseform}
\begin{split}
&\pi(v_i,v_{i+1})^{-1}(h_{q})= {h[\alpha (\Delta (\mu )),q]}^{-1}h[\alpha (\Delta(\mu) ),q'] ={h[\alpha (\Delta (\mu )),q]}^{-1}h[\alpha (\mu ),p],\\
&\pi (v_i,v_{i+1})^{-1}(h_{q'})= {h[\alpha(\Delta(\mu) ),q']}^{-1} h[\alpha (\Delta (\mu )),q+1]={h[\alpha(\mu ),p]}^{-1} h[\alpha (\Delta (\mu )),q+1].
\end{split}
\end{gather}
Moreover, $\pi(v_i,v_{i+1})^{-1}(h(e))=h(e)$ for any other variable. Notice that since $(\alpha(\mu)) = (\alpha(\Delta(\mu)))$, the right-hand sides of Equations (\ref{eq:theseform}) are labels of paths in $\Gamma^{(i)}$.
\end{enumerate}
Thus, we have deduced that there exists a cycle $\cc\in \Gamma$ so that $1\le\e(H(\cc))\le 2\rho$.
\end{proof}

\begin{lem} \label{lem:23-2}
Let $\Omega $ be a generalised equation with no boundary connections. Suppose that $\Omega$ is regular with respect to a periodic structure $\langle \P, R\rangle$. Then there exists a computable function $f_0 (\Omega, {\P},R)$ such that, for every $P$-periodic solution $H$ of $\Omega$ such that $\P(H,P)=\langle \P,R\rangle$ and such that $H$ is minimal with respect to $\AA(\Omega)$, the following inequality holds
$$
|H_k| \leq f_{0} (\Omega, {\P},R) \cdot |P| \ \hbox{ for every $k$}.
$$
\end{lem}
\begin{proof}
Let $H$ be a $P$-periodic minimal solution with respect to the group of automorphisms $\AA(\Omega)$. Notice that $H$ is also minimal with respect to the trivial group of automorphisms.

We first prove that for any regular periodic structure and any $P$-periodic solution $H$ of $\Omega$ minimal with respect to the group of automorphisms $\AA(\Omega)$, the exponent of periodicity of the label of every simple cycle $\cc_e$, $e\notin T$ is bounded by a certain computable function $g_1 (\Omega, {\P}, R)$.

If for every $i$ we have $|H_i|\le 2\rho |P|$, the statement follows. We further assume that there exists $i$ such that  $|H_i|> 2\rho |P|$.

Since the periodic structure is regular, either $|C^{(2)}|=1$ and for all $e \not \in T$ so that $e \in \Sh$ we have $H(\cc_e)=1$, or $|C^{(2)}|=0$.

Assume first that, $C^{(2)}=\{\cc^{(2)}\}$. Since the periodic structure $\langle \P,R\rangle$ is regular, it follows that $H(\cc_e)=1$ for all $e\notin T$, $h(e)\notin \P$. Since ${H}$ is a solution of $\Omega$, it follows that $\e(H(b_\mu))=0$, $\mu \in {\P}$. Therefore, $\e(H(\cc))=0$ for all $\cc$ such that $\tilde{\cc}\in\widetilde{B}$. By (\ref{2.54}), we have that $[\widetilde{Z}_1:\widetilde{B}]<\infty$, hence we get that $\e(H(\cc))=0$ for all $\cc$ such that $\tilde{\cc}\in \widetilde{Z}_1$. Consequently, since the only non-trivial cycle at $v_\Gamma$ is $\cc^{(2)}$, by Lemma \ref{lem:23-1.5}, we get that $|\e(H({\cc}^{(2)}))| \le 2 \rho$. Using factorisation (\ref{2.54}), one can effectively express $\tilde{\cc}_e$, $e \not \in T$ in terms of the elements of the basis:
$$
\tilde{\cc}_e = n_e \tilde{\cc}^{(2)} + \tilde{z}_e^{(1)},\  \tilde{z}_e^{(1)} \in \widetilde{Z}_1.
$$
Hence $|\e(H(\cc_e))| = |n_e \e (H(\cc^{(2)}))| \leq 2 \rho n_e$, and we finally obtain
\begin{equation} \label{2.64}
|\e(H(\cc_e))| \le g_1 (\Omega, {\P}, R),
\end{equation}
where $g_1$ is a certain computable function.

Suppose next that $|C^{(2)}| = 0$, i.e. $\widetilde{Z} = \widetilde{Z}_1$. As we have already seen in the proof of Lemma~\ref{lem:spl1}, $|H(\cc_e)| \leq 2 \rho |P|$, where  $e \not \in T$, $h(e) \not \in {\P}$. Hence, $|\e(H(\cc_e))| \leq 2 \rho$ for all $e$ such that $h(e) \not \in {\P}$. Since $(\widetilde{Z}_{1} : \widetilde{B}) < \infty$, for every $e_0 \not \in T$ one can effectively construct the following equality
$$
n_{e_0} {\tilde{\cc}}_{e_0} = \sum\limits_{h(e) \notin \P} n_e \tilde{\cc}_e + \sum\limits_{\mu \in \P} n_\mu \tilde{b}_\mu.
$$
Hence,
$$
|\e(H({\cc}_{e_0}))| \leq |n_{e_0}\cdot\e ( H({\cc}_{e_0}))|\leq \sum\limits_{h(e)\notin {\P}}
|n_e \e(H({\cc}_e))| \leq 2 \rho \cdot \sum\limits_{h(e) \not \in {\P}} |n_e|.
$$
Thus, $|\e(H({\cc}_{e_0}))|\le g_2(\Omega,\P,R)$.

We now address the statement of the lemma. The way we proceed is as follows. For a $P$-periodic solution $H$ minimal with respect to the group of automorphisms $\AA(\Omega)$, we show that the vector that consists of exponents of periodicity of each of the components $H_k$ of $H$ is bounded by a minimal solution of a linear system of equations whose  coefficients depend only on the generalised equation. Since by Lemma 1.1 from \cite{Makanin}, the components of a minimal solution of a linear system of equations are bounded above by a recursive function of the coefficients of the system, we then get a recursive bound on the exponents of periodicity of the components of the solution $H$.

Let $\delta(k) = P_1^{(k)} P_2^{(k)}$. Denote by $t(\cc, h_k)$ the algebraic sum of occurrences of the edge with the label $h_k$ in the cycle $\cc$, (i.e. edges with different orientation contribute with different signs). For every item $h_k$ that belongs to a closed section from ${\P}$ one can write $$
H_k = P_2^{(k)} P^{n_k} P_1^{(k+1)}.
$$
Note that in the case that $h_k\in \P$, the above equality is graphical. However, in the case that $h_k\notin \P$ and $H_k$ is a subword of $P^{\pm 1}$ there is cancellation. Direct calculations show that
\begin{equation} \label{2.66}
H(\cc) = P^{\left(\sum\limits_k t(\cc, h_k)(n_k+1)\right) -1}.
\end{equation}
By Lemma \ref{lem:23-1.5}, $e_0 \not \in T$ can be chosen in such a way that $\e(H({\cc}_{e_0})) \neq 0$. Let $n_k = | \e (\tilde{\cc}_{e_0})| m_k + r_k$, where $0 < r_k \leq | \e (\tilde{\cc}_{e_0})|$. Equation (\ref{2.66}) implies that the vector $\{ m_k \mid h_k \in {\P} \}$ is a solution of the following system of Diophantine equations in variables $\{ z_k \mid h_k \in {\P} \}$:
\begin{equation} \label{2.67}
\sum\limits_{h_k \in {\P}} t(\cc_e, h_k)(|\e(H({\cc}_{e_0}))| z_k + r_k +1) + \sum\limits_{h_k \not \in {\P}} t(\cc_e, h_k)(n_k +1)  -1 = \e(H({\cc}_e)), \ e \not \in T.
\end{equation}
Note that the number of variables of the system (\ref{2.67}) is bounded. Furthermore, as we have proven above, free terms $\e(H({\cc}_e))$ of this system are also bounded above, and so are the coefficients $|n_k|\leq 2$ for $h_k\not\in {\P}$.

A solution $\{m_k\}$ of a system of linear Diophantine equations is called {\em minimal}, see \cite{Makanin}, if $m_k \geq 0$ and there is no other solution $\{m_k^+\}$ such that $0 \leq m_k^+ \leq m_k$ for all $k$, and at least one of the inequalities $m_k^+ \leq m_k$ is strict. Let us verify that the solution $\{m_k \mid h_k \in {\P} \}$ of system (\ref{2.67}) is minimal.

Indeed, let $\{ m_k^+ \}$ be another solution of system (\ref{2.67}) such that $0 \leq m_k^+ \leq m_k$ for all $k$, and at least for one $k$ the inequality is strict. Let $n_k^+ = |\e(H({\cc}_{e_0}))| m_k^+ + r_k$. Define a vector $H^+$ as follows: $H_k^+ = H_k$ if $h_k \not \in {\P}$, and $H_k^+ = P_2^{(k)} P^{n_k^+}P_1^{(k+1)}$ if $h_k \in {\P}$.

We now show that $H^+$ is a solution of the generalised equation which can be obtained from $H$ by an automorphism from $\AA(\Omega)$.

The vector $H^+$ satisfies all the coefficient equations and all the basic equations of $\Omega$. Indeed, since $\{m_k^+\}$ is a solution of system (\ref{2.67}), $H^+(\cc_e) = P^{\e(H({\cc}_e))} = H(\cc_e)$. Therefore, for every cycle $\cc$ we have $H^+(\cc) = H(\cc)$ and, in particular, $H^+(b_\mu) = H(b_\mu) = 1$. Thus the vector $H^+$ is a solution of the system $\Omega^\ast$.

By construction every component of the solution $H^+$ is non-empty and the words $H^+(\mu)$, $H^+(\Delta(\mu))$ are geodesic. On the other hand, for every $\mu$ we have
$$
H^+(\mu){H^+(\Delta(\mu))}^{-1} = 1.
$$
Since if $\mu \in \P$, then the words $H^+ (\mu)$ and $H^+ (\Delta(\mu))$ are $P$-periodic, it follows that
$$
H^+ (\mu) \doteq H^+ (\Delta(\mu)).
$$
From the definition of $H^+$ it is obvious that $H^+$ satisfies the commutation constraints induced by $\Re_\Upsilon$. Thus, $H^+$ is a solution of the generalised equation $\Omega$.

Denote by $\delta_{ie_0}$ an element from the group of automorphisms $\AA(\Omega)$ defined in the following way. For every $e_i\in T\setminus T_0$, $i=1,\dots,m$ set
$$
\delta_{ie_0}: h(e_j) \mapsto
\left\{
  \begin{array}{ll}
    h(\p(v_\Gamma, v_i)^{-1} \cc_{e_0} \p(v_\Gamma, v_i))h(e_i), & \hbox{for } j=i; \\
    h(e_j), & \hbox{for } j\ne i.
  \end{array}
\right.
$$
Therefore, if $\pi_{{H}'} = \delta_{ie_0}\pi_{H}$ and $h(e_i) = h_k \in {\P}$, then $H_k' =P_2^{(k)} P^{n_k + \e(H(\cc_{{e}_0}))} P_1^{(k+1)}$, and all the other components of $H'$ are the same as in $H$, $H'_l=H_l$, $l\ne k$. Denote by $\delta = \prod\limits_{i=1}^{m} \delta_{ie_0}^{\Delta_i}$, where $h(e_i)=h_{k_i}$, $\Delta_i = (m^+_{k_i} - m_{k_i}) \cdot {\rm \sign}(\e(H(\cc_{{e}_0})))$. Let us
verify the equality
\begin{equation} \label{2.69}
\pi_{{H}^+} = \pi_{{H}} \delta.
\end{equation}

Let $\pi_{{H}^{(1)}} =\delta\pi_{{H}}$. Then, by construction, $H_k^{(1)} = P_2^{(k)} P^{m_k^+} P_1^{(k+1)} = H_k^+$ for all $k$ such that $h_k$ is the label of an edge from $T \setminus T_0$. If the edge labelled by $h_k$ belongs to $T_0$, or $h_k$ does not belong to a closed section from ${\P}$, then $h_k \not \in {\P}$ and $H_k^{(1)}= H_k = H_k^+$. Finally, for every $e \not \in T$ we have $H^{(1)}(\cc_e) = H(\cc_e) = H^+(\cc_e)$. As $\cc_e = \p_1 e \p_2$, where $\p_1$, $\p_2$ are paths in the tree $T$, and for every item $h_k$ which labels an edge from $T$, the equality $H_k^{(1)} = H_k^+$ has already been established, so it follows that $H^{(1)}(e) = H^+(e)$. This proves (\ref{2.69}).

Since $H^+$ is a solution of $\Omega$, since by construction $H^+$ satisfies conditions (\ref{it:minsol3}) and (\ref{it:minsol4}) from Definition \ref{defn:sol<} and by (\ref{2.69}), it follows that ${H}^+ <_{\AA(\Omega)} {H}$. We arrive to a contradiction with the minimality of the solution ${H}$. Consequently, the solution $\{m_k \mid h_k \in {\P} \}$ of system (\ref{2.67}) is minimal.

Lemma 1.1 from \cite{Makanin} states that the components of the  minimal solution $\{m_k \mid h_k \in {\P} \}$ are bounded by a recursive function of the coefficients, the number of variables and the number of equations of the system. Since, as shown above, all of these parameters of system (\ref{2.67}) are bounded above by a computable function, the lemma follows.
\end{proof}

\section{The finite tree $T_0(\Omega )$ and minimal solutions} \label{se:5.3}

In Section \ref{se:5} we constructed an infinite tree $T(\Omega)$. Though for every solution $H$ the path in $T(\Omega)$
$$
(\Omega,H)=(\Omega_{v_0},H^{(0)})\to(\Omega_{v_1},H^{(1)})\to\dots \to (\Omega_{v_t},H^{(t)})
$$
defined by the solution $H$ is finite, in general, there is no global bound for the length of these paths.

Informally, the aim of this section is to prove that the set of solutions of the generalised equation $\Omega$ can be parametrised by a finite subtree $T_0$ (to be constructed below) of the tree $T$,  automorphisms from a recursive group of automorphisms of the coordinate group $\GG_{R(\Omega^*)}$, and solutions of the generalised equations associated to leaves of $T_0$. In other words, we prove that there exists a global bound  $M$, such that for any solution $H$ of $\Omega$ one can effectively construct a path \index{path!$\p(H)$}$\p(H)=(\Omega_{v_0},H^{[0]})\to(\Omega_{v_1},H^{[1]})\to\dots$ in $T(\Omega)$ of length bounded above by $M$ and such that $H^{[i]}<_{\Aut(\Omega)} H$ for all $i$, where ${\Aut(\Omega)}$  is the group of automorphisms of $\GG_{R(\Omega^*)}$ defined in Section \ref{5.5.3}, see  Definition \ref{defn:Aut} (note the abuse of notation: $H^{[i]}$ and $H$ are solutions of different generalised equations; in fact this relation is between two solutions of $\Omega$: a solution induced by $H^{[i]}$ and $H$; see below for a formal definition).

We summarise the results of this section in the proposition below.

\begin{prop}\label{prop:sum5}
For a {\rm(}constrained\/{\rm)} generalised equation $\Omega=\Omega_{v_0}$, one can effectively construct a \emph{finite} oriented rooted at $v_0$ tree $T_0$, $T_0=T_0(\Omega_{v_0})$, such that:
\begin{enumerate}
\item The tree $T_0$ is a subtree of the tree $T(\Omega)$.
\item To the root $v_0$ of $T_0$ we assign a recursive group of automorphisms $\Aut(\Omega)$ {\rm (}see {\rm Definition \ref{defn:Aut})}.
\item For any solution $H$ of a generalised equation $\Omega $ there exists a leaf $w$ of the tree $T_0(\Omega )$, $\tp(w)=1,2$, and a solution $H^{[w]}$ of the generalised equation $\Omega _w$ such that
\begin{itemize}
    \item  ${H}^{[w]}<_{\Aut(\Omega)} H$;
    \item  if $\tp(w)=2$ and the generalised equation $\Omega_{w}$ contains non-constant non-active sections, then there exists a period $P$ such that $H^{[w]}$ is periodic with respect to the period $P$ and the generalised equation $\Omega _w$ is either singular of strongly singular with respect to the periodic structure ${\P}(H^{[w]},P)$.
\end{itemize}
\end{enumerate}
\end{prop}

This section is organised in three subsections.  The aim of Section \ref{5.5.2} is to define the recursive group of automorphisms $\VV(\Omega_v)$ of the coordinate group of the generalised equation $\Omega_v$ associated to $v$, $v\in T(\Omega)$.

In Section \ref{5.5.3}, we define a finite subtree $T_0(\Omega)$ of the tree $T(\Omega)$. In order to define $T_0(\Omega)$ we introduce the notions of prohibited paths of type 7-10, 12 and 15.  Using Lemma \ref{3.2}, we prove that any infinite branch of the tree $T(\Omega)$ contains a prohibited path. The tree $T_0(\Omega)$ is defined to be a subtree of $T(\Omega)$ that does not contain prohibited paths and, by construction, is finite. We then define a recursive group of  automorphisms $\Aut(\Omega)$ that we assign to the root vertex of the tree $T_0(\Omega)$. The group $\Aut(\Omega)$ is generated by conjugates of the groups $\VV(\Omega_v)$, $v\in T_0(\Omega)$, $\tp(v)\ne 1$.

Finally, in Section \ref{5.5.4} for any solution $H$ of $\Omega$ we construct the path $\p(H)$ from the root $v_0$ to $w$, prove that $w$ is a leaf of the tree $T_0(\Omega)$ and show that the leaf $w$ satisfies the properties required in Proposition \ref{prop:sum5}.

\subsection{Automorphisms}\label{5.5.2}

One of the features of our results in this paper is that all the constructions we give are effective. Since, we do not have control over the automorphism groups of $\GG$-residually $\GG$ groups, we are led to narrow the group of automorphisms in such a way that on one hand they are effectively described and, on the other hand, Proposition \ref{prop:sum5} stated above still holds.

Informally, we will see that the automorphisms that we consider (besides the ones already defined in Section \ref{sec:periodstr} for periodic structures) are, in a sense, induced by the canonical epimorphisms in the infinite branches of the tree $T$. Therefore, these automorphisms are compositions of epimorphisms induced by elementary and derived transformations. This will allow us to prove that the automorphisms are completely induced: come from word mappings that, in turn, induce automorphisms of residually free groups considered by Razborov in \cite{Razborov3}. Since the group of automorphisms considered by Razborov are tame and finitely generated, using the decidability of the universal theory of $\GG$, we will get that our groups of automorphisms are recursive.

In his work on systems of equations over a free group \cite{Razborov1}, \cite{Razborov3}, to prove that the groups of automorphisms are finitely generated, Razborov uses a result of McCool, see {\rm \cite{McCool}}. When this paper was already written M.~Day published two preprints, \cite{Day1} and \cite{Day2} on automorphisms of partially commutative groups. The authors believe that using ideas of A.~Razborov and techniques developed by M.~Day, one can prove that the automorphisms groups of the coordinate groups we consider are, in fact, finitely generated.

\begin{defn}\label{defn:complind}
Let $\Omega=\gpo$ be a generalised equation. An automorphism $\theta$ of the coordinate group $\GG_{R(\Omega^*)}$ is called \index{automorphism!completely induced}\emph{completely induced} if there exist an $F$-homomorphism $\tilde \theta$ from $F[h]$ to $F[h]$, where $F$ is the free group on $\cA$ and an $F$-automorphism $\theta'$ of the coordinate group $F_{R(\Upsilon^*)}$, such that the following diagram commutes:
$$
\CD
 \GG_{R(\Omega^*)} @<<<  F[h]           @>>> F_{R(\Upsilon^*)} \\
 @V\theta VV       @V \tilde\theta VV   @VV \theta' V  \\
 \GG_{R(\Omega^*)} @<<<   F[h]           @>>> F_{R(\Upsilon^*)}
\endCD
$$

In this case, the automorphisms $\theta$ and $\theta'$ are called \index{automorphism!dual}\emph{dual}.
\end{defn}

To every vertex $v$ of the tree $T(\Omega)$, $\Omega=\gpo$ we assign a recursive group of automorphisms \glossary{name={$\VV(\Omega_v)$}, description={a recursive group of automorphisms of $\GG_{R(\Omega_v^\ast)}$ that we associate to a vertex $v$ of the tree $T(\Omega)$}, sort=V}$\VV(\Omega_v)$ of $\GG_{R(\Omega_v^\ast)}$.

If $\tp (v)=2$, set  $\VV(\Omega_v)$ to be the group generated by all the  groups of automorphisms $\AA(\Omega_v)$ corresponding to \emph{regular} periodic structures on $\Omega_v$, see Definition \ref{defn:AA}.

\medskip

\begin{rem}
The definition of the automorphism groups $\VV(\Omega_v)$ in the case when $7\le \tp(v)\le 10$ (or, $\tp(v)=12$) given below may seem unnatural. In fact, for the purposes of this paper, the automorphism group $\VV(\Omega_v)$ could be defined as the group of automorphisms of $\GG_{R(\Omega_v^*)}$ that are induced by $F$-\emph{homomorphisms} $\varphi$ of the free group $F[h]$ such that $\varphi(h_i)=h_i$ for all $h_i$ that  belong to the kernel of $\Omega_v$ (or, such that $\varphi(h_i)=h_i$ for all $h_i$ that belong to the non-quadratic part of $\Omega_v$, see below for definition). If defined in this way, the groups of automorphisms $\VV(\Omega_v)$ are recursive and all the proofs in the paper work in the same way.

However, we chose to define the groups $\VV(\Omega_v)$ as below because the groups of automorphisms defined by Razborov are well understood (in the case that the corresponding coordinate group is fully residually free they correspond to the canonical group of automorphisms associated to its JSJ decomposition, see \cite{KMIrc} and \cite{EJSJ}) and we hope that establishing relations between $\VV(\Omega_v)$ and the groups defined by Razborov will help understanding the structure of $\VV(\Omega_v)$ and, consequently, of the splittings of fully residually $\GG$ groups.

For the time being, the only consequences of our definition are that the group of automorphisms $\VV(\Omega_v)$ is ``induced'' by a subgroup of the automorphism groups defined by Razborov and that automorphisms from $\VV(\Omega_v)$ are tame. Recall that an automorphism of a group is \index{automorphism!tame}\emph{tame} if it is induced by an automorphism of the free group.
\end{rem}

Let $7\leq \tp (v)\leq 10$. An automorphism $\varphi$ of the coordinate group $\GG_{R({\Omega_v}^*)}$ is called \index{automorphism!invariant with respect to the kernel}\emph{invariant with respect to the kernel} if it is completely induced (see Definition \ref{defn:complind}) and its dual is invariant with respect to the kernel in the sense of Razborov, \cite{Razborov3}. We now recall the definition of an automorphism invariant with respect to the kernel in the sense of Razborov.

Let $\Upsilon_v$ be a non-constrained generalised equation over the free monoid $\FF$, and $F_{R(\Upsilon_v^*)}$ be its coordinate group over $F$, where $F$ is the free group with basis $\cA$. Let $\widehat{\Upsilon}_v$ be obtained from $\D 3(\Upsilon_v)$ by removing all coefficient equations and all bases from the kernel of $\Upsilon_v$. Let $\pi$ be the  natural homomorphism from $F_{R(\widehat{\Upsilon}_v^*)}$ to $F_{R(\Upsilon_v^*)}$.

An automorphism $\varphi$ of the coordinate group $F_{R({\Upsilon_v}^*)}$ is called \index{automorphism!invariant with respect to the kernel in the sense of Razborov}\emph{invariant with respect to the kernel}, see \cite{Razborov3}, if it is induced by an automorphism $\varphi'$ of the coordinate group $F_{R(\widehat{\Upsilon}_v^*)}$ identical on the kernel of $\Upsilon_v$, i.e. there exists an $F$-automorphism $\varphi'$ of the coordinate group $F_{R(\widehat{\Upsilon}_v^*)}$ so that $\varphi'(h_i)=h_i$ for every variable $h_i\in \Ker(\Upsilon_v)$ and the following diagram commutes
$$
\CD
F_{R(\widehat{\Upsilon}_v^*)}   @>{\pi}>> F_{R({\Upsilon_v}^*)} \\
  @V{\varphi'}VV        @V \varphi VV  \\
  F_{R(\widehat{\Upsilon}_v^*)}  @>{\pi}>> F_{R({\Upsilon_v}^*)}
\endCD
$$

\begin{lem}[Lemma 2.5, \cite{Razborov3}]\label{lem:razker}
There exists a finitely presented subgroup $K$ of the group of automorphisms of a free group such that every automorphism from $K$ induces an automorphism of $\factor{F[h]}{\ncl\langle\Upsilon_v\rangle}$, which, in turn, induces an automorphism of the coordinate group $F_{R({\Upsilon_v}^*)}$. The finitely generated group of all automorphisms $K'$ of $F_{R({\Upsilon_v}^*)}$ induced by automorphisms from $K$ contains all the automorphisms invariant respect to the kernel of $\Upsilon_v$.
\end{lem}

We use the notation of Lemma \ref{lem:razker}. Every automorphism of $\GG_{R(\Omega_v^\ast)}$ invariant with respect to the kernel is completely induced and, by definition, its dual belongs to $K'$. For every automorphism $\varphi$ from $K$, there exists an algorithm to decide whether or not $\varphi$ induces an automorphism of  $\GG_{R(\Omega_v^\ast)}$. Indeed, in order to do so it suffices to check if the following finite family of universal formulas (quasi-identities) holds in $\GG$:
$$
\left\{
\begin{array}{lcl}
\forall H_1,\dots,H_\rho\, \left[\varphi(\Upsilon^*(H))=1\wedge \{\varphi([H_i,H_j])=1\mid \Re_\Upsilon(h_i,h_j)\} \right]&\to&u(H)=1;\\
\forall H_1,\dots,H_\rho\, \left[\varphi(\Upsilon^*(H))=1\wedge \{\varphi([H_i,H_j])=1\mid \Re_\Upsilon(h_i,h_j)\} \right]&\to&[H_i,H_j]=1;\\
\forall H_1,\dots,H_\rho\, \left[\Upsilon^*(H)=1 \wedge \{[H_i,H_j]=1\mid \Re_\Upsilon(h_i,h_j)\}\right]&\to&\varphi(u(H))=1;\\
\forall H_1,\dots,H_\rho\, \left[\Upsilon^*(H)=1 \wedge \{[H_i,H_j]=1\mid \Re_\Upsilon(h_i,h_j)\}\right] &\to&\varphi([H_i,H_j])=1.
\end{array}
\right\}
$$
for every equation $u$ of the system $\Upsilon^*$ and every pair $(H_i, H_j)$ such that $\Re_\Upsilon(h_i,h_j)$. This can be done effectively, since the universal theory of the group $\GG$ is decidable, see \cite{DL}. We therefore get the following lemma.

\begin{lem} \label{lem:autinvkerrec}
The group of all automorphisms of $\GG_{R(\Omega_v^\ast)}$ invariant with respect to the kernel of $\Omega_v$ is tame and recursive.
\end{lem}

In this case, we define $\VV(\Omega_v)$ to be the group of automorphisms of $\GG_{R(\Omega_v^\ast)}$ invariant with respect to the kernel.

\medskip

Let $\tp (v)=15$. Apply derived transformation $\D 3$ and consider the generalised equation $\D 3(\Omega_v)=\widetilde{\Omega} _v$. Notice that, since every boundary in the active part of $\widetilde{\Omega} _v$ that touches a base and intersects another base $\mu$ is $\mu$-tied (i.e. assumptions of Case 14 do not hold), the function $\gamma$ is constant on closed sections of $\widetilde{\Omega}_v$, i.e. $\gamma(h_i)=\gamma(h_j)$ whenever $h_i$ and $h_j$ belong to the same closed section of $\widetilde{\Omega}_v$. Applying $\D 2$, we can assume that every item in the section $[1,j+1]$ is covered  exactly twice (i.e. $\gamma(h_i)=2$ for every $i=1,\dots,j$) and for all $k\ge j+1$ we have $\gamma(h_k)>2$. In this case we call the section $[1,j+1]$ \index{part of a generalised equation!quadratic}\emph{the quadratic part of $\Omega_v$}. We sometimes refer to the set of non-quadratic sections of the generalised equation $\Omega_v$ as to the \index{part of a generalised equation!non-quadratic}\emph{non-quadratic part} of $\Omega_v$.

Let $\tp (v)=12$. Then the \emph{quadratic part of $\Omega_v$} is the whole active part of $\Omega_v$.

A variable base $\mu$ of the generalised equation $\Omega$ is called a \index{base!quadratic}{\em quadratic base} if $\mu $ and its dual $\Delta(\mu)$ both belong to the quadratic part of $\Omega$. A base $\mu$ of the generalised equation $\Omega$ is called a \index{base!quadratic-coefficient}\emph{quadratic-coefficient base} if $\mu$ belongs to the quadratic part of $\Omega$ and its dual $\Delta(\mu)$ does not belong to the quadratic part of $\Omega$.

Let $\tp (v)=12$ or $\tp(v)=15$ and let $[1,j+1]$ be the quadratic part of $\Omega_v$. An automorphism $\varphi$ of the coordinate group $\GG_{R({\Omega_v}^*)}$ is called \index{automorphism!invariant with respect to the non-quadratic part}\emph{invariant with respect to the non-quadratic part} if it is completely induced (see Definition \ref{defn:complind}) and its dual is invariant with respect to the non-quadratic part in the sense of Razborov, \cite{Razborov3}. We now recall the definition of an automorphism invariant with respect to the non-quadratic part in the sense of Razborov.

Let $\Upsilon_v$ be a non-constrained generalised equation over the free monoid $\FF$, and $F_{R(\Upsilon_v^*)}$ be its coordinate group over $F$, where $F$ is the free group with basis $\cA$. Denote by $\widehat{\Upsilon}_v$ the generalised equation obtained from $\Upsilon_v$ by removing all non-quadratic bases, all quadratic-coefficient bases and all coefficient equations.  Let $\phi$ be the  natural homomorphism from $\factor{F[h]}{\ncl\langle\widehat{\Upsilon}_v\rangle}$ to $F_{R(\Upsilon_v^*)}$.

An automorphism $\varphi$ of the coordinate group $F_{R(\Upsilon_v^*)}$ is called \index{automorphism!invariant with respect to the non-quadratic part in the sense of Razborov}\emph{invariant with respect to the non-quadratic part}, see \cite{Razborov3}, if there exists an automorphism $\hat{\varphi}$ of the group $\factor{F[h]}{\ncl(\widehat{\Upsilon}_v)}$ so that for every quadratic-coefficient base $\nu$ of $\Upsilon_v$ one has $\hat{\varphi}(h(\nu))=h(\nu)$, for every $h_i$, $i=j+1,\dots, \rho_{\Upsilon_v}$ (recall that $[1,j+1]$ is the quadratic part of $\Upsilon_v$) one has $\hat{\varphi}(h_i)=h_i$, and the following diagram commutes:
$$
\CD
  \factor{F[h]}{\ncl\langle\widehat{\Upsilon}_v\rangle} @>\phi>> F_{R(\Upsilon_v^*)} \\
  @V \hat{\varphi} VV                               @VV\varphi V  \\
  \factor{F[h]}{\ncl\langle\widehat{\Upsilon}_v\rangle} @>>\phi> F_{R(\Upsilon_v^*)}
\endCD
$$

\begin{lem}[Lemma 2.7, \cite{Razborov3}]\label{lem:raznqp}
There exists a finitely presented subgroup $K$ of the group of automorphisms of a free group such that every automorphism from $K$ induces an automorphism of $\factor{F[h]}{\ncl\langle\Upsilon_v\rangle}$, which, in turn, induces an automorphism of the coordinate group $F_{R({\Upsilon_v}^*)}$. The finitely generated group $K'$ of all automorphisms of $F_{R({\Upsilon_v}^*)}$ induced by automorphisms from $K$ contains all the automorphisms invariant with respect to the non-quadratic part of $\Upsilon_v$.
\end{lem}

We use the notation of Lemma \ref{lem:raznqp}. Every automorphism of $\GG_{R(\Omega_v^\ast)}$ invariant with respect to the non-quadratic part is completely induced and, by definition, its dual belongs to $K'$.  An argument analogous to the one in the case of automorphisms invariant with respect to the kernel shows that for every automorphism $\varphi$ from $K$, there exists an algorithm to decide whether or not $\varphi$ induces an automorphism of $\GG_{R(\Omega_v^\ast)}$. We thereby get the following lemma.

\begin{lem}
The group of all automorphisms of $\GG_{R(\Omega_v^\ast)}$ invariant with respect to the non-quadratic part of $\Omega_v$ is tame and recursive.
\end{lem}

In this case, set $\VV(\Omega_v)$ to be the group of automorphisms of $\GG_{R(\Omega_v^\ast)}$ invariant with respect to the non-quadratic part.

In all other cases set $\VV(\Omega_v)=1$.

\subsection{The finite subtree $T_0(\Omega )$}\label{5.5.3}

As mentioned above, the aim of this section is to construct the finite subtree \glossary{name={$T_0(\Omega)$}, description={the finite subtree of $T(\Omega)$ that does not contain prohibited paths}, sort=T}$T_0(\Omega)$ of $T(\Omega)$ as the subtree that does not contain prohibited paths.

The definition of a prohibited path is designed in such a way that the paths $\p(H)$ in the tree $T$ associated to the solution $H$ do not contain them. Therefore, the nature of the definition of a prohibited path will become clearer in Section \ref{5.5.4}.

\bigskip

By Lemma \ref{3.2}, infinite branches of the tree $T(\Omega)$ correspond to the following cases: $7\leq \tp(v_k)\leq 10$ for all $k$, or $\tp(v_k)=12$ for all $k$, or $\tp(v_k)=15$ for all $k$. We now define prohibited paths of types 7-10, 12 and 15.

\begin{defn}
We call a path $v_1\to v_2 \to\ldots\to v_k$ in $T(\Omega )$ \index{path!prohibited of type 7-10}{\em prohibited of type 7-10} if $7\leq \tp(v_i)\leq 10$  for all $i=1,\dots,k$ and some generalised equation with $\rho$ variables occurs among $\{\Omega_{v_i}\mid 1\le i\le l\}$ at least $2^{4\rho^2\cdot (2^\rho+1)}+1$ times.

Similarly, a path $v_1\to v_2 \to\ldots\to v_k$ in $T(\Omega )$ is called \index{path!prohibited of type 12}{\em prohibited of type 12} if $\tp(v_i)=12$ for all $i=1,\dots,k$  and some generalised equation with $\rho$ variables occurs among $\{\Omega_{v_i}\mid 1\le i\le l\}$ at least $2^{4\rho^2\cdot (2^\rho+1)}+1$ times.
\end{defn}

\bigskip

We now prove that an infinite branch of $T(\Omega)$  of type 7-10 or 12 contains a prohibited path of type 7-10 or 12, correspondingly.

\begin{lem} \label{3.3}
Let $v_0\to v_1 \to \ldots\to v_n\to\ldots $ be an infinite path in the tree $T(\Omega)$, where $7\leq \tp(v_i) \leq 10$ for all $i$, and let $\Omega_{v_0}, \Omega_{v_1}, \dots, \Omega_{v_n}, \dots$ be the sequence of corresponding generalised equations. Then among $\{\Omega_{v_i}\}$ some generalised equation occurs infinitely many times. Furthermore, if $\Omega_{v_k}=\Omega _{v_l}$, then $\pi(v_k,v_l)$ is a $\GG$-automorphism of $\GG_{R(\Omega_{v_k}^*)}$ invariant with respect to the kernel of $\Omega_{v_k}$.
\end{lem}
\begin{proof}
By Lemma \ref{3.1}, we have that $\comp(\Omega_{v_k})\leq \comp (\Omega_{v_0})$ and $\xi(\Omega_{v_k})\leq \xi(\Omega_{v_0})$ for all $k$. We, therefore, may assume that $\comp=\comp(\Omega_{v_k})=\comp(\Omega_{v_0})$ and $\xi(\Omega_{v_k})=\xi(\Omega_{v_0})$ for all $k$. It follows that all the transformations $\ET 5$ introduce a new boundary.

For all $k$, the generalised equations $\Ker (\widetilde{\Upsilon}_{v_k})$ have the same set of bases, recall that $\widetilde{\Upsilon}_{v_k}=\D 3(\widetilde{\Upsilon}_{v_k})$. Indeed, consider the generalised equations $\widetilde{\Upsilon} _{v_k}$ and $\widetilde{\Upsilon}_{v_{k+1}}$. Since $\tp(v_k)\ne 3,4$, the active part of $\widetilde{\Upsilon}_{v_k}$ does not contain constant bases.

If $\tp(v_k)=7,8,10$, then $\widetilde{\Upsilon}_{v_{k+1}}$ is obtained from $\widetilde{\Upsilon} _{v_k}$ by cutting some  base $\mu$ which is eliminable in $\widetilde{\Upsilon} _{v_k}$  and then deleting one of the new bases, which is also eliminable, since it falls under the assumption a) of the definition of an eliminable base. Since every transformation $\ET 5$ introduces a new boundary, the remaining part of the base $\mu$ falls under the assumption b) of the definition of an eliminable base. Therefore, in this case, the set of bases that belong to the kernel does not change.

Let $\tp(v_k)=9$. It suffices to show that, in the notation of Case 9, all the bases of $\widetilde{\Upsilon}_{v_{k+1}}$ obtained by cutting the base $\mu_2$ do not belong to the kernel. Without loss of generality we may assume that $\sigma(\mu_2)$ is a closed section of $\widetilde{\Upsilon}_{v_k}$. Indeed, if $\sigma(\mu_2)$ is not closed, instead, we can consider one of its closed subsections $\sigma'$ in the generalised equation $\widetilde{\Upsilon}_{v_k}$.

Notice that, since $\sigma(\mu_2)$ is closed, every boundary that intersects $\mu_1$ and $\mu_2$ in $\Upsilon_{v_k}$, touches exactly two bases in $\widetilde{\Upsilon}_{v_k}$. Thus, for every boundary connection $(p,\mu_2,q)$ in $\Upsilon_{v_{k+1}}$ either the boundary $p$ or the boundary $q$ touches exactly two bases in $\widetilde{\Upsilon}_{v_{k+1}}$.

Construct an elimination process (see description of the derived transformation $\D 4$) for the generalised equation $\widetilde{\Upsilon}_{v_k}$ and take the first generalised equation $\Upsilon_i$ in this elimination process, such that the base $\nu$ eliminated in this equation was obtained from either $\mu_1$ or $\mu_2$ or $\Delta(\mu _1)$ or $\Delta(\mu _2)$ by applying $\D 3$ to ${\Upsilon}_{v_k}$. The base $\nu$ could not be obtained from $\mu _1$ or $\mu _2$, since every item in the section $\sigma(\mu_2)$ is covered twice and every boundary in this section touches two bases.

If $\nu$ falls under the assumption of case b) of the definition of an eliminable base, then
$$
\hbox{either }\alpha (\nu)\in\{\alpha(\Delta (\mu _1)), \alpha (\Delta (\mu _2))\}\hbox{ or }\beta (\nu)\in\{\beta (\Delta (\mu _1)), \beta(\Delta (\mu _2))\}.
$$
We now construct an elimination process for the generalised equation $\widetilde{\Upsilon}_{v_{k+1}}$. The first $i$ steps of the elimination process for $\widetilde{\Upsilon}_{v_{k+1}}$ coincide with the first $i$ steps of the elimination process constructed for the generalised equation $\widetilde{\Upsilon}_{v_k}$. Then the eliminable base $\nu$ of $\Upsilon_{i}$ corresponds to a base $\nu '$ obtained from either $\mu_2$ or $\Delta(\mu _2)$ by applying $\D 3$ to ${\Upsilon}_{v_{k+1}}$. The base $\nu'$ is eliminable.

Notice that one of the boundaries $\alpha(\nu)$, $\beta(\nu)$, $\alpha(\Delta(\nu))$ or $\alpha(\Delta(\nu))$ touches just two bases. Therefore, after eliminating $\nu'$, this boundary touches just one base $\eta$ that was obtained from $\mu_2$ or $\Delta(\mu_2)$. The base $\eta$ falls under the assumptions b) of the definition of an eliminable base. Repeating this argument, one can subsequently eliminate all the other bases obtained from $\mu _2$ or $\Delta(\mu_2)$. It follows that all the bases of the generalised equation $\widetilde{\Upsilon} _{v_{k+1}}$, obtained from $\mu _2$ or $\Delta(\mu_2)$ do not belong to the kernel.

We thereby have shown that the set of bases is the same for all the generalised equations $\Ker(\widetilde{\Upsilon} _{v_k})$ and thus the set of bases is the same for all the generalised equations $\Ker(\widetilde{\Omega} _{v_k})$. We denote the cardinality of this set by $n'$.

We now prove that the number of bases in the active sections of $\Omega_{v_{k}}$, for all $k$, is bounded above by a function of $\Omega_{v_0}$:
\begin{equation}\label{3.2'}
n_A(\Omega_{v_{k}})\leq 3\comp+6n'+1.
\end{equation}
Indeed, assume the contrary and let $k$ be minimal for which inequality (\ref{3.2'}) fails. Then
\begin{equation}\label{3.3'}
n_A(\Omega_{v_{k-1}})\leq 3\comp +6n'+1, n_A(\Omega_{v_{k}})>3\comp+6n'+1.
\end{equation}
By Lemma \ref{3.1}, $\tp(v_{k-1})=10$. It follows that $\tp(v_{k-1})\ne 5,6,7,8,9$. Therefore, every active section of $\Omega _{v_{k-1}}$ either contains at least three bases or contains some base of the generalised equation $\Ker (\widetilde{\Omega} _{v_{k-1}})$. Let $u_{k-1}$ and $w_{k-1}$ be the number of active sections of $\Omega_{v_{k-1}}$ that contain one base and more than one base, respectively. Hence $u_{k-1}+w_{k-1}\leq\frac{1}{3}n_A(\Omega_{v_{k-1}})+n'$. It is easy to see that
$$
\comp = \sum_{\sigma \in A\Sigma(\Omega)} \max\{0, n(\sigma)-2\}= n_A(\Omega_{v_{k-1}})-2w_{k-1}-u_{k-1}.
$$
Then, $\comp\ge n_A(\Omega_{v_{k-1}})-2(w_{k-1}+u_{k-1})\geq\frac{1}{3}n_A(\Omega_{v_{k-1}})-2n'$, which contradicts  (\ref{3.3'}).

Furthermore, the number $\rho_A(\Omega_{v_k})$ of items in the active part of  $\Omega_{v_k}$ is bounded above:
$$
\rho_A(\Omega_{v_k})\leq \xi(\Omega_{v_k})+n_A(\Omega_{v_{k-1}})+1\leq 3\comp +6n'+\xi(\Omega_{v_0}) +2.
$$
Since the number of bases and number of items is bounded above, the set $\{\Omega _{v_k}\mid k\in\mathbb{N}\}$ is finite and thus some generalised equation occurs in this set infinitely many times.

Let $\Omega _{v_k}=\Omega _{v_l}$. Since, by assumption, the edges $v_j\to v_{j+1}$, $j=k,\dots, l-1$ are labelled by isomorphisms, the homomorphism $\pi (v_k,v_l)$ is an automorphism of the coordinate group $\GG_{R(\Omega_{v_k}^*)}$.

By Lemma \ref{lem:indETepi} and Lemma \ref{lem:indDepi}, there exists an epimorphism $\pi (v_k,v_l)'$  from $F_{R(\Upsilon_{v_k}^*)}$ to $F_{R(\Upsilon_{v_l}^*)}$. Since $F_{R(\Upsilon_{v_k}^*)}$ is a finitely generated residually free group and therefore is residually finite, by a theorem of Mal'cev, $F_{R(\Upsilon_{v_k}^*)}$ is Hopfian. Thus the epimorphism $\pi (v_k,v_l)'$ is an automorphism of $F_{R(\Upsilon_{v_k}^*)}$. This shows that the automorphism $\pi(v_k,v_l)$ is completely induced and $\pi (v_k,v_l)'$ is its dual.  We are left to show that $\pi (v_k,v_l)'$ is invariant with respect to the kernel (in the sense of Razborov, \cite{Razborov3}).

As shown above,
$$
\Ker (\widetilde{\Upsilon}_{v_{k}})=\Ker (\widetilde{\Upsilon}_{v_{k+1}})=\dots=\Ker (\widetilde{\Upsilon} _{v_{l}}).
$$
From the above, it follows that $\Ker(\widetilde{\Upsilon}_{v_{i+1}})$ is obtained from  $\Ker(\widetilde{\Upsilon}_{v_{i}})$ by introducing new boundaries and removing some of the items that do not belong to the kernel of $\widetilde{\Upsilon}_{v_{i+1}}$. Therefore, the number of items that belong to the kernel $\Ker(\widetilde{\Upsilon}_{v_{i+1}})$ can only increase. As $\Omega_{v_k}=\Omega_{v_l}$, so this number is the same for all $i$, $i=k,\dots, l-1$. It follows that $\pi (v_k,v_l)'(h_i)=h_i$ for all $h_i$ that belong to the kernel of $\Upsilon_{v_k}$.

Since the transformations that take the generalised equation $\widetilde{\Upsilon}_{v_{k}}$ to $\widetilde{\Upsilon}_{v_{k+1}}$ do not involve bases that belong to the kernel of $\widetilde{\Upsilon}_{v_{k}}$, the same sequence of transformations can be applied to the generalised equation $\widehat{\widetilde{\Upsilon}}_{v_{k}}$, where $\widehat{\widetilde{\Upsilon}}_{v_{k}}$ is obtained from $\widetilde{\Upsilon}_{v_{k}}$ by removing all coefficient equations and all bases that belong to the kernel of $\widetilde{\Upsilon}_{v_k}$.

Since, by assumption, every time we $\mu$-tie a boundary a new boundary is introduced, we get that the epimorphism from $F_{R(\widehat{\widetilde{\Upsilon}}_{v_{k}}^*)}$ to $F_{R(\widehat{\widetilde{\Upsilon}}_{v_{k+1}}^*)}$ is, in fact, an isomorphism. We therefore get the following commutative diagram (see the definition of an automorphism invariant with respect to the kernel, Section \ref{5.5.2}):
$$
\CD
F_{R(\widehat{\Upsilon}_{v_k}^*)}   @>{\pi}>> F_{R({\Upsilon_{v_k}}^*)} \\
  @VVV        @VV\pi (v_k,v_l)'V  \\
  F_{R({\widehat{\Upsilon}_{v_{l}}}^*)}  @>{\pi}>> F_{R({\Upsilon_{v_{l}}}^*)}
\endCD
$$
It follows that the automorphism $\pi (v_k,v_l)'$ is invariant with respect to the kernel in the sense of Razborov and thus the automorphism $\pi (v_k,v_l)$ is invariant with respect to the kernel of ${\Omega}_{v_{k}}$.
\end{proof}

\begin{cor} \label{cor:proh710}
Let $\p=v_1 \to \ldots\to v_n\to\ldots $ be an infinite path in the tree $T(\Omega)$, and $7\le \tp(v_i)\le 10$ for all $i$. Then
$\p$ contains a prohibited subpath of type 7-10.
\end{cor}

\begin{lem}  \label{lem:c12}
Let $v_0\to v_1 \to \ldots\to v_n\to\ldots $ be an infinite path in the tree $T(\Omega)$, where $\tp(v_i) = 12$ for all $i$, and $\Omega_{v_0}, \Omega_{v_1}, \dots, \Omega_{v_n}, \dots$ be the sequence of corresponding generalised equations. Then among $\{\Omega_{v_i}\}$ some generalised equation occurs infinitely many times. Furthermore, if $\Omega_{v_k}=\Omega _{v_l}$, then $\pi(v_k,v_l)$ is a $\GG$-automorphism of the coordinate group $\GG_{R(\Omega_{v_k}^*)}$ invariant with respect to the non-quadratic part.
\end{lem}
\begin{proof}
Notice that since $\Omega _{v_i}$ is a quadratic generalised equation, quadratic-coefficient bases of $\Omega _{v_i}$ are bases whose duals belong to the non-active part.

Let $\mu_{i}$ be the carrier base of the generalised equation $\Omega _{v_i}$. Consider the sequence $\mu_0,\dots,\mu_i,\dots$ of carrier bases.
By Lemma \ref{3.1}, if $\tp({v_{i}})=12$, then $n_A(\Omega_{v_{i+1}})\le n_A(\Omega_{v_{i}})$. Furthermore, if the carrier base $\mu_i$ is quadratic-coefficient, then this inequality is strict. Hence, it suffices to consider the case when all carrier bases are quadratic.

The number of consecutive quadratic bases in the sequence $\mu_1,\dots,\mu_i,\ldots$ is bounded above. Indeed, by Lemma \ref{3.1}, when the entire transformation is applied, the complexity of the generalised equation does not increase. Furthermore, since the generalised equation is quadratic and does not contain free boundaries, the number of items does not increase. The number of constrained generalised equations with a bounded number of items and bounded complexity is finite and thus some generalised equation occurs in the sequence $\{\Omega_{v_i}\}$ infinitely many times.

Obviously, if $\Omega _{v_k}=\Omega _{v_l}$, then $\pi(v_k,v_l)$ is a $\GG$-automorphism of $\GG_{R(\Omega_{v_k}^*)}$.

By Lemma \ref{lem:indETepi} and Lemma \ref{lem:indDepi}, there exists an epimorphism $\pi (v_k,v_l)'$  from $F_{R(\Upsilon_{v_k}^*)}$ to $F_{R(\Upsilon_{v_l}^*)}$. Since $F_{R(\Upsilon_{v_k}^*)}$ is a finitely generated residually free group and therefore is residually finite, by a theorem of Mal'cev, $F_{R(\Upsilon_{v_k}^*)}$ is Hopfian. Thus the epimorphism $\pi (v_k,v_l)'$ is an automorphism of $F_{R(\Upsilon_{v_k}^*)}$. This shows that the automorphism $\pi(v_k,v_l)$ is completely induced and $\pi (v_k,v_l)'$ is its dual.  We are left to show that $\pi (v_k,v_l)'$ is invariant with respect to the non-quadratic part (in the sense of Razborov, \cite{Razborov3}).

From the definition of the entire transformation $\D 5$, it follows that  the number of items that belong to a given quadratic-coefficient base can only increase. As $\Omega_{v_k}=\Omega_{v_l}$, so this number is the same for all $i$, $i=k,\dots, l$. It follows that $\pi (v_k,v_l)'(h_i)=h_i$ for all $h_i$ that belong to a quadratic-coefficient base.

Since the transformations that take the generalised equation $\widetilde{\Upsilon}_{v_{k}}$ to $\widetilde{\Upsilon}_{v_{k+1}}$ involve only quadratic bases of $\widetilde{\Upsilon}_{v_{k}}$, the same sequence of transformations can be applied to the generalised equation $\widehat{\widetilde{\Upsilon}}_{v_{k}}$, where $\widehat{\widetilde{\Upsilon}}_{v_{k}}$ is the generalised equation obtained from $\widetilde{\Upsilon}_{v_k}$ by removing all non-quadratic bases, all quadratic-coefficient bases and all coefficient equations.

Since, by assumption, every time we $\mu$-tie a boundary a new boundary is introduced, we get that the epimorphism from $F_{R(\widehat{\widetilde{\Upsilon}}_{v_{k}}^*)}$ to $F_{R(\widehat{\widetilde{\Upsilon}}_{v_{k+1}}^*)}$ is, in fact, an isomorphism. We therefore get the following commutative diagram (see the definition of an automorphism invariant with respect to the non-quadratic part, Section \ref{5.5.2}):
$$
\CD
\factor{F[h]}{\ncl\langle\widehat{\Upsilon}_{v_k}\rangle}   @>{\phi}>> F_{R({\Upsilon_{v_k}}^*)} \\
  @VVV        @VV\pi (v_k,v_l)'V  \\
  \factor{F[h]}{\ncl\langle\widehat{\Upsilon}_{v_l}\rangle}  @>{\phi}>> F_{R({\Upsilon_{v_{l}}}^*)}
\endCD
$$
It follows that the automorphism $\pi (v_k,v_l)'$ is invariant with respect to the non-quadratic part of $\Upsilon_{v_k}$ in the sense of Razborov and thus the automorphism $\pi (v_k,v_l)$ is invariant with respect to the non-quadratic part of ${\Omega}_{v_{k}}$.
\end{proof}

\begin{cor}\label{cor:proh12}
Let $\p=v_1 \to \ldots\to v_n\to\ldots $ be an infinite path in the tree $T(\Omega)$, and $\tp(v_i)=12$ for all $i$. Then $\p$ contains a prohibited subpath of type 12.
\end{cor}

\bigskip

Below we shall define  prohibited paths of type 15 in $T(\Omega )$. In this case, the definition of a prohibited path is much more involved.

We need some auxiliary definitions. Recall that the complexity of a generalised equation $\Omega $ is defined as follows:
$$
\comp = \comp (\Omega) = \sum_{\sigma \in A\Sigma_\Omega} \max\{0, n(\sigma)-2\},
$$
where $n(\sigma)$ is the number of bases in $\sigma$. Let \glossary{name={$\tau_v$, $\tau (\Omega_v)$}, description={function of the generalised equation}, sort=T}$\tau_v=\tau (\Omega_v)=\comp(\Omega_v)+\rho -\rho_{v}'$, where $\rho=\rho_\Omega$ is the number of variables in the initial generalised equation $\Omega $ and \glossary{name={$\rho_v'$}, description={the number of free variables that belong to the non-active sections of  $\Omega _v$}, sort=R}$\rho_{v}'$ is the number of free variables belonging to the non-active sections of the generalised equation $\Omega _v$. We have $\rho _{v}'\le \rho$ (see the proof of Lemma \ref{3.2}), hence $\tau_v\geq 0$. If, in addition,  $v_1\to v_2$ is an auxiliary edge, then $\tau_{v_2}< \tau_{v_1}$.

We use induction on $\tau _v$ to construct a finite subtree $T_0(\Omega _v)$ of $T(\Omega _v)$, and the function $\ss(\Omega_v)$.

The tree $T_0(\Omega _v)$ is a rooted tree at $v$ and consists of some of the vertices and edges of $T(\Omega)$ that lie above $v$.

Suppose that $\tau _v=0$. It follows that $\comp(\Omega_v)=0$. Then in $T(\Omega)$ there are no auxiliary edges. Furthermore, since $\comp(\Omega_v)=0$, it follows that every closed active section contains at most two bases and so no vertex of type 15 lies above $v$. We define the subtree $T_0(\Omega _v)$ as follows. The set of vertices of $T_0(\Omega _v)$ consists of all vertices $v_1$ of $T(\Omega)$ that lie above $v$, and so that the path from $v$ to $v_1$ does not contain prohibited subpaths of types 7-10 and 12. By Corollary \ref{cor:proh710} and Corollary \ref{cor:proh12},  $T_0(\Omega _v)$ is finite.

Let
\glossary{name={$\ss(\Omega_v)$}, description={a function that bounds the length of a minimal solution}, sort=S}
\begin{equation}\label{so}
\ss(\Omega_v)=\max\limits_w \max\limits_{\langle {\P},R\rangle}\rho_{\Omega_w}\cdot\{ f_{0} (\Omega_w,{\P},R)\},
\end{equation}
where $\max\limits_w$ is taken over all the vertices of $T_0(\Omega_v)$ for which $\tp (w)=2$ and $\Omega _w$ contains
non-active sections; $\max\limits_{\langle {\P},R\rangle}$ is taken over all regular periodic structures such that the generalised equation  $\widetilde{\Omega}_w$ is regular with respect to $\langle \P, R\rangle$; and $f_{0}$ is the function from Lemma \ref{lem:23-2}.

Suppose now that $\tau_v> 0$. By induction, we assume that for all $v_1$ such that $\tau _{v_1}< \tau _v$ the finite tree $T_0(\Omega _{v_1})$ and $\ss(\Omega _{v_1})$ are already defined. Furthermore, we assume that the full subtree of $T(\Omega)$ whose set of vertices consists of all vertices that lie above $v$ does not contain prohibited paths of type 7-10 and of type 12. Consider a path $\p$ in $T(\Omega)$:
\begin{equation}\label{3.6}
v_1\rightarrow v_2\rightarrow \ldots\rightarrow v_m,
\end{equation}
where $\tp(v_i)=15$, $1\leq i\leq m$ and all the edges are principal. We have $\tau _{v_i}=\tau _v$.

Denote by $\mu _i$ the carrier base of the generalised equation $\Omega_{v_i}$. Path (\ref{3.6}) is called \index{path!$\mu$-reducing}\emph{$\mu$-reducing} if $\mu_1=\mu$ and either there are no auxiliary edges from the vertex $v_2$ and $\mu$ occurs in the sequence $\mu _1,\ldots ,\mu_{m-1}$ at least twice, or there are auxiliary edges $v_2\to w_1$, $v_2\to w_2, \ldots ,v_2\to w_\nn$ and $\mu$ occurs in the sequence $\mu _1,\ldots ,\mu _{m-1}$ at least $\max \limits_{1\leq i\leq \nn} \ss(\Omega _{w_i})$ times. We will show later, see Equation (\ref{3.26}), that, informally, in any $\mu$-reducing path the length of the solution $H$ is reduced by at least $\frac{1}{10}$ of the length of $H(\mu)$, hence the terminology.

\begin{defn}\label{defn:proh15}
Path (\ref{3.6}) is called \index{path!prohibited of type 15}{\em prohibited of type 15}, if it can be represented in the form
\begin{equation}\label{3.7}
\p_1\s_1\ldots \p_l\s_l\p',
\end{equation}
where for some sequence of bases $\eta _1,\ldots ,\eta _l$ the following three conditions are satisfied:
\begin{enumerate}
    \item\label{it:prp2} the path $\p_i$ is $\eta _i$-reducing;
    \item\label{it:prp1} every base $\mu_i$ that occurs at least once in the sequence $\mu_1,\dots, \mu_{m-1}$, occurs at least $40n^2f_{1}(\Omega_{v_2})+20n+1$ times in the sequence $\eta _1,\ldots ,\eta _l$, where $n=|\BS(\Omega_{v_i})|$ is the number of all bases in the generalised equation $\Omega_{v_i}$, and $f_{1}$ is the function from Lemma \ref{2.8}; in other words, in a prohibited path of type 15, for every carrier base $\mu_i$ there exists at least $40n^2f_{1}(\Omega_{v_2})+20n+1$ many $\mu_i$-reducing paths.
    \item\label{it:prp3} every transfer base of some generalised equation of the path $\p$ is a transfer base of some generalised equation of the path $\p'$.
\end{enumerate}
\end{defn}
Note that for any path of the form (\ref{3.6}) in $T(\Omega)$, there is an algorithm to decide whether this path  is prohibited of type 15 or not.

\bigskip

We now prove that any infinite branch of the tree $T(\Omega)$ of type 15 contains a prohibited subpath of type 15.
\begin{lem}
Let $\p=v_1 \to \ldots\to v_n\to\ldots $ be an infinite path in the tree $T(\Omega)$, and $\tp(v_i) =15$ for all $i$. Then
$\p$ contains a prohibited subpath of type 15.
\end{lem}
\begin{proof}
Let $\omega$ be the set of all bases occurring in the sequence of carrier bases $\mu _1,\mu_2,\ldots$ infinitely many times, and $\tilde\omega$ be the set of all bases that are transfer bases of infinitely many equations $\Omega _{v_i}$. Considering, if necessary, a subpath $\tilde \p$ of $\p$ of the form $v_j \to v_{j+1}\to \ldots$, one can assume that all the bases in the sequence $\mu _1,\mu_2 \ldots $ belong to $\omega$ and every base which is a transfer base of at least one generalised equation belongs to $\tilde\omega$. Then for any $\mu\in\omega$ the path $\tilde \p$ contains infinitely many non-intersecting $\mu$-reducing finite subpaths. Hence  there exists a subpath of the form (\ref{3.7}) of $\tilde \p$  which satisfies conditions (\ref{it:prp2}) and (\ref{it:prp1}) of the definition of a prohibited path of type 15, see Definition \ref{defn:proh15}. Taking a long enough subpath $\p'$ of $\p$, we obtain a prohibited subpath of $\p$.
\end{proof}

We now construct the tree \glossary{name={$T_0(\Omega)$}, description={the finite subtree of $T(\Omega)$ that does not contain prohibited paths}, sort=T}$T_0(\Omega)$. Let $T'(\Omega_v)$ be the subtree of $T(\Omega _v)$ consisting of the vertices $v_1$ such that the path from $v$ to $v_1$ in $T(\Omega)$ does not contain prohibited subpaths and does not contain vertices $v_2\ne v_1$ such that $\tau_{v_2}< \tau _v$. Thus, the leaves of $T'(\Omega_v)$ are either vertices $v_1$ such that $\tau_{v_1}< \tau_v$ or leaves of $T(\Omega _v)$.

The subtree $T'(\Omega_v)$ can be effectively constructed. The tree $T_0(\Omega _v)$ is obtained from $T'(\Omega_v)$ by attaching (gluing) $T_0(\Omega _{v_1})$ (which is already constructed by the induction hypothesis) to those leaves $v_1$ of $T'(\Omega _v)$ for which $\tau _{v_1}< \tau _v$. The function $\ss(\Omega _v)$ is defined by (\ref{so}). Set $T_0(\Omega)=T_0(\Omega _{v_0})$, which is finite by construction.

\bigskip

\begin{defn}\label{defn:Aut}
Denote by \glossary{name={$\Aut (\Omega)$}, description={recursive group of automorphisms associated to the root of $T_0(\Omega)$}, sort=A}$\Aut (\Omega)$, $\Omega=\Omega_{v_0}$, the group of  automorphisms of $\GG_{R(\Omega^\ast)}$, generated by all the groups $\pi(v_0,v)\VV(\Omega _v)\pi (v_0,v)^{-1}$, $v\in T_0(\Omega)$, $\tp(v)\ne 1$ (thus $\pi (v_0,v)$ is an isomorphism). Note that by construction the group $\Aut (\Omega )$ is recursive.
\end{defn}

We adopt the following convention. Given two solutions $H^{(i)}$ of $\Omega_{v_i}$ and ${H^{(i')}}$ of $\Omega_{v_i'}$, by \glossary{name={`$H^{(i)}<_{\Aut(\Omega)}{H^{(i')}}$'}, description={if $H^{(i)}$ is a solution of $\Omega_{v_i}$ and ${H^{(i')}}$ is a solution of $\Omega_{v_{i'}}$, $v_i, v_{i'}\in T_0(\Omega)$, then $H^{(i)}<_{\Aut(\Omega)}{H^{(i')}}$ if and only if $H^{(i)}<_{\pi (v_0,v_{i'})^{-1}\Aut(\Omega)\pi(v_0,v_{i})} {H^{(i')}}$}, sort=Z}
$H^{(i)}<_{\Aut(\Omega)}{H^{(i')}}$ we mean $H^{(i)}<_{\pi (v_0,v_{i'})^{-1}\Aut(\Omega)\pi(v_0,v_{i})} {H^{(i')}}$.

\begin{lem} \label{cor:2.1}
Let
$$
(\Omega,H)=(\Omega _{v_0},H^{(0)})\to (\Omega _{v_1}, H^{(1)})\to\ldots\to (\Omega _{v_l}, H^{(l)})
$$
be the path defined by the solution $H$. If $H$ is a minimal solution with respect to the group of automorphisms $\Aut(\Omega)$, then $H^{(i)}$ is a minimal solution of $\Omega _{v_i}$ with respect to the group $\VV(\Omega_{v_i})$ for all $i$.
\end{lem}
\begin{proof}
Follows from Lemma \ref{lem:2.1}
\end{proof}

\subsection{Paths $\p(H)$ are in $T_0(\Omega )$} \label{5.5.4}

The goal of this section is to give a proof of the proposition below.
\begin{prop} \label{3.4}
For any solution $H$ of a generalised equation $\Omega $ there exists a leaf $w$ of the tree $T_0(\Omega )$, $\tp(w)=1,2$, and a solution $H^{[w]}$ of the generalised equation $\Omega _w$ such that
\begin{enumerate}
    \item \label{it:prop1} ${H}^{[w]}<_{\Aut(\Omega)} H$;
    \item \label{it:prop2} if $\tp(w)=2$ and the generalised equation $\Omega_{w}$ contains
non-constant non-active sections, then there exists a period $P$ such that $H^{[w]}$ is periodic with respect to the period $P$, and the generalised equation $\Omega _w$ is singular or strongly singular with respect to the periodic structure ${\P}(H^{[w]},P)$.
\end{enumerate}
\end{prop}

The proof of this proposition is rather long and technical. We now outline the organisation of the proof.

In part (A), for any solution $H$ of $\Omega$ we describe a path $\p(H):(\Omega_{v_0},H^{[0]})\to \dots \to (\Omega_{v_l}, H^{[l]})$, where $H^{[i]}<_{\Aut(\Omega)} H$.

In part (B) we prove that all vertices of the paths $\p(H)$ belong to the tree $T_0$. In order to do so we show that they do not contain prohibited subpaths. In steps (I) and (II) we prove that the paths $\p(H)$ do not contain prohibited paths of type 7-10 and 12, correspondingly. The proof is by contradiction: if $\p(H)$ is not in $T_0$, the fact that a generalised equation repeats enough times, allows us to construct an automorphism that makes the solution shorter, contradicting its minimality.

To prove that the paths $\p(H)$  do not contain prohibited paths of type 15 (step (III)) we show, by contradiction, that on one hand, the length of minimal solutions is bounded above by a function of the excess (see Definition \ref{defn:excess} for definition of excess), see Equation (\ref{3.18}) and, on the other hand the length of a minimal solution $H$ such that $\p(H)$ contains a prohibited path of type 15 fails inequality (\ref{3.18}).

Finally, in part (C) we prove that the pair $(\Omega_w,H^{[w]})$, where $w$ is a leaf of type 2, satisfies the properties required in Proposition \ref{3.4}.

\subsubsection*{\textbf{{\rm(A):} Constructing the paths $\p(H)$}}

To define the path \index{path!$\p(H)$}$\p(H)$ we shall make use of two functions \glossary{name={$\edge$}, description={a function that assigns a pair $(\Omega_{v'},H^{(v')})$ to the pair $(\Omega_v,H^{(v)})$, $\tp(v)\ne 1,2$}, sort=E}$\edge$ and \glossary{name={$\edge'$}, description={a function that assigns a pair $(\Omega_{v'},H^{(v')})$ to the pair $(\Omega_v,H^{(v)})$, $\tp(v)=15$}, sort=E}$\edge'$ that assign to the pair $(\Omega_v,H^{(v)})$ a pair $(\Omega_{v'},H^{(v')})$, where either $v'=v$ or there is an edge $v\to v'$ in $T(\Omega)$. The function $\edge$ can be applied to any pair $(\Omega_v,H^{(v)})$, where $\tp(v)\ne 1,2$ and the function $\edge'$ can only be applied to a pair $(\Omega_v,H^{(v)})$, where $\tp(v)=15$ and there are auxiliary edges outgoing from the vertex $v$.

\medskip

We now define the functions $\edge$ and $\edge'$.

Let $\tp(v)=3$ or $\tp (v)\geq 6$ and let $v\to w_1,\ldots,v\to w_m$ be the list of all \emph{principal} outgoing edges from $v$, then the generalised equations $\Omega _{w_1},\dots ,\Omega_{w_m}$ are obtained from $\Omega _v$ by a sequence of elementary transformations. For every solution $H$ the path defined by $H$ is unique, i.e. for the pair $(\Omega,H)$ there exists a unique pair $(\Omega_{w_i},H^{(i)})$ such that the following diagram commutes:
$$
\xymatrix@C3em{
 \GG_{R(\Omega^\ast)}  \ar[rd]_{\pi_H} \ar[rr]^{\theta_i}  &  &\GG_{R(\Omega_{w_i}^\ast )} \ar[ld]^{\pi_{H^{(i)}}}
                                                                             \\
                               &  \GG &
}
$$
Define a function $\edge$ that assigns the pair $(\Omega _{w_i},H ^{(i)})$ to the pair $(\Omega _v,H)$, $\edge(\Omega_v,H)=(\Omega _{w_i},H ^{(i)})$.

Let $\tp(v)=4$ or $\tp(v)=5$. In these cases there is a single edge $v\to w_1$ outgoing  from $v$ and this edge is auxiliary. We set $\edge(\Omega_v,H)=(\Omega_{w_1},H^{(1)})$.

If $\tp (v)=15$ and there are auxiliary outgoing edges from the vertex $v$, then the carrier base $\mu$ of the generalised equation $\Omega _v$ intersects with $\Delta (\mu)$. Below we use the notation from the description of Case 15.1. For any solution $H$ of the generalised equation $\Omega_v$ one can construct a solution $H'$ of the generalised equation $\Omega_{v'}$ as follows: $H'_{\rho_v+1}=H[1,\beta(\Delta(\mu))]$. We define the function $\edge'$ as follows  $\edge'(\Omega_v,H)=(\Omega_{v'},H')$.

\medskip

To construct the path $\p(H)$
\begin{equation}\label{3.8}
(\Omega,H)\to (\Omega _{v_0},H^{[0]})\to (\Omega _{v_1}, H^{[1]})\to\ldots
\end{equation}
we use induction on its length $i$.

Let $i=0$, we define $H^{[0]}$ to be a solution of the generalised equation $\Omega$ minimal with respect to the group of automorphisms $\Aut(\Omega)$, such that $H^{[0]} <_{\Aut(\Omega)} H$. Let $i \ge 1$ and suppose that the term $(\Omega_{v_i},H^{[i]})$ of the sequence (\ref{3.8}) is already constructed. We construct $(\Omega_{v_{i+1}},H^{[i+1]})$

If $3\le \tp(v_i)\le 6$, $\tp(v_i)=11,13,14$, we set $(\Omega _{v_{i+1}},H^{[i+1]})=\edge(\Omega _{v_i},{H^{[i]}})$.

If $7\leq \tp(v_i)\leq 10$ or $\tp(v_i)=12$ and there exists a minimal solution $H^+$ of $\Omega _{v_i}$ such that $H^+<_{\Aut(\Omega)} H^{[i]}$ and $|H^+|<|H^{[i]}|$, then we set $(\Omega_{v_{i+1}},H^{[i+1]})=(\Omega_{v_i}, H^+)$. Note that, since  $H^{[0]}$ is a minimal solution of $\Omega$ with respect to the group of automorphisms $\Aut(\Omega)$, by construction and by Lemma \ref{cor:2.1}, we have that the solution $H^{[i]}$ is minimal with respect to the group of automorphism $\VV(\Omega_{v_i})$ for all $i$. Although $H^{[i]}$ is a minimal solution, in this step we take a  minimal solution of minimal total length, see Remark \ref{rem:ms}.

Let $\tp(v_i)=15$, $v_i\neq v_{i-1}$ and $v_i\to w_1,\ldots ,v_i\to w_{\nn}$ be the auxiliary edges outgoing from $v_i$ (the carrier base $\mu$ intersects with its dual $\Delta(\mu)$). If there exists a period $P$ such that
\begin{equation}\label{3.9}
{H^{[i]}}[1,\beta (\Delta (\mu))]\doteq P^rP_1,\ P\doteq P_1P_2, \ r\ge \max \limits_{1\leq i\leq {\nn}} \ss(\Omega _{w_i}),
\end{equation}
then we set $(\Omega_{v_{i+1}},H^{[i+1]}) = \edge'(\Omega _{v_i},H^{[i]})$ and declare the section $[1,\beta (\Delta (\mu))]$ non-active.

In all the other cases (when $\tp(v_i)=15$) we set $(\Omega _{v_{i+1}},H^{[i+1]})=\edge(\Omega _{v_i},{H^{[i]}})$.

The path (\ref{3.8}) ends if $\tp (v_i)\leq 2$.

\medskip

A leaf $w$ of the tree $T(\Omega)$ is called \index{leaf!final of the tree $T$}\emph{final} if there exists a solution $H$ of $\Omega_{v_0}$ and a path $\p(H)$  such that $\p(H)$ ends in $w$.

\subsubsection*{\textbf{{\rm (B):} Paths $\p(H)$ belong to $T_0$}}

We use induction on $\tau$ to show that every vertex $v_i$ of the path $\p(H)$ (see Equation (\ref{3.8})) belongs to $T_0(\Omega)$, i.e.  $v_i\in T_0(\Omega)$.  Suppose that $v_i\not\in T_0(\Omega )$ and let $i_0$ be the least among such numbers. It follows from the construction of $T_0(\Omega )$ that there exists $i_1< i_0$ such that the path from $v_{i_1}$ to $v_{i_0}$ contains a prohibited subpath $\s$. From the minimality of $i_0$ it follows that the prohibited path $\s$ goes from $v_{i_2}$, $i_1\le i_2\le i_0$ to $v_{i_0}$.

\subsubsection*{{\rm (I):} Paths $\p(H)$ do not contain prohibited subpaths of type 7-10}

Suppose first that the prohibited path $\s$ is of type 7-10, i.e. $7\leq \tp(v_i)\leq 10$. By definition, there exists a generalised equation $\Omega_{v_{k_1}}$ that repeats $r=2^{4\rho_{\Omega_{v_{k_1}}}^2\cdot(2^{\rho_{\Omega_{v_{k_1}}}}+1)}+1$ times, i.e.
$$
\gpof{v_{k_1}}=\dots=\gpof{v_{k_r}}.
$$
Since the path $\s$ is prohibited, we may assume that $v_{k_i}\ne v_{k_{i+1}}$ for all $i$ and $v_{k_i}\ne v_{k_i+1}$ for all $i$, i.e.
$(\Omega_{v_{k_i+1}},H^{[{k_i+1}]})=\edge(\Omega_{v_{k_i}},H^{[{k_i}]})$.

We now prove that there exist $k_j$ and $k_{j'}$, $k_{j}<k_{j'}$ such that $H^{[k_{j'}]}<_{\Aut(\Omega)}H^{[k_j]}$.

Since, the number of different $2\rho_{\Omega_{v_{k_1}}}\times 2\rho_{\Omega_{v_{k_1}}}$ cancellation matrices (see Definition \ref{defn:nfmatrix}) is bounded above by $r$, if the generalised equation $\Omega_{v_{k_1}}$ repeats $r$ times, there exist $k_j$ and $k_{j'}$ such that $H^{[k_j]}$ and $H^{[k_{j'}]}$ have the same cancellation matrix, i.e. satisfy conditions (\ref{it:minsol3}) and (\ref{it:minsol4}) from Definition \ref{defn:sol<}. Moreover, by Lemma \ref{3.3}, $\pi(v_{k_j},v_{k_{j'}})$ is an automorphism of $\GG_{R(\Omega_{v_{k_j}}^\ast)}$ invariant with respect to the kernel of ${\Omega}_{v_{k_j}}$.

By Remark \ref{rem:leng<}, we have $|H^{[k_j]}| >|{H^{[k_{j'}]}}|$. This derives a contradiction, since, by construction of the sequence (\ref{3.8}) one has $v_{k_j+1}=v_{k_j}$.

\subsubsection*{{\rm (II):} Paths $\p(H)$ do not contain prohibited subpaths of type 12.}
Suppose next that the path $\s$ is prohibited  of type 12, i.e. $\tp(v_i)=12$. An analogous argument to the one for prohibited paths of type 7-10, but using Lemma \ref{lem:c12} instead of Lemma \ref{3.3}, leads to a contradiction. Hence, we conclude that $v_i\in T_0(\Omega)$, where $\tp(v_i)=12$.

\subsubsection*{{\rm (III):} Paths $\p(H)$ do contain prohibited subpaths of type 15.}
Finally, suppose that the path $\s$ is prohibited of type 15, i.e. $\tp(v_i)=15$. Abusing the notation, we consider a subpath of (\ref{3.8})
$$
(\Omega _{v_1},H^{[1]})\to (\Omega_{v_2},H^{[2]})\to\ldots \to(\Omega_{v_m},H^{[m]})\to\ldots,
$$
where $v_1,v_2,\ldots$ are vertices of the tree $T_0(\Omega)$, $\tp(v_i)=15$ and the edges $v_i\to v_{i+1}$ are principal for all $i$. Notice, that by construction the above path is the path defined by the solution $H^{(1)}=H^{[1]}$:
\begin{equation}\label{3.11}
(\Omega _{v_1},H^{(1)})\to (\Omega_{v_2},H^{(2)})\to\ldots \to(\Omega_{v_m},H^{(m)})\to\ldots,
\end{equation}

To simplify the notation, below we write $\rho_i$ for $\rho_{\Omega_{v_i}}$.

Let $\omega =\{\mu _1,\ldots ,\mu_{m}, \ldots\}$ be the set of carrier bases $\mu _i$ of the generalised equations $\Omega_{v_i}$'s and let $\tilde\omega$ denote the set of bases which are transfer bases for at least one generalised equation in (\ref{3.11}). By $\omega_2$ we denote the set of all bases $\nu$ of $\Omega_{v_i}$, $i=1,\dots,m,\dots$ so that $\nu, \Delta(\nu)\notin \omega\cup\tilde \omega$.  Let
$$
\alpha (\omega)=\min \left\{\min\limits_{\mu\in\omega _2}\{\alpha(\mu)\},\rho_A\right\},
$$
where $\rho_A$ is the boundary between the active part and the non-active part.

For every element $(\Omega _{v_i},H^{(i)})$ of the sequence (\ref{3.11}), using $\D 3$, if necessary, we make the section $[1,\alpha (\omega)]$ of the generalised equation $\Omega _{v_i}$ closed and set $[\alpha(\omega),\rho_i]$ to be the non-active part of the generalised equation $\Omega_{v_i}$ for all $i$.

Recall that by $\omega _1$ we denote the set of all variable bases $\nu $ for which either $\nu$ or $\Delta (\nu)$ belongs to the active part $[1,\alpha(\omega)]$ of the  generalised equation $\Omega_{v_1}$, see Definition \ref{defn:excess}.

\subsubsection*{{\rm(III.1):} Lengths of minimal solutions are bounded by a function of the excess}

Let $H$ be a solution of the generalised equation $\Omega$ and let $[1,j+1]$ be the quadratic part of $\Omega$. Set
\glossary{name={$d_1(H)$}, description={length of the ``quadratic part of the solution'' $H$}, sort=D}\glossary{name={$d_2(H)$}, description={length of the ``quadratic-coefficient part of the solution'' $H$}, sort=D}
$$
d _1(H)=\sum \limits_{i=1}^{j}|H_i|, \ d_2(H)=\sum \limits_{\nu}|H(\nu)|,
$$
where $\nu$ is a quadratic-coefficient base.

\begin{lem} \label{2.8}
Let $v$ be a vertex of $T(\Omega)$,  $\tp(v)=15$. There exists a recursive function $f_{1}(\Omega _v)$ such that for any solution $H$ minimal with respect to $\VV(\Omega_v)$ one has
$$
d_1(H)\leq f_{1}(\Omega_v)\max \left\{d_2(H),1\right\}.
$$
\end{lem}
\begin{proof}
Since $\tp(v)=15$, every boundary that touches a base is $\eta$-tied in every base $\eta$ which it intersects.  Instead of $\Omega _v$ below we consider the generalised equation $\widetilde{\Omega}_v=\D 3(\Omega_v)$ which does not have any boundary connections. Then $\GG_{R(\Omega_v^\ast)}$ is isomorphic to $\GG_{R({\widetilde{\Omega}}_v^\ast)}$. We abuse the notation and denote $\widetilde{\Omega}_v$ by $\Omega_v$.

Consider the sequence
$$
(\Omega_v ,H)=(\Omega _{v_0}, H^{(0)})\to (\Omega _{v_1}, H^{(1)})\to\ldots  \to (\Omega _{v_i}, H^{(i)})\to \ldots,
$$
where $(\Omega _{v_{j+1}}, H^{(j+1)})$ is obtained from $(\Omega _{v_j}, H^{(j )})$ by applying the entire transformation $\D 5$ in the quadratic part of $\Omega$. Denote by $\mu_{i}$ the carrier base of the generalised equation $\Omega _{v_i}$ and consider the sequence $\mu_1,\dots,\mu_i,\ldots$

We use an argument analogous to the one given in the proof of Lemma \ref{lem:c12} to show that the number of consecutive quadratic bases in the sequence $\mu_1,\dots,\mu_i,\ldots$ is bounded above. The entire transformation applied in the quadratic part of a generalised equation, does not increase the complexity and the number of items. The number of constrained generalised equations with a bounded number of items and bounded complexity is finite.

We now prove that if a generalised equation $\Omega_{v_{k_1}}$ repeats $r=2^{4\rho_{k_1}^2 \cdot (2^\rho_{k_1}+1)}+1$ times, then there exist $k_j$ and $k_{j'}$ such that  $H^{(k_{j'})}<_{\VV(\Omega_{v_{k_j}})}H^{(k_j)}$.

Since, the number of different $2\rho_{k_1}\times 2\rho_{k_1}$ cancellation matrices (see Definition \ref{defn:nfmatrix}) is bounded above by $r$, if the generalised equation $\Omega_{v_{k_1}}$ repeats $r$ times, there exist $k_j$ and $k_{j'}$ such that $H^{(k_j)}$ and $H^{(k_{j'})}$ have the same cancellation matrix, i.e. satisfy conditions (\ref{it:minsol3}) and (\ref{it:minsol4}) from Definition \ref{defn:sol<}.

Moreover, by Lemma \ref{lem:c12}, the automorphism $\pi(v_{k_j},v_{k_{j'}})$ of $\GG_{R(\Omega_{v_{k_j}}^\ast)}$ is invariant with respect to the non-quadratic part of $\Omega_{v_{k_j}}$.

By Remark \ref{rem:leng<}, we have $|{H^{(k_{j'})}}|< |H^{(k_j)}|$. We, therefore, have that ${H^{(k_{j'})}} <_{\VV(\Omega_{v_{k_j}})} H^{(k_j)}$, contradicting the minimality of $H^{(k_j)}$. Thus, we have proven that the number of consecutive quadratic carrier bases is bounded.

Furthermore, since whenever the carrier base is a quadratic-coefficient base the number of bases in the quadratic part decreases, there exists an integer $N$ bounded above by a computable function of the generalised equation, such that the quadratic part of $\Omega_{v_N}$ is empty.

We prove the statement of the lemma by induction on the length $i$ of the sequence. If $i=N-1$, since the application of the entire transformation to $\Omega_{v_{N-1}}$ results in a generalised equation with the empty quadratic part, the carrier $\mu_{N-1}$ is a quadratic-coefficient base and all the other bases $\nu_{N-1,1},\dots,\nu_{N-1,n_{N-1}}$ are transfer bases and are transferred from the carrier to its dual. Since the length $|H^{(N-1)}(\nu_{N-1,i})|$ of every transferred base is less than the length of the carrier, we get that
$$
d_1(H^{(N-1)})\le\sum\limits_{n=1}^{n_{N-1}}|H^{(N-1)}(\nu_{N-1,n})|\le n_{N-1}|H^{(N-1)}(\mu_{N-1})| = n_{N-1}d_2(H^{(N-1)}).
$$

By induction, we may assume that $d_1(H^{(i+1)})\leq g_{i+1}(\Omega_{v_{i+1}})\max \left\{d_2( H^{(i+1)}),1\right\}$, where $g_{i+1}$ is a certain computable function. We prove that the statement holds for $H^{(i)}$.

Suppose that the carrier base $\mu_i$ is quadratic and let $\nu_{i,1},\dots,\nu_{i,n_{i}}$ be the transfer bases of $\Omega_{v_i}$. Then
$$
|H^{(i)}(\mu_i)|-|H^{(i+1)}(\mu_i)|\le \sum\limits_{n=1}^{n_{i}}|H^{(i+1)}(\nu_{i,n})|,
$$
where $\nu_{i,n}$, $n=1,\dots, n_i$ are bases of $\Omega_{v_{i+1}}$. Thus, by induction hypothesis, we get $$
\sum\limits_{n=1}^{n_{i}}|H^{(i+1)}(\nu_{i,n})|\le n_i g_{i+1}(\Omega_{v_{i+1}})\max \left\{d_2( H^{(i+1)}),1\right\}.
$$
Notice that the sets of quadratic-coefficient bases of $\Omega_{v_i}$ and $\Omega_{v_{i+1}}$ coincide, thus $d_2(H^{(i+1)})=d_2(H^{(i)})$. Therefore,
\begin{gather}\notag
\begin{split}
d_1(H^{(i)})=\left(|H^{(i)}(\mu_i)|\right.&\left.-|H^{(i+1)}(\mu_i)|\right)+ d_1(H^{(i+1)})\le \\
\le \left(|H^{(i)}(\mu_i)|\right.&\left.-|H^{(i+1)}(\mu_i)|\right)+ g_{i+1}(\Omega_{v_{i+1}}) \cdot\max \left\{d_2( H^{(i+1)}),1\right\}\le\\
&\le (n_i+1) \cdot g_{i+1}(\Omega_{v_{i+1}}) \cdot\max \left\{d_2( H^{(i+1)}),1\right\}=g_{i}(\Omega_{v_{i}})\cdot\max \left\{d_2( H^{(i)}),1\right\}.
\end{split}
\end{gather}

Suppose now that the carrier base $\mu_i$ is a quadratic-coefficient base and let $\nu_{i,1},\dots,\nu_{i,n_{i}}$ be the transfer bases of the generalised equation $\Omega_{v_i}$. Since the duals of the transferred quadratic bases become quadratic-coefficient and since $|H^{(i)}(\nu_{i,n})|\le |H^{(i)}(\mu_i)|$, $n=1,\dots, n_i$, we get that
\begin{gather} \notag
\begin{split}
d_2(H^{(i+1)})\le& L+|H^{(i+1)}(\mu_i)+\sum \limits_{n=1}^{n_{i}}|H^{(i)}(\nu_{n_i,n})|\le\\
\le& L+ (n_i+1) |H^{(i)}(\mu_{i})|\le (n_{i}+1)\left(L+|H^{(i)}(\mu_i)|\right)=(n_i+1)\cdot d_2(H^{(i)}),
\end{split}
\end{gather}
where $L=\sum\limits_\lambda|H^{(i)}(\lambda)|$ and the sum is taken over all bases that are quadratic-coefficient in both $\Omega_{v_i}$ and $\Omega_{v_{i+1}}$. Therefore,
\begin{gather} \notag
\begin{split}
  d_1(H^{(i)})=\left(|H^{(i)}(\mu_i)|\right.&- \left.|H^{(i+1)}(\mu_i)|\right)+ d_1(H^{(i+1)})\le \\
            \le |H^{(i)}(\mu_i)|&+g_{i+1}(\Omega_{v_{i+1}}) \cdot\max \left\{d_2( H^{(i+1)}),1\right\}\le\\
                    \le |H^{(i)}&(\mu_i)|+(n_i+1)\cdot g_{i+1}(\Omega_{v_{i+1}})\cdot\max \left\{d_2( H^{(i)}),1\right\}\le \\
&\le(n_{i}+2)\cdot g_{i+1}(\Omega_{v_{i+1}}) \cdot \max  \left\{d_2( H^{(i)}),1\right\}=g_{i}(\Omega_{v_{i}})\cdot \max \left\{d_2(H^{(i)}),1\right\}.
\end{split}
\end{gather}
The statement of the lemma follows.
\end{proof}

Recall, that by $\widetilde{\Omega}$ we denote the generalised equation obtained from $\Omega$ applying $\D 3$. Consider the section  of $\widetilde{\Omega} _{v_1}$ of the form  $[1,\alpha(\omega)]$. The section $[1,\alpha(\omega)]$ lies in the quadratic part of $\widetilde{\Omega} _{v_1}$. Let $B'$ be the set of quadratic bases that belong to $[1,\alpha(\omega)]$ and let $\VV'(\Omega_{v_1})$ be the group of automorphisms of $\GG_{R(\Omega^\ast)}$ that are invariant with respect to the non-quadratic part of $\widetilde{\Omega} _{v_1}$ and act identically on all the bases which do not belong to $B'$. By definition, $\VV'(\Omega_{v_1})\le \VV(\Omega_{v_1})$. Thus the solution $H^{(1)}$ minimal with respect to $\VV(\Omega_{v_1})$ is also minimal with respect to $\VV'(\Omega_{v_1})$. By Lemma \ref{2.8} we have
\begin{equation} \label{3.14}
d_1(H^{(1)})\leq f_{1}(\Omega _{v_1}) \max\left\{d_2(H^{(1)}),1\right\}.
\end{equation}

Recall that  (see Definition \ref{defn:excess})
\begin{equation} \label{eq:M}
d_{A\Sigma}(H)=\sum \limits_{i=1}^{\alpha(\omega)-1}|H_i|, \quad \psi_{A\Sigma}(H)=\sum \limits_{\mu\in \omega_1}|H({\mu})|-2d_{A\Sigma}(H).
\end{equation}

Our next goal is, using inequality (\ref{3.14}), to give an upper bound of the length of the interval $d_{A\Sigma}(H^{(1)})$ in terms of the excess $\psi _{A\Sigma}$ and the function $f_{1}(\Omega _{v_1})$, see inequality (\ref{3.19}).

Denote by \glossary{name={$\gamma_i(\omega)$}, description={the number of bases from $\omega_1$ that contain $h_i$}, sort=G}$\gamma_i(\omega)$ the number of bases $\mu\in\omega_1$ containing $h_i$. Then
\begin{equation}\label{3.15}
\sum\limits_{\mu\in\omega _1}|H^{(1)}({\mu})|=\sum\limits_{i=1}^{\rho}|H_i^{(1)}| \gamma_i(\omega),
\end{equation}
where $\rho =\rho_{\Omega _{v_1}}$. Let $I=\left\{i \mid 1\le i \le\alpha(\omega)-1 \hbox{ and }\gamma_i=2\right\}$, and $J=\left\{i\mid 1\leq i\le \alpha (\omega)-1 \hbox{ and } \gamma_i> 2\right\}$. By (\ref{3.12}) we have:
\begin{equation}\label{3.16}
d_{A\Sigma}(H^{(1)})=\sum \limits_{i\in I}|H_i^{(1)}|+ \sum\limits_{i\in J}|H_i^{(1)}|= d_1(H^{(1)})+
\sum\limits_{i\in J}|H_i^{(1)}|.
\end{equation}

Let  $\lambda ,\Delta(\lambda)$ be a pair of variable quadratic-coefficient bases of the generalised equation $\widetilde{\Omega}_{v_1}$, where $\lambda$ belongs to the non-quadratic part of $\widetilde\Omega _{v_1}$.
When we apply $\D 3$ to $\Omega_{v_1}$ thereby obtaining $\widetilde\Omega _{v_1}$, the pair $\lambda ,\Delta(\lambda)$ is obtained from bases $\mu\in\omega_1$.
There are two types of quadratic-coefficient bases.
\begin{itemize}
    \item[Type 1:] variable bases $\lambda$ such that $\beta(\lambda)\le\alpha(\omega)$. In this case, since $\lambda$ belongs to the non-quadratic part of $\widetilde{\Omega}_{v_1}$, it is a product of items $\{h_i\mid i\in J\}$ and thus $|H({\lambda})|\leq\sum\limits_{i\in J}|H_i^{(1)}|$. Thus the sum of the lengths of quadratic-coefficient bases of Type 1 and their duals is bounded above by $2n\sum\limits_{i\in J}|H_i^{(1)}|$, where $n$ is the number of bases in $\Omega$.
    \item[Type 2:] variable bases $\lambda$ such that $\alpha(\lambda)\ge\alpha(\omega)$. The sum of the lengths of quadratic-coefficient bases of the second type is bounded above by $2\cdot \sum \limits_{i=\alpha(\omega)}^{\rho}|H_i^{(1)}|\gamma _i(\omega)$.
\end{itemize}
We have
\begin{equation}\label{3.17}
d_2(H^{(1)})\le 2n\sum\limits_{i\in J}|H_i^{(1)}|+2\cdot\sum\limits_{i=\alpha(\omega)}^{\rho}|H_i^{(1)}| \gamma_i(\omega).
\end{equation}
Then from (\ref{eq:M}) and (\ref{3.15}) it follows that
\begin{equation}\label{3.18}
\psi_{A\Sigma}(H^{(1)}_i)\ge \sum \limits_{i\in J}|H_i^{(1)}|+\sum\limits_{i=\alpha(\omega)}^{\rho}|H_i^{(1)}|\gamma_i(\omega).
\end{equation}
From  Equation (\ref{3.16}), using inequalities (\ref{3.14}), (\ref{3.17}), (\ref{3.18}) we get
\begin{equation}\label{3.19}
d_{A\Sigma}(H^{(1)})\leq \max\left\{\psi_{A\Sigma}(H^{(1)})(2n f_{1}(\Omega_{v_1})+1), f_{1}(\Omega _{v_1})\right\}.
\end{equation}

\subsubsection*{{\rm (III.2):} Minimal solutions $H$ such that $\p(H)$ contains a prohibited subpath of type 15 fail inequality {\rm(\ref{3.18})}.}

Let the path $v_1\to v_2\to \ldots \to v_m$ corresponding to the sequence (\ref{3.11}) be $\mu$-reducing, that is  $\mu_1=\mu$ and, either there are no outgoing auxiliary edges from $v_2$ and $\mu$ occurs in the sequence $\mu _1,\ldots ,\mu _{m-1}$ at least twice, or $v_2$ does have outgoing auxiliary edges $v_2\to w_1,\dots, v_2\to w_{\nn}$ and the base $\mu$ occurs in the sequence $\mu _1,\ldots ,\mu _{m-1}$ at least $\max\limits_{1\le i\le \nn}\ss(\Omega_{w_i})$ times.

Set $\delta_i=d_{A\Sigma}(H^{(i)})-d_{A\Sigma}(H ^{(i+1)})$. We give a lower bound for $\sum \limits_{i=1}^{m-1}\delta _i$, i.e. we estimate by how much the length of a solution is reduced in a $\mu$-reducing path.

We first prove that if $\mu _{i_1}=\mu _{i_2}=\mu$, $i_1< i_2$ and $\mu _i\ne\mu$ for $i_1< i< i_2$, then
\begin{equation} \label{3.23}
\sum \limits_{i=i_1}^{i_2-1}\delta _i\ge |H^{(i_1+1)}[1,\alpha (\Delta (\mu_{i_1+1}))]|.
\end{equation}
Indeed, if $i_2=i_1+1$ then
$$
\delta _{i_1}=|H^{(i_1)}[1,\alpha(\Delta (\mu))]|=|H^{(i_1+1)}[1,\alpha (\Delta (\mu))]|.
$$
If $i_2 > i_1+1$, then $\mu _{i_1+1}\ne \mu$ and $\mu$ is a transfer base in the generalised equation $\Omega_{v_{i_1+1}}$ and thus
\begin{equation} \label{est1}
\delta_{i_1+1}+|H^{(i_1+2)}[1,\alpha(\mu)]|=|H^{(i_1+1)}[1,\alpha(\Delta(\mu _{i_1+1}))]|.
\end{equation}
Since $\mu$ is the carrier base of $\Omega_{v_{i_2}}$ we have
\begin{equation} \label{est2}
\sum\limits_{i=i_1+2}^{i_2-1}\delta _i\ge |H^{(i_1+2)}[1,\alpha (\mu)]|.
\end{equation}
From (\ref{est2}) and (\ref{est1}) we get (\ref{3.23}).

\bigskip

We want to show that every $\mu$-reducing path reduces the length of a solution $H^{(1)}$ by at least $\frac{1}{10}|H^{(1)}(\mu)|$, see inequality (\ref{3.26}). To prove this we consider the two cases from the definition of a $\mu$-reducing path.

Suppose first that $v_2$ does not have any outgoing auxiliary edges, i.e. the bases $\mu _2$ and $\Delta(\mu _2)$ do not intersect in the generalised equation $\Omega _{v_2}$, then (\ref{3.23}) implies that
$$
\sum\limits_{i=1}^{m-1}\delta_i\ge |H^{(2)}[1,\alpha (\Delta(\mu _2))]|\ge |H^{(2)}({\mu _2})|\ge |H^{(2)}({\mu})|=|H^{(1)}({\mu})|-\delta_1,
$$
which, in turn, implies that
\begin{equation} \label{3.24}
\sum\limits_{i=1}^{m-1}\delta _i\ge\frac{1}{2}|H^{(1)}({\mu})|.
\end{equation}

Suppose now that there are auxiliary edges $v_2 \to w_1,\ldots ,v_2 \to w_{\nn}$. Let $H^{(2)}[1,\alpha (\Delta (\mu _2))]\doteq Q$,
and $P$ be a period such that $Q\doteq P^d$ for some $d\ge 1$, then $H^{(2)}({\mu_2})$ and $H^{(2)}({\mu})$ are initial subwords of the word $H^{(2)}[1,\beta (\Delta (\mu_2))]$, which, in turn, is an initial subword of $P^\infty$.

By construction of the sequence (\ref{3.8}), relation (\ref{3.9}) fails for the vertex $v_2$, i. e. (in the notation of (\ref{3.9})):
\begin{equation}\label{3.25}
H^{(2)}({\mu})\doteq {P}^r \cdot P_1, \ P\doteq P_1\cdot P_2,\ r< \max\limits_{1\leq j\leq \nn} \ss(\Omega _{w_j}).
\end{equation}
Let $\mu _{i_1}=\mu _{i_2}=\mu$, $i_1< i_2$ and $\mu _i\ne \mu$ for $i_1< i< i_2$. If
\begin{equation} \label{3.26a}
|H^{(i_1+1)}({\mu _{i_1+1}})|\ge 2 |P|
\end{equation}
since $H^{(i_1+1)}(\Delta({\mu _{i_1+1}}))$ is a  $Q'$-periodic subword ($Q'$ is a cyclic permutation of $P$) of the $Q'$-periodic word $H^{(i_1+1)}[1,\beta(\Delta(\mu_{i_1+1}))]$ of length greater than $2|Q'|=2|P|$, it follows by Lemma 1.2.9 in \cite{1}, that $|H^{(i_1+1)}[1,\alpha (\Delta(\mu_{i_1+1}))]|\ge k|Q'|$. As $k\ne 0$ ($\mu_{i+1}$ and $\Delta(\mu_{i+1})$ do not form a pair of matched bases), so $|H^{(i_1+1)}[1,\alpha (\Delta(\mu_{i_1+1}))]|\ge |P|$. Together with (\ref{3.23}) this gives that $\sum \limits_{i=i_1}^{i_2-1}\delta _i\ge |P|$. The base $\mu$ occurs in the sequence $\mu _1,\ldots ,\mu _{m-1}$ at least $r$ times, so either (\ref{3.26a}) fails for some $i_1\le m-1$ or
$\sum \limits_{i=1}^{m-1}\delta_i\ge (r-3)|P|$.

If (\ref{3.26a}) fails, then from the inequality $|H^{(i+1)}({\mu_i})|\le |H^{(i+1)}({\mu _{i+1}})|$ and the definition of $\delta_i$ follows that
$$
\sum\limits_{i=1}^{i_1}\delta_i\ge |H^{(1)}({\mu})|-|H^{(i_1+1)}({\mu _{i_1+1}})|\ge (r-2)|P|.
$$
hence in both cases $\sum \limits_{i=1}^{m-1}\delta_i\ge (r-3)|P|$.

Notice that for $i_1=1$, inequality (\ref{3.23}) implies that $\sum\limits_{i=1}^{m-1}\delta _i\ge |Q|\ge |P|$; so $\sum\limits_{i=1}^{m-1}\delta_i\ge \max\{1,r-3\}|P|$. Together with (\ref{3.25}) this implies that
$$
\sum\limits_{i=1}^{m-1}\delta _i\ge \frac{1}{5}|H^{(2)}({\mu})|=\frac{1}{5}(|H^{(1)}({\mu})|-\delta _1).
$$
Finally,
\begin{equation}\label{3.26}
\sum \limits_{i=1}^{m-1}\delta _i\ge\frac{1}{10} |H^{(1)}({\mu})|.
\end{equation}
Comparing (\ref{3.24}) and (\ref{3.26}), we see that for the $\mu$-reducing path inequality (\ref{3.26}) always holds. We thereby have shown that in any $\mu$-reducing path the length of the solution is reduced by at least $\frac{1}{10}$ of the length of the carrier  base $\mu$.

\bigskip

Notice that by property (\ref{it:prp3}) from Definition \ref{defn:proh15}, we can assume that the carrier bases $\mu_i$ and their duals $\Delta(\mu_i)$ belong to the active part $A\Sigma =[1,\alpha(\omega)]$. Then, by Lemma \ref{lem:excess} and by construction of the path (\ref{3.11}), we have that $\psi_{A\Sigma}(H^{(1)})=\dots=\psi_{A\Sigma}(H^{(m)})=\dots$ We denote this number by $\psi_{A\Sigma}$.

We now prove that there exists a base $\mu\in\omega$ such that
\begin{equation}\label{3.27}
|H^{(\kappa)}({\mu })|\ge\frac{1}{2n}\psi _{A\Sigma},
\end{equation}
where $n$ is the number of bases in $\Omega$.
Since, by assumption, the path $v_1\to v_2\to \ldots \to v_m$ corresponding to the path (\ref{3.11}) is prohibited, it can be presented in the form (\ref{3.7}). From the definition of $\psi _{A\Sigma}$, see (\ref{eq:M}), we get that $\sum\limits_{\mu \in\omega_1}|H^{(m)}({\mu})|\ge \psi _{A\Sigma}$, hence the inequality $|H^{(m)}({\mu})|\geq\frac{1}{2n}\psi _{A\Sigma}$ holds for at least one base $\mu\in\omega_1$. Since $H^{(m)}({\mu})\doteq \left(H^{(m)}({\Delta (\mu)})\right)^{\pm 1}$, we may assume that $\mu\in\omega\cup\tilde{\omega}$. Let $\kappa$ be the length of the path $\p_1\s_1\ldots \p_l\s_l$  in (\ref{3.7}).

If $\mu\in\omega$, then inequality (\ref{3.27}) trivially holds.

If $\mu\in\tilde{\omega}$, then  by the third condition in the definition of a prohibited path of type 15 (see Definition \ref{defn:proh15}) there exists $\kappa\leq i\leq m$ such that $\mu$ is a transfer base of $\Omega _{v_i}$. Hence, $|H^{(\kappa)}({\mu_i})|\ge |H^{(i)}({\mu _i})|\ge |H^{(i)}({\mu})|\ge |H^{(m)}({\mu})|\ge \frac{1}{2n}\psi _{A\Sigma}$.

\bigskip

Finally, from conditions (\ref{it:prp2}) and (\ref{it:prp1}) in the definition of a prohibited path of type 15, from the inequality $|H^{(i)}({\mu})|\ge |H^{(\kappa)}({\mu})|$,  $1\leq i\leq \kappa$, and from (\ref{3.26}) and (\ref{3.27}), it follows that
\begin{equation}\label{3.28}
\sum\limits_{i=1}^{\kappa-1}\delta _i\ge \max\left\{\frac{1}{20n}\psi _{A\Sigma},1\right\}\cdot(40n^2f_{1}+20n+1).
\end{equation}

By Equation (\ref{3.22}), the sum in the left part of the inequality (\ref{3.28}) equals $d_{A\Sigma}(H^{(1)})-d_{A\Sigma}(H^{(\kappa)})$, hence
\begin{equation}\label{eq:3.29}
d_{A\Sigma}(H^{(1)})\ge \max \left\{\frac{1}{20n}\psi _{A\Sigma},1\right\}\cdot(40n^2f_{1} +20n +1),
\end{equation}
which contradicts (\ref{3.19}).

Therefore, the assumption that there are prohibited subpaths (\ref{3.11}) of type 15 in the path (\ref{3.8}) led to a contradiction. Hence, the path (\ref{3.8}) does not contain prohibited subpaths. This implies that $v_i\in T_0(\Omega )$ for all $(\Omega _{v_i},H^{(i)})$ in (\ref{3.8}).

In particular, we have shown that final leaves $w$ of the tree $T(\Omega)$ are, in fact, leaves of the tree $T_0(\Omega)$. Naturally, we call  such leaves the \index{leaf!final of the tree $T_0$}\emph{final leaves of $T_0(\Omega)$}.

\subsubsection*{\textbf{{\rm (C):} The pair $(\Omega_w,H^{[w]})$, where $\tp(w)=2$, satisfies the properties required in Proposition \ref{3.4}}}

For all $i$, either $v_i= v_{i+1}$  and $|H^{[i+1]}|<|H^{[i]}|$, or $v_i\to v_{i+1}$ is an edge of a finite tree $T_0(\Omega)$. Hence the sequence (\ref{3.8}) is finite. Let $(\Omega _{w}, H^{[w]})$ be its final term. We show that $(\Omega _{w},H^{[w]})$ satisfies the properties required in the proposition.

Property (\ref{it:prop1}) follows directly from  the construction of $H^{[w]}$.

We now prove that property (\ref{it:prop2}) holds. Let $\tp (w)=2$ and suppose that $\Omega _w$ has non-constant non-active sections. It follows from the construction of (\ref{3.8}) that if $[j,k]$ is an active section of $\Omega_{v_{i-1}}$ and is a non-active section of $\Omega_{v_i}$ then $H^{[i]}[j,k]\doteq H^{[i+1]}[j,k]\doteq\ldots \doteq H^{[w]}[j,k]$. Therefore, (\ref{3.9}) and the definition of $\ss(\Omega _v)$ imply that the word $h_1\ldots h_{\rho_w}$ can be subdivided into subwords $h[1,i_1],\ldots ,h[i_{l'-1},i_{\rho_w}]$, such that for any $l$ either
$H^{[w]}[i_l,i_{l+1}]$ has length $1$, or the word $h[i_l,i_{l+1}]$ does not appear in basic and coefficient equations, or
\begin{equation}\label{3.29}
H^{[w]}[i_l,i_{l+1}]\doteq P_l^r \cdot P_l';\quad P_l\doteq P_l'P_l''; \quad r\ge \rho_w \max\limits_{\langle \P,R\rangle  } \left\{ f_{0}(\Omega_w, \P, R)\right\},
\end{equation}
where $P_l$ is a period, and the maximum is taken over all regular periodic structures of $\widetilde{\Omega}_w$. Therefore, if we choose $P_{l}$ of maximal length, then $\widetilde{\Omega }_w$ is  singular or strongly singular with respect to the periodic structure $\P(H^{[w]},P_l)$. Indeed, suppose that it is regular with respect to this periodic structure. Then, as in $H^{[w]}[i_l,i_{l+1}]$ one has $i_{l+1}-i_l\le \rho_w$, so (\ref{3.29}) implies that there exists $h_k$ such that $|H^{[w]}_k|\ge f_{0}(\Omega_w, \P, R)$. By Lemma \ref{lem:23-2}, this contradicts the minimality of the solution $H^{[w]}$.

This finishes the proof of Proposition \ref{3.4}.

\section{From the coordinate group $\GG_{R(\Omega^*)}$ to proper quotients: \newline the decomposition tree $T_{\dec}$ and the extension tree $T_{\ext}$.} \label{sec:7}

\subsection{The decomposition tree $T_{\dec}(\Omega )$}\label{5.5.5}

We proved in the previous section that for every solution $H$ of a generalised equation $\Omega$, the path $\p(H)$ associated to the solution $H$ ends in a final leaf $v$ of the tree $T_0(\Omega )$, $\tp(v)=1,2$. Furthermore, if $\tp(v)=2$ and the generalised equation $\Omega_{v}$ contains non-constant non-active sections, then the generalised equation $\Omega _v$ is singular or strongly singular with respect to the periodic structure ${\P}(H^{(v)},P)$, see Proposition \ref{3.4}.

The essence of the decomposition tree $T_{\dec}(\Omega)$ is that to every solution $H$ of $\Omega$ one can associate the path $\p(H)$ in \glossary{name={$T_{\dec}(\Omega)$}, description={the decomposition tree of $\Omega$}, sort=T} $T_{\dec}(\Omega)$ such that either all sections of the generalised equation $\Omega_u$ corresponding to the leaf $u$ of $T_{\dec}(\Omega)$  are non-active constant sections or the coordinate group of $\Omega_u$ is a proper quotient of $\GG_{R(\Omega^*)}$.

We summarise the results of this section in the proposition below.
\begin{prop}\label{prop:dectree}
For a {\rm(}constrained{\rm\/)} generalised equation $\Omega=\Omega_{v_0}$, one can effectively construct a finite oriented rooted at $v_0$ tree $T_{\dec}$, $T_{\dec}=T_{\dec}(\Omega_{v_0})$, such that:
\begin{enumerate}
\item The tree $T_0(\Omega)$ is a subtree of the tree $T_{\dec}$.
\item To every vertex $v$ of $T_{\dec}$ we assign a recursive group of automorphisms $A(\Omega_v)$.
\item For any solution $H$ of a generalised equation $\Omega $ there exists a leaf $u$ of the tree $T_{\dec}$, $\tp(u)=1,2$, and a solution $H^{[u]}$ of the generalised equation $\Omega _u$ such that
\begin{itemize}
    \item  $\pi_H= \sigma_0 \pi(v_0,v_1)\sigma_{1} \ldots  \pi(v_{n-1},u)\sigma_n \pi_{H^{[u]}}$, where $\sigma_i\in A(\Omega_{v_i})$;
    \item  if $\tp(u)=2$, then all non-active sections of $\Omega_u$ are constant sections.
\end{itemize}
\end{enumerate}
\end{prop}

To obtain $T_{\dec}(\Omega )$ we add some edges labelled by proper epimorphisms (to be described below) to the final leaves of type $2$ of $T_0(\Omega)$ so that the corresponding generalised equations contain non-constant non-active sections.

The idea behind the construction of this tree is the following. Lemmas \ref{lem:23-1} and \ref{lem:23-1ss} state that given a generalised equation $\Omega$ singular or strongly singular with respect to a periodic structure $\langle\P, R\rangle$, there exist finitely many proper quotients of the coordinate group $\GG_{R(\Omega^*)}$ such that for every $P$-periodic solution $H$ of the generalised equation $\Omega$ such that $\P(H,P)=\langle\P, R\rangle$, an $\AA(\Omega)$-automorphic image $H^+$ of $H$ is a solution of a proper equation. In other words, the $\GG$-homomorphism $\pi_{H^+}: \GG_{R(\Omega^*)} \to \GG$ factors through one of the finitely many proper quotients of $\GG_{R(\Omega^*)}$.

Recall that given a coordinate group of a system of equations, the homomorphisms $\pi$ from this coordinate group to the coordinate groups of constrained  generalised equations determined by partition tables (see Equation (\ref{eq:hompt}) and discussion in Section \ref{sec:relcoordgrgrtr}), are, in general, just homomorphisms (neither injective nor surjective). Our goal here is to prove that in fact, the solution $H^+$ is a solution of one of the finitely many proper generalised equations such that the homomorphism from the coordinate group of the system of equations to the coordinate group of the generalised equation is an epimorphism.

\medskip

Let $v$ be a final leaf of type $2$ of $T_0(\Omega)$ such that $\Omega_v$ contains non-constant non-active sections. For every periodic structure $\langle \P, R\rangle$ on $\Omega_v$ such that $\Omega_v$ is either singular or strongly singular with respect to $\langle \P, R\rangle$, we now describe the vertices that we introduce in the tree $T_{\dec}$ (these vertices will be leaves of $T_{\dec}$), the generalised equations associated to these new vertices and the epimorphisms corresponding to the edges that join the leaf $v$ with these vertices. We then prove that for every $P$-periodic solution $H$ of the generalised equation $\Omega_v$ such that $\Omega_v$ is singular or strongly singular with respect to $\P(H,P)$, there exists an $\AA(\Omega)$-automorphic image $H^+$ of $H$ so that $H^+$ is a solution of one of the generalised equations associated to the new vertices.

\bigskip

We first consider the case when $\Omega_v$ is strongly singular of type (b) with respect to the periodic structure $\langle \P, R\rangle$ (see Definition \ref{defn:singreg}).

In Lemma \ref{lem:23-1ss}, we proved that every $P$-periodic solution $H$ of a generalised equation $\Omega$ strongly singular (of type $(b)$) with respect to a periodic structure $\langle\P, R\rangle$ such that $\P(H,P)=\langle\P, R\rangle$, is a solution of a system of equations $\Omega$ obtained from $\Omega$ by adding new equations. We now explicitly construct a single generalised equation $\Omega_v(\P, R,\{g_i\})$ so that every $P$-periodic solution $H$ of a generalised equation $\Omega$ strongly singular (of type $(b)$) with respect to a periodic structure $\langle\P, R\rangle$ such that $\P(H,P)=\langle\P, R\rangle$, is a solution of $\Omega_v(\P, R,\{g_i\})$.

Let $g_q\in \{g_i\}$ be an element of the family $\{g_i\}$ constructed in Lemma \ref{lem:23-1ss}. As shown in the proof of Lemma \ref{lem:23-1ss}, the element $g_q$ is a commutator of the form $[\varphi_q(h_{i,q}),h_{j,q}]$, where $\varphi_q$ is a generator of the automorphism group $\AA(\Upsilon_v)$, see parts (\ref{spl4}), (\ref{spl5}) of Lemma \ref{2.10''} and Definition \ref{defn:AA}. Let $\varphi_q(h_{i,q})=h_{i_1,q}\cdots h_{i_{k_q},q}$. Define the generalised equation $\Omega_v(\P, R,\{g_i\})=\langle \Upsilon_v(\P, R,\{g_i\}), \Re_{\Upsilon_v(\P, R,\{g_i\})}\rangle$ as follows. Set $\Upsilon_v(\P, R,\{g_i\})=\Upsilon_v$ and define $\Re_{\Upsilon_v(\P, R,\{g_i\})}$ to be the minimal subset of $h\times h$ that contains $\Re_{\Upsilon_v}$, $\{(h_{i_l,q}, h_{j,q})\mid l=1,\dots, k_q\}$ for all $q$, is symmetric and satisfies the condition ($\star$) from Definition \ref{defn:Re}.

The natural homomorphism $\pi_{v,\{g_i\}}:\GG_{R(\Omega^*_v)}\to \GG_{R({\Omega_v(\P, R,\{g_i\})}^\ast)}$ is surjective. Moreover, by construction, $\pi_{v,\{g_i\}}(h(g_q))=1$ for all $g_q\in \{g_i\}$, and therefore $\pi_{v,\{g_i\}}$ is a proper epimorphism.

We introduce a new vertex $w$ and associate the generalised equation $\Omega_w=\Omega_v(\P, R,\{g_i\})$ to it. The vertex  $w$ is a leaf of the tree $T_{\dec}$ and is joined to the vertex $v$ by an edge $v\to w$ labelled by the epimorphism $\pi(v,w)=\pi_{v,\{g_i\}}$.

We now show that every $P$-periodic solution $H$ of $\Omega_v$ such that $\Omega_v$ is strongly singular of type (b) with respect to $\P(H,P)$, is a solution of the generalised equation $\Omega_w$. Obviously, $H$ is a solution of $\Upsilon_w$. We are left to prove that if $\Re_{\Upsilon_w}(h_i,h_j)$, then $H_i\lra H_j$.

In the proof of  Lemma \ref{lem:23-1ss}, on one hand we have shown that  $H_{i,q}\lra H_{j,q}$ and that $\az(H_{i,q})=\az(P)$, thus $\az(P)\lra H_{j,q}$. On the other hand, we have shown that for every generator $\varphi_q$ of the automorphism group $\AA(\Omega_v)$, the image $\varphi_q(h_{i,q})$ is a word $h_{i_1,q}\cdots h_{i_k,q}$ in variables $\{h_{i_l,q}\mid h_{i_l,q}\in \sigma, \sigma \in \P, l=1,\dots, k\}$. Since for every solution $H$ on has $\az(H_{i_l,q})\subseteq \az(P)$ and $\az(P)\lra H_{j,q}$, we get that $H_{i_l,q}\lra H_{j,q}$ and the statement follows.

\bigskip

Let us consider the case when $\Omega_v$ is strongly singular of type (a) with respect to the periodic structure $\langle \P, R\rangle$ (see Definition \ref{defn:singreg}).

In Lemma \ref{lem:23-1ss}, we proved that every $P$-periodic solution $H$ of a generalised equation $\Omega$ strongly singular (of type $(a)$) with respect to a periodic structure $\langle\P, R\rangle$ such that $\P(H,P)=\langle\P, R\rangle$, is a solution of the proper system of equations $\{\Omega_v \cup\{g_i\}\}^*$. Notice that, a priori, $H$ is a solution of the system of equations $\{\Omega_v \cup\{g_i\}\}^*$ over $\GG$. Our goal is to construct a generalised equation ${\Omega_v(\P, R,\{g_i\},\TB)}$ in such a way that $H$ is a solution of the generalised equation ${\Omega_v(\P, R,\{g_i\},\TB)}$ and the homomorphism from $\GG_{R({\Omega_v^\ast)}}$ to $\GG_{R({\Omega_v(\P,R,\{g_i\},\TB)}^\ast)}$ is a proper epimorphism.

Below we use the notation of Section \ref{sec:redgrmon}. For the  system of equations $\{g_i=1\}$, where the elements $g_i$ are defined in Lemma \ref{lem:23-1ss}, consider the subset $\PT'$ of the set $\PT$ of all $\GG$-partition tables of the system $\{g_i\}$, of the form $(V,\GG\ast F(z_1,\dots,z_p))$ that satisfy the following condition
\begin{equation}\label{3.31}
z_1, \dots, z_p \in \langle V_{i,j} \rangle.
\end{equation}
For any generalised equation $\Upsilon_\TB$, $\TB\in \PT'$, the homomorphism $\pi_{\Upsilon_\TB}:\GG_{R(\{g_i\})} \to \GG_{R(\Upsilon^*_\TB)}$ induced by the map $h_i\mapsto P_{h_i}(h',\cA)$ (see Section \ref{sec:relcoordgrgrtr} for definition) is surjective. Consider the set of generalised equations $\{ \Upsilon_\TB \mid \TB \in \PT' \}$ over the free monoid $\FF$. Note that in general, this set of generalised equations should be considered over $\Tr$.  However, we will show that for $P$-periodic solutions $H$ it suffices to consider the set of generalised equations $\{ \Upsilon_\TB \mid \TB \in \PT' \}$ over $\FF$.

Construct the generalised equation $\Omega_v(\P,R,\{g_i\},\TB)=\langle \Upsilon_v(\P,R,\{g_i\},\TB),\Re_{\Upsilon_{v}(\P,R,\{g_i\},\TB)}\rangle$ as follows. The set of items $\tilde h= h\cup h'$ of $\Omega_{v}(\P,R,\{g_i\},\TB)$ is the disjoint union of the items of $\Upsilon_v$ and $\Upsilon_\TB$; the set of coefficient equations of $\Omega_v(\P,R,\{g_i\},\TB)$ is the disjoint union of the coefficient equations of $\Upsilon_v$ and $\Upsilon_\TB$. The set of basic equations of $\Omega_v(\P,R,\{g_i\},\TB)$ consists of: the basic equations of $\Upsilon_v$, the basic equations of $\Upsilon_\TB$, and basic equations of the form $h_k= P_{h_k}(h',\cA)$, where in the left hand side of this equation $h_k$ is treated as a variable of $\Upsilon_v$ and $P_{h_k}(h',\cA)$ is a label of a section of $\Upsilon_\TB$.

The natural homomorphism $\pi'_{v,\{g_i\},\TB}:\GG_{R(\Upsilon^*_v)}\to \GG_{R({\Upsilon_v(\P, R,\{g_i\},\TB)}^\ast)}$, induced by the map
$$
\varpi: h_i\mapsto P_{h_i}(h',\cA),
$$
is surjective, since $\pi_{\Upsilon_\TB}$ is.

The set of relations $\Re_{\Upsilon_{v}(\P,R,\{g_i\},\TB)}\subseteq \tilde h\times \tilde h$ is defined as follows. Let $\varpi(h_i)=h_{i_1}'\cdots h_{i_k}'$ and $\varpi(h_j)=h_{j_1}'\cdots h_{j_l}'$. If $\Re_{\Upsilon_v}(h_i,h_j)$, then we set $\Re_{\Upsilon_{v}(\P,R,\{g_i\},\TB)}(h_{i_m}',h_{j_n}')$ for all $m=1,\dots, k$ and $n=1,\dots,l$. The set $\Re_{\Upsilon_{v}(\P,R,\{g_i\},\TB)}$ is defined as the minimal subset of $\tilde h\times \tilde h$ that contains the above defined set, the set $\Re_{\Upsilon_v}$, is symmetric and satisfies the condition ($\star$) from Definition \ref{defn:Re}. Note that by condition ($\star$), the set $\Re_{\Upsilon_{v}(\P,R,\{g_i\},\TB)}$ is independent of the choice of the map $\varpi$.

The natural homomorphism $\pi_{v,\{g_i\},\TB}:\GG_{R(\Omega^*_v)}\to \GG_{R({\Omega_v(\P, R,\{g_i\},\TB)}^\ast)}$ is surjective, since $\pi'_{v,\{g_i\},\TB}$ is. Moreover, by construction, $\pi_{v,\{g_i\},\TB}(h(g_q))=1$ for all $g_q\in \{g_i\}$, and therefore $\pi_{v,\{g_i\},\TB}$ is a proper epimorphism.

For every partition table $\TB\in \PT'$ as above, we introduce a new vertex $w_\TB$ and associate the generalised equation $\Omega_{w_\TB}=\Omega_v(\P, R,\{g_i\},\TB)$ to this vertex. The vertex  ${w_\TB}$ is  a leaf of the tree $T_{\dec}$ and is joined to $v$ by an edge $v\to w_\TB$ labelled by the epimorphism $\pi_{v, \{g_i\},\TB}$.

We now show that every $P$-periodic solution $H$ of $\Omega_v$ such that $\Omega_v$ is strongly singular of type (a) with respect to $\P(H,P)$ factors through one of the solutions of generalised equations $\Omega_v(\P,R,\{g_i\},\TB)$ for some choice of $\TB$. By Lemma \ref{2.9}, the solution $H$ is a solution of the  system of equations $\{g_q=[h(\cc_{1,q}), h(\cc_{2,q})]=1\mid g_q\in \{g_i\}\}$. It follows that if we write the system $\{[h(\cc_{1,q}), h(\cc_{2,q})]=1\mid g_q\in \{g_i\}\}$ in the form (\ref{*}), and if $r_1r_2\cdots r_l=1$ is an equation of this system and $H_{j_1},\dots, H_{j_l}$ are the respective components of $H$, then the word $H_{j_1}\cdots H_{j_l}$ is trivial in the free group $F(\cA)$, and in the product of two consecutive subwords $H_{j_r}$ and $H_{j_{r+1}}$ either there is no cancellation, or one of these words cancels completely, that is $H_{j_r}\doteq W\cdot {\left(H_{j_{r+1}}\right)}^{-1}$ or, vice-versa $H_{j_{r+1}}\doteq {\left(H_{j_{r}}\right)}^{-1}\cdot W$. This shows that the $\GG$-partition table of the system $\{[h(\cc_{1,q}), h(\cc_{2,q})]=1\mid g_q\in \{g_i\}\}$ for which the solution $H$ factors through satisfies condition (\ref{3.31}).

\bigskip

Finally, we consider the case when $\Omega_v$ is singular with respect to the periodic structure $\langle \P, R\rangle$ (see Definition \ref{defn:singreg}).

In Lemma \ref{lem:23-1}, we proved that for every $P$-periodic solution $H$ of a generalised equation $\Omega$ singular with respect to a periodic structure $\langle\P, R\rangle$ such that $\P(H,P)=\langle\P, R\rangle$ there exists an $\AA(\Omega)$-automorphic image $H^+$ of $H$ that is a solution of a proper equation. Though $H^+$ is a solution of $\Omega^*$, it may be not a solution of $\Omega$, see the example given in Section \ref{sec:explsing}. We construct the generalised equation ${\Omega_v(\P, R,\cc,\TB)}$ in such a way that $H^+$ induces a solution of ${\Omega_v(\P, R,\cc,\TB)}$ and the homomorphism from $\GG_{R({\Omega_v^\ast)}}$ to $\GG_{R({\Omega_v(\P, R,\cc,\TB)}^\ast)}$ is a proper epimorphism.

More precisely, consider a triple $(\langle \P,R\rangle,\cc,\TB)$, where
\begin{itemize}
    \item $\Omega_v$ is singular with respect to the periodic structure $\langle{\P}, R\rangle$,
    \item  $\cc$ is a cycle in $\Gamma$, $\cc\in\{\cc_1,\dots,\cc_r\}$, where $\{\cc_1,\dots,\cc_r\}$ is the set of cycles from Lemma \ref{lem:23-1},
    \item and $\TB$ is a partition table from the set $\PT'$ defined below.
\end{itemize}
For every such triple, we construct a generalised equation ${\Omega_v(\P, R,\cc,\TB)}$ and a proper epimorphism $\pi_{v,\cc,\TB}:\GG_{R(\Omega^*_v)}\to \GG_{R({\Omega_v(\P, R,\cc,\TB)}^\ast)}$, put an edge $v\to u$, $\Omega_u=\Omega_v(\P, R,\cc,\TB)$ and show that every solution $H^+$ of $\Omega_v$ is a solution of one of the generalised equations $\Omega_v(\P, R,\cc,\TB)$ for some $\cc$ and $\TB$. The vertex $u$ is a leaf of $T_{\dec}(\Omega )$.

The way we proceed is the following. By part (\ref{it:23-13}) of Lemma \ref{lem:23-1}, we know that $H^+$ is a solution of the generalised equation $\{h(\mu)=h(\Delta(\mu))\mid \mu \notin \P\}$. This fact pilots the construction of the generalised equation $\check{\Upsilon}_v$ below. On the other hand, we consider the system of equations $\bS$ over $\GG$ corresponding to the bases from $\P$ and the cycle $\cc$ (see below). Since the solution $H^+$ satisfies statement (\ref{it:23-13}) of Lemma \ref{lem:23-1}, we can prove that the $\GG$-partition table $\TB$ associated to $H^+$ satisfies property (\ref{3.31}). For the partition tables $\TB$ that satisfy condition (\ref{3.31}) and the corresponding generalised equations $\Upsilon_\TB$, the homomorphism $\pi_{\Upsilon_\TB}:\GG_{R(\bS)} \to \GG_{R(\Upsilon^*_\TB)}$ is surjective. We construct the generalised equation ${\Omega_v(\P, R,\cc,\TB)}$ from $\Upsilon_\TB$ and $\check{\Upsilon}_v$ by identifying the items they have in common. We refer the reader to Section \ref{sec:expldec} for an example of the constructions given below.

We now formalise the construction of $\GG_{R({\Omega_v(\P, R,\cc,\TB)}^\ast)}$ and $\pi_{v,\cc,\TB}$.

Let $\cc$ be as above and suppose that $h(\cc)=h_{i_1}\cdots h_{i_k}$. Consider the system of equations $\bS \subset \GG[h]$ over $\GG$:
$$
\bS = \left\{ h(\cc_\mu)=1, h(\cc)=h_{i_1}\cdots h_{i_k}=1 \mid \mu \in \P \right\}.
$$

Below we use the notation of Section \ref{sec:redgrmon}. For the system of equations $\bS$ consider the subset $\PT'$ of the set $\PT(\bS)$ of all $\GG$-partition tables of $\bS$ of the form $(V,\GG\ast F(z_1,\dots,z_p))$ that satisfy condition (\ref{3.31}). Consider the set of generalised equations $\{ \Upsilon_\TB \mid \TB \in \PT' \}$ over the free monoid $\FF$. Note that, in general, this set of generalised equations should be considered over $\Tr$. However, we will show that for the minimal solution $H^+$ it suffices to consider the set of generalised equations $\{ \Upsilon_\TB \mid \TB \in \PT' \}$ over $\FF$. For any generalised equation $\Upsilon_\TB$, $\TB\in \PT'$, since the partition table $\TB$ satisfies condition (\ref{3.31}), it follows that the homomorphism $\pi_{\Upsilon_\TB}:\GG_{R(\bS)} \to \GG_{R(\Upsilon^*_\TB)}$ induced by the map $h_i\mapsto P_{h_i}(h',\cA)$ (see Section \ref{sec:relcoordgrgrtr} for definition) is surjective.

Let the generalised equation $\check{\Upsilon}_v$ be obtained from $\widetilde{\Upsilon}_v$ by removing all bases and items that belong to $\P$.

Construct the generalised equation $\Omega_v(\P,R,\cc,\TB)=\langle \Upsilon_v(\P,R,\cc,\TB),\Re_{\Upsilon_{v}(\P,R,\cc,\TB)}\rangle$ as follows. The set of items $\tilde h$ of $\Omega_{v}(\P,R,\cc,\TB)$ is the disjoint union of the items of $\check{\Upsilon}_v$ and $\Upsilon_\TB$; the set of coefficient equations of $\Omega_v(\P,R,\cc,\TB)$ is the disjoint union of the coefficient equations of $\check{\Upsilon}_v$ and $\Upsilon_\TB$. The set of basic equations of $\Omega_v(\P,R,\cc,\TB)$ consists of: the basic equations of $\check{\Upsilon}_v$, the basic equation of $\Upsilon_\TB$, and basic equations of the form $h_k= P_{h_k}(h',\cA)$, $h_k \notin \P$, where in the left hand side of this equation $h_k$ is treated as a variable of $\check{\Upsilon}_v$ and $P_{h_k}(h',\cA)$, $h_k \notin \P$ is a label of a section of $\Upsilon_\TB$.

The natural homomorphism $\pi'_{v,\cc,\TB}:\GG_{R(\Upsilon^*_v)}\to \GG_{R({\Upsilon_v(\P, R,\cc,\TB)}^\ast)}$, induced by a map $\varpi$:
$$
h_i\mapsto \left\{
\begin{array}{ll}
P_{h_i}(h',\cA),&\hbox{ when $h_i\in \P$,} \\
h_i, & \hbox{ otherwise};
\end{array}\right.
$$
is surjective, since $\pi_{\Upsilon_\TB}$ is.

The relation $\Re_{\Upsilon_{v}(\P,R,\cc,\TB)}\subseteq \tilde h\times \tilde h$ is defined as follows. Let $\varpi(h_i)=h_{i_1}'\cdots h_{i_k}'$ and $\varpi(h_j)=h_{j_1}'\cdots h_{j_l}'$. If $\Re_{\Upsilon_v}(h_i,h_j)$, then we set $\Re_{\Upsilon_{v}(\P,R,\cc,\TB)}(h_{i_m}',h_{j_n}')$ for all $m=1,\dots, k$ and $n=1,\dots,l$. The set $\Re_{\Upsilon_{v}(\P,R,\cc,\TB)}$ is defined as the minimal subset of $\tilde h\times \tilde h$ that contains the above defined set, the set $\Re_{\check{\Upsilon}_v}$ (the restriction of the relation $\Re_{\Upsilon_v}$ onto the set of items of $\check{\Upsilon}_v$), is symmetric and satisfies the condition ($\star$) from Definition \ref{defn:Re}. Note that, by condition ($\star$), the set $\Re_{\Upsilon_{v}(\P,R,\cc,\TB)}$ is independent of the choice of the map $\varpi$.

The natural homomorphism $\pi_{v,\cc,\TB}:\GG_{R(\Omega^*_v)}\to \GG_{R({\Omega_v(\P, R,\cc,\TB)}^\ast)}$ is surjective, since $\pi'_{v,\cc,\TB}$ is. Moreover, by construction, $\pi_{v,\cc,\TB}(h(\cc))=1$, and therefore $\pi_{v,\cc,\TB}$ is a proper epimorphism.

We now show that for every $P$-periodic solution $H$ of $\Omega_v$ such that $\Omega_v$ is singular with respect to the periodic structure $\P(H,P)$, there exists an $\AA(\Omega_v)$-automorphic image $H^+$ of $H$ such that $H^+$ factors through  one of the solutions of the generalised equations $\Omega_v(\P,R,\cc,\TB)$, for some choice of $\cc$ and $\TB$. Indeed, let $H^+$  be the solution constructed in Lemma \ref{lem:23-1} and $\cc$ be one of the cycles from Lemma \ref{lem:23-1} for which $H^+(\cc)=1$. The solution $H^+$ is a solution of the system $\bS$. All equations of $\bS$ have the form $h(\cc')=1$, where $\cc'$ is a cycle in the graph of the periodic structure $\P(H,P)$. From condition (\ref{it:23-13}) of Lemma \ref{lem:23-1}, it follows that, on one hand, the word $H_k^+$, $1\le k\le \rho$, is geodesic, and, on the other hand, that if $r_1r_2\cdots r_l=1$ is an equation  of the system $\bS$ and $H_{j_1}^+,\dots, H_{j_l}^+$ are the respective components of $H^+$, then the word $H_{j_1}^+\cdots H_{j_l}^+$ is trivial in the free group $F(\cA)$. Furthermore, in the product of two consecutive subwords $H_{j_r}^+$ and $H_{j_{r+1}}^+$ either there is no cancellation, or one of these words cancels completely, that is either $H_{j_r}^+\doteq W\cdot {\left(H_{j_{r+1}}^+\right)}^{-1}$ or $H_{j_{r+1}}\doteq {\left(H_{j_{r}}\right)}^{-1}\cdot W$. Let $\TB$ be the partition table corresponding to the cancellation scheme of the system $\bS$. The argument above shows that $\TB$ satisfies condition (\ref{3.31}).

Let $v_0=v$,  $\Omega_{v_1}=\Omega_v(\P, R,\cc,\TB)$.

The solution $H^+$ induces a solution $H^+_\bS$ of the system $\bS$. By Lemma \ref{lem:R1}, there exists a solution $H^+_{\Upsilon_\TB}$ of the generalised equation $\Upsilon_\TB$ such that the following diagram
$$
\xymatrix@C3em{
\GG_{R(\bS)} \ar[rd]_{\pi_{H^+_{\bS}}}  \ar[rr]  &     &  \GG_{R(\Upsilon_\TB^\ast)}  \ar[ld]^{\pi_{H^+_{\Upsilon_\TB}}} \\
                                                &  \GG &
}
$$
is commutative.

As $H^+$ satisfies condition (\ref{it:23-13}) of Lemma \ref{lem:23-1}, and the $H_i$'s are geodesic, so the $H^+_i$'s are geodesic. In particular, $H^+_{\Upsilon_\TB}$ is a solution of the generalised equation $\Upsilon_\TB$ over the free monoid $\FF$.

On the other hand, using again part (\ref{it:23-13}) of Lemma \ref{lem:23-1}, if we remove the components $\{H^+_k\mid h_k\in \P\}$ of the solution $H^+$, we get a solution $\check{H}^+$ of the generalised equation $\check{\Upsilon}_v$.

Combining the solutions $\check{H}^+$ and $H^+_{\Upsilon_\TB}$ we get a solution $H^{(v_1)}$ of the generalised equation $\Upsilon_v(\P, R,\cc,\TB)$.  By part (\ref{it:23-13}) of Lemma \ref{lem:23-1}, the solution $H^+$ satisfies the commutation constraints from $\Re_{\Upsilon_v}$. Therefore, from the construction it follows that $H^{(v_1)}$ is a solution of $\Omega_v(\P, R,\cc,\TB)$ and
$$
\pi_{H^{+}}=\pi(v_0,v_1) \pi_{H^{(v_1)}}.
$$
We thereby have shown that for every $P$-periodic solution $H$ of $\Omega_v$ such that $\Omega_v$ is singular with respect to the periodic structure $\P(H,P)$, there exists an $\AA(\Omega_v)$-automorphic image $H^+$ of $H$ such that $H^+$ factors through one of the solutions of the generalised equations $\Omega_v(\P,R,\cc,\TB)$.

\bigskip

To the root vertex $v_0$ of the decomposition tree  $T_{\dec(\Omega)}$ we associate the group of automorphisms $\Aut(\Omega)$ of the coordinate group $\GG_{R(\Omega_{v_0})}$, see Definition \ref{defn:Aut}. To the vertices $v$ such that $v$ is a leaf of $T_0(\Omega)$ but not of $T_{\dec}(\Omega)$ (those vertices to which we added edges), we associate the group of automorphisms generated by the groups $\AA(\Omega_v)$, see Definition \ref{defn:AA}, corresponding to all periodic structures on $\Omega_v$ with respect to which $\Omega_v$ is singular or strongly singular. To all the other vertices of $T_{\dec}(\Omega)$ we associate the trivial group of automorphisms. We denote the automorphism group associated to a vertex $v$ of the tree $T_{\dec}(\Omega)$ by \glossary{name={$A(\Omega_v)$}, description={the automorphism group associated to a vertex $v$ of the trees $T_{\dec}(\Omega)$, $T_{\ext}(\Omega)$ and $T_{\sol}(\Omega)$}, sort=A}$A(\Omega_v)$.

Naturally, we call leaves $u$ such that the paths $\p(H)$ associated to solutions $H$ of $\Omega$ end in $u$, \index{leaf!final of the tree $T_{\dec}$}\emph{final leaves of $T_{\dec}(\Omega)$}.

\subsection{Example} \label{sec:expldec}
Let $\Omega$ be the generalised equation shown on Figure \ref{quadratic}:
$$
h_1h_2h_3h_4=h_4h_5h_6h_7;\  h_1=h_5;\  h_3=h_7;\  h_2=h_8;\  h_6=h_8;\  h_8=a.
$$
Consider the periodic structure $\langle \P, R\rangle$ on $\Omega$ defined in Section \ref{sec:explperstr}. We have shown in the example given in Section \ref{sec:explperstr} that this periodic structure is singular.

In the example given in Section \ref{sec:explsing}, we considered the solution $H$:
$$
\begin{array}{llll}
H_1=(bac)^2b; &  H_3=(cba)^2c; &H_5=(bac)^2b;&  H_7=(cba)^2c;\\
H_2=a; &H_4=(bac)^6;&H_6=a;&H_8=a.
\end{array}
$$

In Section \ref{sec:explsing}, for the solution $H$ we constructed its automorphic image $H^+$ that  satisfies conditions (\ref{it:23-12}) and (\ref{it:23-13}) of Lemma \ref{lem:23-1}:
$$
\begin{array}{llll}
H_1^+=(bac)^2b; &  H_3^+=a^{-1}b^{-1}(bac)^{-2}; &H_5^+=(bac)^2b;&  H_7^+=a^{-1}b^{-1}(bac)^{-2};\\
H_2^+=a; &H_4^+=bac;&H_6^+=a;&H_8^+=a.
\end{array}
$$

Obviously, $H^+$ is a solution of $\Omega^*$.  Notice, however, that $H^+$ is not a solution of $\Omega$. Indeed, for the equation
$h_1h_2h_3h_4=h_4h_5h_6h_7$, the word $H_1^+H_2^+H_3^+H_4^+$ is not geodesic as written.

We construct the system of equations $\bS$ over $\GG$. The system $\bS$ consists of all the equations of $\Omega$ that correspond to the bases that belong to the periodic structure $\langle \P, R\rangle$, see Example \ref{sec:explperstr}, and the equation $h(\cc_{e_7})=h_1h_2h_7=1$:
$$
\bS=\{h_1h_2h_3h_4=h_4h_5h_6h_7, \ h_1=h_5,\  h_3=h_7, \ h_1h_2h_7=1\}.
$$

In the example given in  Section \ref{sec:explsing}, we showed that $H^+(\cc_{e_7})=1$ and therefore, $H^+$ is a solution of $\bS$.  The cancellation scheme for the solution $H^+$ is shown on Figure \ref{fig:perstrdiagr}.

\begin{figure}[!h]
  \centering
   \includegraphics[keepaspectratio,width=5in]{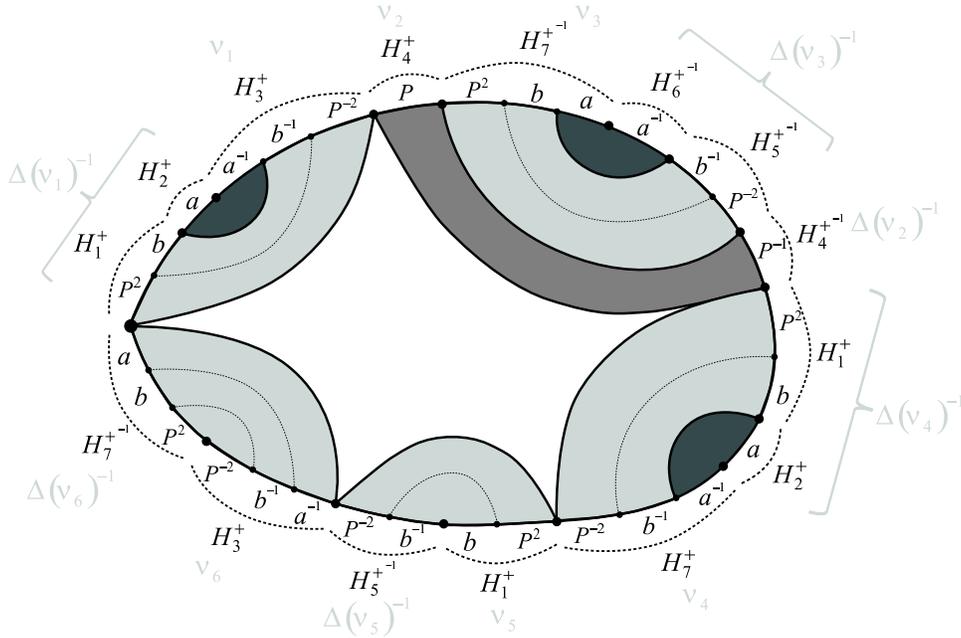}
\caption{The cancellation scheme of $H^+$} \label{fig:perstrdiagr}
\end{figure}

Notice that the partition table corresponding to the cancellation scheme shown on Figure \ref{fig:perstrdiagr} satisfies condition (\ref{3.31}).

We then construct the generalised equation $\Upsilon_\TB$ associated to the system $\bS$ and the generalised equation $\check{\Upsilon}_v$ obtained from $\widetilde{\Upsilon}_v$ by removing all bases and items that belong to $\P$. Using these two equations we construct $\Omega(\P,R,\cc,\TB)$ as shown on Figure \ref{fig:example2}.

We would like to draw the reader's attention to the fact that, though (for simplicity) the generalised equation $\Upsilon_\TB$ shown on Figure \ref{fig:example2} has not been constructed using the procedure described in Section \ref{sec:geneqT}, but the constructed generalised equation is $\approx$-equivalent to the one described there.

\begin{figure}[!h]
  \centering
   \includegraphics[keepaspectratio,width=5in]{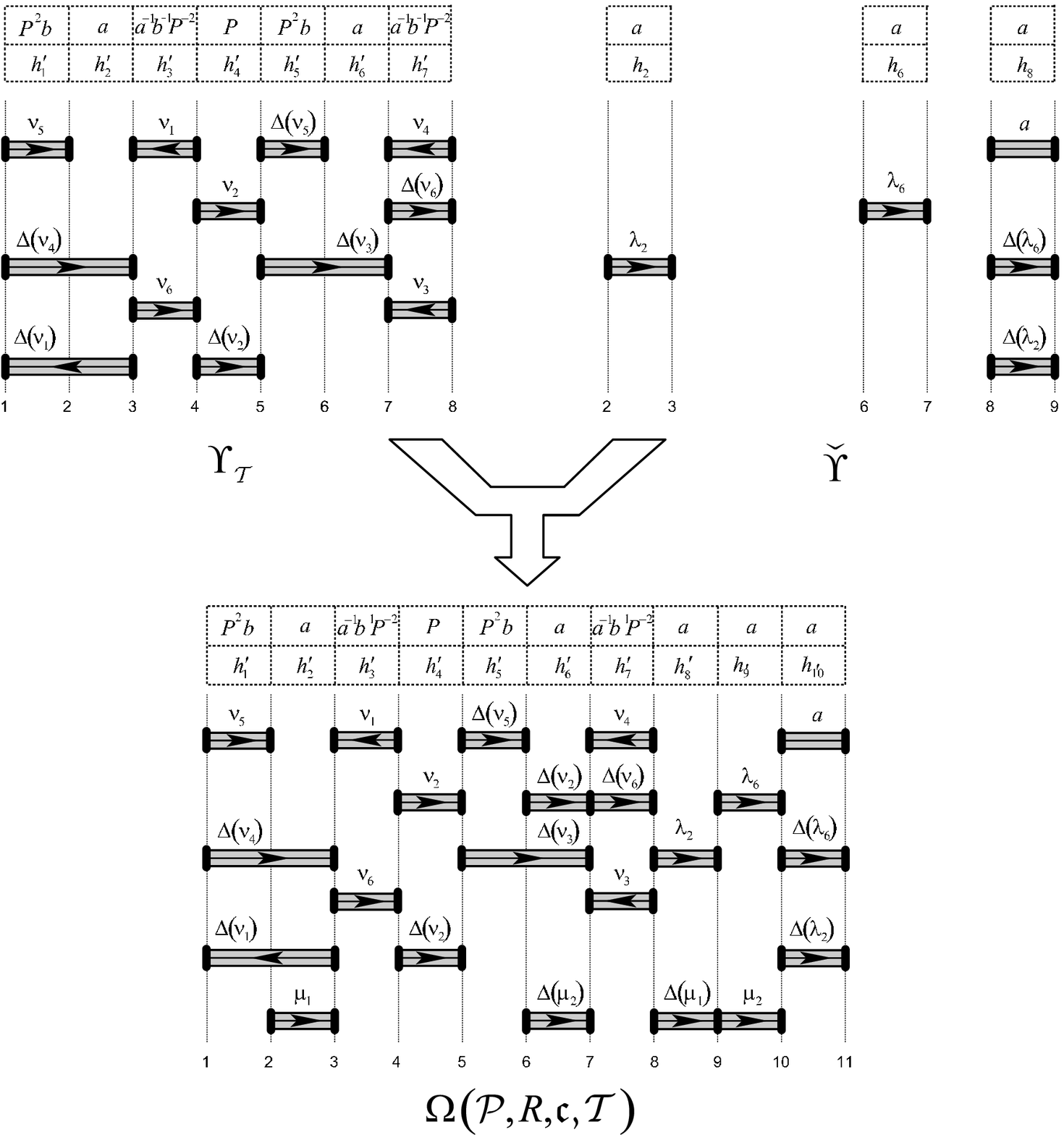}
\caption{The generalised equations $\Upsilon_\TB$, $\check{\Upsilon}$ and $\Omega(\P,R,\cc,\TB)$} \label{fig:example2}
\end{figure}

\subsection{The extension tree $T_{\ext}(\Omega)$}\label{se:5.5'}

Recall that the coordinate group $\GG_{R(\Omega_v^\ast)}$ associated to  a final leaf $v$  of $T_{\dec}(\Omega)$ is either a proper epimorphic image of $\GG_{R(\Omega^\ast)}$ or all the sections of the corresponding generalised equation $\Omega_v$ are non-active constant sections. Since partially commutative groups are equationally Noetherian and thus any sequence of proper epimorphisms of coordinate groups is finite, an inductive argument for those leaves of $T_{\dec}(\Omega)$ that are proper epimorphic image of $\GG_{R(\Omega^\ast)}$ shows that we can construct a tree \glossary{name={$T_{\ext}(\Omega)$}, description={the extension tree of $\Omega$}, sort=T}$T_{\ext}$ with the property that for every leaf $v$ of $T_{\ext}$ all the sections of the generalised equation $\Omega_v$ are non-active constant sections.

We summarise the results of this section in the proposition below.
\begin{prop}\label{prop:exttree}
For a {\rm(}constrained{\rm\/)} generalised equation $\Omega=\Omega_{v_0}$, one can effectively construct a finite oriented rooted at $v_0$ tree $T_{\ext}$, $T_{\ext}=T_{\ext}(\Omega_{v_0})$, such that:
\begin{enumerate}
\item The tree $T_{\dec}(\Omega)$ is a subtree of the tree $T_{\ext}$.
\item To every vertex $v$ of $T_{\ext}$ we assign a recursive group of automorphisms $A(\Omega_v)$.
\item For any solution $H$ of a generalised equation $\Omega $ there exists a leaf $u$ of the tree $T_{\ext}$, $\tp(u)=1,2$, and a solution $H^{[u]}$ of the generalised equation $\Omega _u$ such that
\begin{itemize}
    \item  $\pi_H= \sigma_0 \pi(v_0,v_1)\sigma_{1} \ldots  \pi(v_{n-1},u)\sigma_n \pi_{H^{[u]}}$, where $\sigma_i\in A(\Omega_{v_i})$;
    \item  the sections of $\Omega_u$ are non-active constant sections.
\end{itemize}
\end{enumerate}
\end{prop}

We define a new transformation $L_v$, which we call a \emph{leaf-extension of the tree $T_{\dec}(\Omega)$ at the leaf $v$} in the following way. If there are active sections in the generalised equation $\Omega_v$, we take the union of two trees $T_{\dec}(\Omega)$ and $T_{\dec}(\Omega_v)$ and identify the leaf $v$ of $T_{\dec}(\Omega)$ with the root $v$ of the tree $T_{\dec}(\Omega_v)$, i.e. we extend the tree $T_{\dec}(\Omega)$ by gluing the tree $T_{\dec}(\Omega_v)$ to the vertex $v$.  If all the sections of the generalised equation $\Omega_v$ are non-active, then the vertex $v$ is a leaf and $T_{\dec}(\Omega_v)$ consists of a single vertex, namely $v$. We call such a vertex $v$ \index{terminal vertex}\emph{terminal}.

We use induction to construct the extension tree $T_{\ext}(\Omega)$. Let $v$ be a non-terminal leaf of $T^{(0)} = T_{\dec}(\Omega)$. Apply the transformation $L_v$ to obtain a new tree $T^{(1)} = L_v(T_{\dec}(\Omega))$. If in $T^{(1)}$ there exists a non-terminal leaf $v_1$, we apply the transformation $L_{v_1}$, and so on. By induction we construct a strictly increasing sequence of finite trees
\begin{equation} \label{eq:5.3.1}
T^{(0)} \subset  T^{(1)} \subset \ldots \subset T^{(i)}  \subset \ldots
\end{equation}
Sequence (\ref{eq:5.3.1}) is finite. Indeed, assume the contrary, i.e. the sequence is infinite and hence the union $T^{(\infty)}$ of this sequence is an infinite tree locally finite tree. By  K\"{o}nig's lemma, $T^{(\infty)}$ has an infinite branch. Observe that along any infinite branch in $T^{(\infty)}$ one has to encounter infinitely many proper epimorphisms. This derives a contradiction with the fact that $\GG$ is equationally Noetherian.

Denote by $T_{\ext}(\Omega)$ the last term of the sequence (\ref{eq:5.3.1}).

The groups of automorphisms \glossary{name={$A(\Omega_v)$}, description={the automorphism group associated to a vertex $v$ of the trees $T_{\dec}(\Omega)$, $T_{\ext}(\Omega)$ and $T_{\sol}(\Omega)$}, sort=A}$A(\Omega_v)$ associated to vertices of $T_{\ext}(\Omega)$ are induced, in a natural way, by the groups of automorphisms associated to vertices of the decomposition trees.

Naturally, we call leaves $u$ such that the paths $\p(H)$ associated to solutions $H$ of $\Omega$ end in $u$, \index{leaf!final of the tree $T_{\ext}$}\emph{final leaves of $T_{\ext}(\Omega)$}.

\section{The solution tree $T_{\sol}(\Omega)$ and the main theorem.}\label{se:5.5}

In the previous section we have shown that the generalised equations $\Omega_v=\gpof{v}$ associated to the final leaves $v$ of $T_{\ext}(\Omega)$ contain only non-active constant sections. In other words, the coordinate group $\GG_{R(\Omega_v^*)}$ is isomorphic to
$$
\factor{\GG[h_1,\dots,h_{\rho_{\Omega_v}}]}{R(\{\hbox{coefficient  equations}\}\cup \{[h_i,h_j]\mid \Re_{\Upsilon_v}(h_i,h_j)\})}.
$$

In general, the coordinate group $\GG_{R(\Omega_v^*)}$ may be not fully residually $\GG$, see the example given in Section \ref{sec:explsol}.

The solution  tree $T_{\sol}(\Omega)$ is constructed from the tree $T_{\ext}(\Omega)$ by adding some edges, labelled by homomorphisms of the coordinate groups, to final leaves of the tree $T_{\ext}$. The new vertices $w_{1},\dots, w_n$ connected to a final leaf $v$ of $T_{\ext}$ by an edge have the coordinate groups $\GG_{w_i}$ associated to them and satisfy the following properties.
\begin{enumerate}
  \item \label{it:prs1}For every solution $H^{(v)}$ of $\Omega_v$, there exist a vertex $w_i$ and a solution $H^{(w_i)}$ such that the following diagram commutes:
$$
\xymatrix{
  \GG_{R(\Omega_v^*)} \ar[rr]^{\pi(v,w_i)} \ar[dr]_{\pi_{H^{(v)}}} &  &    \GG_{w_i} \ar[dl]^{\pi_{H^{(w_i)}}}    \\
                & \GG                 }
$$
    \item \label{it:prs2}For every solution solution $H^{(w_i)}$ there exists a homomorphism $\varphi:\GG_{R(\Omega_v^*)} \to \GG $ such that the following diagram commutes:
$$
\xymatrix{
  \GG_{R(\Omega_v^*)} \ar[rr]^{\pi(v,w_i)} \ar[dr]_{\varphi} &  &    \GG_{w_i} \ar[dl]^{\pi_{H^{(w_i)}}}    \\
                & \GG                 }
$$
    \item \label{it:prs3} The coordinate group $\GG_{w_i}$ is a fully-residually $\GG$ free partially commutative group.
\end{enumerate}

The idea behind this construction is the following.

Let $v$ be a final leaf of the tree $T_{\ext}$ and $\Omega_v$ be the constrained generalised equation associated to $v$. By definition, for any $h_i, h_j\in h^{(v)}=h$, such that $\Re_{\Upsilon_v}(h_i,h_j)$, one has that $H_i^{(v)}\lra H_j^{(v)}$ for any solution $H^{(v)}$ of $\Omega_v$, i.e. $H_i^{(v)}\in \BA(H_j^{(v)})$ and $H_j^{(v)}\in \BA(H_i^{(v)})$, see Section \ref{sec:pcgr}. Therefore, any solution of the generalised equation maps $h_i$ and $h_j$ into disjoint canonical parabolic subgroups of $\GG$ that $\lra$-commute.

We encode all canonical parabolic subgroups of the group $\GG$ and the $\lra$-commutativity relation  between them in the graph $\digamma$. There are only finitely many tuples of canonical parabolic subgroups, where the solution $H^{(v)}$ of $\Omega_v$ may map the tuple of variables $h$. The choice of the tuple of canonical parabolic subgroups is encoded by the graph homomorphism $\varphi_{v,i}$ defined below. There is a vertex $w_i$ in the tree $T_{\sol}$ for every such homomorphism $\varphi_{v,i}$.

Using the homomorphism $\varphi_{v,i}$, we construct the coordinate group $\GG_{w_i}$ that we associate to the vertex $w_i$ and define solutions corresponding to this vertex.

We then prove that properties (\ref{it:prs1}), (\ref{it:prs2}) and (\ref{it:prs3}) above hold.

We refer the reader to Section \ref{sec:explsol} for an example of the constructions described in this section.

\bigskip

Define a non-oriented graph \glossary{name={$\digamma$}, description={a non-oriented graph that encodes all canonical parabolic subgroups of the group $\GG$ and the $\lra$-commutativity relation between them}, sort=G}$\digamma$ as follows. The set of vertices $V(\digamma)$ is the set of all full subgraphs of the commutation graph $\mathcal{G}$ of $\GG$ (or, which is equivalent, canonical parabolic subgroups of the group $\GG$). There is an edge  $e\in E(\digamma)$ between two vertices $s_1$ and $s_2$ if and only $V(s_1)\lra V(s_2)$, $V(s_1),V(s_2)\subseteq \cA$.

Let $\Omega_v$ be the generalised equation associated to a final leaf $v$ of the tree $T_{\ext}$. Define a non-oriented graph \glossary{name={$\Pi_v$}, description={a non-oriented graph that encodes coefficient equations of $\Omega_v$ and the relation $\Re_{\Upsilon_v}$}, sort=P}$\Pi_v$ as follows. The set of vertices $V(\Pi_v)$ is $\cA\cup \{h_i\in h^{(v)}\mid h_i \hbox{ does not occur in coefficient equations of $\Omega_v$}\}$.  There is an edge  $e\in E(\Pi_v)$ between two vertices $v_1$ and $v_2$ if and only
$$
\begin{array}{ll}
v_1=a_i, \ v_2=a_j& \hbox{ and } [a_i, a_j]=1 \hbox{ in }\GG,\\
v_1=h_i, \ v_2=h_j& \hbox{ and } \Re_{\Upsilon_v}(h_i, h_j),\\
v_1=h_i, \ v_2=a_j& \hbox{ and there exists a coefficient equation ${h_k}^{\pm 1}=a_j$, and $\Re_{\Upsilon_v}(h_i, h_k)$.}
\end{array}
$$

Consider the set of all graph homomorphisms $\varphi_{v,i}$ from $\Pi_v$ to $\digamma$ that satisfy the following conditions:
\begin{enumerate}
\item[(I)] $\varphi_{v,i}(a_j)=\{a_j\}$, for all $a_j \in \cA$;
\item[(II)] for every $e\in E(\Pi_v)$ we have $\varphi_{v,i}(e)\in E(\digamma)$, i.e. $\varphi_{v,i}(e)$ does not collapse edges.
\end{enumerate}
Note that the set of all such homomorphisms is finite and can be effectively constructed.

For every homomorphism $\varphi_{v,i}$ we construct a new leaf $w_i$ in the tree $T_{\sol}(\Omega)$ and an edge joining $v$ and $w_i$.

The coordinate group $\GG_{w_i}$ associated to the vertex $w_i$ is obtained as follows. Let $G$ be the graph product of groups with the underlying commutation graph $\varphi_{v,i}(\Pi_v)$.

The group associated to the vertex $s_j$ of $\varphi_{v,i}(\Pi_v)$ is defined as follows. For every vertex $s_j$ of $\varphi_{v,i}(\Pi_v)$ we consider the partially commutative group $\GG(s_j)$, where $\GG(s_j)<\GG$ is the canonical parabolic subgroup of $\GG$ corresponding to the full subgraph of $\mathcal{G}$ associated to $s_j$. Consider the decomposition of $\GG(s_j)$ of the form (\ref{eq:decomp}):
$$
\GG(s_j)= \GG(s_j,I_1) \times \cdots \times \GG(s_j,I_{m_j}).
$$
Let $h^{(j)}$ be the set of vertices $h_k\in V(\Pi_v)$ such that $\varphi_{v,i}(h_k)=s_j$. We treat every $h_k\in h^{(j)}$ as a tuple of variables $(h_{k,1},\dots, h_{k,m_j})$, $h_{k,l}\in h^{(j,I_l)}$.
Consider the group
\begin{equation} \label{eq:vertgr}
\GG(s_j,I_1) [h^{(j,I_1)}]\times \cdots \times \GG(s_j,I_m) [h^{(j,I_{m_j})}],
\end{equation}
where if $\GG(s_j,I_l)$ is free abelian, then  $\GG(s_j,I_l) [h^{(j,I_l)}]=\GG(s_j,I_l)\times \langle h^{(j,I_l)}\rangle$, where $\langle h^{(j,I_l)}\rangle$ is the free abelian group with basis $h^{(j,I_l)}$, and if $\GG(s_j,I_l)$  is non-abelian, then
$\GG(s_j,I_l) [h^{(j,I_l)}]=\GG(s_j,I_l) \ast F(h^{(j,I_l)})$. The group we associate to the vertex $s_j$ is the group given in (\ref{eq:vertgr}). It follows, since the graph product of partially commutative groups is again a partially commutative group, that $G$ is a free partially commutative group.

We now turn the group $G$ into a $\GG$-group as follows
\glossary{name={$\GG_{w_i}$}, description={fully residually $\GG$ partially commutative group associated to the leaf of the solution tree $T_{\sol}(\Omega)$}, sort=G}
$$
\GG_{w_i}=\left< \GG, G\,\left| \,C, \, [C_\GG(\GG(s_j,I_k)),h^{(j,I_k)}]=1\right. \hbox{ for all }j, k\right>,
$$
where the relations in $C$ identify the subgroups $\GG(s_j)$ with the corresponding subgroups of $\GG$. This is the group $\GG_{w_i}$ that we associate to the leaf $w_i$ of the tree $T_{\ext}(\Omega)$. Note that, since $G$ is a partially commutative group and the centraliser of a canonical parabolic subgroup is again a canonical parabolic subgroup, the group $\GG_{w_i}$ is a free partially commutative group.

A $\GG$-homomorphism $\psi$ from $\GG_{w_i}$ to $\GG$ such that for every $h_k \in h^{(j)}$ one has $\psi(h_k)\in \GG(s_j)$ is termed a \index{solution!associated to the vertex}\emph{solution associated to the vertex $w_i$}. In other words, by taking restrictions, every solution  associated to the vertex $w_i$, induces a tuple of homomorphisms $(\psi_1,\dots, \psi_{|V(\varphi_{v,i}(\Pi_v))|})$, where $\psi_j$ is a $\GG(s_j)$-homomorphism from $\GG(s_j)[h^{(j)}]$ to $\GG(s_j)$. Conversely, any $|V(\varphi_{v,i}(\Pi_v))|$-tuple of $\GG(s_j)$-homomorphisms from $\GG(s_j)[h^{(j)}]$ to $\GG(s_j)$ uniquely defines a solution $\psi$ associated to the vertex $w_i$. The homomorphisms $\psi_j$ are called \emph{components} of $\psi$.

Every homomorphism $\psi_j$, in turn, induces an $m_j$-tuple of $\GG(s_j,I_l)$-homomorphisms from $\GG(s_j,I_l) [h^{(j,I_l)}]$ to $\GG(s_j,I_l)$, $l=1,\dots, m_j$. Conversely, any $m_j$-tuple of $\GG(s_j,I_l)$-homomorphisms from $\GG(s_j,I_l) [h^{(j,I_l)}]$ to $\GG(s_j,I_l)$ uniquely defines a component $\psi_j$ of a solution $\psi$.

The map $h_k\mapsto h_{k,1} \cdots h_{k,m_j}$ induces a $\GG$-homomorphism $\pi(v,w_i)$ from $\GG_{R(\Omega_{v}^*)}$ to $\GG_{w_i}$. Indeed, by Proposition \ref{prop:soltrdiscr} below, $\GG_{w_i}$ is $\GG$-discriminated by $\GG$. Then, by Theorem \ref{thm:ircoordgr}, $\GG_{w_i}$ is the coordinate group of an irreducible algebraic set. Since every equation from the system $\Omega_{v}^*$ is mapped to the trivial element of $\GG_{w_i}$ and since $\GG_{w_i}$ is a coordinate group, $\pi(v,w_i)$ is a homomorphism.

It follows, therefore,  that property (\ref{it:prs2}) described in the beginning of this section holds, i.e. that for
every solution solution $H^{(w_i)}$ there exists a homomorphism $\varphi:\GG_{R(\Omega_v^*)} \to \GG $ such that the following diagram commutes:
$$
\xymatrix{
  \GG_{R(\Omega_v^*)} \ar[rr]^{\pi(v,w_i)} \ar[dr]_{\varphi} &  &    \GG_{w_i} \ar[dl]^{\pi_{H^{(w_i)}}}    \\
                & \GG                 }
$$
It suffices to take $\varphi:\GG_{R(\Omega_v^*)} \to \GG $ to be the composition of $\pi(v,w_i)$ and $\pi_{H^{(w_i)}}$.

We now prove that property (\ref{it:prs1}) described in the beginning of this section holds, i.e. that for every solution $H$ of $\Omega_{v}$ there exists $i$ and a solution $\psi$ associated to the vertex $w_i$ such that the following diagram is commutative:
$$
\xymatrix@C3em{
 \GG_{R(\Omega_v^*)}  \ar[rd]_{\pi_H} \ar[rr]^{\pi(v,w_i)}  &  &\GG_{w_i} \ar[ld]^{\psi}
                                                                             \\
                               &  \GG &
}
$$

Given a solution $H$ of $\Omega_v$, we construct a non-oriented graph \glossary{name={$\gimel$}, description={a non-oriented graph that describes the canonical parabolic subgroups associated to the solution of a generalised equation}, sort=G}$\gimel$. There are two types of vertices in $\gimel$. For every $a_j\in \cA$, we add a vertex labelled by $a_j$. For all distinct sets $\az(H_j)$, we introduce a vertex of $\gimel$ labelled by a full subgraph of $\mathcal{G}$ generated by $\az(H_j)$. There is an edge between two vertices corresponding to the sets $\az(H_j)$ and $\az(H_k)$ if and only if $\az(H_j) \lra \az(H_k)$. There is an edge between two vertices corresponding to $a_j$ and $a_k$ if and only if $a_j\lra a_k$. There is an edge between two vertices $v_1$, corresponding to the set $\az(H_j)$, and $v_2$, corresponding to $a_k$,  if and only if  $\az(H_j) \lra a_k$. By construction, the graph $\gimel$ is a full subgraph of $\digamma$.

The map $\phi_{v,i}:h_j\mapsto v(h_j)$, where $v(h_j)$ is labelled by the full subgraph of $\mathcal{G}$ generated by $\az(H_j)$ and
$\phi_{v,i}:a_j\mapsto a_j$, where $a_j\in \cA$ extends, to an epimorphism $\varphi$ from $\Pi_v$ to $\gimel$. Since $H$ is a solution of $\Omega_{v}$, if $\Re_{\Upsilon_v}(h_j,h_k)$, then $H_j \lra H_k$. Thus, the homomorphism $\varphi$ satisfies the property (II) above and hence, $\varphi=\varphi_{v,i}$ for some $i$. Setting $\psi=\pi_H$, it follows that the above diagram is commutative.

The proposition below proves that property (\ref{it:prs3}) described in the beginning of this section holds.
\begin{prop} \label{prop:soltrdiscr}
In the above notation, the groups $\GG_{w_i}$ are $\GG$-discriminated by $\GG$ by the family of solutions associated to the vertex $w_i$.
\end{prop}
\begin{proof}
We use induction on the number $N$ of vertices of the underlying graph $\varphi_{v,i}(\Pi_v)$ of the graph product $G$ to prove that the group $\GG_{w_i}$ is obtained from $\GG$ by a chain of extension of centralisers of directly indecomposable canonical parabolic subgroups. It then follows, by Proposition \ref{prop:soltree}, that the group $\GG_{w_i}$ is $\GG$-discriminated by $\GG$.

If $N=1$, then
$$
\GG_{w_1} =\left<\GG,  \GG(s_1,I_1) [h^{(1,I_1)}]\times \cdots \times \GG(s_1,I_{m_1}) [h^{(1,I_{m_1})}]\left|\, C, \, [C_\GG(\GG(s_1,I_k)),h^{(1,I_k)}]=1\right. \hbox{ for all } k\right>,
$$
where the relations in $C$ identify the subgroups $\GG(s_1,I_k)$ with the corresponding subgroups of $\GG$.

We use induction on $m_1$ to prove that the group $\GG_{w_1}$ is obtained from $\GG$ by a chain of extension of centralisers of directly indecomposable canonical parabolic subgroups. If $m_1=1$, then the statement is trivial.

Let
$$
\GG_{w_1}'=\left<\GG, \GG(s_1,I_1) [h^{(1,I_1)}]\times \cdots \times \GG(s_1,I_{m_1-1}) [h^{(1,I_{m_1-1})}],\left| \, C, \, [C_\GG(\GG(s_1,I_k)),h^{(1,I_k)}]=1\right. \hbox{ for all } k\right>.
$$
Without loss of generality we may assume that $|h^{(1,I_{m_1})}|=1$, $h^{(1,I_{m_1})}=\{h\}$. It suffices to show that $\GG_{w_1}$ is isomorphic to
\begin{equation}\label{eq:tocompare}
\langle \GG_{w_1}', h\mid \rel(\GG_{w_1}'), [h, C_{\GG_{w_1}'}(\GG(s_1,I_{m_1}))]=1\rangle.
\end{equation}
Since, by induction, $\GG_{w_1}'$ is a partially commutative group, by Theorem \ref{thm:centr},
$$
C_{\GG_{w_1}'}(\GG(s_1,I_{m_1}))=\langle C_{\GG}(\GG(s_1,I_{m_1})), h^{(1,I_k)}, k=1,\dots, m_1-1\rangle.
$$
Comparing the presentations of $\GG_{w_1}$ and the one given in (\ref{eq:tocompare}), the statement follows in the case when $N=1$.

Take a graph $\Lambda$ with $N$ vertices and consider a full subgraph $\Lambda '$ of $\Lambda$ such that $V(\Lambda)=V(\Lambda')\cup \{s_N\}$. By induction, the $\GG$-group $\GG_{\Lambda'}$ constructed by the graph $\Lambda'$ is a partially commutative group obtained from $\GG$ by a sequence of extensions of centralisers of directly indecomposable canonical parabolic subgroups.

Suppose first that $\GG(s_N)$ is directly indecomposable. Without loss of generality we may assume that $|h^{(N,I_1)}|=1$, $h^{(N,I_1)}=\{h_N\}$. We now prove that $\GG_\Lambda$ is isomorphic to
$$
\GG'=\langle \GG_{\Lambda'}, h_N\mid \rel(\GG_{\Lambda'}), [C_{\GG_{\Lambda'}}(\GG(s_N,I_{1})), h_N]=1\rangle.
$$
Since, by induction, $\GG_{\Lambda'}$ is a partially commutative group, by Theorem \ref{thm:centr},
$$
C_{\GG_{\Lambda'}}(\GG(s_N,I_{1}))=\left<
\begin{array}{l}
 C_{\GG}(\GG(s_N,I_{1})), \ h^{(j,I_k)}, \ \ k=1,\dots, m_j \\
 \hbox{ and there is an edge between $s_j$ and $s_N$ in $\Lambda$}
 \end{array}
 \right>
$$
Comparing the presentations of $\GG'$ and $\GG_{\Lambda}$, the statement follows.

If the group $\GG(s_N)$ is directly decomposable, the proof is analogous to the base of induction and the previous case.
\end{proof}

The solution tree \glossary{name={$T_{\sol}$}, description={the solution tree of a generalised equation}, sort=T}$T_{\sol}(\Omega)$ is obtained from the extension tree $T_{\ext}(\Omega)$ by extending every final leaf of $T_{\ext}(\Omega)$ as above. To each non-leaf vertex $v$ of the tree $T_{\sol}(\Omega)$, we associate the same group of automorphisms \glossary{name={$A(\Omega_v)$}, description={the automorphism group associated to a vertex $v$ of the trees $T_{\dec}(\Omega)$, $T_{\ext}(\Omega)$ and $T_{\sol}(\Omega)$}, sort=A}$A(\Omega _v)$ as in the extension tree $T_{\ext}(\Omega)$. For every leaf $w_i$ of  $T_{\sol}(\Omega)$ we associate the trivial group of automorphisms.

We are now ready to formulate the main result of this paper.
\begin{thm}\label{th:5.3.1}
Let $\Omega = \Omega(h)$ be a constrained generalised equation in variables $h$. Let $T_{\sol}(\Omega)$ be the solution tree for $\Omega$. Then the following statements hold.
\begin{enumerate}
    \item \label{it:th1}  For any solution $H$ of the generalised equation $\Omega$ there exist: a path $v_0\to v_1\to \ldots\to v_n = v$  in $T_{\sol}(\Omega)$ from the root vertex $v_0$ to a leaf $v$, a sequence of automorphisms $\sigma = (\sigma_0, \ldots, \sigma_n)$, where $\sigma_i \in  A(\Omega _{v_i})$ and a solution $H^{(v)}$ associated to the vertex $v$, such that
\begin{equation} \label{eq:5.3.3}
\pi_H=\Phi_{\sigma,H^{(v)}} = \sigma_0 \pi(v_0,v_1)\sigma_{1} \ldots  \pi(v_{n-1},v_n)\sigma_n \pi_{H^{(v)}}.
\end{equation}
    \item \label{it:th2} For any path $v_0\to v_1\to \ldots\to v_n = v$  in $T_{\sol}(\Omega)$ from the root vertex $v_0$ to a leaf $v$, any sequence of automorphisms $\sigma =(\sigma_0, \ldots, \sigma_n )$, $\sigma_i \in  A(\Omega _{v_i})$, and any solution $H^{(v)}$ associated to the vertex $v$, the homomorphism $\Phi_{\sigma,H^{(v)}}$ is a solution of $\Omega^\ast$. Moreover, every solution of $\Omega^\ast$ can be obtained in this way.
\end{enumerate}
\end{thm}
\begin{proof}
The statement follows from the construction of the tree $T_{\sol}(\Omega,\Lambda)$.
\end{proof}

We call leaves $v$ of the tree $T_{\sol}(\Omega)$ such that there exists a path from $v_0$ to $v$ as in statement (\ref{it:th1}) of Theorem \ref{th:5.3.1} \index{leaf!final of the tree $T_{\sol}$}\emph{final leaves of the tree $T_{\sol}(\Omega)$}.

\begin{thm}\label{ge}
Let $\GG$ be the free partially commutative group with the underlying commutation graph $\mathcal{G}$ and let $\widehat{G}$ be a finitely generated {\rm(}$\GG$-{\rm)}group. Then the set of all {\rm(}$\GG$-{\rm)}homomorphisms  $\Hom(\widehat{G},\GG)$ {\rm(}$\Hom_\GG(\widehat{G},\GG)$, correspondingly{\rm)} from $\widehat{G}$ to $\GG$ can be effectively described by a finite rooted tree. This tree is oriented from the root, all its vertices except for the root and the leaves are labelled by coordinate groups of generalised equations. The final leaves of the tree are labelled by fully residually $\GG$ partially commutative groups $\GG_{w_i}$.

Edges from the root vertex correspond to a finite number of {\rm(}$\GG$-{\rm)}homomorphisms from $\widehat{G}$ into coordinate groups of generalised equations.  To each vertex group we assign the group of automorphisms $A(\Omega_v)$. Each edge {\rm(}except for the edges from the root and the edges to the final leaves{\rm)} in this tree is labelled by an epimorphism, and all the epimorphisms are proper. Every {\rm(}$\GG$-{\rm)}homomorphism from $\widehat{G}$ to $\GG$ can be written as a composition of the {\rm(}$\GG$-{\rm)}homomorphisms corresponding to the edges, automorphisms of the groups assigned to the vertices, and a {\rm(}$\GG$-{\rm)}homomorphism $\psi=(\psi_j)_{j\in J}$, $|J|\le 2^\rr$  into $\GG$, where $\psi_j:\HH_j[Y]\to \HH_j$ and $\HH_j$ is the free partially commutative subgroup of $\GG$ defined by some full subgraph of $\mathcal{G}$.
\end{thm}
\begin{proof}
Suppose first that $\widehat{G}$ is the finitely generated $\GG$-group $\GG$-generated by $X$, i.e. $\widehat{G}$ is generated by $\GG\cup X$. Let $S$ be the set of defining relations of $\widehat{G}$. We treat $S$ as a system of equations (possibly infinite) over $\GG$ (with coefficients from $\GG$). Since $\GG$ is equationally Noetherian, there exists a finite subsystem $S_0\subseteq S$ such that $R(S)=R(S_0)$. Since every $\GG$-homomorphism from $\widehat{G}=\factor{\langle \GG, X\rangle}{\ncl\langle S\rangle}$ to $\GG$ factors through $\GG_{R(S_0)}$, it suffices to describe the set of all homomorphisms from $\GG_{R(S_0)}$ to $\GG$. We now run the process for the system of equations $S_0$. Since there is a one-to-one correspondence between solutions of the system $S_0$ and homomorphisms from $\GG_{R(S_0)}$ to $\GG$, by Theorem \ref{th:5.3.1}, we obtain a description of $\Hom_\GG(\GG_{R(S_0)},\GG)$.

Suppose now that $\widehat{G}$ is not a $\GG$-group, $\widehat{G}=\langle X\mid S\rangle$ (note that the set $S$ may be infinite). We treat $S$ as a coefficient-free system of equations over $\GG$. Though, formally, coordinate groups are $\GG$-groups, in this case, we consider instead the group $\GG'_{R'(S)}=\factor{F(X)}{R'(S)}$, where $R'(S)=\bigcap\ker(\varphi)$ and the intersection is taken over all homomorphisms $\varphi$ from $F(X)$ to $\GG$ such that $S\subseteq \ker(\varphi)$. It is clear that $\GG'_{R'(S)}$ is a residually $\GG$ group. Since $\GG$ is equationally Noetherian, there exists a finite subsystem $S_0\subseteq S$ such that $R'(S)=R'(S_0)$. Every homomorphism from $\widehat{G}$ to $\GG$ factors through $\GG'_{R'(S_0)}$, it suffices to describe the homomorphisms from $\GG'_{R'(S_0)}$ to $\GG$. We now run the process for the system of equations $S_0$ and construct the solution tree, where to each vertex $v$ instead of $\GG_{R(\Omega_v^*)}$ we associate $\GG'_{R'(\Omega_v^*)}$ and the corresponding epimorphisms, homomorphisms and automorphisms are defined in a natural way. We thereby obtain a description of the set $\Hom(\GG'_{R'(S_0)},\GG)$.
\end{proof}

\subsection{Example} \label{sec:explsol}
Let $\GG$ be the partially commutative group whose commutation graph is a path of length $3$, $\GG=\langle a,b,c,d\mid [a,b]=1,[b,c]=1,[c,d]=1\rangle$.  Consider the following coordinate group, that, a priori, could be associated to the leaf of the extension tree:
$$
\GG_{R(\Omega_v^*)}=\factor{\GG[h_1,\dots,h_6]}{R\left(\begin{array}{l} \{h_5=a,\, h_6=c\} \cup \\ \{[h_1,h_2]=[h_2,h_3]=[h_3,h_4]=[h_1,h_4]=[h_1,h_5]=[h_4,h_6]=1\}
\end{array}
\right)}.
$$
We show that the group $\GG_{R(\Omega_v^*)}$ is not fully residually $\GG$. We first show that the elements $[h_1,h_3]$ and $[h_2,h_4]$ are non-trivial in $\GG_{R(\Omega_v^*)}$. By the definition of the radical, it suffices to show that there exist homomorphisms $\psi_1$, $\psi_2$ from $\GG_{R(\Omega_v^*)}$ to $\GG$ such that $\psi_1([h_1,h_3])\ne 1$ and $\psi_2([h_2,h_4])\ne 1$. It is easy to check that the map $h_1\to a$, $h_2\to b$, $h_3\to c$ and $h_4\to b$ induces a homomorphism $\psi_1$ such that $\psi_1([h_1,h_3])\ne 1$. Similarly, the map $h_1\to b$, $h_2\to a$, $h_3\to b$ and $h_4\to c$ induces a homomorphism $\psi_2$ such that $\psi_2([h_2,h_4])\ne 1$.

We claim that the family $\{[h_1,h_3], [h_2,h_4]\}$ of elements from $\GG_{R(\Omega_v^*)}$ can not be discriminated into $\GG$, i.e. for every homomorphism $\psi$ from $\GG_{R(\Omega_v^*)}$ to $\GG$ either $\psi([h_1,h_3])=1$ or $\psi([h_2,h_4])=1$. In the proof we make a substantial use of the Centraliser Theorem, see Theorem \ref{thm:centr}.

Indeed, since $[h_1,h_5]=1$ and $h_5=a$, we have that $\psi(h_1)\in C_\GG(a)$. There are three cases to consider: either $\psi(h_1)\in \langle a\rangle$ or $\psi(h_1)\in\langle a,b\rangle$ (and  $\psi(h_1) \notin \langle a\rangle\cup\langle b\rangle$) or $\psi(h_1)\in\langle b \rangle$.

Suppose that $\psi(h_1)\in\langle a\rangle$ (or, $\psi(h_1)\in\langle a,b\rangle$, $\psi(h_1)=w(a,b)$). Since $\psi(h_4)\in C_\GG(c)\cap C_\GG(a)$ (correspondingly, $\psi(h_4)\in C_\GG(c)\cap C_\GG(w)$), it follows that $\psi(h_4)\in \langle b\rangle$. Similarly, we get that  $\psi(h_2)\in C_\GG(a)=\langle a,b\rangle$, and therefore $\psi([h_2,h_4])=1$.

Suppose that $\psi(h_1)\in\langle b\rangle$.  Then either $\psi(h_4)\in\langle b\rangle$ or  $\psi(h_4)\in\langle c\rangle$ or  $\psi(h_4)\in\langle b,c\rangle$. If $\psi(h_4)\in\langle b\rangle$ or $\psi(h_4)\in\langle b,c\rangle$, then $h_3\in C_\GG(b)$ and thus $\psi([h_1,h_3])=1$. Finally, assume that $\psi(h_4)\in\langle c\rangle$. It follows that  either $\psi(h_3)\in\langle c\rangle$ or  $\psi(h_3)\in\langle d\rangle$ or  $\psi(h_4)\in\langle c,d\rangle$. If $\psi(h_3)\in\langle c\rangle$, then $\psi([h_1,h_3])=1$. If  $\psi(h_3)\in\langle c\rangle$ or  $\psi(h_3)\in\langle c,d\rangle$, then since $\psi(h_2)\in C_\GG(\psi(h_1))\cap C_\GG(\psi(h_3))$, it follows that $\psi(h_2)\in \langle c\rangle$, therefore, $\psi([h_2,h_4])=1$.

Therefore, the group $\GG_{R(\Omega_v^*)}$ is not $\GG$-discriminated by $\GG$.

Note that the set of solutions of the generalised equation $\Omega_v$ is empty, while there are solutions of the system $\Omega_v^*$. Indeed, if $H$ is a solution of $\Omega_v$, then $H_1\lra a$, therefore $H_1\in \langle b\rangle$. On the other hand, since $H_4\lra c$, we get that $H_4\in \langle b, d\rangle$. This derives a contradiction since $H_1\not\lra H_4$.

\medskip

Let $\GG$ be the partially commutative group whose commutation graph is a path of length $3$, $\GG=\langle a,b,c,d\mid [a,b]=1,[b,c]=1,[c,d]=1\rangle$.  Consider the following coordinate group:
$$
\GG_{R(\Omega_v^*)}=\factor{\GG[h_1,\dots,h_6]}{R\left(\begin{array}{l} \{h_5=a,\, h_6=b\} \cup \\ \{[h_1,h_2]=[h_2,h_3]=[h_3,h_4]=[h_1,h_4]=[h_1,h_5]=[h_4,h_6]=1\}
\end{array}
\right)}.
$$

We now construct one of the groups $\GG_{w_i}$. The corresponding graphs $\Pi_v$ and $\digamma$ constructed by $\Omega_v$ and $\GG$ are shown on Figure \ref{fig:pivdig}.
\begin{figure}[!h]
  \centering
   \includegraphics[keepaspectratio,width=6in]{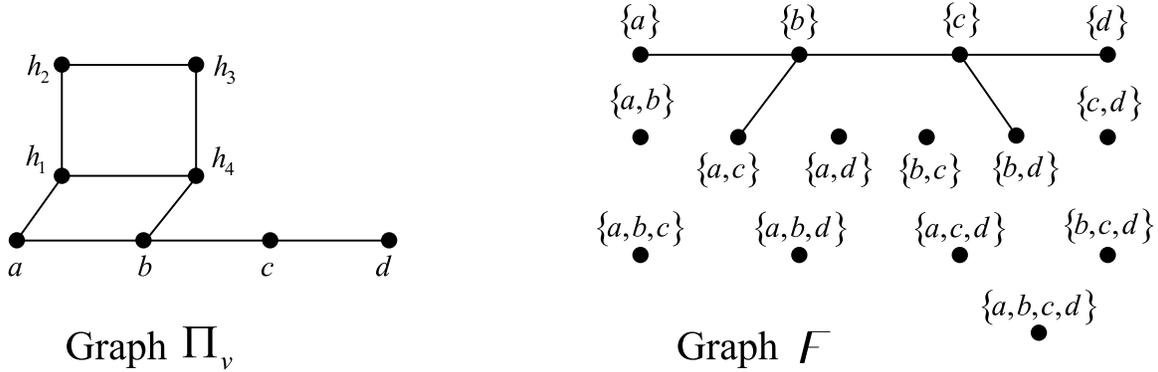}
\caption{The graphs $\Pi_v$ and $\digamma$} \label{fig:pivdig}
\end{figure}

The set of all homomorphisms $\varphi_{v,i}$ from $\Pi_v$ to $\digamma$ that satisfy properties (I) and (II) are listed in the table below:
\begin{center}
\begin{tabular}{l|ccccccccccc}
      & $\varphi_{v,1}$ & $\varphi_{v,2}$ & $\varphi_{v,3}$ & $\varphi_{v,4}$ & $\varphi_{v,5}$ & $\varphi_{v,6}$& $\varphi_{v,7}$& $\varphi_{v,8}$& $\varphi_{v,9}$& $\varphi_{v,10}$& $\varphi_{v,11}$\\  \hline
  $h_1$ & $b$ & $b$       & $b$ & $b$       & $b$       & $b$       & $b$ & $b$         & $b$ & $b$         & $b$\\
  $h_2$ & $a$ & $a$       & $a$ & $\{a,c\}$ & $\{a,c\}$ & $\{a,c\}$ & $c$ & $c$         & $c$ & $c$         & $c$ \\
  $h_3$ & $b$ & $b$       & $b$ & $b$       & $b$       & $b$       & $b$ & $b$         & $b$ & $\{b,d\}$   & $d$\\
  $h_4$ & $a$ & $\{a,c\}$ & $c$ & $a$       & $\{a,c\}$ & $c$       & $a$ & $\{a,c\}$   & $c$ & $c$         & $c$
\end{tabular}
\end{center}

\medskip

Consider the homomorphism $\varphi_{v,6}$ then the corresponding graph $\varphi_{v,6}(\Pi_v)$ is a path of length $2$, whose vertices $s_1$, $s_2$ and $s_3$ are labelled by the subgroups of $\GG$ generated by $\{a,c\}$, $\{b\}$ and  $\{c\}$ correspondingly, see Figure \ref{fig:v2G}. The partially commutative groups $\GG(s_1)=\GG(a,c)$, $\GG(s_2)=\GG(b)$ and $\GG(s_3)=\GG(c)$ are directly indecomposable (the first one is free, and the latter two are cyclic), thus their decompositions of the form (\ref{eq:decomp}) are trivial. The set of vertices $h_k\in V(\Pi_v)$ such that $\varphi_{v,6}(h_k)=s_1$ is $h^{(1)}=\{h_2\}$. Analogously, $h^{(2)}=\{h_1,h_3\}$ and $h^{(2)}=\{h_4\}$. The corresponding group $G$ is a graph product whose underlying graph is a path of length two and the corresponding vertex groups are $\GG(a,c)[h_2]$, $\langle b\rangle\times \langle h_1\rangle\times \langle h_3\rangle$ and $\langle c\rangle \times \langle h_4\rangle$, see Figure \ref{fig:v2G}.

\begin{figure}[!h]
  \centering
   \includegraphics[keepaspectratio,width=6in]{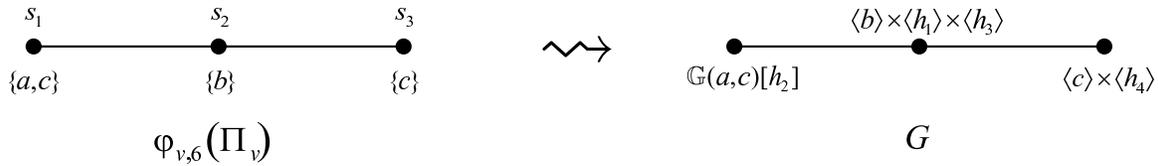}
\caption{Constructing the group $G$ by the graph $\varphi_{v,6}(\Pi_v)$.} \label{fig:v2G}
\end{figure}

We then construct the coordinate group
\begin{gather} \notag
\begin{split}
\GG_{w_6}&=\langle \GG, G\mid\, C,\, [C_\GG(\GG(a,c)),h^{(1)}]=1, [C_\GG(\GG(b)),h^{(2)}]=1, [C_\GG(\GG(c)),h^{(3)}]=1\rangle \\
&=\left< \GG, G\ \left| \begin{array}{l}
                        \mathstrut [b,h_2]=1,  [a,h_1]=[a,h_3]=[b,h_1]=[b,h_3]=[c,h_1]=[c,h_3]=1, \\
                        \mathstrut [b,h_4]=[c,h_4]=[d,h_4]=1, C
                        \end{array}
 \right.\right>\\
&=
\left<
\begin{array}{l}
a,b,c,d, \\
h_1,h_2,h_3,h_4
\end{array}
\left|
\begin{array}{l}
\mathstrut [b,h_2]=1,  [a,h_1]=[a,h_3]=[b,h_1]=[b,h_3]=[c,h_1]=[c,h_3]=1, \\
\mathstrut [b,h_4]=[c,h_4]=[d,h_4]=1, \\
\mathstrut [h_1,h_2]=1, [h_1,h_3]=1, [h_1,h_4]=1, [h_2,h_3]=1, [h_3,h_4]=1.
\end{array}
\right.
\right>
\end{split}
\end{gather}
The coordinate group $\GG_{w_6}$ is the fully residually $\GG$ partially commutative group, whose underlying commutation graph is shown on Figure \ref{fig:lastexpl}.

\begin{figure}[!h]
  \centering
   \includegraphics[keepaspectratio,width=2.4in]{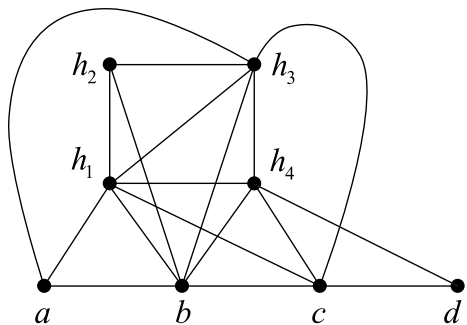}
\caption{The commutation graph of $\GG_{w_6}$.} \label{fig:lastexpl}
\end{figure}

\printindex
\printglossary

\end{document}